\documentclass[leqno]{article} 
\usepackage{amsmath,amsthm}
\usepackage{amssymb,amscd,latexsym}
\usepackage{epic}
\usepackage{epic}
\usepackage{wrapfig}
\usepackage[dvipdfmx]{graphicx}

\theoremstyle{plain}

\newtheorem*{def-theo}{Definition-Theorem}

\newtheorem{thm}{Theorem}[section]
\newtheorem{cor}[thm]{Corollary}
 
\newtheorem{prop}[thm]{Proposition}

 \newtheorem{rem}[thm]{Remark}

\theoremstyle{definition}

\newtheorem{defn}[thm]{Definition}

\newtheorem{theorem}[thm]{Theorem}

\numberwithin{equation}{section}

\def\al{{\alpha}}

\def \0{{\bf 0}}
\def \1{{\bf 1}}

\usepackage[utf8]{inputenc}

\begin{document}

\title{ 
Second homotopy classes\\
associated with non-cancellative monoids \\
}

\author{
\\
 Kyoji Saito
 \medskip
 \\
{\small To my family:  Reiko, Yuko, Ayako and Shingo.}
}
\date{}

\maketitle

 \abstract{

 We develop a method for constructing certain second homotopy classes in a topological space from NC-twins\footnote{ 
A quadruple $(a,b,c,d)$ of elements of a monoid $A^+$ is called {\it non-cancellative tuple (or, NC-tuple) over the kernel} $(b,c)$ if it satisfies the relation $abd\! \sim\! acd$ but $b\! \not \sim\! c$. A pair of NC-tuples over the same kernel is called an {\it NC-twin}    (see Definition \ref{NCtuple} and Footnote 21).
}\!
 in the monoid associated with a semi-positive presentation\footnote
{We call a relation of a discrete group generated by a set $L$ {\it semi-positive}, if it is given by an equality $p\!=\!q$ between two semi-positive words $p$ and $q$ in $L$. 
See Definition 2.1.
}\! 
 of its fundamental group. The construction is achieved by lifting the algebraic defining relations to geometric relations. To each NC-twin we associate a set of second homotopy classes, called the $\Pi$-class,  which forms a single coset of a subgroup of $\pi_2$, called the inertia group. We apply the construction to Yoshinaga's positive presentations for line-arrangement complements and show that, for the generic three-line arrangement, the resulting classes recover Hattori’s second homotopy class. 
 We also associate with each NC-twin of an abstract monoid a canonical second homotopy class in its classifying space, and formulate its comparison problems with the \(\Pi\)-class and the inertia group.

\tableofcontents

 \noindent

\section{Introduction}


A presentation of a fundamental group may contain more information about the underlying space than the fundamental group alone. The present paper develops a method for extracting certain second homotopy classes from such information, and confirms the results by examples.

\smallskip
It is well known that the fundamental group of the regular orbit space associated with an irreducible finite reflection group admits a positive homogeneous presentation (Brieskorn \cite{Br}). The defining relations are the (generalized) Artin braid relations, and the resulting group is the Artin group of finite type \cite{B-S}. The monoid defined by the same presentation is cancellative and embeds into the Artin group as a lattice \cite{B-S} (cf.\ \cite{P}). This structure played an important role in Deligne’s proof \cite{D} (cf.\ \cite{P-Sa}) that the regular orbit spaces are Eilenberg-Mac Lane spaces.

Motivated by these examples, we studied the regular orbit spaces associated with elliptic Weyl group actions on period domains \cite{S3}. Their fundamental groups again admit positive homogeneous presentations, whose defining relations are called elliptic Artin braid relations \cite{S-S}. In contrast to the finite-type case, however, the associated monoids fail to embed in the corresponding elliptic Artin groups. In fact, these monoids admit NC-twins.(see \cite{S5} and Footnote 1). 
This suggests that the elliptic regular orbit space may have non-trivial higher homotopy groups.

 In an attempt to answer  this question,  we  develop a method to construct  certain second homotopy classes, called $\Pi$-classes (Theorem \ref{FinalTheo} and \eqref{eq:FinalPi2}). 
 The basic idea is to regard a failure of cancellation in the presentation monoid as a possible source of two-dimensional homotopy information. 
The construction of the $\Pi$-class  is achieved by liftings of the algebraic relations to geometric relations.     However, we do not yet have a criterion for when the $\Pi$-class does not vanish. Thus, the original question for the elliptic Weyl group orbit spaces remains to be a conjecture. 

Instead, we apply the construction of $\Pi$-classes to Yoshinaga's positive presentations of the fundamental groups of the  complements of  real line arrangements \cite{Y2}, where the presence or absence  of non-trivial second homotopy classes is already known. We study the  two typical cases.

\noindent
 (a)  Three lines in general position in $\mathbb{R}^2$:  This is the case in which Hattori \cite{H} showed that there exists a non-trivial $\pi_2$-class in the complement.\footnote
 {The image of Hattori's homotopy class under the Hurewicz homomorphism vanishes \cite{R}.}

 \noindent
 (b)  Hyperplane arrangement of type $\rm{A}_3$:  This is the case in which we know  that the complement is an Eilenberg-Mac Lane space.
 
  \vspace{0.05cm}
  We examine experimentally  the low-degree parts of each monoid.
  
In case (a), we find three distinguished NC-twins \(\tau_{ij}\) of non-cancellative tuples in the monoid (see §4 for additional twins). Then, we study the $\Pi$-classes of the NC-twins \(\tau_{ij}\) and observe that all three give rise to Hattori’s second homotopy class, as expected.

In case (b), we find 11+6+16  NC-tuples  up to degree 9.  However, {\it none of them forms an NC-twin}, as expected.

\vspace{0.06cm}
\noindent
{\it Note} 1.  The lists of NC-tuples for (a) and (b) have been extended (see Appendix III), providing further support for the above observations.\footnote
{Using custom computer programs, Zitian Xia extended the enumeration of reduced and minimal NC-tuples up to degree 6 for case (a) and degree 17 for case (b) (see Appendix III).
}
 

\smallskip
Let us give a more precise overview of the present paper in {\bf (1)$\sim$(8)}.\!\!

\smallskip
\noindent
 {\bf (1)  Semi-positively presented spaces.}

As the starting point for the general framework, we consider a topological space $X$ equipped with a semi-positive presentation $(L,\mathcal{R})$ (Definition \ref{Semi-positive}) of its fundamental group. In \S2.1, we prepare and fix some notation and terminology on semi-positive presentations, where the associated group and monoid are denoted by $A(L,\mathcal{R})$ $:=\! \langle L\! \mid \! \mathcal{R}\rangle$ and $A^+(L,\mathcal{R}_m)\!:=\! \langle L \! \mid \! \mathcal{R}_m\rangle^+$, respectively.
Examples of semi-positively presented spaces are given in \S2.2. In particular, we study the following two cases of positive homogeneous presentations given by Yoshinaga  \cite{Y2} for the fundamental groups of the complements $M\!(\mathcal{A})$ of line arrangements $\mathcal{A}$ defined over $\mathbb{R}$.  

\smallskip 
\noindent
 (a) Generic 3-line arrangement $\mathcal{A}_3$. This is the case where there exists a non-trivial class in $\pi_2(M(\mathcal{A}_3),*)$, described by Hattori \cite{H,H-K}. 
   
\smallskip
 \noindent
 (b) Type ${\rm A}_3$ arrangement $\mathcal{A}_{{\rm A}_3}$. This is the system of reflection hyperplanes of the Weyl group of type ${\rm A}_3$ and the complement $M(\mathcal{A}_{{\rm A}_3})$ is well-known to be an Eilenberg-Mac Lane space (no higher homotopy classes). 

%

\smallskip
\noindent
 {\bf (2)  A geometric realization of the monoid  $A^+(L,\mathcal{R}_m)$.}  

In \S3, we realize the monoid $A^+(L,\mathcal{R}_m)$ as certain geometric equivalence classes of loops in $X$. For the purpose, we lift the presentation  $(L,\mathcal{R})$ to a geometric presentation  $(W, \widetilde{L}, \mathcal{R}_g)$ (Definitions \ref{W}  and \ref{GeoFundRel}), where $W$ is a closed connected subset of $X$ homotopy equivalent to a one-dimensional  CW-complex and $\widetilde{L}$ is a generator system of $\pi_1(W,*)$, identified with the set $L$. 
Then, for a fundamental relation $p\!=\!q$ in $\mathcal{R}$ ($p,q\!\in\! L^*$), we choose a subset $\mathcal{R}_g(p,q)$ of  $[p,q]$, where $[p,q]$ \eqref{eq:Equivalence} is the set of  homotopy classes of  homotopies from  $p$ to $q$ in $X$.\footnote{The set $[x,y]$ is embedded into the relative homotopy group $\pi_2(X,W,*)$ by the relative realization morphism $\chi$ \eqref{eq:immersion}. 
However, we do not take that point of view, since we need to define the product  ``$\cdot$" 
\eqref{eq:cdot} and to develop in \S5.1 its division theory which is crucial in constructing the $\Pi$-classes.\!\!\!} 
  By concatenating these lifted geometric relations, we obtain a subsystem $[x,y]_{\mathcal{R}_g}\subset [x,y]$ \eqref{eq:R^gSet}  for $x,y\in \widetilde{L}^*\simeq L^*$. Thus,  two loops   $x$ and $y \in \widetilde{L}^*$ give rise to the same element in  $A^+(L,\mathcal{R}_m)$ (denoted by $x\sim y$) if and only if $[x,y]_{\mathcal{R}_g}\not=\emptyset$.

\smallskip
\noindent
{\bf (3)  Twin of non-cancellative tuples}

In \S4.1, for any  monoid $A^+$,  we introduce the concepts of a  non-cancellative tuple (NC-tuple) and a twin of non-cancellative tuples (NC-twin) (Footnote 1), together with the notions of reducedness, minimality  and, experimentally,  indecomposability (see Remark 4.2).\footnote
{There are meaningless NC-tuples constructed trivially (see, for instance, \S2.2, {\bf 7}). 

Thus we impose these conditions. See {\bf Definition 4.1}.
}
In \S4.2, we analyze NC-tuples and twins of the monoids for the line arrangements (a) and (b), and observe a clear difference between them.

\smallskip 
\noindent
 (a)   Generic 3-line arrangement $\mathcal{A}_3$.  We find 7 reduced and minimal NC-tuples for the monoid $A^+(L, \mathcal{R}_m(\mathcal{A}_3))$ of length $\le5$ (see Appendix III for more data).  Among them, we distinguish  three  left-right NC-twins, called $\tau_{12}, \tau_{23}$ and $\tau_{31}$, which are reduced and minimal (see Footnote 23). 
  
\smallskip
 \noindent
 (b) Type $\mathrm{A_3}$ arrangement $\mathcal{A}_{\rm{A}_3}$. 
We find 33\! reduced and minimal  NC-tuples of length $\le 9$ for $A^+(L,\mathcal{R}_m(\mathcal{A}_{\rm{A}_3}))$. None of them forms an NC-twin  (see Appendix III for more data). 

\smallskip
\noindent
{\bf Conjecture for $\mathcal{A}_{\mathrm{A_3}}$.}  
For every minimal kernel \((b,c)\), there exists at most one reduced NC-tuple over \((b,c)\).

\smallskip
\noindent
 {\bf (4)  The class $\Pi(\tau)$ associated with an  NC-twin $\tau$.}\!\!\!


In \S5, starting from  an {\it NC-twin $\tau$}  of a monoid $A^+(L,\mathcal{R}_m)$, we construct a set $\Pi(\tau)$ of second homotopy classes of $X$, called the $\Pi$-class associated with $\tau$, which forms a coset of the associated inertia group.\!\!\!  



Let an NC-tuple $(a,b,c,d)\in A^+(L,\mathcal{R}_m)^4$ of the monoid be given. Let $(\tilde{a},\tilde{b},\tilde{c},\tilde{d})\! \in\! (\widetilde{L}^*)^4$ be a geometric lifting of it. By the relation \eqref{eq:NCRelation}, we have  $[\tilde{a}\tilde{b}\tilde{d},\tilde{a}\tilde{c}\tilde{d}]_{\mathcal{R}_g}\!\not=\! \emptyset$. Apply the division morphism in  Proposition \ref{QuotientComposition},  and obtain the non-empty division set 
 $\tilde{a}^{-1}[\tilde{a}\tilde{b}\tilde{d},\tilde{a}\tilde{c}\tilde{d}]_{\mathcal{R}_g}\tilde{d}^{-1} \! \subset\! [\tilde{b},\tilde{c}]$.  
%
%
If we apply the relative realization morphism $\chi$ \eqref{eq:immersion}  on the division set, it is a subset of  the relative homotopy group $\pi_2(X,W,*)$ (=the inverse image of $bc^{-1}\in \pi_1(W,*)$,  cf.\ Footnote 5). 
%
%
Thus, for a twin  $\tau:=\big((a_1,b,c,d_1),(a_2,b,c,d_2)\big)$ of NC-tuples over the same kernel $(b,c)$,  we obtain two division subsets
in  $[\tilde b,\tilde c]$ whose differences belong to the second homotopy group of $X$.  
%

\medskip
\noindent
{\bf Theorem.}
{\it  Let $\tau$ be an NC-twin of the monoid $A^+(L,\mathcal{R}_m)$.  
After fixing the geometric fundamental relations, we have}

\smallskip
\noindent
{\bf 1.} {\it The set \(\Pi(\tau)\) \eqref{eq:FinalPi2}  of  differences of two division sets depends only on \(\tau\) and not on the choices of representatives used to lift the NC-twin.
}

\smallskip
\noindent
{\bf 2.} {\it There is a subgroup $G(\tau)$ \eqref{eq:FinalInertia} of $\pi_2(X,*)$ depending only on $\tau$, called the inertia group of $\tau$, such that $\Pi(\tau)$ is a single coset of $G(\tau)$ in  $\pi_2(X,*)$.}

 
%

\smallskip
\noindent
 {\bf (5) Hattori's $\pi_2$-classes  and the $\Pi$-classes of line arrangements.\!\!}

In \S6, we compare the $\Pi$-classes for the twins $\tau_{12}, \tau_{23}$ and $\tau_{31}$ given in \S4.2 Example (a) with Hattori's $\pi_2$-class for the line arrangement $\mathcal{A}_3$.  

 \medskip
 \noindent
 {\bf Theorem.}  {\it Consider a geometric lifting of Yoshinaga's fundamental relations for $M(\mathcal{A}_3)$. 
 Then, due to the balancedness of twins (see Corollary 4.5), the $\Pi$-classes $\Pi(\tau_{ij})$ and the inertia groups $G(\tau_{ij})$   
are independent of the choice of the geometric  liftings of the fundamental relations.  
 
\noindent
 {\rm (1)} The inertia groups $G(\tau_{ij})$ 
 are trivial.
 
 \noindent
 {\rm (2)}  The $\Pi$-classes  $\Pi(\tau_{ij})$ 
 give rise  to Hattori's class in $\pi_2(M(\mathcal{A}_3),*)$.} 

 \smallskip
\noindent
{\it Note} 2. The inertia group is a measure of how sharp the \(\Pi\)-class is. Thus, in the present case, the triviality of the inertia groups means that the \(\Pi\)-classes \(\Pi(\tau_{ij})\) are as sharp as possible.

 
\smallskip
\noindent
 {\bf (6)  Second homotopy classes in the classifying space of a monoid.}

In Appendix I, an NC-twin of an abstract monoid \(A^+\) gives rise to a canonical second homotopy class in the classifying space \(BA^+\) of $A^+$. When the monoid \(A^+=A^+(L,R_m)\) arises from a semi-positive presentation of a space \(X\), we formulate comparison problems between this canonical class and the geometrically constructed \(\Pi\)-class, in both directions \(BA^+\!\to\! X\) and \(X \! \to \! BA^+\), and relate these problems to the inertia group.\!\!\!\!




\smallskip
\noindent
 {\bf (7) Linear independence  of Yoshinaga's fundamental relations.} 
 
 In Appendix II,  we prove a certain linear independence property of Yoshinaga's fundamental relations,
 which gives a general proof of the balancedness of NC-twins studied in \S4 (in \S4, this was verified by hand in the examples), implying the $\Pi$-classes in \S6 are independent of the choice of the geometric fundamental relations (cf.\ Corollary 4.5 and part (1) of the proof of Theorem \ref{ThmY-H}).\!\!

  \smallskip
 \noindent
{\bf (8)  Problems.} Several basic questions remain. 

 (1) What are the $\Pi$-classes (is there a characterization of them)?  
 What are the images of the $\Pi$-classes under the Hurewicz homomorphism? 
 When is $\Pi(\tau)$ non-vanishing  in $\pi_2(X,*)/G(\tau)$? Can one find a criterion?


(2) Which semi-positive presentations \((L,R)\) of a space \(X\), which geometric liftings \((W,\widetilde L,R_g)\), and which NC-twins of the associated monoids (cf.\ Footnote 21) give meaningful answers to the questions in (1)?

As we asked in (2), 
question (1) depends on the choice of the presentation $(L,\mathcal{R})$ and its lifting $(W,\widetilde{L},\mathcal{R}_g)$. 
Actually, we may modify $\mathcal{R}$ to $\mathcal{R}'$ so that we  either ``artificially create"  NC-tuples (\S2.2,\ 7), or  kill  ``artificial"  NC-tuples (Remark\! \ref{Reduction}).\ Thus, we need to distinguish good presentations from artificial presentations.
We do not know a good criterion for that. However, we see in \S4 and \S6 that the examples of Yoshinaga's presentation for line arrangement complements are good in two senses: 

 {\footnotesize (1)} {\it The inertia groups  are trivial and the $\Pi$-classes are singletons.} 

{\footnotesize (2)} {\it The NC-twins are balanced (Corollary \ref{balanced})\footnote
{If there are relations among relations, the treatment should be modified.}
 so that the $\Pi$-classes do not depend on the lifting of the fundamental relations.}\!\!\! 

\noindent

\smallskip
\noindent
We conjecture that the following three classes of presentations satisfy conditions (1) and (2) above and admit geometric liftings for which the NC-twins of the associated monoids describe the second homotopy groups.


1.  Elliptic Artin braid relations $\mathcal{R}$ associated with an elliptic root system (\S2.2 Ex.\ 2 (iv); \cite{S3,S-S,S5}). 

2. Yoshinaga's fundamental relations of the fundamental group of $M(\mathcal{A})$ for any line arrangement $\mathcal{A}$ (\S2.2 Ex.\ 6, \cite{Y2}).  

3. Certain Zariski-van Kampen relations of the fundamental group of complements of free divisors of Sekiguchi polynomials (\S2.2 Ex.\ 3, \cite{I,I-S}).

\section{\!\!\! Semi-positive presentation of groups}
In \S2.1, we fix concepts and terminology on semi-positive presentations of  groups and monoids. Then in \S2.2, we give  examples of topological spaces whose fundamental groups are semi-positively presented.


\subsection{\normalsize Semi-positive presentation  $(L,\mathcal{R})$ of groups}

\begin{defn}
\label{Semi-positive}  (1) A semi-positive presentation of a group is a pair 
\begin{equation}
\vspace{-0.1cm}
\label{eq:SemiPositive}
( L, \mathcal{R})
\end{equation}
where  
 $L=\{a_\lambda\}_{\lambda\in \Lambda}$ is a set of letters 
 and $\mathcal{R}=\{p_\mu=q_\mu\}_{\mu\in M}$ is a set of equalities (fundamental relations) 
 such that  $p_\mu, q_\mu \in L^*:=\cup_{n\in \mathbb{Z}_{\ge0}} L^n$ ($\mu\in M$) are finite words over the letters $L$, including the empty word.
 
\noindent
 (2) If  $p_\mu, q_\mu $ are not the empty word, i.e.\ they have positive length, we call  $(L, \mathcal{R})$ a positive presentation.
 
\noindent
 (3)  If $\deg_w(p_\mu)\! =\! \deg_w(q_\mu)$ ($\forall \mu\in M$), we call  $( L, \mathcal{R})$ a weighted homogeneous presentation, where $\deg_w(a)$ for $a\in L^*$ is the weighted degree defined using the weight system $w\!:\!L\!\to\! \mathbb{Z}_{\ge0}$ (c.f.\ \S2.1 {\it Note} 3.).
\end{defn}

For a given semi-positive presentation $(L,\mathcal{R})$, 
we denote by 
\begin{equation}
\vspace{-0.1cm}
\label{eq:Group}
A(L,\mathcal{R})  :=  \langle L\mid \mathcal{R}\rangle, 
\end{equation} 
the associated group. Next we introduce the monoid associated with the same defining relations. Since we need to distinguish the equality in the group and in the monoid,\footnote{
 This distinction is necessary, since two non-equivalent elements in the monoid may give rise to  equal elements in the group. Thus, {\it in the fundamental relations $\mathcal{R}_{m}$ for the monoid, we use the notation  ``$\sim$" instead of ``$=$"}. 
In \S3, we shall introduce another equivalence relation: the {\it geometric} fundamental relations $\mathcal{R}_{g}$, where we shall use the notation ``$\sim_{\mathcal{R}_{g}}$".   
}
we shall use ``$=$'' for the equivalence in the group and ``$\sim$'' for the equivalence in the monoid. Accordingly, let 
\begin{equation}
\vspace{-0.1cm}
\label{eq:MonoidRelation}
\mathcal{R}_m:= \{p_\mu\!\sim\! q_\mu\}_{\mu\in M}
\end{equation}
be a ``copy'' of $\mathcal{R}$ in such manner that each relation $p_\mu=q_\mu$ in $\mathcal{R}$ is replaced by the relation  $p_\mu\! \sim\! q_\mu$.
Then we define 
\begin{equation}
\vspace{-0.1cm}
\label{eq:Monoid}
A^+(L,\mathcal{R}_m)  :=  \langle L\mid \mathcal{R}_m\rangle^+ 
\end{equation} 
to be  the quotient of the free monoid by the congruence generated by  $\mathcal{R}_m$.  For later use, in the following (1) and (2), we describe this congruence explicitly in terms of elementary transformations.
 
\smallskip
\noindent
(1)
The notation $\langle L \mid \mathcal{R}_{m} \rangle^+$ in the RHS of \eqref{eq:Monoid} means the quotient set $L^*/\! \sim$ of the free monoid $L^*:=\sqcup_{l=0}^\infty L^l$ generated by $L$ and divided by the equivalence relation ``$ \sim  $" generated by $\mathcal{R}_m$ in the following way.\footnote
{In \cite{B-S}, the relation ``$\sim$" is called positive equivalence and denoted by 
``$\underset{{}\cdot{}}{=}$", cf.\ \cite{D} (1.10). Actually, this congruence relation gives the quotient object in the category of monoids.
} 



%
%
%
Two words $U,V \in L^*$ are equivalent and denoted by $U\sim  V$, if 
there exists a finite sequence $W_0,W_1,\cdots,W_n$ ($n\ge0$) of words of $L^*$ such that $W_0=U$,  $W_n=V$,  there are words $A_i,B_i,C_i,D_i\in L^*$  
  for $i=1,\cdots,n$ such that $W_{i-1}=A_iB_iD_i$, $W_i=A_iC_iD_i$,  and  either ``$B_i\sim C_i$" or ``$C_i\sim B_i$" belongs to $\mathcal{R}_{m}$ (case $n=0$ means $U=W_0=W_n=V$).  

  We shall call each step $W_{i-1}\sim W_{i}$ an {\it elementary transformation} or {\it e-transformation} for short. 
  

\smallskip
\noindent
(2) The product structure on $L^*/\!\sim$ is induced from concatenations of words in $L^*$ (since, if $U_1\!\sim\! U_2$ and $V_1\!\sim\! V_2$ in $L^*$ then $U_1V_1\!\sim\! U_2V_2$ in $L^*$), where the empty word $\phi\in L^0$ plays the role of the unit element and is denoted by $1$. 
We say an element $a\in A^+(L,\mathcal{R}_m)$ is left (resp.\ right) divisible by $b\in A^+(L,\mathcal{R}_m)$ (or, $b$ divides $a$ from the left (resp.\ right)) and denote $b\mid^l \! a$ (resp.\ $b\mid^r \! a$) if there are representative words $U$ and $V$ of $a$ and $b$ and a word $W$ such that $U=VW$ (resp.\ $U=WV$).

\medskip
An alternative description of the monoid  $A^+(L,\mathcal{R}_m)$ in terms of graphs is given as follows. Consider the labeled and oriented graph:
\vspace{-0.09cm}
\begin{equation}
\label{eq:MonoidGraph3}
G(L^*,\mathcal{R}_m)
\vspace{-0.11cm}
\end{equation}
 where (i) the set of vertices  is $L^*:=\sqcup_{l=0}^\infty L^l
$, and (ii) there exists an edge between $U,V \in L^*$ if and only if there exists an e-transformation $U\sim V$.  The edge is labeled and oriented by the oriented  fundamental relation $p_\mu\sim q_\mu$ or  $q_\mu\sim p_\mu$.  Then we have a natural isomorphism of monoids 
\vspace{-0.07cm}
\begin{equation}
\label{eq:MonoidGraph}
A ^+(L,\mathcal{R}_m)  \quad \simeq \quad   \pi_0(G(L^*,\mathcal{R}_m),1)
\vspace{-0.07cm}
\end{equation}
where the product structure in the RHS is defined by the one naturally induced from that on $L^*$ (i.e.\ the concatenation of words). 

Under some weak assumptions\footnote
{ Assume: (1) there do not exist two relations of the form $a_ipb_i\sim a_iqb_i$ ($i=1,2$) for some $a_i,b_i, p,q\in L^*$, and (2) there do not exist two relations $p_i \!\simeq \! 1$ ($i\!=\!1,2$) such that $p_2$ is a cyclic rotation of letters in $p_1$. Since the relations need not be reduced, sometimes, these assumptions  are reasonable (e.g.\ \S2.2 {\bf 6.}).\!\!
}
 on  the presentation $(L,\mathcal{R})$, there exists at most one edge between two vertices and the possible e-transformation between them is uniquely determined. Then, there is one-to-one correspondence  between the set of sequences $W_0, \cdots,W_n$ between $U,V\in L^*$ performing the equivalence in the above description (2) of the monoid  and the set of paths in the graph $G(L^*,\mathcal{R}_m)$ between $U,V\in L^*$.

\smallskip 
Note that if the presentation is positive, then any non-trivial element of $A ^+(L,\mathcal{R}_m)$ has neither a left nor a right inverse, since a product of non-trivial elements cannot be  equivalent to the unit elements.







\medskip
The natural identity $L\overset{\sim\  }{\to} L$ of the generators induces a monoid homomorphism, called the {\it localization morphism}, 
\begin{equation}
\label{eq:Plocalization}
A^+(L,\mathcal{R}_m) \ \rightarrow  \  A(L,\mathcal{R}).
\end{equation}
The localization in general is neither injective  nor surjective. In particular, {\it the equivalence  relation in $A^+(L,\mathcal{R}_m)$ may be weaker than that} in $ A(L,\mathcal{R})$.  One subject of the present paper is to understand such differences from a point of view of higher homotopy classes (see \S3.2.1).

\medskip
\noindent
{\it Note} 3. 
If all the relations in $\mathcal{R}_{m}$ are (weighted) homogeneous w.r.t. a positive weight $w:L\to \mathbb{Z}_{\ge 0}$,  then  we obtain a degree homomorphism:
\begin{equation}
\label{eq:Degree}
\mathrm{deg}_w  : A^+(L,\mathcal{R}_m) \rightarrow  \mathbb{Z}_{\ge0},   \quad W \mapsto \ell(W)=\text{weighted length of $W$}.\!\!\!
\end{equation}
If the monoid $A^+(L,\mathcal{R}_m)$ is cancellative, the existence of a positive integral valued degree homomorphism \eqref{eq:Degree} 
allows us to define the skew-growth function
$N_{A^+(L,\mathcal{R}_m),\mathrm{deg}}(t)$ (see \cite{S4,I-S}), which is important for our future applications.
 However, we shall not use the degree morphism 
  $\mathrm{deg}$ except for the study  in the examples in \S4.
 
 \smallskip
 \noindent
 {\it Note} 4.  Let $\langle L, \mathcal{R}\rangle$ be a presentation of a group in the classical sense that $\mathcal{R}=\{ r_\mu=1\}_{\mu\in M}$ where $r_\mu\in (L\cup L^{-1})^*$ for $\mu \in M$. Then, we may reformulate it as a semi-positive presentation $(L\cup L', \mathcal{R}')$ for the same group, where $L'=\{a_\lambda'\}_{\lambda\in \Lambda}$ is a copy of $L$ and  $\mathcal{R}'=\{r'_{mu}\}_{\mu\in M}\}\cup \{a_\lambda a_\lambda'=1\}_{\lambda\in \Lambda}$ such that $r_\mu'$ is a word in the letters $L\cup L'$ obtained from $r_\mu$ by replacing any letter $a_\lambda^{-1}$ by the letter $a_\lambda'$. 
 However, this presentation  does not provide any new information.  
 One would like to find a better formulation of a semi-positive presentation to avoid such  ``absurd'' cases (See Final Remarks in \S5).

\subsection{\normalsize Semi-positively presented topological spaces} 

In this subsection, we list examples {\bf 1.-6.}\ of topological spaces, whose fundamental groups have semi-positive presentations to illustrate their variety. All examples except  the first admit positive weighted homogeneous presentations. All examples satisfy the assumptions (1) and (2) given in the Footnote 10.  
 
 A reader in a hurry may skip {\bf 1.-5.}\  and proceed directly  to Example {\bf 6.}, since it is the only case studied in the sequel of the present paper.\!\!\!

\medskip
\noindent
{\bf 1.  Example of semi-positive but not positive relation: }

{\it Choose $m\!\in\!\mathbb{Z}_{\ge1}$.\! Let  $X$\! be a 2-dimensional CW-complex, where 0-cells $\!=\!\{*\}$,\! 1-cells\! $\!=\!\!\{a,b\}$,\! 2-cells $\!=\!\{c\}$\! s.t.\  $\partial c\!=\! (ab)^ma$.\! Then, $W:=a\cup_* b$, $L\!:=\!\{a,b\}$ and $\mathcal{R}\!:=\!\{(ab)^ma\!=\!1\}$ define a semi-positive presentation.\! Set\!\!
$$
A(L,\mathcal{R})\!=\!\langle a,b \mid (ab)^ma\!=\!1 \rangle \quad \text{ and }\quad A^+(L,\mathcal{R}_m)\!=\!\langle a,b \mid (ab)^ma\!\sim\!1 \rangle^+.
$$
Then, we have natural isomorphisms: $A^+(L,\mathcal{R}_m)\!=\!A(L,\mathcal{R})\!\simeq\!\mathbb{Z}, a\! \mapsto \! m, b\! \mapsto \! -\!m\!-\!1$, and  a homotopy equivalence: 
 $X\simeq S^1$.}
\begin{proof} It is sufficient to show that $A^+(L,\mathcal{R}_m)\simeq\mathbb{Z}$. Then,   $A(L,\mathcal{R})\simeq\mathbb{Z}$.

We first show that the relation $ab\! \sim \! ba$ holds due to the  e-transformations:
$ab \! \sim \! ab(ab)^{m}a  = (ab)^maba\sim \! ba$.  So, in any positive word in $A^+(L,\mathcal{R}_m)$, we can move the letters $a,b$ freely within the word. So, any word is equivalent to $a^pb^q$ for some $p,q\in\mathbb{Z}_{\ge0}$. 
If $p\ge m+1$ and $q\ge m$, we can factor out the word $(ab)^ma$. Factoring out the relation as far as possible,  the remaining equivalence classes are represented by the following  list.
$$
a^pb^q\  \text{for } p\in \mathbb{Z}_{\ge0}, \ 0\le q\le m\!-\!1 \ \text{ and }  \ a^pb^q \ \text{ for }  0\le p\le m, \ q\in\mathbb{Z}_{\ge m} .
$$
We define a monoid morphism $A^+(L,\mathcal{R}_m)\to \mathbb{Z}$ by assigning $a\mapsto m$ and $b\mapsto -m-1$. Then we see that the above list is in bijection with $\mathbb{Z}$.
\end{proof}

\noindent
{\bf 2. Regular orbit spaces for reflection groups:}

The regular orbit spaces of the following examples of reflection groups admit positive and homogeneous presentations by simple generators S with Artin braid relations $R$: 

 i) Complexifications of finite real reflection groups  \cite{F-N},\ \cite{Br},\ \cite{D},\ \cite{S2}
 
 ii) Complexifications of affine Weyl groups \cite{Ba}, \cite{Du}, 
 
 iii) Unitary reflection groups \cite{Be2}, 
 
 iv) Elliptic Weyl groups \cite{S3,S-S}, 
 

\medskip
\noindent
{\bf 3.  Logarithmically free weighted homogeneous polynomials:}

In the above examples 2. i), ii) and iv), the regular orbit space can also be considered as the complement of associated discriminant loci in the smooth orbit space (homeomorphic to  a complex affine space), where the discriminant is a weighted homogeneous logarithmically free divisors in the sense \cite{S1}. Let us give more examples of weighted homogeneous logarithmically free polynomials (see Question at the end of \S6 of \cite{S-I}).

  In \cite{Sek}, J. Sekiguchi has listed 17 weighted homogeneous logarithmically free polynomials of three variables. The fundamental groups of the complements of the zero loci in $\mathbb{C}^3$ of the polynomials were calculated by the Zariski-van Kampen method in \cite{S-I} \S4. It turns out that {\it all the 17 cases admit positive homogeneous presentations according to some generators in Zariski-pencils in $z$-variable}. Among them, 5 are Artin groups and 8 are free abelian groups. However, for the remaining 4 types 
\vspace{-0.1cm}
$$
\begin{array}{rcl}
\Delta_{\rm B_{ii}} & := & z(-2y^3+4x^3z+18xyz+27z^2), \\
\Delta_{\rm B_{iv}}&:= & z(9x^2y^2-4y^3+18xyz+9z^2), \\
\Delta_{\rm H_{ii}} &:= & 100x^3y^4+y^5+40x^4y^2z-10xy^3z+4x^5z^2-15x^2yz^2+z^3, \\
\Delta_{\rm H_{iii}} &:= & 8x^3y^4+108y^5-36xy^3z-x^2yz^2+4z^3, 
\end{array}
$$
\vspace{-0.03cm}
the associated monoids are not Gaussian so that they cannot be reduced to Artin groups. Also from a point of view of exponents, they cannot be reduced to elliptic Artin groups. In particular, it is shown that the type ${\rm B_{ii}}$ case admits essentially two different positive and homogeneous presentations. 

\begin{theorem} \label{Ishibe}
(Ishibe \cite{I} 3.1)  {\it For any choice of Zariski-van Kampen generator system
$\{a,b,c\}$ (up to a permutation), the fundamental group of type ${\rm B_{ii}}$ admits only
one of the following two presentations} I {\it and} II
\vspace{-0.1cm}
$$
\begin{array}{rrcl}
 \mathrm{I} \ : & \langle\ a,b,c & \mid & cbb=bba,bc=ab, ac=ca \  \rangle \\
\mathrm{II}: & \langle\ a,b,c & \mid & ababab=bababa, b=c, aabab=baaba \  \rangle
\end{array}
$$
\end{theorem}
\noindent
{\it Note} 5.  In \cite{I2}, T. Ishibe informed the author that {\it a finite cyclic covering space of the complement of the free divisor of type ${\rm B_{ii}}$ is isomorphic to $\mathbb{C}^\times \!\!\times\!(\mathbb{C}^2\!\setminus\!\{\text{three generic lines}\})$}.  According to Hattori \cite{H}, this space carries non-trivial second homotopy group. 

\medskip
\noindent
{\bf 4.  Dual presentation of finite type Artin groups:}

 Artin groups of finite type which appeared already in the above examples {\bf 2.} i) and {\bf 3.}  have another homogeneous presentation using larger generating set $T=\cup_{m\in Z}\ (c_P)^mS$  (where $c_P$ is a Coxeter element and $S$ is a simple generator system of the Artin group) 
by David Bessis (\cite{Be1}, cf.\ \cite{B-K-L} for type ${\rm A_l}$ and \cite{Bra} for 3 generators case, cf.\ also \cite{I-S}).

\medskip
\begin{theorem}\!\!(Bessis \cite{Be1}) {\it The generator system $T$ satisfies quadratic fundamental relations\!\!
$$
 r\cdot s = s^*\cdot r
 $$
for non-crossing pairs $r, s \in T$ and $s^*\in T$, chosen to be the element satisfying the  relation  $\overline{r}\cdot\overline{s}=\overline{s^*}\cdot\overline{r}$ in the corresponding Coxeter group.  }
\end{theorem}
The relations are called the {\it dual Artin braid relations}.


\medskip
\noindent
{\bf 5. Regular orbit spaces of extended Coxeter groups:}


Looijenga \cite{Lo} introduced an action of an extended Coxeter group $\widetilde{W}$ on a tube domain  $\Omega$, and  studied the invariant theory for the orbit space $\widetilde{W}\backslash \Omega$. Then van der Lek \cite{Lek} described a {\it weighted homogeneous} presentation of the fundamental group of its regular orbit space. 

A triplet $\bigl(\mathbb{V},B,\vee\bigr)$,  where $\mathbb{V}$ is a real vector space  of finite dimension, $B=\{\al_i\}_{i=1}^l$ is a linearly independent subset of $\mathbb{V}$ and $\vee: \mathbb{V} \to V^*$, is called a {\it root basis} by Looijenga \cite{Lo} if it satisfies suitable axioms. It induces the Cartan matrix $N$ given by $n_{ij}=\al_i^\vee(\al_j)$, the Coxeter group $W\subset GL(V)$ generated by the reflections  and its extension $\widetilde{W}$ by the translation lattice $\sum_{i=1}^l\mathbb{Z}\al_i$, called the extended Coxeter group, and the Tits cone $\mathbb{X}\ :=\ \bigcup_{w\in W} w(\overline{\mathcal{C}})$ where $\mathcal{C}$ is the fundamental chamber defined by $\mathcal{C}\ :=\ \bigl\{h\in\mathbb{V} \,\bigl|\,
\bigl\langle h, \al_i^\vee\bigr\rangle>0  \ i=1,\cdots,l\}$. Finally, the tube domain $\Omega$ is introduced as an open subset of the complexification $\mathbb{V}_\mathbb{C}$ of $\mathbb{V}$ by
$
\Omega := \mathrm{Im}^{-1} (\overset{\circ}{\mathbb{X}})
$
on which $\widetilde W$ acts properly discontinuously. A point of $\Omega$ is irregular if and only if it belongs to a reflection hyperplane of $\widetilde W$.

The fundamental group of the regular orbit space is generated by the half loops $s_i\  (i=1,\cdots,l)$ around simple reflection hyperplanes and  simple 
translations $t_i\ (i=1,\cdots,l)$ of  the lattice $\sum_{i=1}^l\mathbb{Z}\al_i$. 

In order to state the result, we set, for $0\leq i\ne j\leq l$, 
\vspace{-0.1cm}
\begin{equation}
m_{i,j}:=\begin{cases}
2 & \mbox{if }n_{i,j}n_{j,i}=0,\\
3 & \mbox{if }n_{i,j}n_{j,i}=1,\\
4 & \mbox{if }n_{i,j}n_{j,i}=2,\\
6 & \mbox{if }n_{i,j}n_{j,i}=3,\\
\infty & \mbox{if }n_{i,j}n_{j,i}\geq 4,\\
\end{cases}\label{eq:(3.5.8)}\end{equation}

\begin{thm}{\rm (Lek \cite{Lek})}
The fundamental group of the regular orbit space is  generated by 
${\bf s}_i,{\bf t}_i\,(0 \!\leq\! i \! \leq\! l)$ satisfying the fundamental relations{\rm :}
\begin{itemize}
\item[{\rm (i)}] For $0\leq i,j\leq l$, 
\vspace{-0.15cm}
\[
{\bf s}_i{\bf s}_j{\bf s}_i\cdots ={\bf s}_j{\bf s}_i{\bf s}_j\cdots\quad 
\mbox{{\rm (}$m_{i,j}$ letters in both sides{\rm )}}.
\vspace{-0.15cm}
\]
\item[{\rm (ii)}] ${\bf t}_i{\bf t}_j={\bf t}_j{\bf t}_i$ for $0\leq i,j\leq l$, 
\item[{\rm (iii)}] For $0\leq i,j\leq l$
\vspace{-0.15cm}
\[
\begin{array}{llll}
{\bf s}_i{\bf t}_j{\bf t}_i^r={\bf t}_j{\bf t}_i^r {\bf s}_i & 
\mbox{if $n_{i,j}$ is even, and $n_{i,j}=-2r$,}\\
{\bf s}_i{\bf t}_j{\bf t}_i^r{\bf s}_i={\bf t}_j{\bf t}_i^{r+1} & 
\mbox{if $n_{i,j}$ is odd, and $n_{i,j}=-2r-1$.}
\end{array}
\vspace{-0.15cm}
\]
\end{itemize}
\end{thm}
\vspace{-0.15cm}
\noindent
{\it Note}  6. The presentation is weighted homogeneous with the weights: 
\vspace{-0.25cm}
$$
{\bf s}_i\mapsto 1,\quad {\bf t}_i \mapsto 2 \quad  \,(0\leq i\leq l).
$$ 

\noindent
{\bf 6.  Real line arrangements in $\mathbb{R}^2$:}

Let $\mathcal{A}$ be a set of $n$ distinct affine lines: $H_i=\{\al_i\!=\!0\}$ ($i=1,\cdots,n<\infty$) in $\mathbb{R}^2$. Set $M(\mathcal{A}):=\mathbb{C}^2\setminus \cup_{H\in \mathcal{A}} H\otimes\mathbb{C}$. Then, $M(\mathcal{A})$ is known to be homotopic to a minimal CW-complex (e.g.\ \cite{Y1}). 

Masahiko Yoshinaga \cite{Y2} gave an explicit minimal stratification $M(\mathcal{A})= T_2^3\! \supset\!  X_1\! \supset\! X_2$, using real quadratic surfaces $Im(\al_i\overline{\al}_{i\!-\!1})\! =\!0$ {\small ($i\!=\!1,\cdots,n$)}, and depending on a choice of a generic flag $\mathcal{F}$ in $\mathbb{R}^2$ (giving an ordering of the set $\mathcal{A}$). 
Let $\eta_1,\cdots,\eta_n$ be a system of  generators of the fundamental group dual to the stratification. Then, the alternative system of generators: $\gamma_i\!=\!\eta_i\eta_{i+1}^{-1}$ ($i\!=\!1,\cdots,n-1$) and $\gamma_n\!=\!\eta_n$ (each  $\gamma_i$  turns once around the line $H_i$, cf.\ Figure 2.1, (ii)) satisfies  positive homogeneous relations.\footnote{The generator system is homotopic to a Zariski-van Kampen generator system in the pencil $\mathcal{F}^1_{\mathbb{C}}$. However, the relations \eqref{eq:linesPositiveHomogeneous} are not Zariski-van Kampen relations. 
}

\begin{theorem} (Yoshinaga \cite{Y2} Theorem 6.5)  {\it The fundamental group of $M(\mathcal{A})$ is isomorphic to the group presented by
\begin{equation}
\label{eq:linesPositiveHomogeneous}
\quad \langle \tilde{\gamma}_1,\cdots,\tilde{\gamma}_n \mid \tilde{\gamma}_1\cdots \tilde{\gamma}_n=\tilde{\gamma}_{i_1(C)}\cdots\tilde{\gamma}_{i_n(C)} \ \ for \ \ C\in {\mathrm ch}_2^\mathcal{F}(\mathcal{A})\rangle . 
\end{equation}
where $L=\{\tilde{\gamma}_1,\cdots,\tilde{\gamma}_n\}$ is the system of free generators corresponding to ${\gamma}_1,\cdots,{\gamma}_n\in \pi_1(M(\mathcal{A}),\mathcal{F}^0)$,  ${\mathrm ch}_2^\mathcal{F}(\mathcal{A})$ is the  set of chambers which are disjoint with the flag $\mathcal{F}^1$, and $i(C):=(i_1(C),\!\cdots\!,i_n(C))$ is a permutation of the letters $(1,\!\cdots\!,n)$ associated with $C\!\in\! {\mathrm ch}_2^\mathcal{F}(\mathcal{A})$.}\footnote
{Note  that $i(C)$ is a bijection of the set $\{1,\!\cdots\!,n\}$ such that $i(C)^{-1}(j)=i_j(C)$\ ($1\le j\le n$).
}
\end{theorem}

\noindent
{\it Note} 7. Note that the relations $\mathcal{R}(\mathcal{A})$ in \eqref{eq:linesPositiveHomogeneous} are weighted homogeneous for any weight system $w$. In particular, we consider the weight system $w(\tilde{\gamma}_i)=1$ for all letters $\tilde{\gamma}_i$, and call its degree \eqref{eq:Degree} the  {\it length}.

\smallskip
In the present paper, we study the following two contrasting examples (a) and (b) (see \S4.2 and  \S6 for their further studies).

\medskip
\noindent
{\bf (a)}   Three generic lines $\mathcal{A}_3=\{H_1,H_2,H_3\}$ in $\mathbb{R}^2$.   

This is the example where $\pi_2(M(\mathcal{A}_3),*)$ is nontrivial (see \S6).
Choose  a generic flag $\mathcal{F}$ and the numbering of chambers as in Fig.\ 2.1, (i). Accordingly, the generators $\{\gamma_1,\!\gamma_2,\gamma_3\}$ are given as in Figure 2.1, (ii).

 \vspace{0.25cm}
\hspace{0.95cm} 
\includegraphics[width=0.627\textwidth]{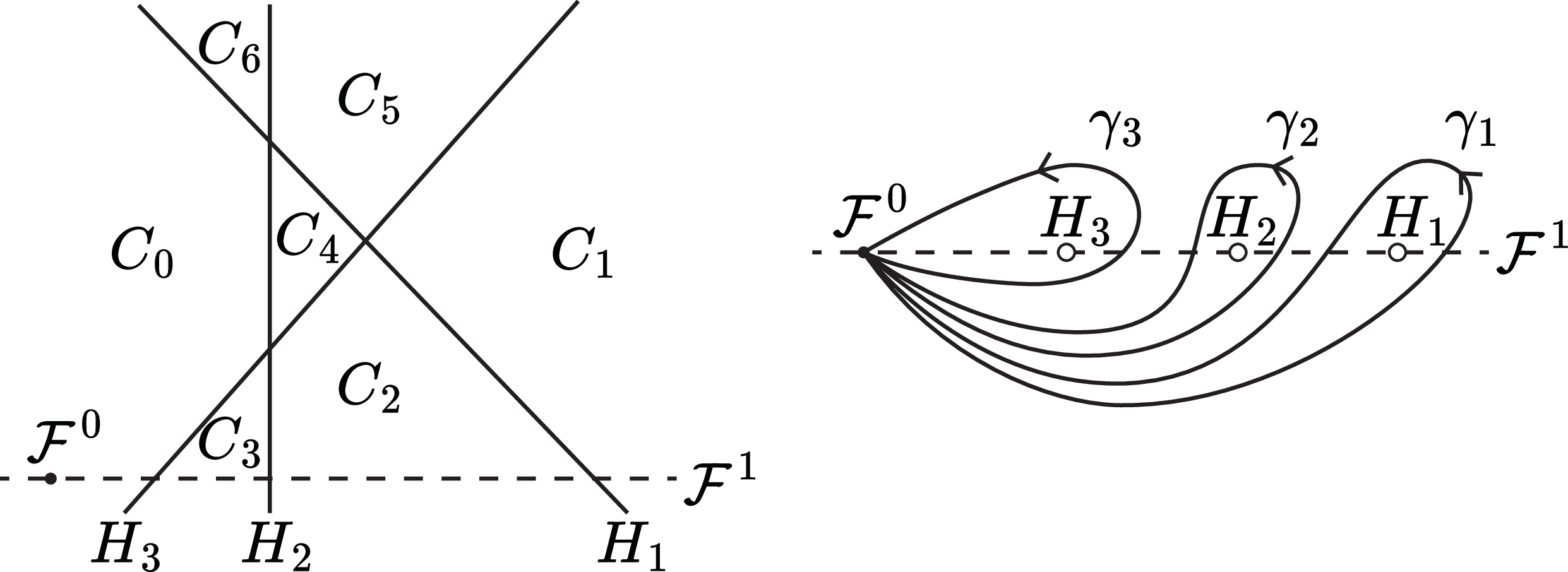}

\vspace{0.1cm}
\centerline{\qquad \ \  Figure 2.1: (i)  \qquad\qquad \qquad   Figure 2.1: (ii)   \quad }

\vspace{0.2cm}

\noindent
Then we have
\vspace{-0.2cm}
$$
\begin{array}{rcl}
i(C_4)=(1,3,2), & i(C_5)=(3,1,2), & i(C_6)=(2,3,1).
\end{array}
$$

\noindent
{\bf (b)}   $\mathrm{A}_3$-arrangement $\mathcal{A}_{\rm{A}_3}$ (\cite{Y2} Example 6.7.).

This is the example where $\pi_2(M(\mathcal{A}),*)$ is trivial.  Consider an arrangement of five lines exhibited in Figure 2.2 (it is a degeneration of \cite{Y2} Fig.5.4), where  the choice of a generic flag $\mathcal{F}$ and the numbering   of chambers are given. 
Accordingly, the generators $\{\gamma_1,\cdots,\gamma_5\}$ of the fundamental group are described in a manner similar to that in Figure 2.1 (ii).

 \vspace{0.3cm}
\hspace{1.9cm} 
\includegraphics[width=0.41\textwidth]{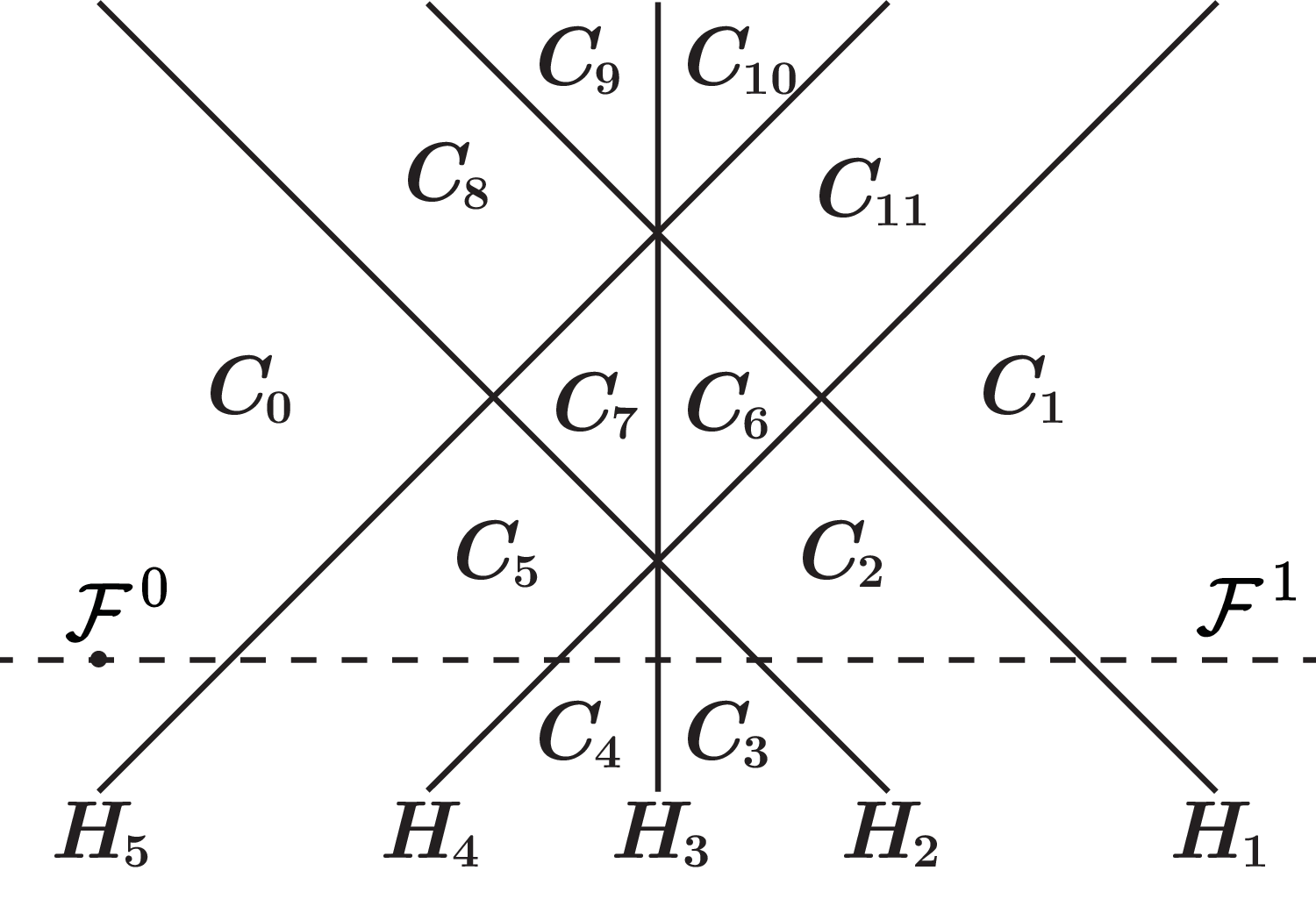}

\vspace{0.1cm}
\centerline{ Figure 2.2:  $A_3$-arrangement     }

\vspace{0.1cm}

\noindent
Then we have
\vspace{-0.1cm}
$$
\begin{array}{rcl}
i(C_6)=(1,4,2,3,5), &
i(C_7)=(1,3,4,2,5),&
i(C_8)=(1,3,4,5,2), \\
i(C_9)=(3,4,5,1,2), &
i(C_{10})=(4,5,1,2,3), &
i(C_{11})=(4,1,2,3,5). 
\end{array}
$$

\noindent
7. {\bf Warning.} Let a semi-positive presentation $(L,\mathcal{R})$ for the fundamental group of a topological space $X$ be given. Then, it has another presentation  $(L,\mathcal{R}')$ where $\mathcal{R}'=\{r_\mu p_\mu s_\mu = r_\mu q_\mu s_\mu\}_{\mu\in M}$ for some $r_\mu,s_\mu\in L^*$ ($\mu\!\in\! M$).  This is an artificial presentation since  $r_\mu,s_\mu\in L^*$ ($\mu\! \in\! M$) do not reflect anything of $X$. Unfortunately, the present  Definition \ref{Semi-positive}  of a semi-positive presentation cannot exclude such cases. A better formulation is desirable.

\vspace{-0.1cm}
\section{\!\!\! Geometric realization of  
 $A^+(L,\mathcal{R}_m)$}

\vspace{-0.1cm}
In the present section, we consider the geometric realization of 
 a semi-positive presentation studied in the previous section. 
 It is achieved by two steps:  1.\ the lifting of generators $L$ (Definition  \ref{W}) and 2.\ the lifting of relations $\mathcal{R}$ (Definition \ref{GeoFundRel}). The lifting data together lead further to the {\it geometric description of the monoid  $A^+(L,\mathcal{R}_m)$} (Proposition \ref{GeometricMonoid}) and to the  second homotopy classes $\Pi(\tau)$  of $X$ (Theorem \ref{FinalTheo}).
 All steps are described in the frame of elementary homotopy theory. Due to the experimental nature of the present work, we describe every step in a down-to-earth manner.\!\!\!\!

\subsection{\normalsize A geometric description of the group $A(L,\mathcal{R})$}

\noindent
{\bf 3.1.1}   Let $X$ be a connected topological space, whose fundamental group is semi-positively presented by   $(L,\mathcal{R})$ \eqref{eq:SemiPositive}.  We  ``lift'' the generator set $L$ of the fundamental group  to a set of loops in $X$ in the following sense.

\medskip
\begin{defn}
\label{W}
Let $X$ and  $(L,\mathcal{R})$ be as above. 
A {\it lifting of the set of generators $L$} consists of a pair $(W,\widetilde{L})$ where 

1)  $W$ is  a closed connected subset of $X$ containing the  base point $*$, which has the homotopy type of a connected 1-dimensional CW-complex such that the inclusion $\iota: W\subset X$ induces a surjective homomorphism:
\begin{equation}
\label{eq:WtoX}
\iota_*:\ \pi_1(W,*) \ \to \ \pi_1(X,*)  \ \to \ 0
\end{equation}

 2)   $\widetilde{L}$ is a system of free generators of the  group $\pi_1(W,*)$ with a bijection
 \begin{equation}
 \label{eq:CalL}
   \widetilde{L} \quad \simeq \quad L 
   \end{equation}  
\noindent
which induces the homomorphism \eqref{eq:WtoX}.\footnote{We may regard $(W,\widetilde{L})$ as a  ``topological analogue of  a Zariski pencil and a Zariski-van Kampen generator system in the pencil". However, even in that case, $\mathcal{R}$ may not necessarily be Zariski-van Kampen relations (see Examples in \S2.2, {\bf 6.} and Footnote 11).}
\end{defn}

  We identify the free group
\vspace{-0.1cm}
\begin{equation}
\label{eq:FreeGroup}
\mathcal{F}\ :=\ \pi_1({W},*)
\end{equation}
generated by $\widetilde{L}:=\{\tilde a_\lambda\}_{\lambda\in \Lambda}$ with the free group generated by $L$ through the bijection  \eqref{eq:CalL}.   The product of two elements $a,b$ is denoted either $ab$ or $a\cdot b$. 
The kernel of the morphism $\iota_*  :  \mathcal{F}  \to  \pi_1(X,*)\simeq A(L,\mathcal{R})$ is the normal subgroup generated by a copy in terms of $\widetilde{L}$  of the fundamental relations $\mathcal{R}$ \eqref{eq:SemiPositive}.

\smallskip
\noindent
{\bf 3.1.2}  We switch the point of view. Instead of taking the quotient of $\mathcal{F}$ algebraically by the normal subgroup generated by a copy of $\mathcal{R}$, we consider the  set $[x,y]$ of homotopy classes of homotopies from  $x$ to $y\in \mathcal{F}$ in $X$:
\vspace{-0.04cm}
\begin{equation}
\label{eq:Equivalence}
[x,y] \ := \{\text{homotopies from $x$ to $y$ in }X\}\  \big/ \sim
 \end{equation}
where, precisely, the RHS is defined by the following rules  (1) and (2):\footnote
    {To be exact, $x$ and $y$ are not loops but are homotopy equivalent classes in $\mathcal{F}$. However, this ``ambiguity" is absorbed  in the homotopy equivalence $\sim$ in the description (2).
    }

  (1) By a homotopy from $x$ to $y$, we mean a continuous map, say $H: [0,1]^2 \to X$, satisfying the following conditions a) and b). 
  
  \noindent
a) The restrictions $H|_{[0,1]\times 0}$ and $H|_{[0,1]\times 1}$ are maps  $[0,1]\to W$ and represent  the classes 
 $x, y \in \mathcal{F}$.

\noindent
b) The restrictions $H|_{0\times [0,1]}$ and $H|_{1\times[0,1]}$ are constant $*$-valued.

\smallskip
\noindent
That is: the first variable $u$ in $H(u,v)$ for each fixed $v$ descr-

\noindent 
 ibes a\! loop\! in\! $X$\! based\! at\! $*$ and  the second\! variable $v$ describes

\noindent
homotopy between  the two loops $x$  at $v=0$ and  $y$ at $v\!=\!1$.

\noindent

\vspace{-1.3cm}
\qquad\qquad \qquad \qquad \qquad 
\qquad\qquad \qquad \qquad \qquad 
\qquad\qquad \quad \ \  \
\includegraphics[width=0.12\textwidth]{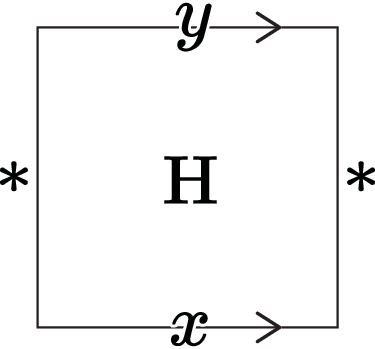}

\vspace{-0.18cm}
\noindent
The map $H$ is symbolically represented  as in Figure.   \hfill{Figure.\ \ }
 
  \smallskip  
   (2) The relation $\sim$ in 
    \eqref{eq:Equivalence} is the homotopy equivalence, i.e.\ $H_1$ and $H_2$ are equivalent if they belong to the path-connected component with respect to the compact-open topology of the set of all homotopy equivalence maps $: [0,1]^2 \to X$ from $x$ to $y$ in $X$. One should not confuse present ``$\sim$" with that introduced in \S2.1 in order to define the monoid $A^+(L,\mathcal{R}_m)$.

\medskip 
\noindent
One obvious reason for introducing the system $\{[x,y]\}_{x,y\in \mathcal{F}}$ is the following, which we shall use to introduce a weaker equivalence relation \eqref{eq:R^gSet}.\!\!

\smallskip 
 \noindent
 {\bf Fact.} {\it Two elements $x,y\in \mathcal{F}$ induce the same element in the group $A(L,\mathcal{R})$ if and only if $[x,y]\! \not=\! \emptyset$. That is, the set $\{(x,y)\mid [x,y]\! \not=\! \emptyset\}_{x,y\in \mathcal{F}}$ is the equivalence relation on the set $\mathcal{F}$ defining the quotient group $A(L,\mathcal{R})$.} 
 
 \smallskip
 

 
  
  The system \eqref{eq:Equivalence} shall be the basic tool to describe the monoid $A^+(L,\mathcal{R}_m)$ geometrically and to construct  the second homotopy classes associated with NC-twins of the monoid.  
  Therefore, in the remainder of \S3.1.2-3., we describe in a down-to-earth manner some basic and classical operations on the system  \eqref{eq:Equivalence}:  two product structures ``$\ \circ\ $" and ``$\ \cdot\ $",  and the relative realization  morphism $\chi$  of the system into the relative homotopy group $\pi_2(X,W,*)$ (\S3.1.3).

\newpage
 \noindent
 \begin{rem}\!
 {\rm { From Whitehead's\! viewpoint,\! the\! boundary\! homomorphism}
{\footnotesize $ \partial\!:\!\pi_2(X,W,*)\!\to\!\pi_1(W,*)$}
together with the natural action of {\footnotesize \(\pi_1(W,*)\)} is the fundamental crossed module of the pair {\footnotesize \((X,W)\)} \cite{W}. In particular, the sets {\footnotesize  \([x,y]\)} may be viewed as concrete descriptions of the fibers {\footnotesize  \(\partial^{-1}(xy^{-1})\)}.}
\end{rem}
  


\smallskip
 \begin{prop} 
 \label{PropA} {\it
The system $\{[x,y]\}_{x,y\in \mathcal{F}}$ \eqref{eq:Equivalence}  has the following algebraic properties a), b), c), d), e), f)  and g). 
}
\end{prop}

\vspace{0.1cm}
{\it
a)  For $[H_1]\in [x,y]$ and $[H_2]\in [y,z]$  ($x,y,z\!\in \! \mathcal{F}$), we define a product $[H_1]\circ[H_2]\in [x,z]$ by concatenating  the second variable (in the sense given in the proof, see Figure 3.1):
\begin{equation}
\label{eq:circ}
\circ\ : \ \ [x,y] \times [y,z] \ \rightarrow\  [x,z], \  [H_1]\times [H_2] \mapsto [H_1]\circ[H_2] .
\end{equation}
Then the  product ``\ $\circ$" is associative.

\smallskip
b)  For all $x\in \mathcal{F}$, $[x,x]$ contains a unique element $[1_x]$, called the identity, such that, for any $[H]\! \in \! [x,y]$, we have
$
[1_x]\circ[H]  =  [H]  =  [H]\circ [1_y]
$.

\smallskip
c) For $[H]\in [x,y]$ ($x,y\in \mathcal{F})$, let $1/H$ be the map  reversing the second variable $v$ to $1\!-\!v$ of $H$.\ Then $[1/H]\! \in\! [y,x]$ depends only on $[H]$ and induces\!\!\!\!
\begin{equation}
\label{eq:circInverse0}
[x,y] \ \overset{\sim}{\longrightarrow} \ [y,x], \quad  [H] \ \mapsto \ 1/[H]:=[1/H] 
\end{equation}
such that 
$
[H]\circ[1/H]\!=\!1_x$ and $[1/H]\circ[H]\! = \! 1_y
$.\!\!

\smallskip
d)  For $[H]\in [x,y]$ and $[H']\in [z,w]$  ($x,y,z,w\!\in \! \mathcal{F}$), we define product $[H]\cdot [H'] \! \in\! [xz,yw]$ by concatenating  the first variable (see Figure 3.1):
\begin{equation}
\label{eq:cdot}
\cdot \ : \ [x,y] \times [z,w] \ \to \ [xz,yw], \  [H]\times [H'] \mapsto [H]\cdot [H'] .
\end{equation}
Then the  product ``\ $\cdot$" is associative. 

e)  For $[H]\in [x,y]$ ($x,y\in \mathcal{F})$, by reversing the first variable of $H$, we obtain a map $H^{-1}$ s.t. $[H^{-1}]\in [x^{-1},y^{-1}]$. So we obtain a bijection: 
\vspace{-0.1cm}
\begin{equation}
\label{eq:LoopInverse}
[x,y] \ \overset{\sim}{\longrightarrow} \ [x^{-1},y^{-1}], \quad  [H] \ \mapsto \ [H]^{-1}:= [H^{-1} ] 
\vspace{-0.1cm}
\end{equation}
such that 
$
[H]\cdot [H^{-1}]=[H^{-1}]\cdot [H] = \! 1_e \in [e,e]$ where $e\in \mathcal{F}$ is the unit element. In particular, we have the equality:  $[1_x^{-1}]=[1_{x^{-1}}]$ for any $x\in \mathcal{F}$. 

f) For $[H_1],[H_2], [H_1'],[H_2']$, we have the following commutativity.
\begin{equation}
\label{eq:CdotCirc}
([H_1]\circ[H_2]) \cdot ([H_1']\circ[H_2']) =
([H_1]\cdot[H_1']) \circ ([H_2]\cdot[H_2']) .
\end{equation}

g)  For a  loop $a\!\in\! Loop(W,*)$, consider 5 maps, say $H_{**}^{**}$, $H_{aa^{-1}}^{**}$, $H_{**}^{aa^{-1}}$, $H_{aa^{-1}}^{**}\circ H_{**}^{aa^{-1}}$ and, $H_{aa^{-1}}^{aa^{-1}}$,
   defined on the following 5 figures\}.

\hspace{-1.0cm} 
 \includegraphics[width=0.92\textwidth]{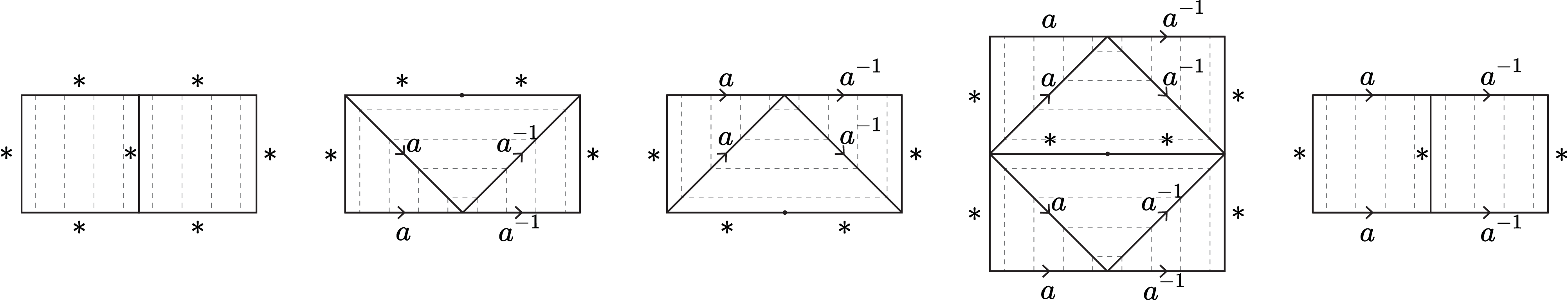}

\noindent
\vspace{-0.2cm}
\hspace{-0.4cm}
{\rm Figure: $H_{**}^{**}$, \qquad\quad  $H_{aa^{-1}}^{**}$, \ \qquad \quad \ $H_{**}^{aa^{-1}}$, \qquad $H_{aa^{-1}}^{**}\circ H_{**}^{aa^{-1}}$, \qquad  $H_{aa^{-1}}^{aa^{-1}}$,}

\vspace{0.5cm}
\noindent
by taking constant values  along the dotted lines.\footnote
   {For the meaning of the notation $H_{aa^{-1}}^{**}\circ H_{**}^{aa^{-1}}$, see the proof of a).}
 Then, all the maps are homotopy equivalent to each other. That is, they define the same elements  $[H_{**}^{**}]\!=\! [H_{aa^{-1}}^{**}] \!=\! [H_{**}^{aa^{-1}}]\! =\! [H_{aa^{-1}}^{**}\circ H_{**}^{aa^{-1}}]\! =\! [H_{aa^{-1}}^{aa^{-1}}]$ in the same group $[e,e]= [aa^{-1},e]=[e, a a^{-1}]=[aa^{-1}, a a^{-1}]=[aa^{\!-1}\!, a a^{\!-1}]$,\! respectively.\footnote{
There are three notational remarks: (1) here, we shall confuse $a$ and $a^{\!-1}$\! with their images in $\mathcal{F}$, in that case, $aa^{\!-1}\!=e$, (2) we shall denote by $*$ also the constant loop staying at $*$, which represents the unit $e\in\mathcal{F}$, and (3) the domain of the 5 maps for the first variable ``$u$" is the interval $[0,2]$ instead of $[0,1]$. This abuse may be allowed by rescaling ``$u$" to ``$2u$".
}
}

\medskip
 \noindent
 {\bf Notational remark.}  We need to distinguish the inverse operations of the two product structures  ``$\cdot$'' and ``$\circ$'' on  \eqref{eq:Equivalence}.    
  We denote the inverse with respect to ``$\cdot$'' by ``$H^{-1}$'', and the inverse with respect to  ``$\circ$'' by ``$1/H$'' . This convention is chosen to agree with the notation ``$a^{-1}$'' for the inverse of a loop ``$a$''.

  \vspace{0.1cm}
\hspace{1.6cm} 
\includegraphics[width=0.5\textwidth]{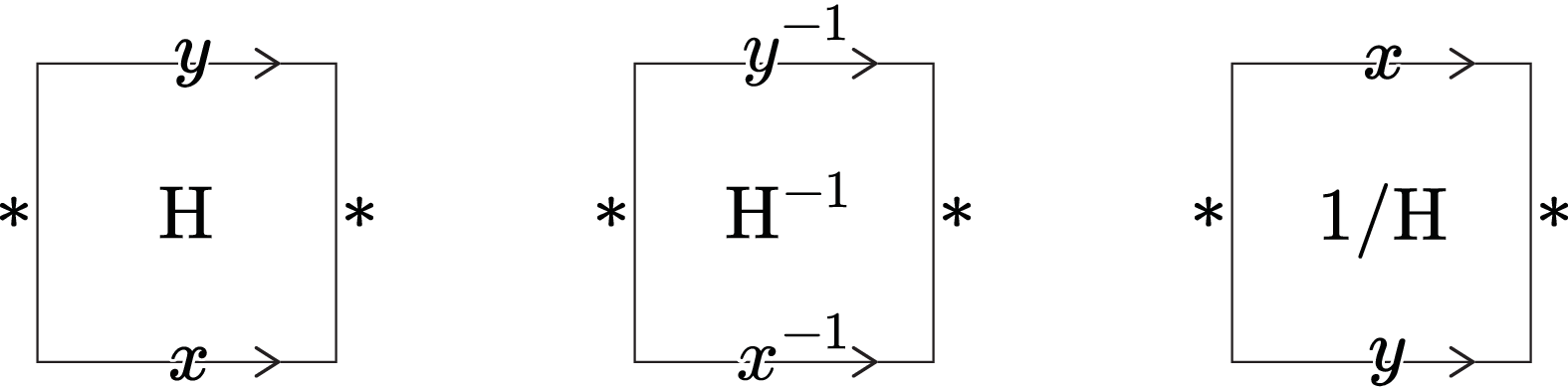}

\centerline{  Figure 3.2: \quad  Two inverses \ \ $H^{-1}$\ \ and \ \ $1/H$.}

 \begin{proof} 
 \!a) Choose representative maps $H_1$ and $H_2$ for $[H_1] \!\in \! [x,y]$ and $[H_2] \!\in \! [y,z]$ such that the restrictions $H_1\! \mid_{[0,1]\times\{1\}}$ and $H_2\! \mid_{[0,1]\times\{0\}}$ are the same representative of $y$.
 Then, a representative of $[H_1] \circ [H_2]$ is given by  
 the map  obtained by the map $H_1\circ H_2$ obtained by concatenating the second variable $v$ of $H_1(u,v)$ from $x$ to $y$ and that of $H_2(u,v)$ from $y$ to $z$ (see Figure 3.1). 
 The homotopy class  of $ [H_1\circ H_2]$ does not depend on representatives\! {\small $H_1$\! and\! $H_2$}.\! The\! associativity\! follows\! from\! the\! concatenation.\!\!
 
\smallskip
b) $1_x:[0,1]^2\to {W}^{(1)}$ is a map given by $1_x(u,v):=x(u)$ for a representative $x:[0,1]\to W$ of  $x$
  (we use the same notation ``$x$").
  
  \smallskip
 c) Set the map $1/H$ by  reversing the second variable:
$ (1/H)(u,v):=H(u,1-v)$.
 A homotopy equivalence from $1_x$ to $H\circ (1/H)$ is given by the concatenation of the second variables of $H(u,tv)$ and $H(u,t(1-v))$.

 \vspace{0.25cm}
\hspace{2.cm} 
\includegraphics[width=0.47\textwidth]{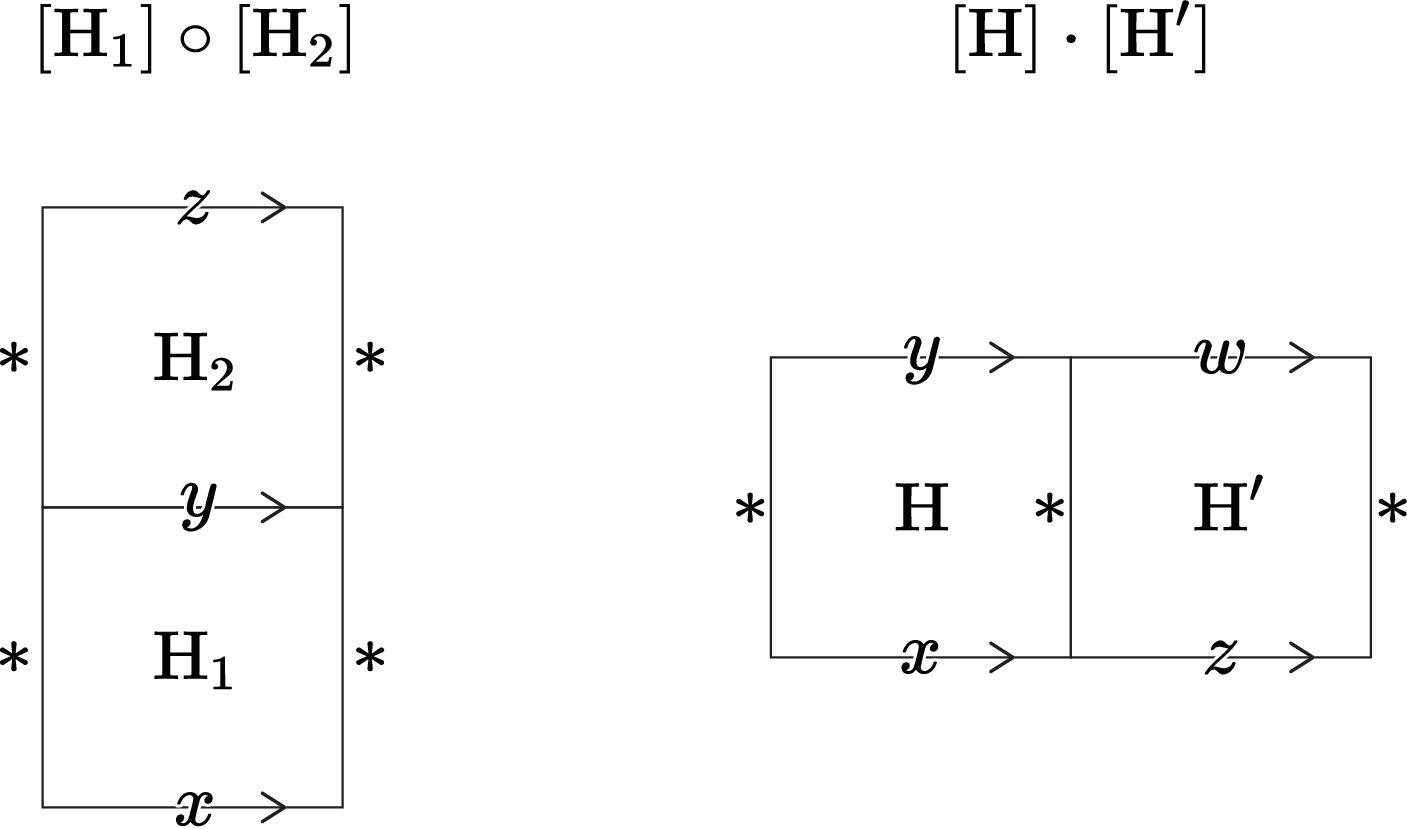}

\centerline{\qquad \qquad Figure 3.1:  \quad Two products \  ``\ $\circ$\ " \ and \ ``\ $\cdot$\ " . \qquad }

\medskip
d)  Choose representatives $H$ and $H'$ for $[H] \!\in \! [x,y]$ and $[H'] \!\in \! [z,w]$.
 Then, a representative of $[H] \cdot [H']$ is given by  
 the map  obtained by $H\cdot H'$ concatenating the first variable $u$ of $H(u,v)$ from $x$ to $y$ and that of $H'(u,v)$ from $z$ to $w$ (see Figure 3.1) . 
 The homotopy class  of the concatenation $ [H\cdot H']$ does not depend on the representatives $H$ and $H'$. The associativity is clear from that of the concatenation.

\smallskip
e) We define $[H^{-1}]$ as the class of the map 
$H^{-1}(u,v):=H(1-u,v). 
$
A homotopy equivalence (parameter $t\in[0,1]$) from $1_e$ to $H\cdot H^{-1}$ is given by the concatenation of the first variable of $H(tu,v)$ and $H(t(1-u),v)$.

\smallskip
f) Both sides of \eqref{eq:CdotCirc} are presented by the following same map:  

\medskip
\centerline{\includegraphics[width=0.22\textwidth]{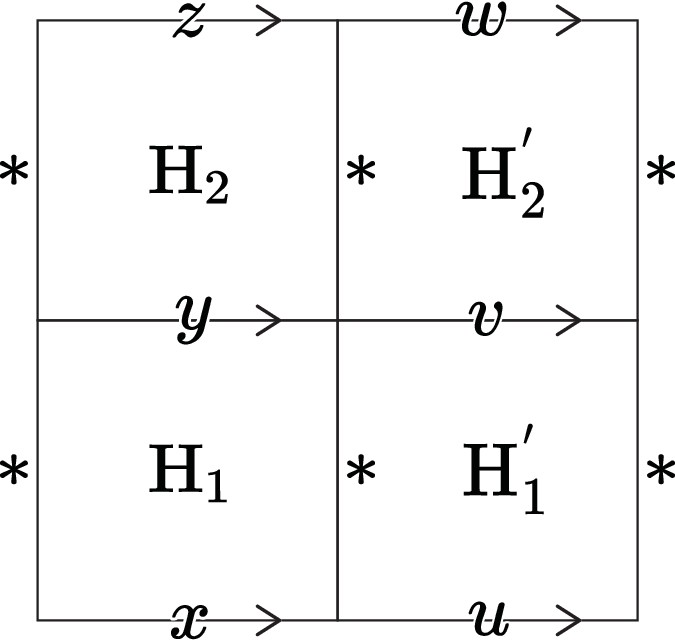}}

\vspace{0.15cm}
\centerline{ Figure 3.3 : {\footnotesize $([H_1]\circ[H_2]) \cdot ([H_1']\circ[H_2']) =
([H_1]\cdot[H_1']) \circ ([H_2]\cdot[H_2']) $}.}

\medskip
g)  We first introduce two homotopy equivalences:  $H_{aa^{-1}}^* \overset{A}{\Rightarrow} H_{**}^{**}$  and $H_{aa^{-1}}^*\overset{B}{\Rightarrow} H_{aa^{-1}}^{aa^{-1}}$ as one  parameter families in $t\in [0,1]$ as follows.\footnote
{Be cautious that the notations $\overset{A}{\Rightarrow}$ and $\overset{B}{\Rightarrow}$ look  a bit misleading, since the homotopy equivalence relation is symmetric. This is caused by the orientation of the variable $t$ defining the homotopy equivalence. In fact, if we reverse the variable $t$ of $A$ and $B$ to the variable $1-t$, we obtain the homotopy equivalences in the opposite direction, respectively. 
}

\vspace{0.25cm}
\hspace{-0,3cm} 
\includegraphics[width=0.36\textwidth]{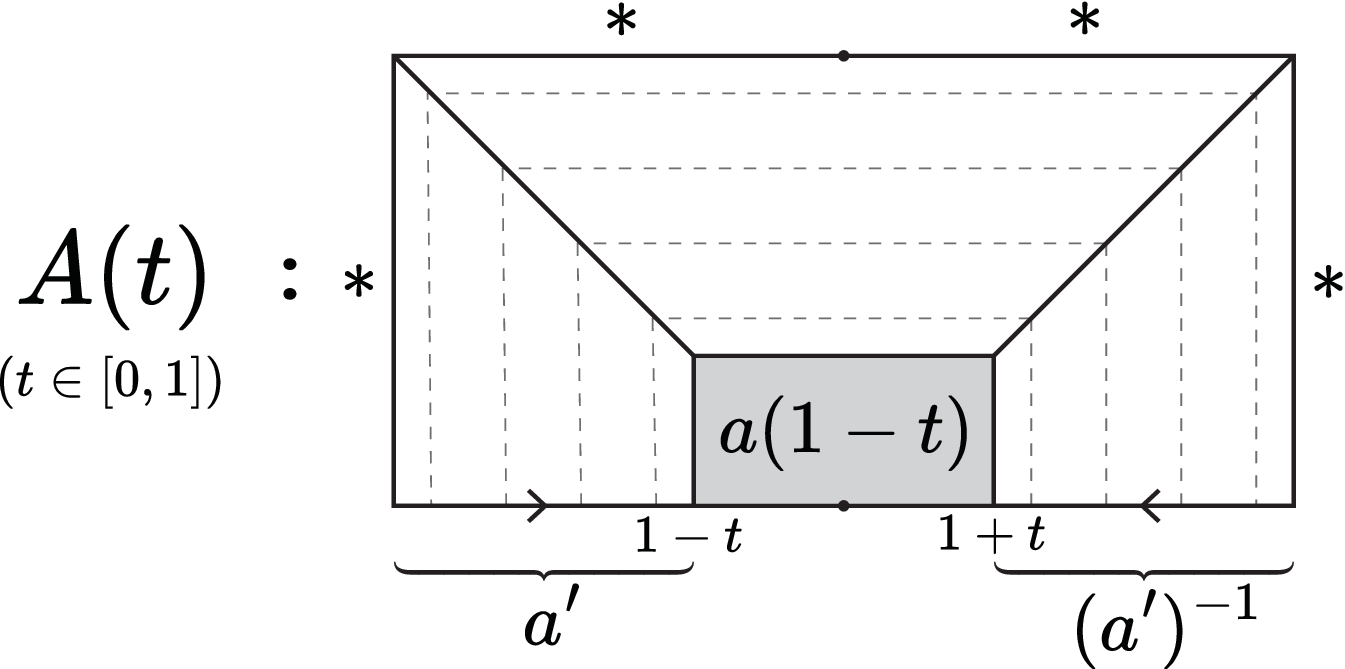}


\vspace{-2.35cm}
\hspace{4.8cm} 
\includegraphics[width=0.36\textwidth]{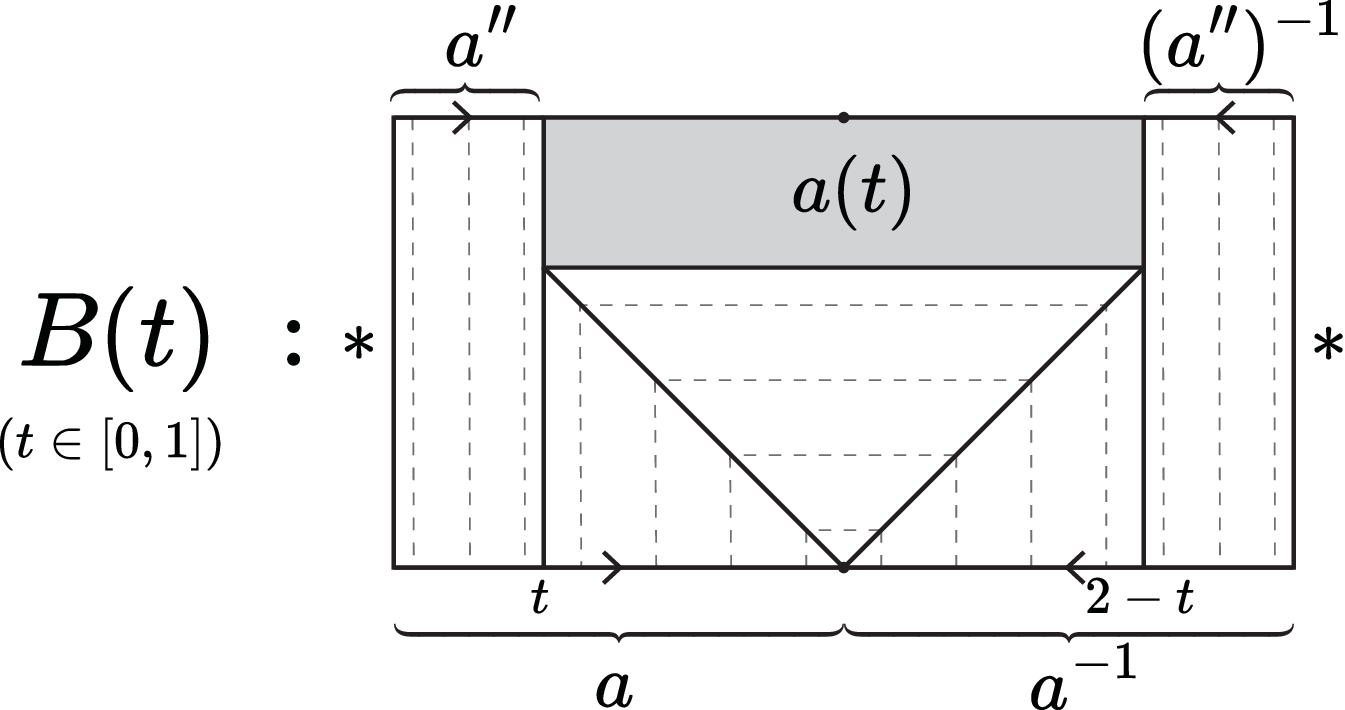}

\noindent
\  \ \ Figure 3.4: \ {\footnotesize $H_{aa^{-1}}^{**}\overset{A}{\Rightarrow} H_{**}^{**} $} \qquad \qquad \ \ 
 Figure 3.5: \ {\footnotesize $ H_{aa^{-1}}^{**}\overset{B}{\Rightarrow}  H_{aa^{-1}}^{aa^{-1}}  $}

\medskip
\noindent
where $a'$ and $a''$ are the arcs of the loop $a$ obtained by the restrictions $a|_{[0,1-t]}$ and  $a|_{[0,t]}$, respectively, and the maps $A(t)$ and $B(t)$ take constant values $a(1-t)$ or $a(t)$ on each grey zone, respectively.


\smallskip
Then, we obtain the following list of homotopy equivalences:
\vspace{-0.2cm}
$$
\vspace{-0.1cm}
\begin{array}{ccccccccc}
\qquad \quad H_{**}^{**} & \overset{A}{\Longleftarrow} & H_{aa^{-1}}^{**} & \overset{B}{\Longrightarrow} & H_{aa^{-1}}^{aa^{-1}} \qquad \qquad\\
\qquad \quad H_{**}^{**} & \overset{1/A}{\Longleftarrow} & H_{**}^{aa^{-1}} & \overset{1/B}{\Longrightarrow} &  H_{aa^{-1}}^{aa^{-1}}\qquad \qquad  \\
\!\!\! H_{**}^{**} \sim H_{**}^{**}\! \circ\! H_{**}^{**} & \!\! \overset{A\circ (1/A)\!\!}{\Longleftarrow} \!\! & \!H_{aa^{-1}}^{**} \! \circ \! H_{**}^{aa^{-1}} \! & \!\! \overset{B\circ (1/B)\!}{\Longrightarrow} \!\! &\! H_{aa^{-1}}^{aa^{-1}} \!\circ\! H_{aa^{-1}}^{aa^{-1}} \sim H_{aa^{-1}}^{aa^{-1}}
\end{array}
$$
Here $1/A$ and  $1/B$ are the maps obtained by reversing the second variable $v$ to $1\! -\! v$, and  
$[A\! \circ\! (1/ A)]$ and $[B\! \circ\! (1/B)]$ are families of  identities.\!\!

\smallskip
\noindent
All these together give the homotopy equivalences among all 5 maps.
 \end{proof}
 
\begin{rem}
{\rm For the purpose of the proof of g) of Proposition \ref{PropA} , the homotopy equivalences  $H_{aa^{-1}}^{**}\overset{A}{\Rightarrow} H_{**}^{**} $ and $ H_{aa^{-1}}^{**}\overset{B}{\Rightarrow}  H_{aa^{-1}}^{aa^{-1}}$ are sufficient. However, we may additionally consider a homotopy equivalence:

\smallskip
\vspace{0.1cm}
\hspace{4.4cm} 
\includegraphics[width=0.35\textwidth]{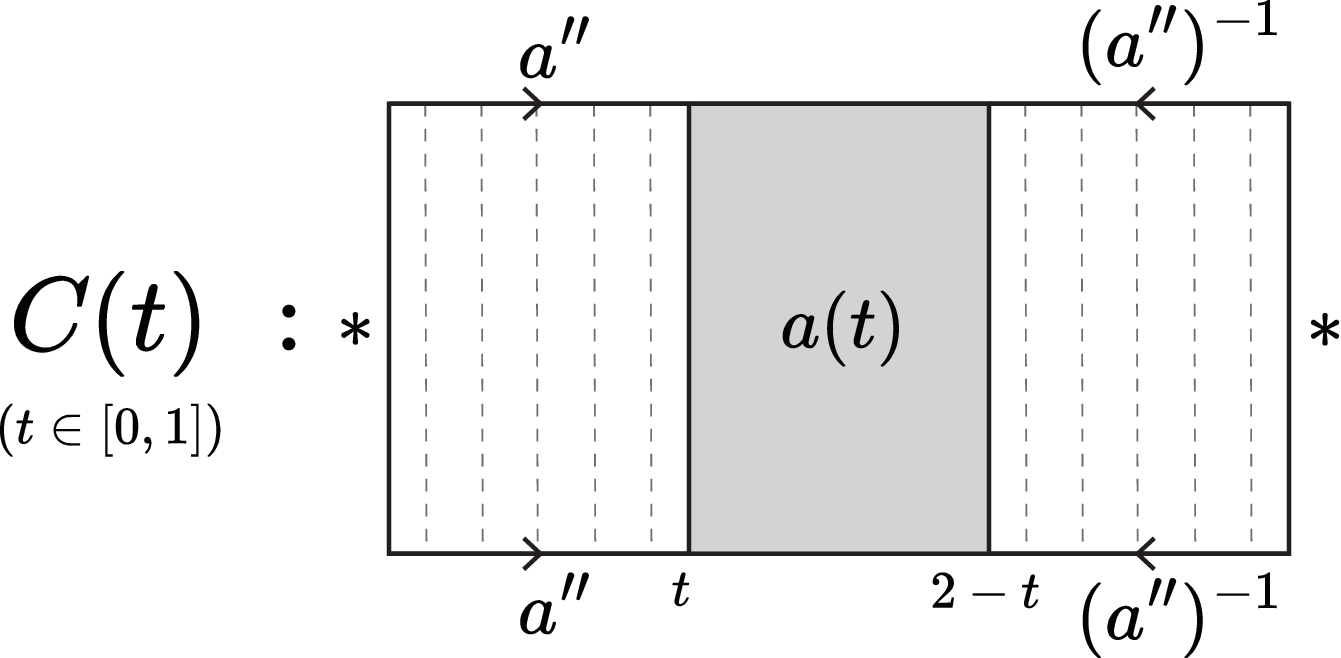}

\vspace{-1.1cm}
\!\!\!\!{Figure 3.6: \ $H_{**}^{**}\overset{C}{\Rightarrow} H_{aa^{-1}}^{aa^{-1}} $}

\vspace{1.cm}
\noindent
We obtain two homotopy equivalences from 
{\footnotesize $H_{**}^{**}$ to $H_{**}^{**} $: $(\overset{A}{\Rightarrow})^{-1}(\overset{B}{\Rightarrow})(\overset{C}{\Rightarrow})^{-1}$ and $(\overset{A^{-1}}{\Rightarrow})^{-1}(\overset{B^{-1}}{\Rightarrow})(\overset{C}{\Rightarrow})^{-1}$.} What classes they represent 
  in the fundamental group of the space of homotopy equivalence maps from $H_{**}^{**}$ to itself? }
\end{rem}
 
 \medskip
 As an immediate consequence of Proposition \ref{PropA}, we have:

\begin{cor}
\label{SimpleTransitive}

1) If $[x,y]\not=\emptyset$, the left (resp.\  right) action by the multiplication ``$\circ$''  of $[x,x]$ (resp.\ $[y,y]$) on $[x,y]$  is simply transitive.  

2)  The right (resp.\ left) ``$\ \cdot$"-multiplication of $ 1_a$  induces an associative right (resp.\ left) bijective action of  $a\in \mathcal{F}$ on the system $\{[x,y]\}_{x,y\in \mathcal{F}}$ to the system $\{[xa,ya]\}_{x,y\in \mathcal{F}}$ (resp.\ $\{[ax,ay]\}_{x,y\in \mathcal{F}}$), which, instead of writing ``$\cdot 1_a$ (resp.\ ``$1_a\cdot$"), we sometimes write $\cdot a$ (resp.\ $a\cdot$).
\end{cor}

 \smallskip
 \noindent
 {\bf 3.1.3 \quad Relative realization morphism $\chi$ of  $\{[x,y]\}_{x,y\in \mathcal{F}}$ in  $\pi_2(X, W,*) $}.\!\!\!
 
 \smallskip
 The relationship of the sets $[x,y]$ with the homotopy groups of $X$ is given by realizing \footnote
 {We use the term ``realization", since, as we shall see immediately below, that the map $\chi$ is injective in the ``local level" on $[x,y]$ but is not injective in the ``global level" on $\cup_{x,y\in \mathcal{F}}[x,y]$.}
 them in the relative homotopy group $\pi_2(X, W,*) $ where the realization depends on additional choices (see Footnote 19).

\smallskip
Recall 
the homotopy long exact sequence:
$$
\begin{array}{l}
\!\!\!\!\overset{\partial_2}{\to} \pi_2(W,*) 
\to \pi_2(X,*)  \to  \pi_2(X, W,*) 
\overset{\partial_1}{\to} \pi_1(W,*) \overset{\iota_*}{\to}  \pi_1(X,*) \to \! 1
\end{array}
$$
where the leftmost term $\pi_2(W,*)$ is the unit group $1$ since $W$ has the homotopy type of one -dimensional CW-complex. Then we obtain {\it a bijection of $[x,y]$ with the subset of  the second relative homotopy group}:\footnote
{Actually, there are 4 realization morphisms $[x,y]\to \pi_2(X, W,*)$ depending on a choice of orientation on $[0,1]^2$ (clockwise or counter-clockwise) and on a choice of base parts of $[0,1]^2$ (either $0\times[0,1]$ or $1\times[0,1]$). Accordingly, $[x,y]$ is bijective to the $\partial_1^{-1}$  image of either $xy^{-1}, yx^{-1}, x^{-1}y$ or $y^{-1}x$ (which are, of course,  naturally bijective to each other).
}
\begin{equation}
\label{eq:immersion}
  \chi\ :\ [x,y] \quad \overset{\sim}{\longrightarrow} \quad \partial_1^{-1}(xy^{-1}) \quad \subset \ \pi_2(X, W, *) 
\end{equation}
by regarding a representative $H$ of $[H]\in [x,y]$ as a map between triples:
\begin{equation}
\label{eq:Relativisation}
H\ :\ ([0,1]^2,\partial([0,1]^2),\{0\} \!\times \! [0,1]) \rightarrow
(X, W,*)
\end{equation}
where  we choose 

(1) the counter-clockwise orientation on $\partial[0,1]^2$, 

(2) the set $\{0\} \!\times \! [0,1]$ as the base part of $[0,1]^2$.
\footnote
{According to our choice,  the set $[xw,yw]$ $\forall w\!\in\! \mathcal{F}$ is mapped on to the same set $\partial_1^{-1}\!(xy^{\!-1})$.\!
}

\smallskip
\noindent
We shall call the bijection $\chi$ in \eqref{eq:immersion} the {\it relative realization morphism}. 

\smallskip
\noindent
{\it Proof.} We have only to identify the equivalence relation ``$\sim$" used to define $[x,y]$ and that of $\partial_1^{-1}(xy^{-1})$. Details are left to the reader. \hfill{$\Box$

\smallskip
The relative realization morphism $\chi$ \eqref{eq:immersion} maps the unit element $[1_x]\in [x,x]$ ($x\in\mathcal{F}$) to the unit element $e$ of the relative homotopy group. Further, {\it by the morphism $\chi$,  the ``$\ \circ$"-product structure \eqref{eq:circ} and the ``$\ \circ$"-inverse structure \eqref{eq:circInverse0} are identified with the corresponding product structure ``\ $\cdot_2 $" and inverse in the relative homotopy group. That is,  following diagrams commute. }
\vspace{-0.1cm}
\begin{equation}
\label{eq:twoProducts}\!\!\!\!\!\!\!\!\!\!\!\!
\begin{array}{ccccccc}
\vspace{0.2cm}
\!\!\!\!\!\!\!\!\!\!\!\!\!\!\!\!\!\!\!\!\!\!\!\!\!\! [x,y] \!&\!\!\!\! \overset{\sim}{\rightarrow} \!\! &\! \partial_1^{-1}(xy^{-1})&\!\!\!&
[x,y] \! \times \![y,z] \!& \!\!\overset{\sim}{\rightarrow} \!\!&\! \partial_1^{- \!1}(xy^{-1}) \! \times \! \partial_1^{-1}(yz^{-1})\!\!\!\!\!\!\!\!\!\!\!\!\!\! \!\!\\
\vspace{0.2cm}
\!\!\!\!\!\!\!\!\!\!\!\!\!\!\! \downarrow 1/(-)\!\!\!\!\!\!\!& &\quad \downarrow 1/(-)\!\!\!\! & & \downarrow \circ&&\ \ \downarrow \cdot_2 \\
\!\!\!\!\!\!\!\!\!\!\!\!\!\!\!\!\!\!\!\!\!\!\!\!\! [y,x]\!\! &\!\!\!\! \overset{\sim}{\rightarrow} \!\! &\! \partial_1^{-1}(yx^{-1}) && [x,z] &\!\! \overset{\sim}{\rightarrow} \!\! &\ \ \partial_1^{-1}(xz^{-1})
\end{array}
\end{equation}
\begin{proof} A homotopy $h$ from the unit element $e$ to $1_x$  is given by $h(u,v,t)$ $:=x(ut)$ for $(u,v,t)\in[0,1]^3$. 

We verify that the product ``$\circ$" given in Proposition \ref{PropA} is compatible with that ``$\cdot_2$" in the relative homotopy group $\pi_2(X, W,*)$. 
%
Let $H_1$ and $H_2$ be maps from $S_1=[0,1]^2$ and $S_2:=[0,1]^2$ to $X$ representing the classes  $[H_1]\!\in \! [x,y]$ and $[H_2] \! \in \! [y,z]$, respectively (without a loss of generality, we can choose the representatives satisfying $H_1|_{[0,1]\times\{1\}}\!=\!H_2|_{[0,1]\times\{0\}}$). Then, the class $[H_1]\! \circ \! [H_2]$ is represented by  the map defined on the space obtained  by gluing the upper edge $[0,1]\times\{1\}$ of $S_1$ to the lower edge $[0,1]\! \times \! \{0\}$ of $S_2$ with opposite orientations. The union of the left edges  $\{0\}\! \times \! [0,1]$ of $S_1$ and $S_2$ is taken as the base part. 

\vspace{0.3cm}
\hspace{2.2cm} 
\includegraphics[width=0.40\textwidth]{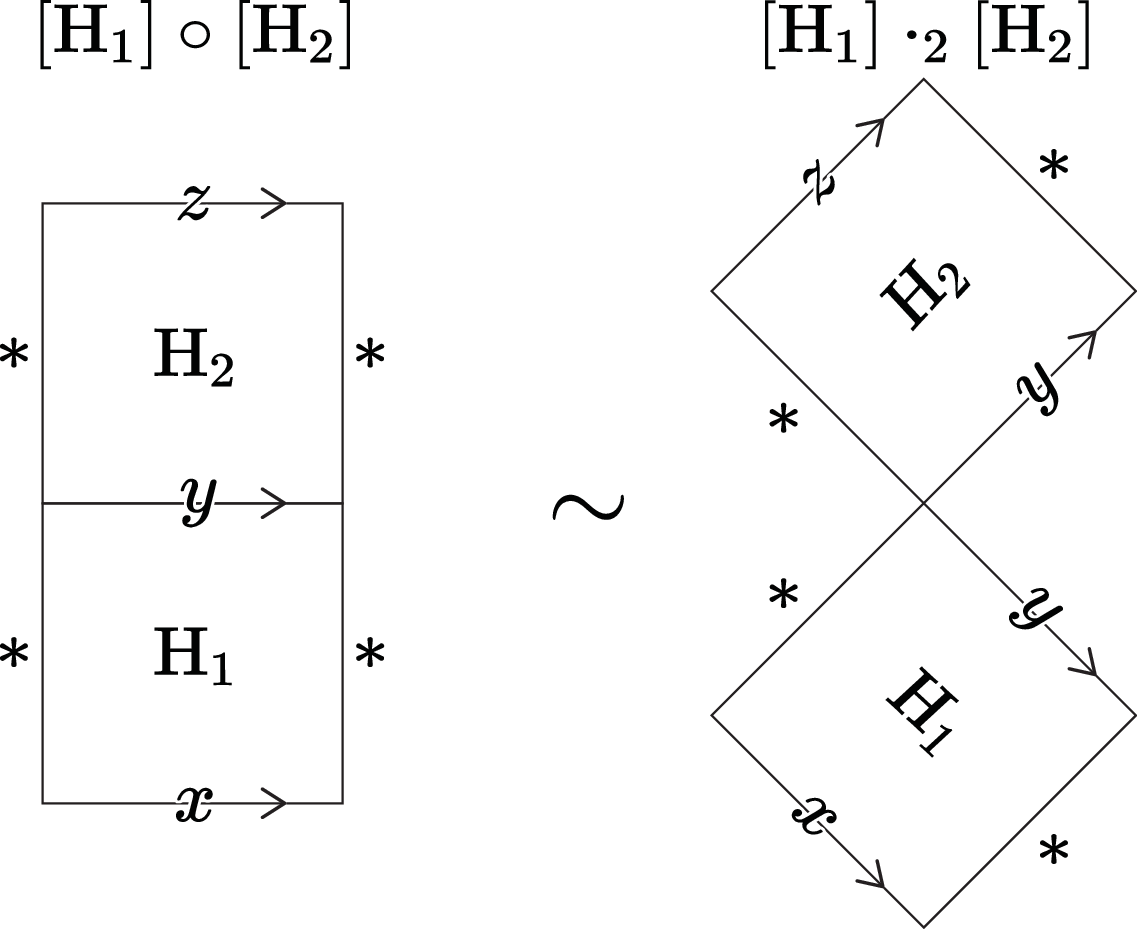}

\centerline{\ Figure 3.7: \ \  Products ``$\circ$" and ``$\cdot_2$" in the relative Homotopy group. \quad }

\vspace{0.25cm}

On the other hand, the map defined on the gluing  $S_1$ and $S_2$ along the left edges $\{0\}\! \times \! [0,1]$ of $S_1$ and $S_2$ with opposite orientations (where the identified left edges is considered as the base part) gives a representative of the product $[H_1]\cdot_2 [H_2]$ of $[H_1]$ and $[H_2]$ in the relative homotopy group.  

Since the identifications of points $(0,1)\in S_1$ 
and $(0,0)\in S_2$ are common to both gluings, we see that the two maps are homotopic to each other and give the same class in the relative homotopy group. 

Compatibility of two inverse structures \eqref{eq:twoProducts} is shown similarly.
\end{proof}

\noindent
In rest of the present paper, {\it we fix  the relative realization morphism  $\chi$ \eqref{eq:immersion} by fixing} (1) {\it and} (2). 
As consequences, we obtain the following.

\medskip
\noindent
\begin{prop}
\label{PropB}
 By the relative realization morphism $\chi$ \eqref{eq:immersion} (recall  Footnote 20), the group $[x,x]$ w.r.t.\ the product ``\ $\circ$" for any $x\in \mathcal{F}$ is canonically isomorphic to the second homotopy group $[e,e]=\pi_2(X,*)$ w.r.t.\ the product ``\ $\cdot_2$".  In particular, it is an abelian group. 
Then the isomorphism  induces a left action of $\pi_2(X,*)$ on $[x,y]$ for any $y\in \mathcal{F}$. If $[x,y]\not=\emptyset$, the  action of $\pi_2(X,*)$ on $[x,y]$ is simply transitive.  
\end{prop}
\begin{proof} The first half follows from $[x,x]\simeq \partial_1^{-1}(xx^{-1})=\partial_1^{-1}(e) \simeq [e,e]$ and from the commutativity in \eqref{eq:twoProducts}.
The latter half follows from the fact that  the second homotopy group $\pi_2(X,*)$ is embedded in the second relative homotopy group
as a normal subgroup so that $\partial_1^{-1}(xy^{-1})$ forms its single coset of the normal subgroup.
\end{proof}
\begin{cor}  Suppose $[x,y]\!\not=\!\emptyset$. 

\noindent
{\rm (1)} For any $[H]\!\in\! [x,y]$, the adjoint action\!\! 
\begin{equation}
\label{eq:AdjointIsomorphism}
\mathrm{Ad}(H): [x,x] \longrightarrow [y,y],\quad [G] \mapsto [1/H]\circ[G]\circ [H]
\end{equation}
is an isomorphism with respect to the product “\(\circ\)” and is independent of the choice of \([H]\).

\noindent
\rm{(2)}  The adjoint isomorphism $\mathrm{Ad}(H)$ \eqref{eq:AdjointIsomorphism} commutes with the identity map of \(\pi_2(X,*)\) under the relative realization morphism $\chi$
\begin{equation}
\begin{array}{cccc}
&[x,x] & \overset{\mathrm{Ad}(H)}{\longrightarrow} & [y,y]\\
&\chi \downarrow \ \ && \ \ \downarrow \chi \\
&\pi_2(X,*) & \overset{id}{\longrightarrow} & \pi_2(X,*)
\end{array}
\end{equation}


\end{cor}
  \begin{proof} (1) The isomorphism follows from the fact that $\rm{Ad}(1/H)$ gives the inverse action.
The independence of the choice of  $[H]$  follows from the equality: 
  $$
 \vspace{-0.1cm}
  \mathrm{Ad}(H_1)[G]=[1/H_2]\circ([H_2/H_1][G][H_1/H_2])\circ[H_2]=\mathrm{Ad}(H_2)[G]
 \vspace{-0.1cm}
  $$
where $[G]$ and $[H_1/H_2]$ as elements of $\pi_2(X,*)$ are commutative. 

(2) By the relative realization morphism $\chi$, the  $[G]\in [x,x]$ and $\mathrm{Ad}(H)([G])$ are realized as in the following figure:

\vspace{0.2cm}
\hspace{2.3cm}
\includegraphics[width=0.38\textwidth]{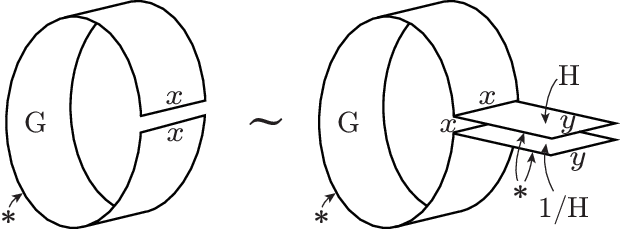}

\centerline{Figure 3.8: \ \ Relative realization of the adjoint action of $H$ on $G$. }

\noindent
where the edges $*$ in the frontside shrink to the base point $*\in X$, and the edges in the backside shrink to points which are not necessarily in the base point.  It may be clear that we can retract the union of the faces  $[H]\!\cup\! [1/H]$ in the diagram on the right to the edge $x$ to obtain the homotopy equivalence $\chi(\mathrm{Ad}(H)[G]) \sim \chi([G])$ in  $X$.\!\!\!\!\!
\end{proof}  
  
 \begin{rem}
 {\rm Similar to Proposition \ref{PropB}, we get the right action of  $\pi_2(X,*)$ on $[x,y]$ by changing the orientation of $[0,1]^2$ in \eqref{eq:Relativisation}.  Then,  $\chi$  \eqref{eq:immersion} gets  different target set $\partial_1^{-1}(yx^{-1})$. Details are left to the reader. 
 }
 \end{rem}
 

\subsection{\normalsize A geometric description of the monoid $A^+(L,\mathcal{R}_m)$}


{\bf 3.2.1\quad  Geometric fundamental relations $\mathcal{R}_g$ and  $\mathcal{R}_g$-sets}. 

We lift the algebraic fundamental relations $\mathcal{R}$ \eqref{eq:SemiPositive} to  a system of geometric fundamental relations  $\mathcal{R}_g$. 
We first identify the free monoid 
\begin{equation}
\label{eq:PositiveMonoid}
\mathcal{F}^+\ :=\ \text{free sub-monoid of $\mathcal{F}$ generated by } \widetilde{L}=\{\tilde a_\lambda\}_{\lambda\in \Lambda}
\end{equation}
 with the free monoid $L^*\!:=\! \cup_{k\in\mathbb{Z}_{\ge0}}^\infty L^k$ using \eqref{eq:CalL}. So, a word $p$  in $L^*$ shall also induce an element $\mathcal{F}\simeq \pi_1(W,*)$, denoted by the same $p$. 

A {\it system of geometric fundamental relations} is a ``lifting'' of the system of algebraic fundamental relations  $\mathcal{R}_m =\{p_\mu\sim q_\mu\}_{\mu\in M}$ defined as follows.

\medskip
\begin{defn}
\label{GeoFundRel}
A system of  geometric fundamental relations $\mathcal{R}_g=\{\mathcal{R}_g(p_\mu,q_\mu)\}_{\mu\in M}$ is a choice  
of {\it non-empty} subsets 
\begin{equation}
\label{eq:GeometricFundamentalRelation}
\mathcal{R}_g(p_\mu,q_\mu)\subset [p_\mu,q_\mu] 
\end{equation} 
  for all $\mu\in M$.
We also define 
  \begin{equation}
\label{eq:Reflexive}
\mathcal{R}_g(q_\mu,p_\mu):=1/\mathcal{R}_g(p_\mu,q_\mu) \subset  [q_\mu,p_\mu]
\end{equation}
If  $\#\mathcal{R}_g(p_\mu,q_\mu)=1$ for $\mu\in M$, then we call the system a {\it singleton lifting}. 

\end{defn}

 \medskip
\noindent
{\it Note} 8. The only condition for the choice of the subsets \eqref{eq:GeometricFundamentalRelation} is that it is non-empty. The choice does not affect  Proposition \ref{GeometricMonoid}. However, the choice strongly influences  the size of  local and global inertia groups (Definition \ref{LocalInertia}), which should preferably be  small (see  {\bf Question 3.2.3.} at the end of \S3.2 and the warning at the end of \S5.2).

\medskip
\noindent
{\bf Construction of $\mathcal{R}_g$-set.} 

Depending on  the choice of $\mathcal{R}_g$, we introduce a subset 
\begin{equation}
\label{eq:R^gSet}
\qquad \qquad [x,y]_{\mathcal{R}_g} \  \subset \   [x,y]   \qquad for \ \  x,y\in \mathcal{F}^+
\end{equation}
called {\it $\mathcal{R}_g$-set}, generated by $\mathcal{R}_g$ as the union of the sets obtained by  the following procedures 
%
 a) and b).

\smallskip
\noindent
 a) For all $x\in \mathcal{F}^+$,  the identity element  $1_x \in [x,x]$ belongs to $[x,x]_{\mathcal{R}_g}$.

\noindent
b)  For $x,y \in \mathcal{F}^+$, suppose there exists a finite sequence $z_0,z_1,\cdots,z_n$ for $n\in \mathbb{Z}_{>0}$ of elements of $\mathcal{F}^+$ such that (i) $z_0=x$ and $z_n=y$, and (ii) for $i=1,\cdots,n$, there are words $a_i,b_i,c_i,d_i\in \mathcal{F}^+$  such that $z_{i-1}\! =\! a_i b_id_i$, $z_i\! =\! a_ic_id_i$ and either $\{b_i\! \sim \! c_i\}\! \in\! \mathcal{R}_m$ or $\{c_i\! \sim \! b_i\}\! \in\! \mathcal{R}_m$. 

Then, we consider  the following product set 
\begin{equation}
\label{eq:EquivalenceGenerated}
\{id_{a_1}\cdot \mathcal{R}_g(b_1,c_1) \cdot id_{d_1}\}\circ \{id_{a_2} \cdot \mathcal{R}_g(b_2,c_2) \cdot id_{d_2}\}\circ \cdots \circ \{id_{a_n} \cdot \mathcal{R}_g(b_n,c_n) \cdot id_{d_n}\}.
\end{equation}

\begin{rem}
{\rm The above construction of  $\mathcal{R}_g$-sets 
$\{[x,y]_{\mathcal{R}_g}\}_{x,y\in \mathcal{F}^+}$ generated from the geometric relations $\mathcal{R}_g$ is parallel to the construction of algebraic equivalence relation ``$\sim$" generated from the algebraic relations $\mathcal{R}_m$ described at the definition of  the monoid $A^+(L,\mathcal{R}_m)$ (see the description (1) and (2) following \eqref{eq:Monoid}). 
Actually, the $\mathcal{R}_g$-set $[x,y]_{\mathcal{R}_g}$  is a subset of $[x,y]$ 
 obtained by replacing each e-transformation $apd\sim aqd$ or $aqd\sim apd$ (=an edge in the graph $G(L^*,\mathcal{R}_m)$ \eqref{eq:MonoidGraph3}, cf.\ also \eqref{eq:LocalInertia})  by the set of homotopy maps $id_a\cdot\mathcal{R}_g(p,q)\cdot id_d$ or  $id_a\cdot\mathcal{R}_g(q,p)\cdot id_d$.}
\end{rem}

\medskip
Thus the following Proposition may be almost tautological. 
 
\medskip
\begin{prop}
\label{GeometricMonoid}
 1. Let us define a binary relation on $\mathcal{F}^+$. Set   
\begin{equation}
\label{eq:RgEquivalence}
\quad  x \sim_{\mathcal{R}_g} y  \quad \Leftrightarrow_{def}   \quad [x,y]_{\mathcal{R}_g}\not=\emptyset   \quad \text{ for } \ x,y\in \mathcal{F}^+.
\!\!\!\!\!\!
\end{equation}
Then the  relation $\sim_{\mathcal{R}_g}$ on $\mathcal{F}^+$ 
 is an equivalence relation.
 
 \medskip
 2. The relation  $\sim_{\mathcal{R}_g}$ is compatible with product structure ``\ $\cdot$" on $\mathcal{F}^+$ so that it induces a monoid structure on the quotient set  $\mathcal{F}^+/\! \sim_{\mathcal{R}_g}$.  
 
 \medskip
3. The quotient monoid is naturally isomorphic to $A^+(L,\mathcal{R}_m)$:
\begin{equation}
\label{eq:GeometricArtinMonoid}
 \mathcal{F}^+/\! \sim_{\mathcal{R}_g} \quad \simeq \quad A^+(L,\mathcal{R}_m). 
\end{equation}
 \end{prop}
\begin{proof} 1. 
We have to show the following a), b) and c).

 a) For all $x\in \mathcal{F}^+$, the $\mathcal{R}_g$-set $[x,x]_{\mathcal{R}_g}$ contains  $1_x$.

\noindent
({\it Proof.}  This follows from {\bf Construction of $\mathcal{R}_g$-set} a). \hfill{ $\Box$)}

\smallskip
b) For all $x,y\in \mathcal{F}^+$, the $\mathcal{R}_g$-sets are closed under the ``$\circ$"-inversion: 
$$
[x,y]_{\mathcal{R}_g} \ \overset{\sim}{\longrightarrow} \ [y,x]_{\mathcal{R}_g}, \quad  [H] \ \mapsto \ [1/H]. 
$$
({\it Proof.}  Let $[H]\in [x,y]$ be presented by  a sequence $z_0,\cdots,z_n$ connecting $x$ to $y$ in the sense of Definition b).  Consider the element $[H']\in [y,x]$ which is presented by the reversed sequence $z_n,\cdots,z_0$  connecting $y$ to $x$ in the sense of Definition b) (use  \eqref{eq:Reflexive}). Let us show $[H]'=[1/H]$. Because of the associativity of the product ``\ $\circ$", we have only to check this for the sequence of length 1, i.e for an e-transformation. Then the proof is reduced to the condition \eqref{eq:Reflexive} on the geometric fundamental relations.   \hfill{ $\Box$)}

   c)  For $x,y,z\in \mathcal{F}^+$, the $\mathcal{R}_g$-sets are closed under composition ``$\ \circ$":
\begin{equation}
\vspace{-0.1cm}
\label{eq:circ2}
\circ\ :\ [x,y]_{\mathcal{R}_g} \times [y,z]_{\mathcal{R}_g} \ \rightarrow\  [x,z]_{\mathcal{R}_g}.
\vspace{-0.cm}
\end{equation}
\noindent
({\it Proof.}   If $z_0,\cdots,z_n$ and $w_0,\cdots,w_m$ are the sequences connecting $x$ to $y$ and $y$ to $z$, respectively, in the sense of b) of {\bf Construction of $\mathcal{R}_g$-set}. Then, $z_0,\cdots,z_n=$ $y=w_0,\cdots,w_m$ is a sequence connecting $x$ and $z$ in the same sense.  
\hfill{ $\Box$)}

2.  For $[H]\in[x,y]_{\mathcal{R}_g}$ and $[H']\in[z,w]_{\mathcal{R}_g}$, the concatenation, say $H\cdot H'$, of them in the first variable (recall Proposition \ref{PropA}(d)) 
gives rise to an element $[H\cdot H']$ belonging to $[x\cdot z,y\cdot w]_{\mathcal{R}_g}$.
\begin{equation}
\vspace{-0.1cm}
\label{eq:cdot2}
\cdot \ :\ [x,y]_{\mathcal{R}_g} \times [z,w]_{\mathcal{R}_g} \ \rightarrow\  [xz,yw]_{\mathcal{R}_g}.
\end{equation}
({\it Proof.}   If $u_0,\cdots,u_n$ and $v_0,\cdots,v_m$ are the sequences connecting $x$ to $y$ and $z$ to $w$, respectively, in the sense of b) of {\bf Construction of $\mathcal{R}_g$-set}. Then, $xz=u_0z,\cdots,u_nz=yz=yv_0,\cdots,yv_m=yw$ is a sequence connecting $xz$ and $yw$ which gives rise to some $\mathcal{R}_g$-set in $[xz,yw]_{\mathcal{R}_g}$. \hfill{ $\Box$)}

3.   For $x,y\in \mathcal{F}^+$, we have only to note the equivalences:
\vspace{-0.2cm}
$$
\vspace{-0.1cm}
\begin{array}{rll}
 x \sim_{\mathcal{R}_g} y &
\Leftrightarrow_{def} & [x,y]_{\mathcal{R}_g}\not=\emptyset \\
&\Leftrightarrow & \text{Either $x=y$ and $1_x\in [x,x]_{\mathcal{R}_g}$, or, there exists a seq-}\\
&& \text{uence $z_0,\cdots,z_n$ connecting $x$ and $y$ satisfying b) of  }\\
&& \text{{\bf Construction of $\mathcal{R}_g$-set}.}\\
&\Leftrightarrow & \text{Either $x=y$ in $L^*$, or, there exists a finite sequence}\\
&& \text{$z_0,\cdots,z_n$ connecting $x$ and $y$ where  $z_{i}$ is an element-}\\
&& \text{ary transformation of $z_{i-1}$ ($i=1,\cdots,n$).}\\
&\Leftrightarrow_{def} & x \sim y 
\end{array}
\vspace{-0.6cm}
$$
\end{proof}
\noindent
{\it Note}  9. Proposition \ref{GeometricMonoid} 3.\ implies, in particular, the relation $ \sim_{\mathcal{R}_g}$ does not depend on the choice of a geometric fundamental relations $\mathcal{R}_g$.

\smallskip
\noindent
{\it Note}  10.  Using the ``$\circ$"-invertibility of $\mathcal{R}_g$-sets (see 1.\ b) in the proof), we observe that a left (resp.\ right)  ``$\circ$"-multiplication \eqref{eq:circ} of an element $[H_1]\in [x,y]_{\mathcal{R}_g}$ (resp.\ $[H_2]\in [y,z]_{\mathcal{R}_g}$) induces bijections between $\mathcal{R}_g$-sets:
\begin{equation}
\vspace{-0.1cm}
\label{eq:R-invert}
[H_1]\circ\ :\  [y,z]_{\mathcal{R}_g} \simeq [x,z]_{\mathcal{R}_g} 
\quad  \text{and} \quad  
\circ [H_2]\ :\  [x,y]_{\mathcal{R}_g} \simeq [x,z]_{\mathcal{R}_g} .
\end{equation}

However, ``$\cdot$"-multiplications \eqref{eq:cdot} either from the left or right may not induce a bijection. In other words, the inverse operation (division) of  ``$\cdot$" may not exist in the frame of $[\cdot,\cdot]_{R_g}$ (cf.\ \S5.1 Division morphism). 

\medskip
\noindent
{\bf 3.2.2 \quad Local inertia group}

\smallskip 
\noindent
We introduce local {\it inertia groups} associated with the geometric fundamental relations $\mathcal{R}_g$. It is  used as a measure of {\it how sharp (or, how loose) the geometric fundamental relations $\mathcal{R}_g$ are}. Namely, $\mathcal{R}_g$ is considered to be sharper (or looser)  if the following inertia group is smaller (or larger).


\begin{defn}
\label{LocalInertia}
The set  $[x,x]_{\mathcal{R}_g}$ is a subgroup of $[x,x]$ w.r.t.\ the product ``\ $\circ$" and  shall be called the {\it local inertia group} at $x\in \!\mathcal{F}^+$ with respect to  the geometric fundamental relations $\mathcal{R}_g$.  

\end{defn}

The definition of an inertia group is paraphrased and symbolically expressed with the  help of the graph $G(L^*,\mathcal{R}_m)$ \eqref{eq:MonoidGraph3} as follows.
\vspace{-0.07cm}
\begin{equation}
\label{eq:LocalInertia}
Loop(G(L^*,\mathcal{R}_m),x)_{ \mathcal{R}_g}\quad \to \quad  [x,x]_{\mathcal{R}_g} \quad \to \ 0.
\vspace{-0.cm}
\end{equation}

\smallskip
{\bf Explanation:}  We mean by $Loop(G(L^*,\mathcal{R}_m ),x)$ the monoid consisting of closed loops in the graph $G(L^*,\mathcal{R}_m)$ \eqref{eq:MonoidGraph3}  starting from and ending at the fixed vertex $x\in L^*$, where the product is given by concatenation of loops. 
Then $Loop(G(L^*,\mathcal{R}_m),x)_{\mathcal{R}_g}$  is
an extension of the monoid  $Loop(G(L^*,\mathcal{R}_m ),x)$, given as the union 
$  \bigcup_{P\in Loop(G(L^*,\mathcal{R}_m),x)} P_{\mathcal{R}_g}$
where $P_{\mathcal{R}_g}$ is the set of pairs of $P$ and an element of the subset \eqref{eq:EquivalenceGenerated}  of $[x,x]$ 
associated with the sequence $z_0=x,z_1,\cdots,z_n=x$ giving the loop $P$. The product is defined as $(P, u)\times (Q,v) \mapsto (PQ,u\circ v)$ for $(P,u)\in P_{\mathcal{R}_g}$, $(Q,v)\in Q_{\mathcal{R}_g}$ and $P,Q\in Loop(G(L^*,\mathcal{R}_m ),x)$. 
Then the  morphism $Loop(G(L^*,\mathcal{R}_m),x)_{\mathcal{R}_g} \to [x,x]_{\mathcal{R}_g}$ \eqref{eq:LocalInertia} is given by $(P,u) \mapsto u$.




Suppose that the lifting $\mathcal{R}_g$ is  singleton  (i.e.\ $\#(\mathcal{R}_g(p_\mu,q_\mu))\! =\! 1$ for $\mu\! \in\! M$),\! then $P_{ \mathcal{R}_g}$ consists of a single element whose second part depends only on the  class  $P$ in $\pi_1(G(L^*,\mathcal{R}_m),x)$ (any homotopy equivalence in  $G(L^*,\mathcal{R}_m)$ is given by compositions of deletion or insertion of consecutive e-transformation  and its inverse. Then associated geometric singleton relations cancel each other: $\mathcal{R}_g(p,q)\circ \mathcal{R}_g(q,p)\! =\! \{1_p\}$ so that it does not affect  the set \eqref{eq:EquivalenceGenerated}. So, we have the surjective homomorphism:
\vspace{-0.08cm}
\begin{equation}
\label{eq:SingletonLocalInertia}
\quad \pi_1(G(L^*,\mathcal{R}_m),x) \quad \to \quad  [x,x]_{\mathcal{R}_g} \quad \to \ 0.
\end{equation}


\smallskip
\begin{prop}
\label{InertiaAction}
 {\it 
Suppose the set $[x,y]_{\mathcal{R}_g}$ for $x,y\in \mathcal{F}^+$ is non-empty.\!

\noindent
(1) The local inertia group $[x,x]_{\mathcal{R}_g}$ (resp.\ $[y,y]_{\mathcal{R}_g}$) acts simply and transitively on $[x,y]_{\mathcal{R}_g}$ from the left (resp.\ right). 

\smallskip
\noindent
(2) Suppose $[x,y]_{\mathcal{R}_g}\not=\emptyset$. Then the adjoint isomorphism $\mathrm{Ad}(H)$ \eqref{eq:AdjointIsomorphism} for any $[H]\in [x,y]$ induces the isomorphism of the local inertia groups $[x,x]_{\mathcal{R}_g} \simeq [y,y]_{\mathcal{R}_g}$ which is independent of the choice of $[H]$.\ 

\smallskip
\noindent
(3)  Under the relative realization morphism $\chi$ \eqref{eq:immersion}
,  
\vspace{-0.07cm}
we have the equality}
\begin{equation}
\label{}
\chi( [x,x]_{\mathcal{R}_g}) =\chi( [y,y]_{\mathcal{R}_g}) .
\vspace{-0.07cm}
\end{equation}
\end{prop}
\begin{proof}
(1)  The simplicity of the action of $[x,x]_{\mathcal{R}_g}$ is a consequence of the simplicity of the action of  $[x,x]$ on $[x,y]$.

 Transitivity: for any two elements $[H_1], [H_2]\in [x,y]_{\mathcal{R}_g}$, let us show the existence of a unique element $g\in [x,x]_{\mathcal{R}_g}$ such that $[H_1]=g[H_2]$. By the definition of $[x,y]_{\mathcal{R}_g}$, there is sequence $z_0,\cdots,z_n$ connecting $x$ and $y$ such that $H_1$ is an element of 
$\{id_{a_1}\cdot \mathcal{R}_g(b_1,c_1) \cdot id_{d_1}\}\circ \{id_{a_2} \cdot \mathcal{R}_g(b_2,c_2) \cdot id_{d_2}\}\circ \cdots \circ \{id_{a_n} \cdot \mathcal{R}_g(b_n,c_n) \cdot id_{d_n}\}$.
Similarly, there exists a sequence $w_0,\cdots,w_m$ connecting $x$ and $y$ 
and $e_i,f_i,g_i,h_i\in \mathcal{F}^+$ ($i=1,\cdots,m$) such that $H_2$ is an element of 
$\{id_{e_1}\cdot \mathcal{R}_g(f_1,g_1) \cdot id_{h_1}\}\circ \{id_{e_2} \cdot \mathcal{R}_g(f_2,g_2) \cdot id_{h_2}\}\circ \cdots \circ \{id_{e_m} \cdot \mathcal{R}_g(f_m,g_m) \cdot id_{h_m}\}$, respectively.
Then, we put 
$g:=[H_1] [/H_2]\in 
\{id_{a_1}\cdot \mathcal{R}_g(b_1,c_1) \cdot id_{d_1}\}\circ \{id_{a_2} \cdot \mathcal{R}_g(b_2,c_2) \cdot id_{d_2}\}\circ \cdots \circ \{id_{a_n} \cdot \mathcal{R}_g(b_n,c_n) \cdot id_{d_n}\}
\circ \{id_{e_m} \cdot \mathcal{R}_g(g_m,f_m) \cdot id_{h_m}\}
\circ \cdots
\circ \{id_{e_2} \cdot \mathcal{R}_g(g_2,f_2) \cdot id_{h_2}\}
\circ \{id_{e_1}\cdot \mathcal{R}_g(g_1,f_1) \cdot id_{h_1}\}
\subset [x,x]_{\mathcal{R}_g}$.

\smallskip
\noindent
(2)  If $[x,y]_{\mathcal{R}_g}\not=\emptyset$, then for an element $[H]\in [x,y]_{\mathcal{R}_g}$, the adjoint isomorphism $\mathrm{Ad}(H)$ \eqref{eq:AdjointIsomorphism} preserves the $\mathcal{R}_g$-subsets $[x,x]_{\mathcal{R}_g}\to [y,y]_{\mathcal{R}_g}$.   Using $\mathrm{Ad}(1/H)$, we see that it is a bijection. Then, we recall Corollary 3.7 that, for any $H'\in [x,y]$ (without the restriction to be in the $\mathcal{R}_g$-set), the two actions $\mathrm{Ad}(H)$ and $\mathrm{Ad}(H')$ are the same. 

\smallskip
\noindent
(3) This follows immediately from (2) and Corollary 3.7 (2). 
\end{proof}

\noindent
{\bf Question 3.2.3.}  When are the local inertia groups trivial?  Find conditions for the geometric liftings $(W,\widetilde{L},\mathcal{R}_g)$ such that the groups are trivial.

\section{Twin of non-cancellative tuples}

We prepare two concepts: a non-cancellative tuple (NC-tuple) and a twin  of NC-tuples (NC-twin) for any abstractly given monoid $A^+$ (see Definition \ref{NCtuple}). The subject  seems to be rich and interesting, and it may be worthwhile to develop a general framework for these concepts. Here we restrict our attention to two typical examples of line arrangements.

\subsection{\normalsize Non-cancellative tuples and non-cancellative twins}
\noindent
Let $A^+$ be any monoid, where we use the notation ``$\sim$" to express the\! equality of elements in $A^+$ for the same reason as in Footnote 9.

\begin{defn} 
\label{NCtuple}
1. Consider a pair $(b,c)\!\in\!$ {\small \! $A^+ \! \times \!A^+$} such that $b\not\sim c$.  We call  $(a,b,c,d) \in (A^+)^4$  a {\it non-cancellative tuple (NC-tuple for short) over the kernel  $(b,c)$} if it satisfies the following {\it non-cancellative  relation} (NC-relation):
\begin{equation}
\label{eq:NCRelation}
a\ b\ d \ \sim \ a\ c\ d .
\vspace{-0.2cm}
\end{equation}
We set 
\vspace{-0.2cm}
\begin{equation}
\label{eq:ncp}
NC_{A^+}(b,c):=\{(a
,d)\in (A^+)^2 \mid \text{solution of the equation\ } \eqref{eq:NCRelation} \}\!\!\!
\vspace{-0.1cm}
\end{equation}
or $NC(b,c)$ if there is no confusion. If the monoid is graded (cf.\ {\it Note} 3 in \S2.1), then we shall call $\deg(abd)=\deg(acd)$ the {\it degree of the NC-tuple} and $\deg(b)=\deg(c)$ the {\it degree of the kernel}. 
 If $a=1$ (resp.\ $d=1$), we call $(1,b,c,d)$ (resp.\ $(a,b,c,1)$) a right (resp.\ left) NC-tuple.

\smallskip
\noindent
2.  An NC-tuple $(a
,d)\in NC(b,c)$ is called {\it reduced} if whenever there are $a',d'\in A^+$ such that $a' \! \mid^r \! a$, $d' \!\mid^l \! d$ and $(a',b,c,d') \!\in \! NC(b,c)$, 
then $a\sim a'$ and $d\sim d'$.\!\!

\smallskip
\noindent
3.  The kernel $(b,c)$ of an NC-tuple is called {\it minimal} if whenever there are $a',d'\in A^+$ such that $a'\mid^l b$, $a'\mid^l c$, $d'\mid^r b$ and $d'\mid^r c$, then $a'\sim d'\sim1$.

 
 \smallskip
 \noindent
 4.  An ordered pair of  NC-tuples 
\vspace{-0.1cm}
\begin{equation}
\label{eq:Twin}
\tau:= \big( (a_1,b,c,d_1),(a_2,b,c,d_2)\big) 
 \end{equation}
over  a common kernel $(b,c)$ is called an {\it NC-twin} over the kernel $(b,c)$.\footnote
{The ``NC-twin"  appeared first in the study of elliptic Artin monoid \cite{S5} where NC-tuples appeared as pairs of  ``left'' and ``right'' NC-tuples.   Similar phenomena are observed  in the line arrangement $\mathcal{A}_3$ (see \eqref{eq:Twin23}, \eqref{eq:Twin31} and \eqref{eq:Twin12}). 
The appearance of left and right NC-tuples simultaneously  seems not an accident. However, at this  stage of the study, we do not know a reason why such pairs exist, and cannot characterize them. Hence, we tentatively call any  pair of NC- tuples  defined over the same kernel an {\it NC-twin}. We expect a future redefinition.\!\!\!}
\end{defn}

\medskip
\noindent
{\it Note}  11.1  By definition,  $NC_{A^+}(b,c)\!\simeq\! NC_{A^+}(c,b)$.\  But,  an NC-tuple has  the {\it orientation} of the order between  $b$ and $c$. Also, a  twin, as an ordered pair of NC-tuples, has  the other  {\it orientation} (see Remark \ref{twinReverse}).\!\!\!

\smallskip
\noindent
11.2  Among the monoids $A^+(L,\mathcal{R}_m)$ in  \S2.2 Examples,  
 the cases  1., 2.\ i) \cite{B-S, D,Gar}, ii) \cite{Du}, 3.\ $B_{ii}$ \cite{S-I} and 4 \cite{Be1} are cancellative, and  the cases  2.\ iv)  for simply-laced case \cite{S5}, and 5 are non-cancellative. 

\smallskip
\begin{rem} 
\label{Examples}
{\rm   The above conditions 2 and 3  are intended to capture the notion of  ``atoms" among  NC-tuples (cf.\ Remark \ref{GeometricAtoms} and Remark \ref{reducedness}). To define an ``atom'', we may need to add one more condition:  ``indecomposability''.  Experimentally, we  formulate two conditions as follows.

\smallskip
\noindent
{ {\bf K-indecomposability:}}  {\footnotesize an NC-tuple over the kernel $(b,c)$ is called K-indecomposable if there does  not exist NC-tuples over the kernels {\footnotesize $(b_1,c_1)$} and {\footnotesize $(b_2,c_2)$} and an element {\footnotesize $f\in A^+$} such that {\footnotesize $b\sim b_1fb_2$} and {\footnotesize $c\sim c_1fc_2$}.}

\smallskip
\noindent
{ {\bf T-indecomposability:}}  {\footnotesize an NC-tuple $(a,b,c,d)$ is called T-indecomposable if there does  not exist NC-tuples {\footnotesize $(a_1,b_1,c_1,d_1)$} and {\footnotesize $(a_2,b_2,c_2,d_2)$} and an element {\footnotesize $g\in A^+$} such that {\footnotesize $abd\sim a_1b_1d_1ga_2b_2d_2$} and {\footnotesize $acd\sim a_1c_1d_1ga_2c_2d_2$}.}

\smallskip
In Appendix III, we apply these conditions to NC-tuples in examples (a) and  (b) to examine how these conditions reduce the number of cases.
} 
\end{rem}




\subsection{\normalsize NC-tuples and NC-twins for  line arrangements} 

\noindent
We analyze two cases (a) and (b)  of  Yoshinaga's  positive homogeneous presentation \eqref{eq:linesPositiveHomogeneous}  of line arrangements  in $\mathbb{R}^2$  from the point of view of NC-tuples and NC-twins. 
As we will observe, there is  an interesting  ``resonance'' between geometry and algebra:

\smallskip
(a)   Existence of $\pi_2$-classes  \ $\Leftrightarrow$ \ Existence of  NC-twins,

\smallskip
(b)   Non-existence of a $\pi_2$-class  \ $\Leftrightarrow$ \ Non-existence of an NC-twin,

\smallskip
\noindent
The purpose of the present paper is to give a topological foundation to the 
observation of the resonance.

\smallskip
In the following list of  NC-tuples,  we use the concept of the {\it length}  on the monoid (recall \S2.2 6. {\it Note} 7.).  We display one representative from each pair obtained by reversing the kernel \((b,c)\leftrightarrow(c,b)\).

\medskip
\noindent
(a) Three generic lines $\mathcal{A}_3=\{H_1, H_2, H_3\}$  in $\mathbb{R}^2$.

\smallskip
\noindent
The monoid  $A^+(L,\mathcal{R}_m(\mathcal{A}_3))$ has the presentation (\S2.2 Example 6  (a)):\!\!\!
\begin{equation}
\label{eq:3linesMonoid}
 \langle \tilde{\gamma}_1,\tilde{\gamma}_2,\tilde{\gamma}_3\mid \tilde{\gamma}_1\tilde{\gamma}_2\tilde{\gamma}_3 \!\overset{A}{\sim} \!
\tilde{\gamma}_1\tilde{\gamma}_3\tilde{\gamma}_2, 
\tilde{\gamma}_1\tilde{\gamma}_2\tilde{\gamma}_3 \! \overset{B}{\sim} \!
\tilde{\gamma}_3 \tilde{\gamma}_1\tilde{\gamma}_2,
\tilde{\gamma}_1\tilde{\gamma}_2\tilde{\gamma}_3 \! \overset{C}{\sim} \!
\tilde{\gamma}_2\tilde{\gamma}_3\tilde{\gamma}_1 \rangle^+\!\! 
\end{equation}

 We put labels $A,B$ and $C$ on the three fundamental relations ``$\sim$". We regard  the labeled relations $\overset{A}{\sim}$, $\overset{B}{\sim}$ and $\overset{ C}{\sim}$ as oriented, and denote the same relations with the opposite orientations by $\overset{\ \ 1/A}{\!\sim}$, $\overset{\ \ 1/B}{\!\sim}$ and $\overset{\ \ 1/C}{\!\sim}$ (see \eqref{eq:l=3}  - \eqref{eq:l=5} for their use). These additional data shall be used in the construction of second homotopy classes in \S6  Example (a).

\noindent
\begin{prop} 
\label{NCList}
The following is a partial list of reduced, minimal NC-tuples of length $\le 5$ of the monoid $A^+(L,\mathcal{R}_m(\mathcal{A}_3))$.\footnote
{
We use the notation $\gamma_1,{\gamma}_2,{\gamma}_3, etc.$ for the images in $A^+(L,\mathcal{R}_m)$ of the free generators $\tilde{\gamma}_1,\tilde{\gamma}_2,\tilde{\gamma}_3, etc.$ in $L$. Strictly speaking, this is not correct, since  $\gamma_1,{\gamma}_2,{\gamma}_3$ were already used to describe the elements in the group $A(L,\mathcal{R})$ (see \S2.2 Example 6). We allow this abuse, since we shall not use the notation for an element of the group in the sequel of the paper.\!\!\! 
}

\smallskip
\noindent
1.  Length=3 case.
\vspace{-0.2cm}
\begin{equation}
\label{eq:ANoncancellDeg3}
\begin{array}{rcl}
(\gamma_1,1) & \in & NC(\gamma_2\gamma_3,\gamma_3\gamma_2)\\
(1,\gamma_2) & \in & NC(\gamma_3\gamma_1,\gamma_1\gamma_3)\\
\end{array}
\vspace{-0.1cm}
\end{equation}
2. Length=4 case.
 \vspace{-0.2cm}
 \begin{equation}
\label{eq:ANoncancellDeg4}
(\gamma_1\gamma_3,1), (1,\gamma_3\gamma_2) \in  NC(\gamma_1\gamma_2,\gamma_2\gamma_1)
\vspace{-0.1cm}
\end{equation}
3.  Length=5 case.
 \vspace{-0.2cm}
\begin{equation}
\label{eq:ANoncancellDeg5}
\begin{array}{rcl}
(\gamma_1\gamma_2\gamma_3,1) & \in & NC(\gamma_3\gamma_1,\gamma_1\gamma_3)\\
(1,\gamma_1\gamma_2\gamma_3) & \in & NC(\gamma_2\gamma_3,\gamma_3\gamma_2)\\
(\gamma_1,\gamma_2) & \in & NC(\gamma_3\gamma_2\gamma_1,\gamma_2\gamma_1\gamma_3)\\
(\gamma_3,\gamma_2\gamma_3) & \in & NC(\gamma_1\gamma_2,\gamma_2\gamma_1)\\
\end{array}
\end{equation}
\end{prop}
\begin{proof}  We describe briefly how we list the tables \eqref{eq:ANoncancellDeg3}, \eqref{eq:ANoncancellDeg4} and \eqref{eq:ANoncancellDeg5}. In order to show  the NC-relation \eqref{eq:NCRelation}, we need to lift the elements in the monoid $A^+(L,\mathcal{R}_m(\mathcal{A}_3))$ to the free monoid $L^*$. 
Suppose an NC-tuple is given and consider a sequence of e-transformations giving the NC-relation \eqref{eq:NCRelation}. Due to the restriction of  length$\le 5$  of NC-tuples, any two fundamental words in the two successive e-transformations have non-trivial common substring of length 1 or 2.  We list such pairs as follows.
{\small
$$\begin{array}{rl}
length\ 1\ case:&(\tilde{\gamma}_1\tilde{\gamma}_2\tilde{\gamma}_3, \tilde{\gamma}_3 \tilde{\gamma}_1\tilde{\gamma}_2),  \ 
(\tilde{\gamma}_1 \tilde{\gamma}_3\tilde{\gamma}_2,   \tilde{\gamma}_2\tilde{\gamma}_3\tilde{\gamma}_1), \  
(\tilde{\gamma}_ 3\tilde{\gamma}_1\tilde{\gamma}_2,   \tilde{\gamma}_2\tilde{\gamma}_3\tilde{\gamma}_1), \\  
&( \tilde{\gamma}_2\tilde{\gamma}_3\tilde{\gamma}_1,  \tilde{\gamma}_1\tilde{\gamma}_2\tilde{\gamma}_3), \  
 (\tilde{\gamma}_2\tilde{\gamma}_3\tilde{\gamma}_1,  \tilde{\gamma}_1\tilde{\gamma}_3\tilde{\gamma}_2), \\
length\ 2\ case:&(\tilde{\gamma}_1\tilde{\gamma}_2\tilde{\gamma}_3, \tilde{\gamma}_2\tilde{\gamma}_3\tilde{\gamma}_1), \ 
(\tilde{\gamma}_2\tilde{\gamma}_3 \tilde{\gamma}_1,\tilde{\gamma}_3 \tilde{\gamma}_1\tilde{\gamma}_2),  \ 
 (\tilde{\gamma}_3 \tilde{\gamma}_1\tilde{\gamma}_2, \tilde{\gamma}_1\tilde{\gamma}_2\tilde{\gamma}_3).
 \end{array}
 $$
 }
 Then, it is an elementary but an involved task to list all cases  when these e-transformations together form a full sequence for a reduced and minimal NC-tuple. In Proposition \ref{NCList}, we give a complete list of NC-tuples for the degree 3 and 4, but for the case degree 5 the list is partial but covers sufficient NC-tuples  to form NC-twins to be studied in the present paper.  
 \end{proof}

\noindent
{\it Note} 12.  In Appendix III (a), we list all reduced and minimal NC-tuples up to degree 6, where we find a lot of new twins. 

\medskip
  For each NC-tuple in Proposition \ref{NCList},  we  give an explicit sequence of e-transformations giving the NC-relations \eqref{eq:NCRelation}. Actually, the sequence is unique for each NC-tuple, since we shall see in \S6.1, {\bf 2.\ Inertia group}  that the component of $G(L^*,\mathcal{R}_m(\mathcal{A}_3))$ presenting the tuple is a tree.
  
  We regard each e-transformation as an action of the symbols $\overset{A}{\sim}$, \ $\overset{B}{\sim}$, \ $\overset{C}{\sim}$ and their inverses $\overset{\ 1/A}{\!\sim}$, $\overset{\ 1/B}{\!\!\sim}$ or $\overset{\ 1/C}{\!\sim}$ on the words from the right. The subword on which the relation acts is explicitly indicated by separating 
  \vspace{-0.1cm}
  the part by spaces.
\begin{equation}
\label{eq:l=3}
\begin{array}{rl}
\tilde{\gamma}_1(\tilde{\gamma}_2\tilde{\gamma}_3)\!\!&\!\! 
=\tilde{\gamma}_1 \tilde{\gamma}_2 \tilde{\gamma}_3
 \overset{A}{\sim} \tilde{\gamma}_1\tilde{\gamma}_3 \tilde{\gamma}_2 
 =\tilde{\gamma}_1(\tilde{\gamma}_3 \tilde{\gamma}_2 )\\
(\tilde{\gamma}_3 \tilde{\gamma}_1) \tilde{\gamma}_2\!\!&\!\!
=\tilde{\gamma}_3\tilde{\gamma}_1\tilde{\gamma}_2
\overset{1/B}{\!\sim} \tilde{\gamma}_1\tilde{\gamma}_2\tilde{\gamma}_3
\overset{A}{ \sim} 
 \tilde{\gamma}_1\tilde{\gamma}_3\tilde{\gamma}_2
 = (\tilde{\gamma}_1\tilde{\gamma}_3)\tilde{\gamma}_2 
\end{array}
\end{equation}
\vspace{-0.4cm}
\begin{equation}
\label{eq:l=4}
\begin{array}{rl}
\tilde{\gamma}_1\tilde{\gamma}_3 (\tilde{\gamma}_1\tilde{\gamma}_2) 
\!\! &\!\! 
=\tilde{\gamma}_1 \ \tilde{\gamma}_3 \tilde{\gamma}_1\tilde{\gamma}_2
\overset{1/B}{\!\sim} 
\tilde{\gamma}_1 \ \tilde{\gamma}_1 \tilde{\gamma}_2\tilde{\gamma}_3
\overset{C}{\sim}
\tilde{\gamma}_1 \tilde{\gamma}_2 \tilde{\gamma}_3 \ \tilde{\gamma}_1 \\
\!\!&\!\! 
\overset{A}{ \sim }
\tilde{\gamma}_1\tilde{\gamma}_3\tilde{\gamma}_2 \tilde{\gamma}_1
 = 
\tilde{\gamma}_1\tilde{\gamma}_3  (\tilde{\gamma}_2  \tilde{\gamma}_1 )\\
(\tilde{\gamma}_1\tilde{\gamma}_2)\tilde{\gamma}_3 \tilde{\gamma}_2 \!\! &\!\! 
=\tilde{\gamma}_1\tilde{\gamma}_2\tilde{\gamma}_3 \ \tilde{\gamma}_2  
\overset{C}{\sim} 
\tilde{\gamma}_2 \ \tilde{\gamma}_3 \tilde{\gamma}_1 \tilde{\gamma}_2  
\overset{1/B}{\!\sim} 
\tilde{\gamma}_2 \ \tilde{\gamma}_1 \tilde{\gamma}_2 \tilde{\gamma}_3\\
\!\!&\!\! \overset{A}{\sim} 
\tilde{\gamma}_2 \tilde{\gamma}_1\tilde{\gamma}_3  \tilde{\gamma}_2
=(\tilde{\gamma}_2 \tilde{\gamma}_1) \tilde{\gamma}_3  \tilde{\gamma}_2
\end{array}
\end{equation}
\vspace{-0.4cm}
\begin{equation}
\label{eq:l=5}
\begin{array}{rl}
\tilde{\gamma}_1\tilde{\gamma}_2\tilde{\gamma}_3(\tilde{\gamma}_3\tilde{\gamma}_1) \!\!&\!\! 
=\tilde{\gamma}_1\tilde{\gamma}_2\tilde{\gamma}_3 \ \tilde{\gamma}_3\tilde{\gamma}_1 
\overset{A}{\sim} \tilde{\gamma}_1\tilde{\gamma}_3\ \tilde{\gamma}_2\tilde{\gamma}_3\tilde{\gamma}_1 \overset{1/C}{\!\sim} 
\tilde{\gamma}_1\ \tilde{\gamma}_3\tilde{\gamma}_1\tilde{\gamma}_2\ \tilde{\gamma}_3 \\
\!\!  &\!\! 
 \overset{1/B}{\!\sim} 
\tilde{\gamma}_1\ \tilde{\gamma}_1\tilde{\gamma}_2\tilde{\gamma}_3\ \tilde{\gamma}_3
\overset{C}{\sim} 
\tilde{\gamma}_1\tilde{\gamma}_2\tilde{\gamma}_3\tilde{\gamma}_1\tilde{\gamma}_3 
=\tilde{\gamma}_1\tilde{\gamma}_2\tilde{\gamma}_3(\tilde{\gamma}_1\tilde{\gamma}_3)  \\
(\tilde{\gamma}_2 \tilde{\gamma}_3) \tilde{\gamma}_1 \tilde{\gamma}_2 \tilde{\gamma}_3 \!\!&\!\!
= \tilde{\gamma}_2 \tilde{\gamma}_3 \  \tilde{\gamma}_1 \tilde{\gamma}_2 \tilde{\gamma}_3
\overset{A}{\sim} \tilde{\gamma}_2 \tilde{\gamma}_3 \tilde{\gamma}_1\ \tilde{\gamma}_3 \tilde{\gamma}_2
\overset{1/C}{\!\sim} \tilde{\gamma}_1 \tilde{\gamma}_2 \tilde{\gamma}_3\ \tilde{\gamma}_3 \tilde{\gamma}_2 \\
\!\!&\!\! 
\overset{B}{\sim} \tilde{\gamma}_3\ \tilde{\gamma}_1 \tilde{\gamma}_2 \tilde{\gamma}_3\ \tilde{\gamma}_2 
\overset{C}{\sim} \tilde{\gamma}_3 \tilde{\gamma}_2 \ \tilde{\gamma}_3 \tilde{\gamma}_1 \tilde{\gamma}_2 
\overset{1/B}{\!\sim} \tilde{\gamma}_3 \tilde{\gamma}_2 \tilde{\gamma}_1 \tilde{\gamma}_2 \tilde{\gamma}_3 \\
\!\!&\!\! 
=(\tilde{\gamma}_3 \tilde{\gamma}_2) \tilde{\gamma}_1 \tilde{\gamma}_2 \tilde{\gamma}_3 \\
\tilde{\gamma}_1(\tilde{\gamma}_3\tilde{\gamma}_2\tilde{\gamma}_1)\tilde{\gamma}_2 \!\!&\!\! 
=\tilde{\gamma}_1\tilde{\gamma}_3\tilde{\gamma}_2 \ \tilde{\gamma}_1\tilde{\gamma}_2 
\overset{1/A}{\sim} \tilde{\gamma}_1\tilde{\gamma}_2 \ \tilde{\gamma}_3\tilde{\gamma}_1\tilde{\gamma}_2 \overset{1/B}{\!\sim} 
\tilde{\gamma}_1 \tilde{\gamma}_2 \ \tilde{\gamma}_1\tilde{\gamma}_2\tilde{\gamma}_3 \\
\!\!  &\!\! 
\overset{A}{\sim}   \tilde{\gamma}_1\tilde{\gamma}_2\tilde{\gamma}_1\tilde{\gamma}_3\tilde{\gamma}_2   = \tilde{\gamma}_1(\tilde{\gamma}_2\tilde{\gamma}_1\tilde{\gamma}_3)\tilde{\gamma}_2  \\

\tilde{\gamma}_3( \tilde{\gamma}_1 \tilde{\gamma}_2) \tilde{\gamma}_2 \tilde{\gamma}_3 \!\!&\!\!
= \tilde{\gamma}_3 \tilde{\gamma}_1  \tilde{\gamma}_2\ \tilde{\gamma}_2 \tilde{\gamma}_3
\overset{1/B}{\sim} \tilde{\gamma}_1 \tilde{\gamma}_2 \tilde{\gamma}_3\ \tilde{\gamma}_2 \tilde{\gamma}_3
\overset{C}{\!\sim} \tilde{\gamma}_2 \tilde{\gamma}_3\  \tilde{\gamma}_1\tilde{\gamma}_2 \tilde{\gamma}_3 \\
\!\!&\!\! 
\overset{A}{\sim} \tilde{\gamma}_2 \tilde{\gamma}_3 \tilde{\gamma}_1\ \tilde{\gamma}_3 \tilde{\gamma}_2 
\overset{1/C}{\sim} \tilde{\gamma}_1 \tilde{\gamma}_2  \tilde{\gamma}_3 \ \tilde{\gamma}_3 \tilde{\gamma}_2 
\overset{B}{\!\sim} \tilde{\gamma}_3 \ \tilde{\gamma}_1 \tilde{\gamma}_2 \tilde{\gamma}_3 \ \tilde{\gamma}_2 \\
\!\!&\!\! 
\overset{C}{\sim} \tilde{\gamma}_3 \tilde{\gamma}_2 \ \tilde{\gamma}_3 \tilde{\gamma}_1\tilde{\gamma}_2 \overset{1/B}{\!\sim} \tilde{\gamma}_3 \tilde{\gamma}_2 \tilde{\gamma}_1 \tilde{\gamma}_2 \tilde{\gamma}_3 
=\tilde{\gamma}_3( \tilde{\gamma}_2 \tilde{\gamma}_1) \tilde{\gamma}_2 \tilde{\gamma}_3 .\\

\end{array}
\end{equation}
\noindent
\begin{defn}
\label{Weight}
 For 
 the chosen e-transformation sequence representing its NC-relation, we define the weight of the sequence to be
 the formal sum of the  number  of relations  $\overset{A}{\sim}, \overset{B}{\sim}$ and $\overset{C}{\sim}$ (counted with orientation as  an element of $\mathbb{Z}A\oplus\mathbb{Z}B\oplus \mathbb{Z}C$)  in the NC-equivalence sequence. We call it the {\it weight of the sequence} (we shall justify this concept  in Corollary 8.3 in Appendix II by showing that it depends only on the kernel of the tuple).
\end{defn}

\smallskip
We list all twin pairs among the list of NC-tuples in Proposition  \ref{NCList}.\!\!\! 

\smallskip
1. There exists a twin over the kernel  $(\gamma_2\gamma_3, \gamma_3\gamma_2)$: 
\begin{equation}
\label{eq:Twin23}
\tau_{23}:=\{(\gamma_1
,1), (1
,\gamma_1\gamma_2\gamma_3) \} \subset NC(\gamma_2\gamma_3, \gamma_3\gamma_2).
\end{equation}

2.  There exists a twin over the kernel $(\gamma_3\gamma_1, \gamma_1\gamma_3)$:
\begin{equation}
\label{eq:Twin31}
\tau_{31}:=\{(\gamma_1\gamma_2\gamma_3
,1), (1
,\gamma_2) \} \subset  NC(\gamma_3\gamma_1, \gamma_1\gamma_3).
\end{equation}

3. There is a triplet NC-tuples $(\gamma_1\gamma_3
,1), (1,\gamma_3\gamma_2)$ and $(\gamma_3,\gamma_2\gamma_3)$ over the kernel $(\gamma_1\gamma_2, \gamma_2\gamma_1)$. Therefore, there are three twins between them. We first choose the twin of the left- and right- NC-tuples:
\begin{equation}
\label{eq:Twin12}
\tau_{12}:=\{(\gamma_1\gamma_3
,1), (1
,\gamma_3\gamma_2)\}\subset  NC(\gamma_1\gamma_2, \gamma_2\gamma_1).
\end{equation}
and then consider additionally the following two twins:\footnote
{The reason for the  priority given to the pair of left and right NC-tuple is not logical but is based on Footnote 21, where we noticed there might be distinctions between twins.  
} 
\begin{equation}
\begin{array}{cccc}
\label{eq:Twin*}
\tau_{12}^*:=\{(\gamma_1\gamma_3
,1), (\gamma_3
,\gamma_2\gamma_3)\}\!  & \! \subset \! &NC(\gamma_1\gamma_2, \gamma_2\gamma_1) \\
\tau_{12}'^{*}:=\{(1,\gamma_3\gamma_2
), (\gamma_3
,\gamma_2\gamma_3)\}\!  & \! \subset \! &NC(\gamma_1\gamma_2, \gamma_2\gamma_1)
\end{array}
\vspace{-0.1cm}
\end{equation}
\begin{proof} We have only to show that the twins are balanced. This is merely to count the weights of the sequences in \eqref{eq:l=3}, \eqref{eq:l=4} and \eqref{eq:l=5} explicitly. However, we shall see in Corollary to Proposition \ref{YoshiIndependence} in Appendix II that, in general, the weight of an NC-tuple depends only on its kernel. 
\end{proof}

\begin{cor}
  The weights of the two NC-tuples  in an NC-twin in the above list  are equal (see Corollary \ref{balanced} in Appendix II for the reason). We shall call this property of the twins {\bf balanced}.
\end{cor}

\noindent
{\it Note} 13. Let us illustrate  an example of non-minimal kernel. Consider $(\tilde\gamma_1,1)\in NC(\tilde{\gamma}_1\tilde{\gamma}_2\tilde{\gamma}_3, \tilde{\gamma}_3\tilde{\gamma}_2\tilde{\gamma}_1)$ (since $\tilde{\gamma}_1\tilde{\gamma}_1\tilde{\gamma}_2\tilde{\gamma}_3\overset{C}{\sim} \tilde{\gamma}_1\tilde{\gamma}_2\tilde{\gamma}_3\tilde{\gamma}_1 \overset{A}{\sim} \tilde{\gamma}_1\tilde{\gamma}_3\tilde{\gamma}_2\tilde{\gamma}_1$). However,  since we have   
$\tilde{\gamma}_1\tilde{\gamma}_2\tilde{\gamma}_3 \overset{B}{\sim}
\tilde{\gamma}_3\tilde{\gamma}_1\tilde{\gamma}_2$, the kernel is left divisible by $ \tilde{\gamma}_3$.

\vspace{0.3cm}
\noindent
{\bf (b)}   $\mathrm{A}_3$-type arrangement $\mathcal{A}_{\rm{A}_3}$ in $\mathbb{R}^2$.

The monoid has the following presentation (\S2.2 Example 6(b)):
\begin{equation}
\begin{array}{l}
\label{eq:A3Monoid}
A^+(L,\mathcal{R}_m(\mathcal{A}_{\rm{A}_3})) \ \ :=\ \ \langle \tilde{\gamma}_1,\tilde{\gamma}_2,\tilde{\gamma}_3,\tilde{\gamma}_4,\tilde{\gamma}_5
\mid 
\\ 
\qquad\  \tilde{\gamma}_1\tilde{\gamma}_2\tilde{\gamma}_3  \tilde{\gamma}_4\tilde{\gamma}_5 \overset{1}{\sim} \tilde{\gamma}_1\tilde{\gamma}_4\tilde{\gamma}_2\tilde{\gamma}_3\tilde{\gamma}_5 ,\quad 
\tilde{\gamma}_1\tilde{\gamma}_2\tilde{\gamma}_3  \tilde{\gamma}_4\tilde{\gamma}_5\overset{2}{\sim} \tilde{\gamma}_1\tilde{\gamma}_3\tilde{\gamma}_4\tilde{\gamma}_2\tilde{\gamma}_5  \\
\qquad\  \tilde{\gamma}_1\tilde{\gamma}_2\tilde{\gamma}_3  \tilde{\gamma}_4\tilde{\gamma}_5\overset{3}{\sim} \tilde{\gamma}_1\tilde{\gamma}_3\tilde{\gamma}_4\tilde{\gamma}_5\tilde{\gamma}_2,  \quad 
\tilde{\gamma}_1\tilde{\gamma}_2\tilde{\gamma}_3  \tilde{\gamma}_4\tilde{\gamma}_5 \overset{4}{\sim} \tilde{\gamma}_3\tilde{\gamma}_4\tilde{\gamma}_5\tilde{\gamma}_1\tilde{\gamma}_2  \\
\qquad \ \tilde{\gamma}_1\tilde{\gamma}_2\tilde{\gamma}_3  \tilde{\gamma}_4\tilde{\gamma}_5\overset{5}{\sim} \tilde{\gamma}_4\tilde{\gamma}_5\tilde{\gamma}_1\tilde{\gamma}_2\tilde{\gamma}_3  ,\quad 
\tilde{\gamma}_1\tilde{\gamma}_2\tilde{\gamma}_3  \tilde{\gamma}_4\tilde{\gamma}_5\overset{6}{\sim} \tilde{\gamma}_4\tilde{\gamma}_1\tilde{\gamma}_2\tilde{\gamma}_3\tilde{\gamma}_5 
 \rangle^+
 \end{array} 
\vspace{-0.1cm}
\end{equation}


\noindent
\begin{prop} 
\label{NCforA_3} 
The following is the complete list of reduced and minimal  NC-tuples of length $\le 9$.   None of them forms a twin. 

\smallskip
1.  Length=5 case.
\vspace{-0.3cm} 
\begin{equation}
\label{eq:A3NoncancellDeg5}
\begin{array}{rcl}
(1,\gamma_5) & \in & NC(\gamma_1\gamma_2\gamma_3\gamma_4,\gamma_4\gamma_1\gamma_2\gamma_3)\\
(1,\gamma_5) & \in & NC(\gamma_1\gamma_3\gamma_4\gamma_2,\gamma_4\gamma_1\gamma_2\gamma_3)\\
(1,\gamma_2) & \in & NC(\gamma_1\gamma_3\gamma_4\gamma_5,\gamma_3\gamma_4\gamma_5\gamma_1)\\
(\gamma_1,1) & \in & NC(\gamma_2\gamma_3\gamma_4\gamma_5,\gamma_3\gamma_4\gamma_5\gamma_2)\\
(\gamma_1,1) & \in & NC(\gamma_4\gamma_2\gamma_3\gamma_5,\gamma_3\gamma_4\gamma_5\gamma_2)\\
(\gamma_4,1) & \in & NC(\gamma_5\gamma_1\gamma_2\gamma_3,\gamma_1\gamma_2\gamma_3\gamma_5)\\
(\gamma_1,\gamma_5) & \in & NC(\gamma_2\gamma_3\gamma_4,\gamma_4\gamma_2\gamma_3)\\
(\gamma_1,\gamma_5) & \in & NC(\gamma_2\gamma_3\gamma_4,\gamma_3\gamma_4\gamma_2)\\
(\gamma_1,\gamma_5) & \in & NC(\gamma_3\gamma_4\gamma_2,\gamma_4\gamma_2\gamma_3)\\
(1
,\gamma_2\gamma_3\gamma_5)& \in& NC(\gamma_1\gamma_4,\gamma_4\gamma_1)\\
(\gamma_1\gamma_3\gamma_4
,1)& \in & NC(\gamma_2\gamma_5,\gamma_5\gamma_2).
\end{array}
\end{equation}

2. There are no NC-tuples of length 6.

3. Length 7 case.
\vspace{-0.3cm} 
\begin{equation}
\label{eq:A3NoncancellDeg7}
\begin{array}{rcl}
(\gamma_1,\gamma_2) & \in & NC(\gamma_4\gamma_2\gamma_3\gamma_5 \gamma_1,\gamma_2\gamma_1\gamma_3\gamma_4\gamma_5)\\
(\gamma_1,\gamma_2) & \in & NC(\gamma_3\gamma_4\gamma_2\gamma_5\gamma_1,\gamma_2\gamma_1\gamma_3\gamma_4\gamma_5)\\
(\gamma_1,\gamma_2) & \in & NC(\gamma_3\gamma_4\gamma_5\gamma_2\gamma_1,\gamma_2\gamma_1\gamma_3\gamma_4\gamma_5)\\
(\gamma_4,\gamma_5) & \in & NC(\gamma_1\gamma_2\gamma_3\gamma_5\gamma_4,\gamma_5\gamma_1\gamma_4\gamma_2\gamma_3)\\
(\gamma_4,\gamma_5) & \in & NC(\gamma_1\gamma_2\gamma_3\gamma_5\gamma_4,\gamma_5\gamma_1\gamma_3\gamma_4\gamma_2)\\
(\gamma_4,\gamma_5) & \in & NC(\gamma_1\gamma_2\gamma_3\gamma_5\gamma_4,\gamma_5\gamma_4\gamma_1\gamma_2\gamma_3)\\
\end{array}
\end{equation}

4. There are no NC-tuples of length 8.

\smallskip
5.  Length 9 case.
\vspace{-0.3cm}
\begin{equation}
\label{eq:A3NoncancellDeg9}
\begin{aligned}
(\gamma_{4},\gamma_{2}) &\ \in \ NC(\gamma_{1}\gamma_{2}\gamma_{3}\gamma_{5}\gamma_{4}\gamma_{5}\gamma_{1},\,\gamma_{5}\gamma_{1}\gamma_{2}\gamma_{1}\gamma_{3}\gamma_{4}\gamma_{5}),\\
(\gamma_{4},\gamma_{5}) &\ \in \ NC(\gamma_{1}\gamma_{2}\gamma_{3}\gamma_{5}\gamma_{4}\gamma_{5}\gamma_{4},\,\gamma_{5}\gamma_{4}\gamma_{5}\gamma_{1}\gamma_{4}\gamma_{2}\gamma_{3}),\\
(\gamma_{4},\gamma_{5}) &\ \in \ NC(\gamma_{1}\gamma_{2}\gamma_{3}\gamma_{5}\gamma_{4}\gamma_{5}\gamma_{4},\,\gamma_{5}\gamma_{4}\gamma_{5}\gamma_{1}\gamma_{3}\gamma_{4}\gamma_{2}),\\
(\gamma_{4},\gamma_{5}) &\ \in \ NC(\gamma_{1}\gamma_{2}\gamma_{3}\gamma_{5}\gamma_{4}\gamma_{5}\gamma_{4},\,\gamma_{5}\gamma_{4}\gamma_{5}\gamma_{4}\gamma_{1}\gamma_{2}\gamma_{3}),\\
(\gamma_{1},\gamma_{2}) &\ \in \ NC(\gamma_{3}\gamma_{4}\gamma_{2}\gamma_{5}\gamma_{1}\gamma_{2}\gamma_{1},\,\gamma_{2}\gamma_{1}\gamma_{2}\gamma_{1}\gamma_{3}\gamma_{4}\gamma_{5}),\\
(\gamma_{1},\gamma_{2}) &\ \in \ NC(\gamma_{4}\gamma_{2}\gamma_{3}\gamma_{5}\gamma_{1}\gamma_{2}\gamma_{1},\,\gamma_{2}\gamma_{1}\gamma_{2}\gamma_{1}\gamma_{3}\gamma_{4}\gamma_{5}),\\
(\gamma_{1},\gamma_{2}) &\ \in \ NC(\gamma_{3}\gamma_{4}\gamma_{5}\gamma_{2}\gamma_{1}\gamma_{2}\gamma_{1},\,\gamma_{2}\gamma_{1}\gamma_{2}\gamma_{1}\gamma_{3}\gamma_{4}\gamma_{5}),\\
(\gamma_{1},\gamma_{5}) &\ \in \ NC(\gamma_{3}\gamma_{4}\gamma_{2}\gamma_{5}\gamma_{1}\gamma_{2}\gamma_{4},\,\gamma_{2}\gamma_{4}\gamma_{5}\gamma_{1}\gamma_{3}\gamma_{4}\gamma_{2}),\\
(\gamma_{1},\gamma_{5}) &\ \in \ NC(\gamma_{3}\gamma_{4}\gamma_{2}\gamma_{5}\gamma_{1}\gamma_{2}\gamma_{4},\,\gamma_{2}\gamma_{4}\gamma_{5}\gamma_{1}\gamma_{4}\gamma_{2}\gamma_{3}),\\
(\gamma_{1},\gamma_{5}) &\ \in \ NC(\gamma_{3}\gamma_{4}\gamma_{2}\gamma_{5}\gamma_{1}\gamma_{2}\gamma_{4},\,\gamma_{2}\gamma_{4}\gamma_{5}\gamma_{4}\gamma_{1}\gamma_{2}\gamma_{3}),\\
(\gamma_{1},\gamma_{5}) &\ \in \ NC(\gamma_{3}\gamma_{4}\gamma_{5}\gamma_{2}\gamma_{1}\gamma_{2}\gamma_{4},\,\gamma_{2}\gamma_{4}\gamma_{5}\gamma_{1}\gamma_{3}\gamma_{4}\gamma_{2}),\\
(\gamma_{1},\gamma_{5}) &\ \in \ NC(\gamma_{3}\gamma_{4}\gamma_{5}\gamma_{2}\gamma_{1}\gamma_{2}\gamma_{4},\,\gamma_{2}\gamma_{4}\gamma_{5}\gamma_{1}\gamma_{4}\gamma_{2}\gamma_{3}),\\
(\gamma_{1},\gamma_{5}) &\ \in \ NC(\gamma_{3}\gamma_{4}\gamma_{5}\gamma_{2}\gamma_{1}\gamma_{2}\gamma_{4},\,\gamma_{2}\gamma_{4}\gamma_{5}\gamma_{4}\gamma_{1}\gamma_{2}\gamma_{3}),\\
(\gamma_{1},\gamma_{5}) &\ \in \ NC(\gamma_{4}\gamma_{2}\gamma_{3}\gamma_{5}\gamma_{1}\gamma_{2}\gamma_{4},\,\gamma_{2}\gamma_{4}\gamma_{5}\gamma_{1}\gamma_{3}\gamma_{4}\gamma_{2}),\\
(\gamma_{1},\gamma_{5}) &\ \in \ NC(\gamma_{4}\gamma_{2}\gamma_{3}\gamma_{5}\gamma_{1}\gamma_{2}\gamma_{4},\,\gamma_{2}\gamma_{4}\gamma_{5}\gamma_{1}\gamma_{4}\gamma_{2}\gamma_{3}),\\
(\gamma_{1},\gamma_{5}) &\ \in \ NC(\gamma_{4}\gamma_{2}\gamma_{3}\gamma_{5}\gamma_{1}\gamma_{2}\gamma_{4},\,\gamma_{2}\gamma_{4}\gamma_{5}\gamma_{4}\gamma_{1}\gamma_{2}\gamma_{3}).
\end{aligned}
\end{equation}

\end{prop}
\begin{proof}
{\bf  Case for length $\ell=5$}.   

Since all fundamental relations \eqref{eq:A3Monoid} are among 7 words of length 5, we have only to divide the equivalent 7 words  from the left and right to become minimal,  and obtain the list of 11 minimal NC-tuples. They are automatically reduced. The inertia groups are trivial, since we shall see in \S6 that the component of the graph $G(L^*,\mathcal{R}_m(\mathcal{A}_{\rm{A}_3}))$ is a tree.

\medskip
\noindent
{\bf Case for length $\ell=  6, 7, 8$  and $9$.}  

\smallskip
We first prepare some notations and observations  for the proof.

\smallskip
(1) Let us denote the letters $\tilde{L}=\{\tilde{\gamma}_1,\tilde{\gamma}_2,\tilde{\gamma}_3,\tilde{\gamma}_4,\tilde{\gamma}_5\}$  by $\{1,2,3,4,5\}$ for simplicity.  The unit element in the monoid $\tilde{L}^*$ shall be denoted by $\emptyset$.

(2) For a given word $W \in \tilde{L}^*$, the set of  letters in $W$ counted with multiplicity is denoted by $[W]$ and called the {\it underlying set} of $W$.  If two words are equivalent in $A^+({L},\mathcal{R}(\mathcal{A}_{\mathrm{A}_{3}}))$, then their underlying sets are equal, since all fundamental  relations preserve the underlying sets. 

(3) Let us call the word $12345$ and the 6 words in the right-hand-side of the fundamental relations \eqref{eq:A3Monoid}  the {\it fundamental words}. They are
$$
12345,\ \ 14235,\ \ 13425,\ \ 13452, \ \ 34512,\ \ 45123,\ \ 41235
$$

(4) We ``complete'' the system of fundamental relations \eqref{eq:A3Monoid} by adding   equivalence between any two fundamental words to the system. This does not change the monoid and NC-tuples, but does simplify the proof by shortening the sequence defining the equivalences among words.\!\!\! 

(5)  We list cases when a nontrivial\footnote
{Here, ``nontrivial'' means that the part is neither empty nor the full word.
}
tail part of a fundamental  word coincides with a nontrivial head part of another fundamental word. 
Let us call the coinciding part a {\it core}.
Actually, such core appears as one half of a  commuting pair as formulated below. 

Let us  call the following  {\it un-ordered pair}  of words a {\it commuting pair}.
\begin{equation}
\label{eq:CommutingPair}
(3, 4512), \ \ \  (12, 345),  \ \ \  (45,123)\ .
\end{equation}

\noindent
{\bf Assertion.} 
{\it For any commuting pair $(P,Q)$ in \eqref{eq:CommutingPair}, the following holds.}

(i) {\it $PQ$ and $Q P$ are fundamental words. In particular, $PQ\sim Q P$.}

(ii) {\it If a fundamental word $W$ is left-divisible by $P$, then $W=PQ$.}

(iii) {\it If a fundamental word $W$ is right-divisible by $Q$, then $W=PQ$.}

(iv) {\it A word $W$ is a core, if and only if $W$ is one half of a commuting pair in \eqref{eq:CommutingPair}. Explicitly, $W$ is an element of the following list.}
\begin{equation}
\label{eq:FundamentalPhrais}
\{\ 3, \ 12, \ 45, \ 123, \ 345, \ 4512\ \}.
\end{equation}
\begin{proof} These follow by checking the completed fundamental relations.\!\!\!\!\!
\end{proof}

We start a proof of Proposition \ref{NCforA_3}.  Let  $(a,b,c,d)$ be an NC-tuple of length $\ell$ ($6\le \ell \le 9$), and $(\tilde{a},\tilde{b},\tilde{c},\tilde{d})$ be its lifting in $\tilde{L}^*$.  Consider a shortest sequence of words $W_0=\tilde{a}\tilde{b}\tilde{d}, W_1,\cdots, W_n=\tilde{a}\tilde{c}\tilde{d}$ such that $W_i$ is an e-transformation of $W_{i-1}$ ($1\le i\le n$) w.r.t. the completed system of fundamental relations.\footnote
{The assumption ``shortest'' requires that none of the e-transformations in the sequence should  be an identity.  It requires also that the lifting $(\tilde b,\tilde c)$ should be so chosen that the first (resp.\ the last)  e-transformation touching $\tilde b$ (resp.\ $\tilde c$)  should not be contained in $\tilde b$ (resp.\ $\tilde c$). There are similar conditions on $\tilde a$ and $\tilde d$. But we omit them,
}
The length $n$ is not 0 since $W_0\not=W_n$. The underlying multi-set $[W_0]=\cdots=[W_n]$ contains the set $\{1,2,3,4,5\}$. The difference set $D:=[W_0]\setminus\{1,2,3,4,5\}$ consists of $\ell-5$ elements.\!\!\!

The reducedness of the tuple (of length $\ge6$) implies $n\ge2$. 

For $1\le i \le n$, let  $J_{i-1} \subset W_{i-1}$ and $I_i\subset W_i$ be the substrings  from $k_i$-th letter to $k_i+4$-th letter for $1\le k_i\le \ell-4$ such that  $J_{i-1}$ is elementary transformed to $I_i$ by a fundamental relation.\footnote
{The fact that $k_i$ is uniquely defined is non-trivial but easily checked from the relations.}
  By the shortestness of the sequence and the completeness of the relations, $k_i\not= k_{i+1}$ ($1\le i<n$).

The substrings  $J_i$ and $I_i$ in $W_i$  for $1\le i\le n-1$ may intersect as  $\deg(I_i\cap J_i)\ge 10-\ell>0$. The overlapped part, say $P_i:=I_i\cap J_i$, is a core in the sense of above (5).  Let $Q_i$ be the other half partner of the commuting  pair \eqref{eq:CommutingPair}  containing $P_i$. Then, due to (ii) and (iii) of {\bf Assertion}, we see that the word $W_i$ contains a substring $I_i\cup J_i=Q_iP_iQ_i$ and $|k_{i+1}-k_i|=|Q_i|$. In particular, the difference set $D$ of the tuple contains the underlying set $[Q_i]$ ($1\!\le\! i \!\le\! n\!-\!1$). 

\medskip
\noindent
{\bf Assertion} (v) {\it  We have  $k_i\not= k_{i+2}$  for   $1\le i\le n-2$.}
\vspace{-0.1cm}
\begin{proof}  Assume $k_{i+1}>k_i$. If $k_i=k_{i+2}$,  the head of the word $J_{i+1}$ in $W_{i+1}$ for the $i+2$-th e-transformation is equal to $Q_i$ (inheriting from the word $W_i$ since that part is unchanged by the $i+1$-th e-transformation). Then {\bf Assertion} (ii), implies that the fundamental word $J_{i+1}$  is $Q_iP_i$. This means that the head of $I_{i+1}$ is $P_i$ and hence $I_{i+1}\!=\!P_iQ_i$ ({\bf Assertion} (ii)). That is,  the  $i\!+\!1$-th e-transformation is an identity. A contradiction! 

A completely parallel proof works for the case $k_{i+1}<k_i$.
\end{proof}

From now on, the analysis depends on the length $\ell$ of the tuple.

\smallskip
$\ell=6$ :  Since $|D|=1$,  among the list \eqref{eq:FundamentalPhrais} only $\{Q_i\}=\{3\}$ is possible. Due to {\bf Assertion} (v), $n=2$ and  $W_1=345123$ and $k_1=1,k_2=2$ or its opposite sequence. Then, the NC assumption implies either  the first or the last letter of $W_0$ and $W_2$ coincides to 3. That is, either $W_0$ or $W_2$ is equal to $W_1$, and the e-transformation there is trivial.  

$\ell=7$:  Since $|D|=2$,  among the list \eqref{eq:FundamentalPhrais}  $\{Q_i\}$ is either equal to $\{3\}$, $\{12\}$ or $\{45\}$ (they cannot mix). 

(1) Suppose $\{Q_i\}=\{3\}$.  Either $k_1=1,k_2=2,k_3=3$ or its opposite sequence. In the first case, $W_0=A33$ and $W_3=33B$, where $A,B$ are fundamental words.  NC assumption implies either the prefix of $A$ or the suffix of $B$ is equal to 3,  implying the e-transformation there is trivial.  

(2) Suppose  $\{Q_i\}=\{12\}$. Due to {\bf Assertion} (v), $n=2$ and  $W_1=1234512$. So, $k_1=1,k_2=3$ or its opposite. Suppose $k_1=1$ and $W_0=A12$ and $W_2=12B$, where $A,B$ are fundamental words. s.t.  $A$ is not left divisible by $12$ and $B$ is not right divisible by $12$, otherwise the e-transformation there becomes trivial.  The prefix of $W_0$ is 1 and the suffix of $W_2$ is 2, otherwise an e-transformation occur inside the kernel, violating the minimality. These conditions are satisfied when $A$ is either $14235,13425$ or $13452$ and $B$ is $13452$. So, $(1,2)$ is a member of $NC(42351,21345),
NC(34251,21345)$ and $NC(34521,21345)$.

(3) Suppose  all $Q_i= \{45\}$. For the same reasons as in (2), we may assume $n=2$, $k_1=1$  and $W_0=A45$ and $W_2=45B$ for fundamental words $A$ and $B$. The prefix of $A$ is 4 and the suffix of $B$ is 5 (but not  divisible by 45). This is possible when the head of $A$ is $41235$ and $B$ is either $14235, 13425$ or $41235$. So, $(4,5)$ is a member of $NC(12354,51423),
NC(12354,51342)$ and $NC(12354,54123)$.

The minimality of the NC-tuples in the list of (2) and (3) follows from the fact that their  kernels consist of word of length 5 which  {\it are not} fundamental words.  Reducedness is automatic. 

$\ell=8$:  Since $|D|=3$,  among the list \eqref{eq:FundamentalPhrais}  $\{Q_i\}$ is either a subset of  $\{3, 12, 123\}$ or of $\{3,45, 345\}$.

(1) Suppose there exists $Q_i=123$ for some $1\le i\le n-1$. Then, neither $1< i$ nor $i+1<n$. 

\noindent
{\it Proof.}\ Since $k_i=1, k_{i+1}=4$ or opposite, we may assume $k_i=1$.  Thus $W_{i-1}=A 123, W_i=12345123$ and $W_{i+1}=123B$ for fundamental words $A$ and $B$. If $1<i$, then $k_{i-1}$ is possibly 2  or 3, but 2 is not possible  since there does not exist a fundamental word $I_{i-1}$ of suffix 1 and  3  is not possible since the only fundamental word $I_{i-1}$ ending 12 is 34512 (Assertion (iii)) implying that $J_{i-1} $ is 12345 (Assertion (iii)) so that the $i$-th e-transformation is trivial. 
If $i<n$, then $k_{i+1}$ is 2 or 3, but 2 is not possible since there does not exist a fundamental word $J_{i+1}$ whose prefix is 2 and 3 is not possible since the only fundamental word $J_{i+1}$, whose prefix is 3, is $34512$ (Assertion (ii))  implying that $I_{i+1}=45123$, implying that the $i+1$-th e-transformation is trivial. $\Box$  

Thus, $i=1=n-1$ and $W_0=A123$ and $W_2=123B$. To  become an NC-tuple, the suffix of $B$ needs to be 3 (for the same reason as in length 7 case). However, this means $B=45123$ (Assertion (iii)). Thus, the second e-transformation becomes trivial. This is not possible.

 (2) Suppose there exists $Q_i=345$ for some $1\le i\le n-1$. Then, neither $1< i$ nor $i+1<n$. 
 
 \noindent
 {\it Proof.}\ Assume, for simplicity, $k_i=1$. Then,  there is no possible fundamental word for $k_{i-1}$ is 2 or 3. Similarly, $k_{i+1}$ is only possibly 2, where the first two letters of the $i+1$-th e-transformation are $45$. This implies the first 3 letters of the $i$-th elementary  transformation consist of $1,2,3$. The only such fundamental word is $12345$ implying that the $i+1$-th e-transformation is trivial. $\Box$
 
Thus, $i=1=n-1$. Assuming $k_1=1$, the prefix of $W_2$ is 3. However, prefix of $W_0$ cannot be 3, else the first e-transformation is trivial. Thus $\tilde{a}=\emptyset$ so that $\tilde{d}\not=\emptyset$. If $\tilde{d}$ is $45$ or longer than it, then the second e-transformation is trivial, so that $\tilde{d}=5$. However, this contradicts  the shortestness of $n$ and the minimality of the tuple.
 
(3) Let us list all remaining possibilities of the intermediate words $W_i$ ($1\le i\le n-1$) using the expression $Q_i \cdot P_i\cdot  Q_i$ for $Q_i\in \{12, 45, 3\}$, where we put numbers (1)$\sim$(7) for later references.
$$
\begin{array}{c}
(1)\ 12\cdot 345\cdot 12\cdot *, \ (2)\ *\cdot 12\cdot 345\cdot 12, \ (3)\  45\cdot 123\cdot *, \ (4)\ *\cdot 45\cdot 123\cdot 45 \\
(5)\ 3\cdot4512\cdot 3\cdot **, \ \  (6)\ *\cdot 3\cdot4512\cdot 3\cdot *, \ \  (7)\  **\cdot 3\cdot4512\cdot 3\cdot 
\end{array}
$$
Here ``$*$'' is an unfixed letter.  There are four observations.

\smallskip
1.  The pairs of words (1) and (7),  and (4) and (5) become equal by choosing $*=3$ and $**=12$ or $45$, respectively.

2. Among the pairs (2) and (5),  (3) and (7), (5) and (6), and (6) and (7), there is an e-transformation by choosing $*$ and $**$ suitably.

3.  The above 1. and 2. give the complete list 
 
4. By adding possible words $W_0$ and $W_n$ to the above list, we obtain the following list of all possible candidates for the sequence $W_0,\cdots, W_n$. 

$$
\begin{array}{l}
1. \  \ A123 \sim 12345123 \sim  123 B\\
2. \  \ A 345 \sim 34512345\sim 345 B\\
3. \ \ 312A  \sim 3 12345 12 \sim 34512312 \sim B312 \\
4. \ \  A453 \sim  45123453 \sim 45345123 \sim 453 B\\
5. \  \ A333\sim 34512333\sim 33451233\sim 33345123\sim  333B
\end{array}
$$
where  $A$ and $B$ are the unknown fundamental words.
Let us examine each case what are the possibilities of $A$ and $B$.

1.  If the suffix of $B$ is 3, then $B=45123$ (Assertion (iii)) so the e-transformation become trivial. Thus, the prefix of $A$ needs to be 1. Actually, in order to be the shortest (to keep the second e-transformation), the head of  $A$ needs to be 123, but it implies the first e-transformation is trivial.

2.  If the prefix of $A$ is 3 the first e-transformation becomes trivial., so that suffix of $B$ is 5. Actually, to be shortest (to keep the first elementary transformation), tail of $B$ needs to be 345. But then the second e-transformation becomes trivial.

3.  If the prefix of $B$ is 3, then the last e-transformation becomes trivial. So the kernel starts from the first letter. Its length  should be shorter  than 5 (else  it contradicts  shortestness. But then it implies the tail of $A$ is 312, implying the first e-transformation becomes trivial.

4. If the suffix of $B$ is 3, then the last e-transformation is trivial. Thus the kernel ends at the last letter.  If its length is shorter than 5 then the sequence is not the shortest. Therefore 453 is the head of $A$, implying the first e-transformation is trivial.

5.  If the prefix of $A$ is 3, then the first e-transformation becomes trivial. If the suffix of $B$ is 3, the last e-transformation becomes trivial.

\smallskip
This completes the proof of non-existence of NC-tuples of length 8.

\smallskip
$\ell=9$:  We can proceed the same type of proof of the case of length 9 as in the cases of length 6-8, and obtain \eqref{eq:A3NoncancellDeg9}. We omit its details, since it is just to use Assertion (i)-(v) systematically, and to apply the NC-relation \eqref{eq:NCRelation}, but  would lead to rather lengthy calculations.

\smallskip
This completes the proof of Proposition \ref{NCforA_3}.
\end{proof}

\begin{rem}
{\rm Recall that we conjectured  in Introduction (3)  that  the monoid $A^+(L, \mathcal{R}_m(\mathcal{A}_{\mathrm{A}_3}))$ does not have a twin of reduced and minimal NC-tuples. 
 The background for this is  Proposition \ref{NCforA_3}, showing that there does not exist a twin of reduced and minimal  NC-tuples of length $\le9$.  

In Appendix III (b), we give further computational data. A separate cumulative check of the computed reduced and minimal NC-tuples of degree at most \(11\) found no two distinct tuples with the same ordered kernel \((b,c)\), before imposing either indecomposability condition. The two experimental indecomposability conditions are examined through degree \(17\).


\smallskip
Another remark is that we observed already in Proposition \ref{NCforA_3} the first step of the following ``periodicity'' of NC-tuples:  for $n\in \mathbb{Z}_{>0}$, we have 
\begin{equation}
\label{eq:A3NoncancellDeg5+2n}
\begin{array}{rcl}
(\gamma_1,\gamma_2) & \in & NC(\gamma_4\gamma_2\gamma_3\gamma_5 \gamma_1(\gamma_2\gamma_1)^{n-1},(\gamma_2\gamma_1)^n\gamma_3\gamma_4\gamma_5)\\
(\gamma_1,\gamma_2) & \in & NC(\gamma_3\gamma_4\gamma_2\gamma_5\gamma_1(\gamma_2\gamma_1)^{n-1},(\gamma_2\gamma_1)^n\gamma_3\gamma_4\gamma_5)\\
(\gamma_1,\gamma_2) & \in & NC(\gamma_3\gamma_4\gamma_5\gamma_2\gamma_1(\gamma_2\gamma_1)^{n-1},(\gamma_2\gamma_1)^n\gamma_3\gamma_4\gamma_5)\\
(\gamma_4,\gamma_5) & \in & NC(\gamma_1\gamma_2\gamma_3(\gamma_5\gamma_4)^n,(\gamma_5\gamma_4)^{n-1}\gamma_5\gamma_1\gamma_4\gamma_2\gamma_3)\\
(\gamma_4,\gamma_5) & \in & NC(\gamma_1\gamma_2\gamma_3(\gamma_5\gamma_4)^n,(\gamma_5\gamma_4)^{n-1}\gamma_5\gamma_1\gamma_3\gamma_4\gamma_2)\\
(\gamma_4,\gamma_5) & \in & NC(\gamma_1\gamma_2\gamma_3(\gamma_5\gamma_4)^n,(\gamma_5\gamma_4)^{n-1}\gamma_5\gamma_4\gamma_1\gamma_2\gamma_3)\\
\end{array}
\end{equation}
Such ``periodicity'' is seemingly uninteresting repetition of NC-tuples. However, the author does not know how to exclude such repetitions.
}
\end{rem}

\section{\! $\Pi$-class associated with an NC-twin }



We come back to the setting in \S3 for a topological space equipped with a geometric lifting of fundamental relations.  The goal of the present section is to construct the set $\pi_2$-classes $\Pi(\tau)$  \eqref{eq:FinalPi2}, which we call the $\Pi$-class of $\tau$,  and to show that it is a single coset of the global inertia group $G(\tau)$ \eqref{eq:FinalInertia} associated with an NC-twin $\tau$ of the monoid $A^+(L,\mathcal{R}_m)$ (Definition \ref{FinalDefn} and  Theorem  \ref{FinalTheo}).  
 
 \smallskip
Let us review the plan of the construction of the $\Pi$-class in \S5.\!\!\!
 

 (1) For a given NC-tuple $(a,b,c,d) \in A^+(L,\mathcal{R}_m)^4$ (Definition \ref{NCtuple} 1.), let $(\tilde{a},\tilde{b},\tilde{c},\tilde{d})\in (\mathcal{F}^+)^4$ be its lifting. The NC-relation \eqref{eq:NCRelation} lifts to certain homotopy equivalence set  $[\tilde{a}\tilde{b}\tilde{d},\tilde{a}\tilde{c}\tilde{d}]_{\mathcal{R}_g} \subset[\tilde{a}\tilde{b}\tilde{d},\tilde{a}\tilde{c}\tilde{d}$]. 
 We embed this set 
 into the set of homotopy equivalences $[\tilde{b},\tilde{c}]$ by a  division theory developed in \S5.1, and obtain the division $\mathcal{R}_g$-set
 $\tilde{a}^{-1}[\tilde{a}\tilde{b}\tilde{d},\tilde{a}\tilde{c}\tilde{d}]_{\mathcal{R}_g}\tilde{d}^{-1}$ \eqref{eq:DivisionRgset}, 
 on which the division inertia groups \eqref{eq:InertiaDivision} act simply and transitively.

(2)   Let $\tau=\{(a_1,b,c,d_1),(a_2,b,c,d_2)\}$ be an NC-twin  (Definition  \ref{NCtuple}  4.).
Then we consider  the $\circ$-difference set  of the two division $\mathcal{R}_g$-sets $\tilde{a}_i^{-1} [\tilde{a}_i\tilde{b}\tilde{d}_i,\tilde{a}_i\tilde{c}\tilde{d}_i]_{\mathcal{R}_g}\tilde{d}_i^{-1}$ ($i\!=\!1,2$), and  realize it in  $\pi_2(X,*)$ by the relative realization morphism  $\chi$ \eqref{eq:immersion}. The image $\Pi(\tau)$ \eqref{eq:FinalPi2}  is the $\Pi$-class, on which the global inertia group \eqref{eq:FinalInertia} acts simply transitively.\!\! 

(3) Finally, we show  that  {\it the $\Pi$-class and the global inertia group do not depend on the choice of the liftings of the NC-tuples} and that {\it the $\Pi$-class forms a single coset of the global inertia group} (Theorem\ \ref{FinalTheo}).

\medskip
\subsection{\normalsize  Division morphism}

\smallskip
\noindent
We prepare a division morphism \eqref{eq:Relative2class} in the setting of \S3.1.   

\smallskip
Let  $\tilde{a},\tilde{b},\tilde{c},\tilde{d},\tilde{e}$ be arbitrary elements in the free group $\mathcal{F}$.
%
Recall the natural (bijective) embedding (recall Corollary \ref{SimpleTransitive}, 2)
\begin{equation}
\label{eq:SplitEmbedding}
\quad [\tilde{b},\tilde{c}] \quad \to\quad  [\tilde{a}\tilde{b}\tilde{d},\tilde{a}\tilde{c}\tilde{d}], \ \ \quad    [H] \mapsto [T]:=\tilde{a}\cdot [H] \cdot \tilde{d}:= [1_{\tilde a}\cdot H \cdot 1_{\tilde d}]
\end{equation} 
We construct its inverse map. For any   class $[T] \in [\tilde{a}\tilde{b}\tilde{d},\tilde{a}\tilde{c}\tilde{d}]$, we consider its extension to $id_{\tilde{a}^{-1} }\! \cdot\! [T] \! \cdot \! id_{\tilde{d}^{-1}} \! \in \! [\tilde{a}^{-1}\tilde{a}\tilde{b}\tilde{d}\tilde{d}^{-1},$ $\tilde{a}^{-1}\tilde{a}\tilde{c}\tilde{d}\tilde{d}^{-1}]$ (here $\tilde{a}^{-1}, \tilde{d}^{-1}$ are inverses in $\mathcal{F}$ w.r.t.\ the ``$\cdot$"-product in $W$ (not the ``$\circ$"-inverse  \eqref{eq:circInverse0}) in $X$).  
Composing this map with the homotopy equivalences $\tilde{a}^{\!-\!1}\tilde{a}\!\sim\!\! *$ and $\tilde{d}\tilde{d}^{-1}\! \sim\! *$ (= the constant loop at $*$), we obtain a homotopy equivalence class  $[H] \in [*\tilde{b}*,*\tilde{c}*]$. 
A representing map $H$ of $[H]$ can be obtained from a representing map $T$ of $[T]$ as in Figure 5.1 with  the rules (i)$\sim$(iii) below:

\vspace{0.2cm}
\hspace{1.2cm} 
 \includegraphics[width=0.55\textwidth]{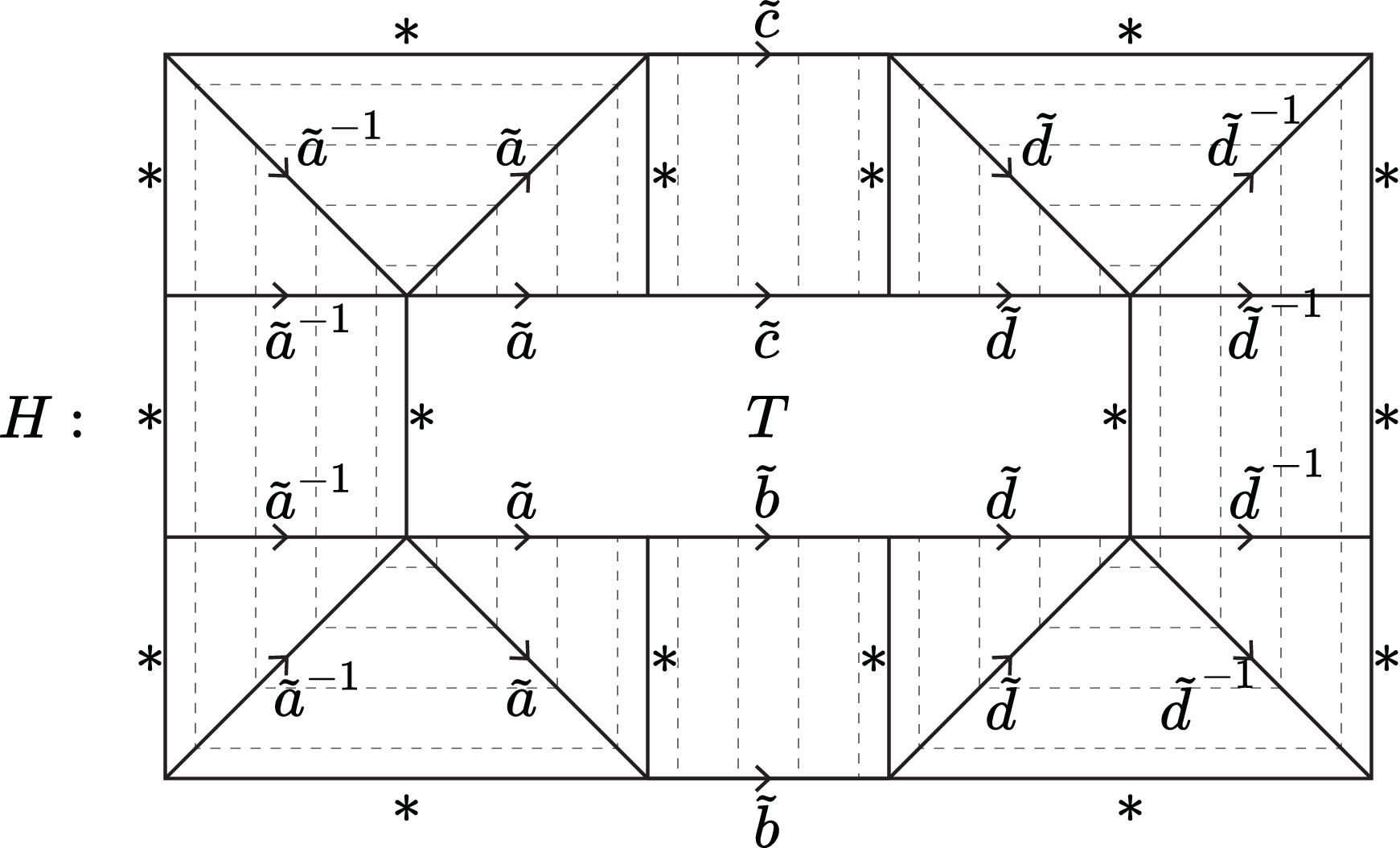}

\centerline{\quad Figure 5.1: \ Homotopy equivalence  $[H]=:\tilde{a}^{-1}  [T]\tilde{d}^{-1}  \in[\tilde{b},\tilde{c}]$\quad }

\bigskip
\noindent

(i) Labels $\tilde{a}^{\pm1},\tilde{b},\tilde{c},\tilde{d}^{\pm1}$ and $*$ on the edges of Figure 5.1 imply that the restriction of  $H$ on the edges coincide (up to homotopy) with the labels. 

(ii) The map $H$ takes constant values along the dotted lines. Actually, the lines form a foliation in the defining domain of $H$ , and in Figure 5.1, only samples of leaves are drawn. 

(iii)  On the central small rectangle, $H$ coincides  with $T$.

\medskip
We denote the class $[H]$ by
$\tilde{a}^{-1} [T]\tilde{d}^{-1}  \in [\tilde{b},\tilde{c}]$, and obtain a {\it bijective} map\!\! 
\begin{equation}
\label{eq:Relative2class}
 [\tilde{a}\tilde{b}\tilde{d},\tilde{a}\tilde{c}\tilde{d}] \quad \rightarrow \quad 
  [\tilde{b},\tilde{c}],  \qquad [T]\mapsto [H]=:
 \tilde{a}^{-1}  [T]\tilde{d}^{-1} .
\end{equation}
which we shall call the {\it $\cdot$-division morphism}.
Be cautious that even in the case $\tilde{a},\tilde{b},\tilde{c},\tilde{d}\in \mathcal{F}^+$,  
the  $\cdot$-division morphism does not preserve the $\mathcal{R}_g$-sets, since its construction involves operations $\tilde{a}^{-1}$ and $\tilde{d}^{-1}$ which do not belong to $\mathcal{F}^+$. This is why the $\cdot$-division morphism is of interest.

\medskip
\begin{prop}
\label{QuotientComposition}
1. The natural embedding \eqref{eq:SplitEmbedding} composed with the $\cdot$-division morphism \eqref{eq:Relative2class} is the identity map on $[\tilde{b},\tilde{c}]$.

\smallskip 
2.  For $\tilde{a},\tilde{b},\tilde{c},\tilde{d},\tilde{e} \in \mathcal{F}$, the following diagram is commutative.\!\!
\begin{equation}
\label{eq:divisionComposit}
\begin{array}{rccc}
\vspace{0.1cm}
& [\tilde{a}\tilde{b}\tilde{d},\tilde{a}\tilde{c}\tilde{d}]\times  [\tilde{a}\tilde{c}\tilde{d},\tilde{a}\tilde{e}\tilde{d}]   & \rightarrow &   [\tilde{b},\tilde{c}] \times [\tilde{c},\tilde{e}] \\
\vspace{0.1cm}
& \downarrow && \downarrow \\
&[\tilde{a}\tilde{b}\tilde{d},\tilde{a}\tilde{e}\tilde{d}] & \rightarrow & [\tilde{b},\tilde{e}]
\end{array}
\end{equation}
where horizontal arrows are $\cdot$-division morphisms and down arrows are $\circ$-products.  That is, $\circ$-product commutes with $\cdot$-division morphism. 
\end{prop}
\begin{proof} 
1. This is tautological, since for any $[H] \in [\tilde{b},\tilde{c}]$, we see immediately that $\tilde{a}^{-1}  (1_{\tilde a}\cdot [H]\cdot 1_{\tilde d})\tilde{d}^{-1} $ is homotopic to $[H]$.

2. The two routes, left and right, around the  diagram are given by the maps defined on the  following domains.

\vspace{0.1cm}
\hspace{-0.5cm} 
\includegraphics[width=0.85\textwidth]{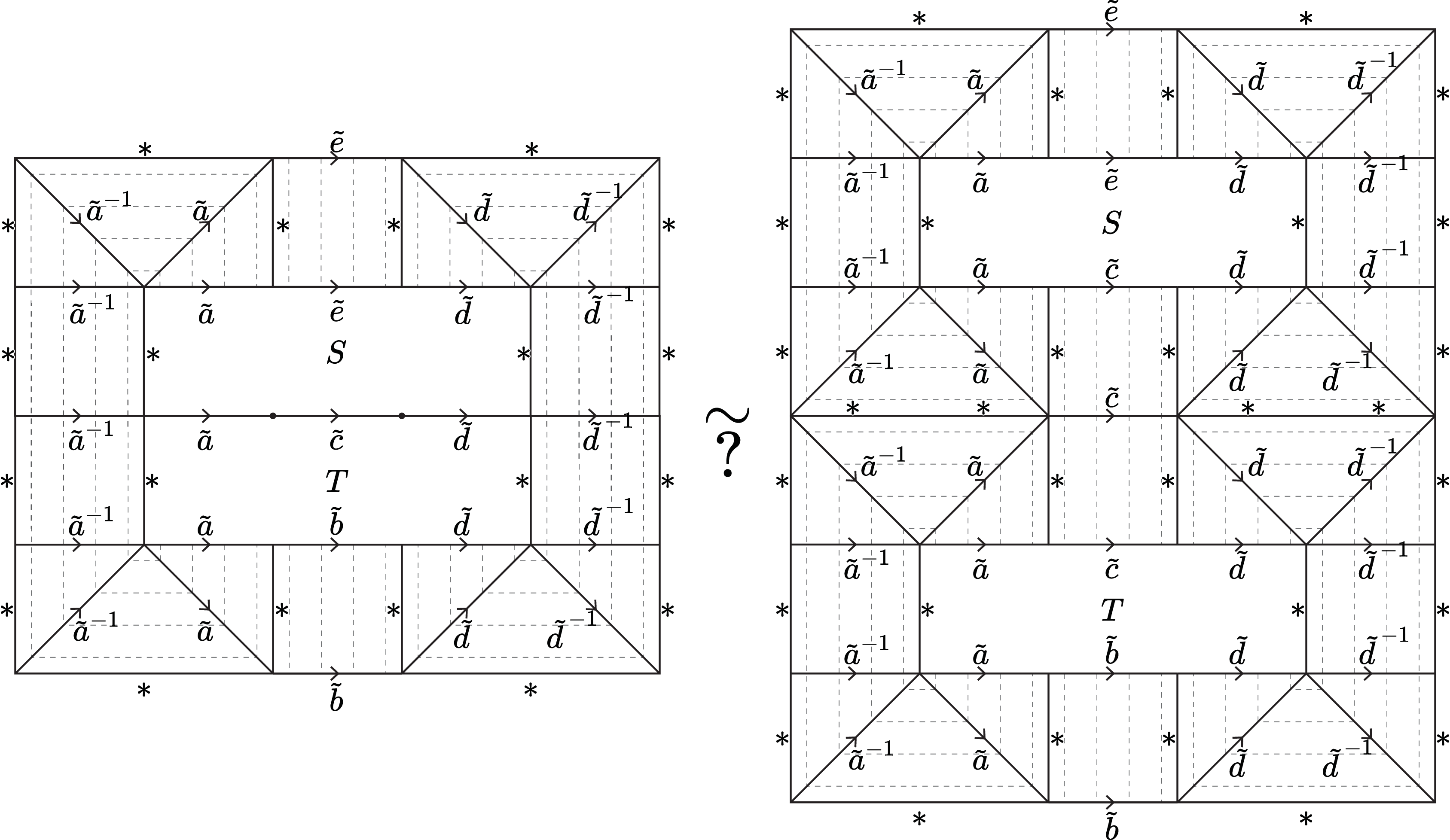}

\centerline{ Figure 5.2: \quad Left turn domain \qquad\qquad  \qquad  Right turn domain\qquad \qquad \quad  }

\bigskip
We apply Proposition \ref{PropA} g)  to the third and the fourth rows of blocks in the right turn domain in Figure 5.2. That is, the application of the homotopy equivalence of the 4th and 5th figures of g)  of  Proposition \ref{PropA} to the row 
implies that the row becomes  an identity homotopy from $\tilde{a}^{-1}\tilde{a}\tilde{c}\tilde{d}\tilde{d}^{-1}$ to itself. Shrinking the identity homotopy, we obtain the left turn domain in Figure 5.2.
\vspace{-0.05cm}
\end{proof}

\begin{cor} 
 The $\cdot$-division morphism induces  an isomorphism of \! $\circ$-groups.\!\!\!\!
\vspace{-0.1cm}
\begin{equation}
\label{eq:SplitProjection}
 [\tilde{a}\tilde{b}\tilde{d},\tilde{a}\tilde{b}\tilde{d}]  \to  [\tilde{b},\tilde{b}]
\text{\quad  and \quad }
 [\tilde{a}\tilde{c}\tilde{d},\tilde{a}\tilde{c}\tilde{d}]   \to   [\tilde{c},\tilde{c}] 
\end{equation}
which makes the following diagram commutative
\begin{equation}
\label{eq:DivisionCommutative}
\begin{array}{rccc}
\vspace{0.01cm}
&[\tilde{a}\tilde{b}\tilde{d},\tilde{a}\tilde{b}\tilde{d}]\times [\tilde{a}\tilde{b}\tilde{d},\tilde{a}\tilde{c}\tilde{d}]\times  [\tilde{a}\tilde{c}\tilde{d},\tilde{a}\tilde{c}\tilde{d}]   & \rightarrow &   [\tilde{b},\tilde{b}]\times [\tilde{b},\tilde{c}] \times [\tilde{c},\tilde{c}] \\
\vspace{0.02cm}
& \downarrow && \downarrow \\
&[\tilde{a}\tilde{b}\tilde{d},\tilde{a}\tilde{c}\tilde{d}] & \rightarrow & [\tilde{b},\tilde{c}]
\end{array}
\end{equation}
where the horizontal arrows are  $\cdot$-divisions and down arrows are  combinations  of left and right $\circ$-action on the middle factor.
\end{cor}

\medskip
\noindent
\subsection{\normalsize  Relative $\pi_2$-classes associated with an NC-tuple}

\smallskip
\noindent
{\bf We return to the setting in \S3.2.} 
Let $T:=(a,b,c,d)\in A^+(L,\mathcal{R}_m)^4$ be an NC-tuple, satisfying the NC-relation \eqref{eq:NCRelation}, and let  $\widetilde{T}=(\tilde{a},\tilde{b},\tilde{c},\tilde{d})\in (\mathcal{F}^+)^4$ be a lifting of $T$. So, $[\tilde{a}\tilde{b}\tilde{d},\tilde{a}\tilde{c}\tilde{d}]_{\mathcal{R}_g}$  is a non-empty  $\mathcal{R}_g$-homotopy equivalence set (Proposition \ref{GeometricMonoid}).  We apply now the $\cdot$-division morphism \eqref{eq:Relative2class} to this set. Consequently, we consider the following set and groups.  
}

(1) Consider the bijective image of the $\mathcal{R}_g$-set $ [\tilde{a}\tilde{b}\tilde{d},\tilde{a}\tilde{c}\tilde{d}]_{\mathcal{R}_g}$ by the division morphism:
\begin{equation}
\label{eq:DivisionRgset}
\qquad \Pi(\widetilde{T}) \ := \ \tilde{a}^{-1}  [\tilde{a}\tilde{b}\tilde{d},\tilde{a}\tilde{c}\tilde{d}]_{\mathcal{R}_g}\tilde{d}^{-1}  \quad \subset \quad  [\tilde{b},\tilde{c}],
\end{equation}
which we shall call  the {\it division $\mathcal{R}_g$-set}. Note that it 
 is not an $\mathcal{R}_g$-set, since, by assumption $\tilde{b}\not \sim \tilde{c}$ of an NC-tuple, we have $ [\tilde{b},\tilde{c}]_{\mathcal{R}_g}=\emptyset$. 

(2) Consider the bijective image of the local inertia groups $[\tilde{a}\tilde{b}\tilde{d},\tilde{a}\tilde{b}\tilde{d}]_{\mathcal{R}_g}$ and $ [\tilde{a}\tilde{c}\tilde{d},\tilde{a}\tilde{c}\tilde{d}]_{\mathcal{R}_g}$ by the division morphisms:
\begin{equation}
\label{eq:InertiaDivision}
 \begin{array}{c}
\qquad G(\widetilde{T},\tilde{b})\ :=\ \tilde{a}^{-1}  [\tilde{a}\tilde{b}\tilde{d},\tilde{a}\tilde{b}\tilde{d}]_{\mathcal{R}_g}\tilde{d}^{-1}  \quad \subset \quad  [\tilde{b},\tilde{b}]\\
\qquad \ G(\widetilde{T},\tilde{c})\ :=\  \tilde{a}^{-1}  [\tilde{a}\tilde{c}\tilde{d},\tilde{a}\tilde{c}\tilde{d}]_{\mathcal{R}_g}\tilde{d}^{-1}  \quad \subset \quad  [\tilde{c},\tilde{c}],
\end{array}
\end{equation}
which we shall call the {\it division inertia groups}. They are closed under the product ``$\circ$'', shown by multiple applications of Proposition \ref{QuotientComposition} 2. 
They contain  the local inertia groups $ [\tilde{b},\tilde{b}]_{\mathcal{R}_g}$ and $ [\tilde{c},\tilde{c}]_{\mathcal{R}_g}$ at $\tilde{b}$ and $\tilde{c}$ as subgroups, respectively (recall  Proposition \ref{QuotientComposition} 1. \eqref{eq:SplitEmbedding}).\footnote
{It may be a bit confusing that the division inertia  groups are isomorphic to local inertia groups at $\tilde{a}\tilde{b}\tilde{d}$ and $\tilde{a}\tilde{c}\tilde{d}$. 
However, this is necessary to distinguish them from those at $\tilde{b}$ and $\tilde{c}$.
}

In the following proposition, we see that 
1.\ the division inertia groups act on the division $\mathcal{R}_g$-set simply and transitively,  and 
2. the division $\mathcal{R}_g$-set and the division inertia groups are independent of the choices of the liftings $\tilde{a}$ and $\tilde{d}$ of $a$ and $d$.\!\!

\medskip
\begin{prop} 
\label{LocalToDvisionInertia}
{\rm (1)}  The actions of division inertia groups $\tilde{a}^{-1}  [\tilde{a}\tilde{b}\tilde{d},\tilde{a}\tilde{b}\tilde{d}]_{\mathcal{R}_g}\tilde{d}^{-1}$  and $\tilde{a}^{-1}   [\tilde{a}\tilde{c}\tilde{d},\tilde{a}\tilde{c}\tilde{d}]_{\mathcal{R}_g}\tilde{d}^{-1} $  on $[\tilde{b},\tilde{c}] $ from the left and right, respectively,  preserve the division $\mathcal{R}_g$-set   
$ \tilde{a}^{-1} [\tilde{a}\tilde{b}\tilde{d},\tilde{a}\tilde{c}\tilde{d}]_{\mathcal{R}_g}\tilde{d}^{-1} $ invariant. 
So we obtain a diagram of left and right actions of the groups: 
\begin{equation}
\label{eq:R_g-invariance1}
\begin{array}{c}
\vspace{0.2cm}
 (\tilde{a}^{-1}  [\tilde{a}\tilde{b}\tilde{d},\tilde{a}\tilde{b}\tilde{d}]_{\mathcal{R}_g}\tilde{d}^{-1}  )
\times
(\tilde{a}^{-1}  [\tilde{a}\tilde{b}\tilde{d},\tilde{a}\tilde{c}\tilde{d}]_{\mathcal{R}_g}\tilde{d}^{-1}  ) 
\times
(\tilde{a}^{-1}   [\tilde{a}\tilde{c}\tilde{d},\tilde{a}\tilde{c}\tilde{d}]_{\mathcal{R}_g}\tilde{d}^{-1} ) \!\!\!\\
 \to  \qquad \quad \
\tilde{a}^{-1}  [\tilde{a}\tilde{b}\tilde{d},\tilde{a}\tilde{c}\tilde{d}]_{\mathcal{R}_g}\tilde{d}^{-1}  \qquad \qquad .
\end{array}
\end{equation}
The action of each division inertia group is simply transitive.

The two  division inertia groups  are isomorphic to each other by the adjoint action of any element in the division $\mathcal{R}_g$-set. 

{\rm (2)} Let $\tilde a',\tilde b',\tilde c', \tilde d' \!\in \! \mathcal{F}^+$  be  other representatives of $a,b,c,d \! \in\!  A^+(L,\mathcal{R}_m)$.\!\!  Choose $[B]\in [\tilde b,\tilde b']_{\mathcal{R}_g}$ and $[C]\in [\tilde c,\tilde c']_{\mathcal{R}_g}$.   Then, the bijections
$$
\begin{array}{rccc}
\!\!&  [\tilde b,\tilde b] \simeq [\tilde b',\tilde b'], &   [\tilde b,\tilde c] \simeq   [\tilde b',\tilde c'] &  [\tilde c,\tilde c]
    \simeq  [\tilde c',\tilde c'] \\
\!\!& [H] \mapsto [1/B]\!\! \circ\! [H] \! \circ\! [B],&  [H] \mapsto [1/B]\!\! \circ\! [H]\! \circ \![C],&
[H] \mapsto [1/C] \!\!\circ \! [H] \!\circ\! [C] 
\end{array}
$$
(where {\small $[1/B] \!\in\! [\tilde b',\tilde b]_{\mathcal{R}_g}$ and 
$[1/C] \! \in\! [\tilde c',\tilde c]_{\mathcal{R}_g}$} 
are the ``\ $\circ$"-inverses given in  Prop.\ \ref{PropA}, c))
 induce the bijection between the division $\mathcal{R}_g$-sets 
\begin{equation}
\label{eq:R_g-invariance2}
  \tilde{a}^{-1}  [\tilde{a}\tilde{b}\tilde{d},\tilde{a}\tilde{c}\tilde{d}]_{\mathcal{R}_g}\tilde{d}^{-1} 
  \       \simeq  \  
\tilde{a}'^{-1}  [\tilde{a}'\tilde{b}'\tilde{d}',\tilde{a}'\tilde{c}'\tilde{d}']_{\mathcal{R}_g}\tilde{d}'^{-1} .
\end{equation} 
and the isomorphisms between the division inertia groups 
\begin{equation}
\label{eq:InertiaDivisions}
 \begin{array}{c}
 \tilde{a}^{-1}  [\tilde{a}\tilde{b}\tilde{d},\tilde{a}\tilde{b}\tilde{d}]_{\mathcal{R}_g}\tilde{d}^{-1}  
\ \simeq \
 \tilde{a}'^{-1}  [\tilde{a}'\tilde{b}'\tilde{d}',\tilde{a}'\tilde{b}'\tilde{d}']_{\mathcal{R}_g}\tilde{d}'^{-1} 
\\
 \tilde{a}^{-1}  [\tilde{a}\tilde{c}\tilde{d},\tilde{a}\tilde{c}\tilde{d}]_{\mathcal{R}_g}\tilde{d}^{-1}  
\ \simeq \
 \tilde{a}'^{-1}  [\tilde{a}'\tilde{c}'\tilde{d}',\tilde{a}'\tilde{c}'\tilde{d}']_{\mathcal{R}_g}\tilde{d}'^{-1} .
\end{array}
\end{equation}

\end{prop}
\begin{proof} 
(1)   Let's take elements $[H]\in \tilde{a}^{-1} [\tilde{a}\tilde{b}\tilde{d},\tilde{a}\tilde{c}\tilde{d}]_{\mathcal{R}_g}\tilde{d}^{-1}$, $[H_b]\in \tilde{a}^{-1}  [\tilde{a}\tilde{b}\tilde{d},\tilde{a}\tilde{b}\tilde{d}]_{\mathcal{R}_g}\tilde{d}^{-1}$ and $[H_c]\in \tilde{a}^{-1}  [\tilde{a}\tilde{c}\tilde{d},\tilde{a}\tilde{c}\tilde{d}]_{\mathcal{R}_g}\tilde{d}^{-1}$. We set $[T]:=\tilde{a}\cdot [H]\cdot \tilde{d}:=[1_{\tilde a}\cdot H \cdot 1_{\tilde d}] \in [\tilde{a}\tilde{b}\tilde{d},\tilde{a}\tilde{c}\tilde{d}]_{\mathcal{R}_g}$,
We want to show $[H_b]\circ [H]\circ [H_c]\in \tilde{a}^{-1} [\tilde{a}\tilde{b}\tilde{d},\tilde{a}\tilde{c}\tilde{d}]_{\mathcal{R}_g}\tilde{d}^{-1}$.  

Using the commutativity \eqref{eq:CdotCirc} (Figure 3.3), we calculate
$$
\begin{array}{ccl}
&& \tilde{a}\cdot ([H_b]\circ [H]\circ [H_c]) \cdot \tilde{d}\\
&:=& [1_{\tilde a}\cdot (H_b \circ H\circ H_c) \cdot 1_{\tilde d}] \\
&=&  [(1_{\tilde a}\circ 1_{\tilde a}\circ 1_{\tilde a})\cdot (H_b \circ H\circ H_c) \cdot (1_{\tilde d}\circ 1_{\tilde d}\circ 1_{\tilde d})] \\
&=&  [(1_{\tilde a}\cdot H_b \cdot 1_{\tilde d}) \circ (1_{\tilde a} \cdot H \cdot 1_{\tilde d}) \circ (1_{\tilde a}\cdot  H_c \cdot 1_{\tilde d})] \\
&=&  [1_{\tilde a}\cdot H_b \cdot 1_{\tilde d}] \circ [1_{\tilde a} \cdot H \cdot 1_{\tilde d}] \circ [1_{\tilde a}\cdot  H_c \cdot 1_{\tilde d}] \\
\end{array}
$$
where, by assumptions on $H_b, H_c$ and $H$, we know that 
\vspace{-0.1cm}
$$
[1_{\tilde a}\cdot H_b \cdot 1_{\tilde d}] \in [\tilde{a}\tilde{b}\tilde{d}, \tilde{a}\tilde{b}\tilde{d}]_{\mathcal{R}_g},
 [1_{\tilde a} \cdot H \cdot 1_{\tilde d}]  \in [\tilde{a}\tilde{b}\tilde{d}, \tilde{a}\tilde{c}\tilde{d}]_{\mathcal{R}_g},
   [1_{\tilde a}\cdot  H_c \cdot 1_{\tilde d}] 
    \in [\tilde{a}\tilde{c}\tilde{d}, \tilde{a}\tilde{c}\tilde{d}]_{\mathcal{R}_g}.
$$
That is, all ``$\circ$"-product factors of the last expression belong to the $\mathcal{R}_g$-sets so that the product belongs to the $\mathcal{R}_g$-set $[\tilde{a}\tilde{b}\tilde{d}, \tilde{a}\tilde{c}\tilde{d}]_{\mathcal{R}_g}$. That is, $[H_b]\circ [H]\circ [H_c]$ belongs to the division  $\mathcal{R}_g$-set, i.e.\ we get  \eqref{eq:R_g-invariance1}.

That each division inertia group \eqref{eq:InertiaDivision} acts on the division $\mathcal{R}_g$-set \eqref{eq:DivisionRgset} transitively follows from the fact that, before the division,  inertia groups act on $\mathcal{R}_g$-set transitively (Proposition \ref{InertiaAction}) and the commutativity \eqref{eq:divisionComposit} of the actions with $\cdot$-division. That the actions are simple is a consequence of the fact that the groups are subgroups of $[\tilde{b},\tilde{b}]$ and $[\tilde{c},\tilde{c}]$, respectively, which act simply (Corollary \ref{SimpleTransitive}, 1)). Then, the isomorphism of two division inertia groups is a trivial consequence of this.

\smallskip
(2)  Choose also  $[ A]\in [\tilde a,\tilde a']_{\mathcal{R}_g}$ and $[D]\in [\tilde d,\tilde d']_{\mathcal{R}_g}$, and consider the``$\cdot$"-compositions  $[A][B][D] (:=[A]\cdot[B]\cdot[D])\in [\tilde{a}\tilde{b}\tilde{d},\tilde{a}' \tilde{b}' \tilde{d}' ]_{\mathcal{R}_g}$ and  
 $[A][C][D] (:=[A]\cdot[C]\cdot[D])\in [\tilde{a}\tilde{c}\tilde{d},\tilde{a}\tilde{c}' \tilde{d}' ]_{\mathcal{R}_g}$. 
 Then,  the $\circ$-invertibility \eqref{eq:R-invert} implies 
the natural bijection of $\mathcal{R}_g$-sets 
\vspace{-0.1cm}
$$
\begin{array}{ccl}
[\tilde a\tilde b \tilde d,\tilde a \tilde c \tilde d]_{\mathcal{R}_g}
 &  \tilde{-\!\!\! \to} & \!\!
[\tilde a' \tilde b'  \tilde d' ,\tilde a'  \tilde c'  \tilde d' ]_{\mathcal{R}_g},  
\\
\ [T] & \mapsto &\! \! [T']:=(1/[A][B][D])\!\circ [T]\circ ([A][C][D]).\!\!
  \end{array}
\vspace{-0.1cm}
$$
and the natural bijection of local inertia groups
$$
\begin{array}{ccl}
[\tilde a\tilde b \tilde d,\tilde a \tilde b \tilde d]_{\mathcal{R}_g}
&  \tilde{-\!\!\! \to} &\!\!
[\tilde a' \tilde b'  \tilde d' ,\tilde a'  \tilde b'  \tilde d' ]_{\mathcal{R}_g},
\\
\ [T] & \mapsto &\! \! [T']:=(1/[A][B][D])\!\circ [T]\circ ([A][B][D]).\\
\!\!\![\tilde a\tilde c \tilde d,\tilde a \tilde c \tilde d]_{\mathcal{R}_g}
&  \tilde{-\!\!\! \to} &\!\!
[\tilde a' \tilde c'  \tilde d' ,\tilde a'  \tilde c'  \tilde d' ]_{\mathcal{R}_g},
\\
\  [T] & \mapsto &\! \! [T']:=(1/[A][C][D])\!\circ [T]\circ ([A][C][D]).\\
  \end{array}
\vspace{-0.1cm}
$$

\medskip
Next, let us consider the following 3 diagrams:
\begin{equation}
\label{eq:Commutative}
\begin{array}{cccccc}
 \ [\tilde b,\tilde c]  &\!\!\!\! \tilde{-\!\!\! \rightarrow} \!\!\!\!& \qquad [\tilde b',\tilde c']  \qquad\ , & [H]  & \!\!\! \mapsto \!\!\! & \ \   [1/B] \circ [H] \circ [C]   \\
\vspace{-0.35cm}
\\
 \downarrow  &                                   &  \!\!\! \downarrow  & \downarrow  &  &\quad  \downarrow \!\!\!\!
 \\
\vspace{-0.35cm}
\\
\!\!\!\! [\tilde a\tilde b \tilde d,\tilde a \tilde c \tilde d]
& \!\!\!\!  \tilde{-\!\!\! \to} \!\!\!\! &\!
[\tilde a' \tilde b' \tilde d' ,\tilde a'  \tilde c'  \tilde d' ],  & [\tilde a H \tilde d] &\!\!\! \mapsto \!\!\!& \!\!\!\! \!\!\! [X] \quad \!\!\! \overset{?}{=} \!\!\quad [Y], \  \ \\
\\
\ [\tilde b,\tilde b]  &\!\!\!\! \tilde{-\!\!\! \rightarrow} \!\!\!\!& \qquad [\tilde b' ,\tilde b' ]  \qquad\ , & [H]  & \!\!\! \mapsto \!\!\! & \ \   [1/B] \circ [H] \circ [B]   \\
\vspace{-0.35cm}
\\
 \downarrow  &                                   &  \!\!\! \downarrow  & \downarrow  &  & \quad \downarrow \!\!\!\!\!
 \\
\vspace{-0.35cm}
\\
\!\!\!\! [\tilde a\tilde b \tilde d,\tilde a \tilde b \tilde d]
& \!\!\!\!  \tilde{-\!\!\! \to} \!\!\!\! &\!
[\tilde a' \tilde b' \tilde d' ,\tilde a'  \tilde b'  \tilde d' ],  & [\tilde a H \tilde d] &\!\!\! \mapsto \!\!\!& \!\!\!\!\! \! [X] \quad \!\!\! \overset{?}{=} \!\!\quad [Y], \  \ \\
\\
\ [\tilde c,\tilde c]  &\!\!\!\! \tilde{-\!\!\! \rightarrow} \!\!\!\!& \qquad [\tilde c' ,\tilde c' ]  \qquad\ , & [H]  & \!\!\! \mapsto \!\!\! & \ \  [1/C] \circ [H] \circ [C]   \\
\vspace{-0.35cm}
\\
 \downarrow  &                                   &  \!\!\! \downarrow  & \downarrow  &  &\quad  \downarrow \!\!\!\!\!
 \\
\vspace{-0.35cm}
\\
\!\!\!\! [\tilde a\tilde c \tilde d,\tilde a \tilde c \tilde d]
& \!\!\!\!  \tilde{-\!\!\! \to} \!\!\!\! &\!
[\tilde a' \tilde c'  \tilde d' ,\tilde a'  \tilde c'  \tilde d' ],  & [\tilde a H \tilde d] &\!\!\! \mapsto \!\!\!& \!\!\!\!\!\!  [X] \quad \!\!\! \overset{?}{=} \!\!\quad [Y]. \  \ \\
\\
\end{array}
\vspace{-0.45cm}
\end{equation}
where we ask, for each diagram,  whether the image $[X]$ obtained from the starting datum $[H]$ by the left-turn of the diagram  coincides with the image $[Y]$ obtained by the right turn of the diagram. 

We show the coincidence of $[X]$ and $[Y]$ only for the first case, since the other two cases are shown by completely parallel arguments.

Let us describe explicitly the maps $X$ and $Y$ representing the classes $[X]$ and $[Y]$.

\vspace{0.1cm}
\hspace{-0.28cm} 
\includegraphics[width=0.83\textwidth]{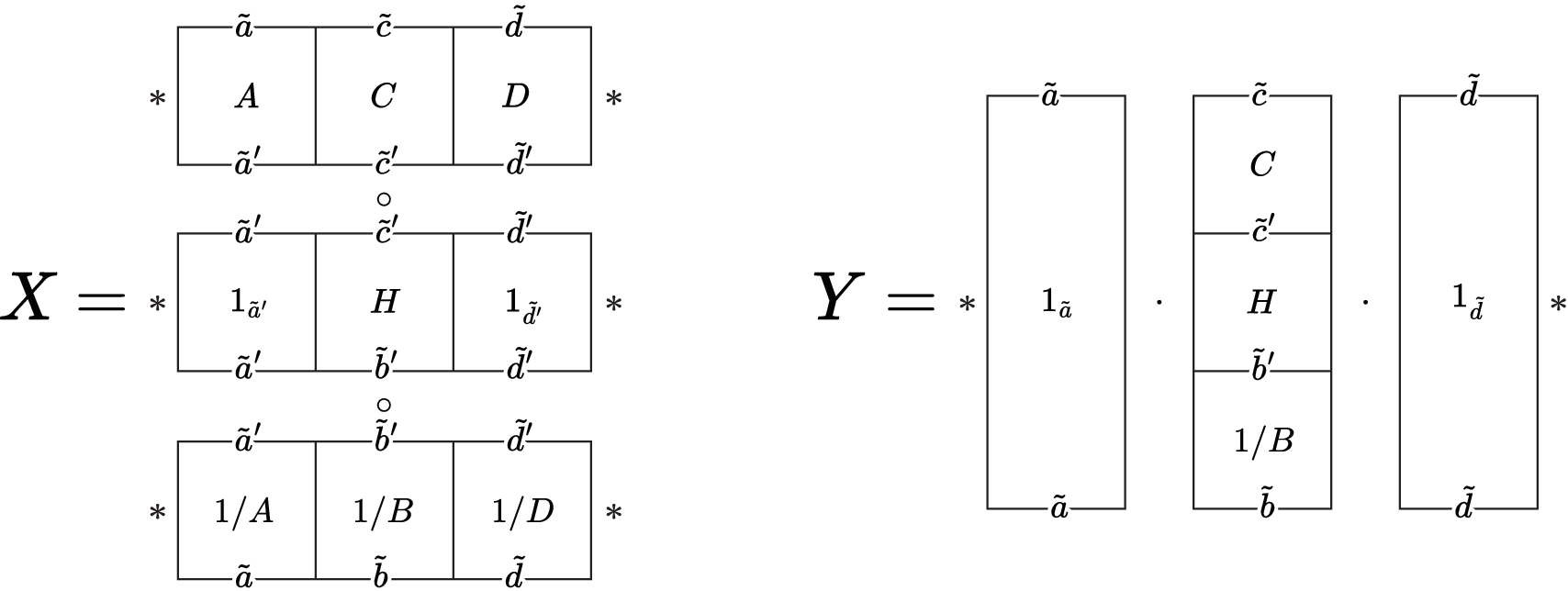}

\vspace{0.0cm}
\centerline{\  Figure 5.3: \ 2 presentations of a relative homotopy class in $[\tilde{a}' \tilde{b}' \tilde{d}' ,\tilde{a}' \tilde{c}' \tilde{d}' ]$ }

\bigskip

Here, recall Proposition \ref{PropA}, c) for $1/A, 1/B$ and $1/D$. Recall also the commutativity \eqref{eq:CdotCirc} of Proposition \ref{PropA}, f), so that the above figure implies that, in order to show the homotopy equivalence $X\sim Y$, it is sufficient to show the homotopy equivalences of the left and right  sides of $X$ and $Y$.  This is achieved trivially: 
$$
1/A\circ 1_{\tilde a}\circ A \sim 1/A\circ A \sim 1_{\tilde a'}   \quad \text{and} \quad
1/D\circ 1_{\tilde d}\circ D \sim 1/D\circ D \sim 1_{\tilde d'} .
$$
(use Proposition \ref{PropA}, c), i.e.\  the pair $1/A$ and $A$ and also the pair $1/D$ and $D$ cancel each other so that their roles were auxiliary). 

Suppose  $[H]$ at the starting belongs to  
$  \tilde{a}^{-1}  [\tilde{a}\tilde{b}\tilde{d},\tilde{a}\tilde{c}\tilde{d}]_{\mathcal{R}_g}\tilde{d}^{-1} $. It is equivalent to say that it is the division image (recall \eqref{eq:Relative2class}) of an element $[T]$ belonging to  
$ [\tilde{a}\tilde{b}\tilde{d},\tilde{a}\tilde{c}\tilde{d}]_{\mathcal{R}_g}$.  So its image $[T' ] (:=[X])$ in the RHS belongs to $[\tilde a' \tilde b'  \tilde d' ,\tilde a'  \tilde c' \tilde d' ]_{\mathcal{R}_g}$. The equality $[X]=[Y]$ implies that the down image $[T'] (:=[Y])$ of $[1/B] \circ [H] \circ [C] \in [\tilde b' ,\tilde c' ] $ belongs to $[\tilde a' \tilde b'  \tilde d' ,\tilde a'  \tilde c'  \tilde d' ]_{\mathcal{R}_g}$. That is, $[1/B] \circ [H] \circ [C]$ belongs to $\tilde{a}'^{-1}  [\tilde{a}'\tilde{b}'\tilde{d}',\tilde{a}'\tilde{c}'\tilde{d}']_{\mathcal{R}_g}\tilde{d}'^{-1}  $. 

So, the restriction of the correspondence $[\tilde b,\tilde c]  \to [\tilde b',\tilde c']$ to the division $\mathcal{R}_g$-set \eqref{eq:DivisionRgset} induces one direction of  the morphism \eqref{eq:R_g-invariance2}. Applying the above result to the inverse correspondence $[\tilde b' ,\tilde c']  \to [\tilde b,\tilde c]$, we obtain the opposite direction morphism so that we get the bijection  \eqref{eq:R_g-invariance2}. 
\end{proof} 

\begin{cor}
\label{adfree}
The division $\mathcal{R}_g$-set $\tilde{a}^{-1}  [\tilde{a}\tilde{b}\tilde{d},\tilde{a}\tilde{c}\tilde{d}]_{\mathcal{R}_g}\tilde{d}^{-1}  $ and the division inertia groups 
$ \tilde{a}^{-1}  [\tilde{a}\tilde{b}\tilde{d},\tilde{a}\tilde{b}\tilde{d}]_{\mathcal{R}_g}\tilde{d}^{-1} $ and  
$\tilde{a}^{-1}   [\tilde{a}\tilde{c}\tilde{d},\tilde{a}\tilde{c}\tilde{d}]_{\mathcal{R}_g}\tilde{d}^{-1} $ 
do not depend on the choices of the liftings $\tilde a$ and $\tilde d$, but on the NC-tuple $(a,b,c,d)$ and on the liftings $\tilde b$ and $\tilde c$.
\end{cor}

\smallskip
The images under the relative realization morphism \(\chi\) of the two division inertia groups are equal as subgroups of \(\pi_2(X,*)\), by Proposition 5.3(1) and Corollary 3.7(2).

\smallskip
\begin{rem} 
\label{Reduction}
Consider an NC-tuple $(a,d)\in NC(b,c)$ for the monoid $A^+(L,\mathcal{R}_m)$. Then, we can introduce a new system of fundamental relations by $\mathcal{R}':=\mathcal{R}\cup \{\tilde {b}=\tilde{c}\}$ and its geometric lifting by $\mathcal{R}_g':= \mathcal{R}_g\cup \{(\tilde{a}^{-1}  [\tilde{a}\tilde{b}\tilde{d},\tilde{a}\tilde{c}\tilde{d}]_{\mathcal{R}_g}\tilde{d}^{-1}  )^{\pm1}\}$, where $A^+(L,\mathcal{R}_m')$ does not carry the NC-tuple $(a,b,c,d)$. The renewed graph $G(L^*,\mathcal{R}'_m)$ \eqref{eq:MonoidGraph3} obtains new edges (including the one between $\tilde b$ and $\tilde c$), and  the inertia groups \eqref{eq:LocalInertia} may increase. However, this is an undesired situation, since we want to keep the inertia groups as small as possible.\!\!
\end{rem}


\smallskip
\begin{rem} 
\label{GeometricAtoms}
We examine the meanings of  reducedness and minimality  (Definition 4.1) for the monoid $A^+=A^+(L,\mathcal{R}_m)$. Since we shall not use (explicitly) the results, the proofs are left to the reader. 

\smallskip
\noindent
{\bf Reducedness.}  Let $(a,d) \in NC(b,c)$. Then for any $a', d'\in A^+$, we have the inclusion:
\vspace{-0.1cm}
\begin{equation}
\label{eq:Reducedness}
 \tilde{a}^{-1}  [\tilde{a}\tilde{b}\tilde{d},\tilde{a}\tilde{c}\tilde{d}]_{\mathcal{R}_g}\tilde{d}^{-1} \ \subset \ (\widetilde{a'a})^{-1}  [\widetilde{a'a}\tilde{b}\widetilde{dd'},\widetilde{a'a}\tilde{c}\widetilde{dd'}]_{\mathcal{R}_g}(\widetilde{dd'})^{-1}  
\vspace{-0.1cm}
\end{equation}
\noindent
{\bf Minimality.}  Let $(a,d) \in NC(b,c)$. 
Suppose  there are $a',b',c',d'\in A^+(L,\mathcal{R}_m)$ such that  $b\!\sim\!  a'b'd', c \!\sim \! a'c'd'$.\ Then we have a natural inclusion\!\!\! 
\vspace{-0.1cm}
\begin{equation}
\label{eq:minimality}
\tilde{a}^{-1}  [\tilde{a}\tilde{b}\tilde{d},\tilde{a}\tilde{c}\tilde{d}]_{\mathcal{R}_g}\tilde{d}^{-1}   \ \supset \ (\widetilde{aa'})^{-1}  [\widetilde{aa'}\tilde{b'}\widetilde{d'd},\widetilde{aa'}\tilde{c'}\widetilde{d'd}]_{\mathcal{R}_g}(\widetilde{d'd})^{-1} 
\vspace{-0.1cm}
\end{equation}
(here, the inclusion map is given by left $\tilde a'$ and right $\tilde d'$ multiplications).\!\!

\end{rem}

\noindent 
\subsection{\normalsize The $\pi_2$-classes $\Pi(\tau)$ associated with an NC-twin $\tau$.}   

Let an NC-twin of NC-tuples $T_i=(a_i,b,c,d_i)\in A^+(L,\mathcal{R}_m)^4$ 
($i=1,2$) be given. We choose their liftings $\widetilde{T}_i=(\tilde{a}_i, \tilde{b},\tilde{c},\tilde{d}_i)\in (\mathcal{F}^+)^4$ 
($i=1,2$). 
Accordingly, we obtain  two division $\mathcal{R}_g$-sets
\begin{equation}
\label{eq:TwinDivisionRgset}
\!\! \Pi(\widetilde{T}_1), \ \ \Pi(\widetilde{T}_2) \quad  \subset \quad   [\tilde{b},\tilde{c}]  ,
\end{equation}
and two pairs of left and right division inertia groups 
\begin{equation}
\label{eq:TwinDivisionInertia}
\begin{array}{c}
G(\widetilde{T}_1,\tilde b) \ \subset\ [\tilde b,\tilde b]  \quad \text{and} \quad  
G(\widetilde{T}_1,\tilde c) \ \subset\ [\tilde c,\tilde c] \\
\text{ and }\\
G(\widetilde{T}_2,\tilde b) \ \subset\ [\tilde b,\tilde b]  \quad \text{and} \quad  
G(\widetilde{T}_2,\tilde c) \ \subset\ [\tilde c,\tilde c] 
\end{array}
\end{equation}
such that the left and right division inertia groups  act from the left and right on the division $\mathcal{R}_g$-sets simply and transitively. We recall  Corollary \ref{adfree} that these  data do not depend on the liftings $\tilde a_i$, $\tilde d_i$ ($i=1,2$).

Our goal is to construct, out of these data, some $\pi_2$-classes of $X$ which are independent of the liftings $\tilde b$ and $\tilde c$. One step before that,  we first construct some ``intermediate" classes, so to speak, pre-$\Pi$-class and pre-inertia group,  inside  $[\tilde{b},\tilde{b}]$ and $[\tilde{c},\tilde{c}]$ by taking the ``$\circ$''-differences of the two division sets associated with $\tau$, as follows. This construction is later supported by the identification of ``$\circ$''-product with ``$\cdot_2$''-product \eqref{eq:twoProducts}. 

  Set 
\begin{equation}
\label{eq:TildePi2}
\begin{array} {lll}
 \Pi(\tau; \tilde{b}) & := & \Pi(\widetilde{T}_1)\circ (1/\Pi(\widetilde{T}_2)) \\
&\ =& \{[H_1]\circ [1/H_2]\mid [H_i] \in \Pi(\widetilde{T}_i), i=1,2\}, \\ 
 \Pi(\tau; \tilde{c}) & := & (1/\Pi(\widetilde{T}_2)\circ \Pi(\widetilde{T}_1)) \\
&\ =& \{ [1/H_2]\circ [H_1]  \mid [H_i] \in \Pi(\widetilde{T}_i), i=1,2\}, \\ \end{array}
\end{equation}
\vspace{-0.2cm}
\begin{equation}
\label{eq:TildeInertia}
\begin{array} {lll}
 G(\tau;\tilde{b})& :=& G(\widetilde{T}_1,\tilde b)\circ (1/G(\widetilde{T}_2, \tilde b)) \\
    &\ =& \{  [G_1]\circ [1/G_2]\mid [G_i]\in G(\widetilde{T}_i,\tilde b), i=1,2\} \\
 G(\tau;\tilde{c})& :=& G(\widetilde{T}_1,\tilde c)\circ (1/G(\widetilde{T}_2, \tilde c)) \\
    &\ =& \{  [G_1]\circ [1/G_2]\mid [G_i]\in G(\widetilde{T}_i,\tilde c), i=1,2\} 
%
\end{array}
\end{equation}
where we note that $ \Pi(\tau; \tilde{b})\subset [\tilde b,\tilde b]$,  $ \Pi(\tau; \tilde{c})\subset [\tilde c,\tilde c]$, $G(\tau;\tilde{b})\subset [\tilde b,\tilde b]$ and $G(\tau;\tilde{c})\subset [\tilde c,\tilde c]$

\medskip
\noindent
{\bf Notational remark.} In the above formulae \eqref{eq:TildePi2} and \eqref{eq:TildeInertia}, we used a notation $1/*$ (instead of $*^{-1}$) in order to  adjust with the {\bf Notational remark} in \S3.1. We warn that the ambient groups $[\tilde b,\tilde b], [\tilde c,\tilde c]$ are abelian so that the notation actually implies the differences but not the quotients.

\medskip
\begin{prop} 
\label{Prop.5.7}
Let the notations be as above.

\noindent
{\rm (1)} For any element $[H]\in [\tilde b,\tilde c]$, the adjoint isomorphism \eqref{eq:AdjointIsomorphism} induces bijections
\begin{equation}
\begin{array}{rcl}
\label{eq:LRInertiaIso}
\mathrm{Ad}(H)& : & \Pi(\tau; \tilde b) \ \overset{\sim}{\longrightarrow}\ \Pi(\tau;\tilde c)\\
\mathrm{Ad}(H)&  :&  G(\tau; \tilde b) \ \overset{\sim}{\longrightarrow}\ G(\tau;\tilde c)
\end{array}
\end{equation}
Both bijections are independent of the choice of $[H]$.
 
 \smallskip
\noindent
{\rm (2)}  The pre-$\Pi$-class $\Pi(\tau; \tilde{b})$ is a single coset of the subgroup $G(\tau;\tilde b)$ in $[\tilde{b},\tilde{b}]$. In particular, the group $G(\tau; \tilde{b})$ acts simply transitively on $\Pi(\tau;\tilde b)$.
\end{prop}
\begin{proof}
(1) \  In addition to $[H]\in [\tilde b,\tilde c]$,  choose also any $[H_i]\in \Pi(\widetilde{T}_i)$ ($i=1,2$), whose existence is guaranteed by the NC-relations for $T_i$. Due to Corollary 3.7 (1), all the adjoint isomorphisms \eqref{eq:AdjointIsomorphism} associated with  $[H]$ and $[H_i]$  coincide to each other. 

In particular, we have
$$
\mathrm{Ad}(H)([H_1]\circ [1/H_2])=[1/H_2]\circ [H_1]
$$
for $[H_1]\circ [1/H_2]\in \Pi(\tau,\tilde{b})$ because \(\operatorname{Ad}(H)=\operatorname{Ad}(H_2)\) by Corollary 3.7, and we have 
$$
\qquad \qquad \qquad \mathrm{Ad}(H) \mid_{G(\widetilde{T}_i,\tilde{ b})} 
\ \ = \  \mathrm {Ad}(H_i) \mid_{G(\widetilde{T}_i,\tilde b)} \quad (i=1,2).
$$
On the other hand, since $[H_i] \in\Pi(\widetilde{T}_i)$, Proposition \ref{LocalToDvisionInertia}, 1. implies that
$$
\mathrm {Ad}(H_i) \ :\  G(\widetilde{T}_i,\tilde b) \ \overset{\sim}{\longrightarrow} \ G(\widetilde{T}_i,\tilde c).
$$
Therefore, 
$$
\begin{array}{ccl}
\mathrm{Ad}(H) (G(\tau;\tilde{ b})) 
&=&\mathrm {Ad}(H) \big(G(\widetilde{T}_1,\tilde b)\circ (1/G(\widetilde{T}_2,\tilde b))\big)\\
&= & G(\widetilde{T}_1,\tilde{ c})\circ (1/G(\widetilde{T}_2,\tilde{ c}) )
\ \ = \ \ G(\tau;\tilde c) .
\end{array}
$$

\smallskip
\noindent
(2) \ Fix $[H_i]\in \Pi(\widetilde{T}_i)$  ($i=1,2$) and set an element $[H_0]:=[H_1] \circ[1/H_2]$ of the pre-$\Pi$-class \eqref{eq:TildePi2}. By Proposition \ref{LocalToDvisionInertia}, 1., the left action of $G(\widetilde{T}_i,\tilde b)$ on $\Pi(\widetilde{T}_i)$ is simply transitive. Hence for every $[H_i']\in \Pi(\widetilde{T}_i)$, there exists a  unique $[G_i]\in G(\widetilde{T}_i,\tilde b)$ such that $H'_i]=[G_i]\circ [H_i]$.  Therefore,  
$$
[H'_1]\circ[1/H'_2]] =[G_1] \circ [H_1]\circ[1/H_2]\circ[1/G_2]=[H_0]\circ([G_1]\circ[1/G_2]),
$$
where we used the commutativity of $[G_1] $ and $[H_1]\circ[1/H_2]$ in $[\tilde b,\tilde b]$ (Proposition \ref{PropB}). 
As $G_1$ and $G_2$ vary , the last factor runs through $G(\tau;\tilde b)$. Thus,
$$
\Pi(\tau;\tilde b)=[H_0] \circ G(\tau;\tilde b).
$$ proving that $\Pi(\tau;\tilde b)$ is a single coset of $G(\tau;\tilde b)$.
\end{proof}

Recall that the relative realization morphism \(\chi\) identifies \([\widetilde b,\widetilde b]\) and \([\widetilde c,\widetilde c]\) with \(\pi_2(X,*)\).

 \smallskip
\begin{defn}
\label{FinalDefn} Set images under the relative realization morphism $\chi$  of the set \eqref{eq:TildePi2} in $\pi_2(X,*)$: 
  \vspace{-0.1cm}
\begin{equation}
\begin{array}{rcl}
\label{eq:FinalPi2}
 \Pi(\tau) & :=& \chi( \Pi(\tau; \tilde{b}) ) \\
  &\ =& \chi( \Pi(\tau; \tilde{c}) )
  \end{array}
 \end{equation}
and call it the {\it  $\Pi$-class associated with the NC-twin} $\tau$ \eqref{eq:Twin}. 

 Set the images under the relative realization morphism $\chi$  of the  groups    \eqref{eq:TildeInertia} in $\pi_2(X,*)$:
\begin{equation}
\label{eq:FinalInertia}
\begin{array}{rcl}
 G(\tau)&:=&\chi(G(\tau; \tilde{b}))\\
 &\ =& \chi(G(\tau; \tilde{c}))
\end{array}
 \vspace{-0.1cm}
 \end{equation}
and call it the {\it global inertia group associated with the NC-twin} $\tau$ \eqref{eq:Twin}.

\smallskip
\noindent
The equalities \eqref{eq:FinalPi2} and \eqref{eq:FinalInertia}  follow immediately from Proposition \ref{Prop.5.7} (1) and Corollary 3.7 (2).
 \end{defn}

\smallskip
\begin{rem}
\label{twinReverse}
{\rm Let $\tau^{t}:=\{(a_2,b,c,d_2),(a_1,b,c,d_1)\}$ be the NC-twin obtained by transposing NC-tuples in $\tau$ \eqref{eq:Twin}.  Then, by the construction, we have, }
$$
\Pi(\tau^{t})=-\Pi(\tau)\quad \text{and} \quad G(\tau^{t})=G(\tau).
$$
\end{rem}

\smallskip
A justification of the above Definition \ref{FinalDefn} and its notation is given by the following final step of constructing the $\Pi$-class.

\smallskip
\begin{theorem}
\label{FinalTheo}
1. {\it  The  $\Pi$-class \eqref{eq:FinalPi2} and the global inertia group  \eqref{eq:FinalInertia} associated with the NC-twin \eqref{eq:Twin} are independent of  the   liftings $\tilde{b},\tilde{c}$. }

\noindent
2.  {\it The  $\Pi$-class forms a single coset of the global inertia group. 

\noindent
3.  The $\Pi$-class is a singleton  
%
 $\Leftrightarrow$ The local inertia groups $[\tilde{a}_i\tilde{b}\tilde{d}_i,\tilde{a}_i\tilde{b}\tilde{d}_i]_{\mathcal{R}_g} 
 $  ($i=1,2$)  for a (and then for all)  lifting of the NC-twin is trivial. 
}
\end{theorem}

\begin{proof}
1. The proof for the inertia group is parallel to that for the set of $\Pi$-classes. So, we proceed the proof only for the $\Pi$-classes.

 Let  us consider two liftings $\tilde{b}, \tilde{b}' \in \mathcal{F}^+$ of $b$ and  $\tilde{c}, \tilde{c}' \in \mathcal{F}^+$ of $c$, and choose any $[B]\in [\tilde{b},\tilde{b}']_{\mathcal{R}_g}$ and $[C]\in [\tilde{c},\tilde{c}']_{\mathcal{R}_g}$, respectively.  Recall Proposition \ref{LocalToDvisionInertia} 2. so that we have the bijections
\vspace{-0.07cm}
$$
\begin{array}{cccrcl}
\tilde{a}_1^{-1} [\tilde{a}_1\tilde{b}\tilde{d}_1,\tilde{a}_1\tilde{c}\tilde{d}_1]_{\mathcal{R}_g}\tilde{d}_1^{-1}   \quad\!\! & \simeq &
\tilde{a}_1^{-1} [\tilde{a}_1\tilde{b}'\tilde{d}_1,\tilde{a}_1\tilde{c}'\tilde{d}_1]_{\mathcal{R}_g}\tilde{d}_1^{-1} , \\
\  [H_1] & \mapsto &  [1/B]\! \circ\! [H_1]\! \circ\!  [C]\!=\![(1/B)\!\circ\! H_1\!\circ C ]\\
\\
\tilde{a}_2^{-1} [\tilde{a}_2\tilde{b}\tilde{d}_2,\tilde{a}_2\tilde{c}\tilde{d}_2]_{\mathcal{R}_g}\tilde{d}_2^{-1}  \quad\!\! &\simeq  &
\tilde{a}_2^{-1} [\tilde{a}_2\tilde{b}'\tilde{d}_2,\tilde{a}_2\tilde{c}'\tilde{d}_2]_{\mathcal{R}_g}\tilde{d}_2^{-1} ,  \\
\ [H_2] &  \mapsto  &   [(1/B)]\!\circ\! [H_2] \! \circ\! [C] \!=\! [(1/B)\!\circ\! H_2\!\circ\! C]
\end{array}
 \vspace{-0.07cm}
$$
which induce a bijection:
$$
\begin{array}{ccc}
\ [H_1] \circ [1/H_2]  &  \mapsto  &  [(1/B)\!\circ\! H_1 \!\circ \!C] \circ  (1/( [(1/B)\! \circ \! H_2 \! \circ \!C])
 \end{array}
 \vspace{-0.05cm}
 $$
 between the two pre-$\Pi$-classes $ \Pi(\tau; \tilde{b})$ and $ \Pi(\tau; \tilde{b}')$ (bijectivity follows from the fact that we can construct the inverse map by exchanging the roles of the liftings   $\tilde{b}, \tilde{b}'$ of $b$ and  $\tilde{c}, \tilde{c}'$ of $c$). 

 Then the  embedded images in $\pi_2(X,*)$ of these elements
 $[H_1] \circ [1/H_2]$ and  $[(1/B)\!\circ\! H_1 \!\circ\! C] \circ(1/([1/B)\!\circ\! H_2 \! \circ \! C])$
  by the relative realization morphism $\chi$ \eqref{eq:immersion} can be obtained by the maps from the following figure to $X$ where the maps on the front faces and the back faces of the cylindrical figures are given  by the constant maps $*$, and the front faces become the base points of the $\pi_2$-classes. 
 
 
 \vspace{0.3cm}
\hspace{-0.2cm} 
\includegraphics[width=0.79 \textwidth]{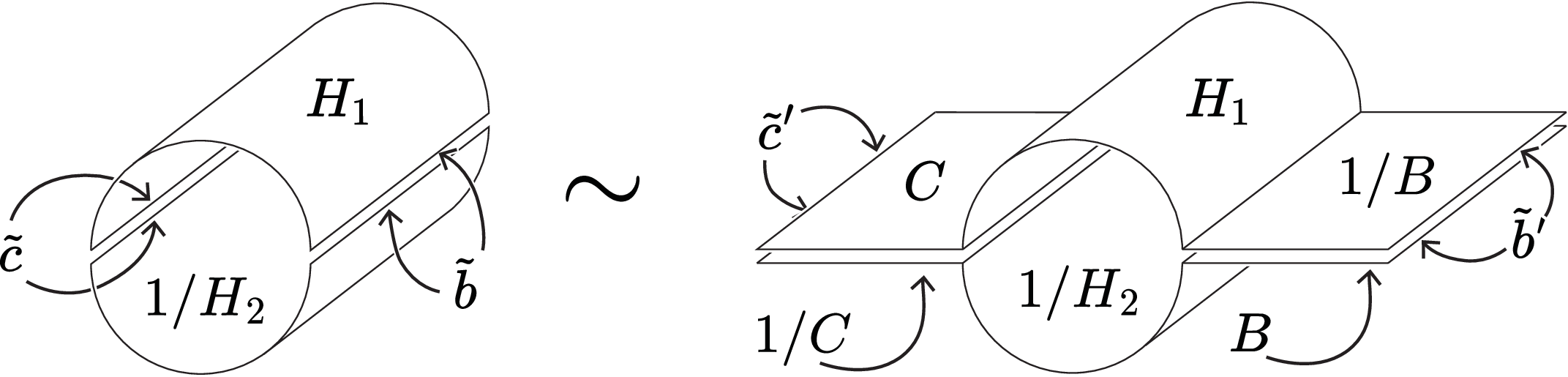}

\vspace{0.2cm}
\centerline{\quad Figure 5.4: \ $\pi_2$-classes in $\Pi(\tau)$}

\bigskip
 
 One sees that,  in the right hand figure, the images of   $(1/B)\circ B$ and $C\circ(1/ C)$ can be retracted to the edges $\tilde b$ and $\tilde c$, respectively, so that right and left figures give the same second homotopy classes. That is, the $\pi_2$-class is independent of the liftings $\tilde b$ and $\tilde c$.  
 
 2. This is obvious, since the pre-$\Pi$-class forms a single coset of the pre-inertia group.

3.  By 2, the  $\Pi(\tau)$ is a singleton if and only if $G(\tau)$ is trivial. By the definition of $G(\tau)$as the product of the two division inertia subgroups, this is equivalent to the triviality of both division  inertia groups, hence of the corresponding  local inertia groups.  
  \end{proof}
 \begin{rem}
 \label{reducedness}
 {\rm According to the geometric interpretation of reducedness and  minimality given in Remark \ref{GeometricAtoms}, we have:}
 
\smallskip
\noindent
{\bf reducedness.}  For an NC-twin $\tau$ \eqref{eq:Twin} and for any $a_1',a_2', d_1',d_2'\in A^+(L,\mathcal{R}_m)$, consider an NC-twin $\tau':=\{(a_1'a_1,b,c,d_1d_1'),(a_2'a_2,b,c, d_2d_2')\}$. Then, we have 
\begin{equation}
\label{eq:Reducedness}
\Pi(\tau) \ \subset \  \Pi(\tau')
\end{equation}
\noindent
{\bf Minimality.}   For an NC-twin $\tau$ \eqref{eq:Twin}, suppose  there are $a',b',c',d'\in A^+$ such that  $b\sim a'b'd', c\sim a'c'd'$. Set $\tau':=\{(a_1a',b',c',d'd_1),(a_2a',b',c',d'd_2)\}$. Then we have 
\begin{equation}
\label{eq:minimality}
\Pi(\tau) \supset \Pi(\tau') .
\end{equation}

\smallskip
\noindent
\end{rem}
 
\noindent
{\bf   \qquad  \ A summary  of the construction of the $\Pi$-class $\Pi(\tau)$ 

\noindent
\qquad associated with an NC-twin $\tau$ in the monoid $A^+(L,\mathcal{R}_m)$.} 

\smallskip
\noindent
(1) We start with a topological space $X$ whose fundamental group has a semi-positive presentation $(L,\mathcal{R})$ (\S2). We fix  its geometric lifting $(W,\widetilde{L}, \mathcal{R}_g)$ (\S3). 
Suppose  that the associated monoid $A^+(L, \mathcal{R}_m)$
has an NC-twin $\tau=\{({a}_1,{b},{c},{d}_1), ({a}_2,{b},{c},{d}_2)\}$ of  NC-tuples  (\S4 Definition \ref{NCtuple}).

\smallskip
\noindent
(2) Choose any representatives $\tilde{a}_i,\tilde{b}, \tilde{c},\tilde{d}_i \in \mathcal{F}^+$ ($i\!=\!1,2$) \eqref{eq:PositiveMonoid} of the tuples. The non-cancellative relations \eqref{eq:NCRelation} imply  the non-emptiness   of the  $\mathcal{R}_g$-sets  $[\tilde{a}_i\tilde{b}\tilde{d}_i,\tilde{a}_i\tilde{c}\tilde{d}_i]_{\mathcal{R}_g}\not=\emptyset$ ($i=1,2$) (Proposition \ref{GeometricMonoid}), on which the local inertia groups $[\tilde{a}_i\tilde{b}\tilde{d}_i,\tilde{a}_i\tilde{b}\tilde{d}_i]_{\mathcal{R}_g}$ and $[\tilde{a}_i\tilde{c}\tilde{d}_i,\tilde{a}_i\tilde{c}\tilde{d}_i]_{\mathcal{R}_g}$
($i=1,2$) (Definition \ref{LocalInertia}) act  simply and transitively from the left and right (Proposition \ref{InertiaAction}). 

\smallskip
\noindent
(3) 
Apply the division morphism \eqref{eq:Relative2class} to the   $\mathcal{R}_g$-sets: 
\vspace{-0.1cm}
$$
[\tilde{a}_i\tilde{b}\tilde{d}_i,\tilde{a}_i\tilde{c}\tilde{d}_i]_{\mathcal{R}_g} \to 
 [\tilde{b},\tilde{c}], \quad  [T_i] \mapsto [H_i]   \quad (i=1,2)
\vspace{-0.1cm}
$$
and get  {\footnotesize $\tilde{a}_i^{-1} [\tilde{a}_i\tilde{b}\tilde{d}_i,\tilde{a}_i\tilde{c}\tilde{d}_i]_{\mathcal{R}_g}\tilde{d}_i^{-1} $}, called the division $\mathcal{R}_g$-sets \eqref{eq:TwinDivisionRgset}. The division inertia groups  \eqref{eq:TwinDivisionInertia}
 $  \tilde{a}_i^{-1} [\tilde{a}_i\tilde{b}\tilde{d}_i,\tilde{a}_i\tilde{b}\tilde{d}_i]_{\mathcal{R}_g}\tilde{d}_i^{-1}$ and   
$\tilde{a}_i^{-1} [\tilde{a}_i\tilde{c}\tilde{d}_i,\tilde{a}_i\tilde{c}\tilde{d}_i]_{\mathcal{R}_g}\tilde{d}_i ^{-1}$ act simply and transitively on the division $\mathcal{R}_g$-sets (Proposition  \ref{LocalToDvisionInertia} 1). 

\smallskip
\noindent
(4) Consider the  set  \eqref{eq:TildePi2} of differences $[H_1]\circ[1/H_2]$ of elements of the two division $\mathcal{R}_g$-sets, and realize it  in $\pi_2(X,*)$ by the relative realization morphism  $\chi$ \eqref{eq:immersion}. The image $\Pi(\tau)$ \eqref{eq:FinalPi2} is the  {\it  $\Pi$-class associated with the NC-twin $\tau$}.

The $\Pi$-class $\Pi(\tau)$ is independent of the liftings $\tilde{a}_i,\tilde{b}, \tilde{c},\tilde{d}_i$ ($i=1,2$), and forms a single coset  of the global inertia group $ G(\tau)$ \eqref{eq:FinalInertia}, where the $G(\tau)$-action on $\Pi(\tau)$  is given in the following (5).

\smallskip
\noindent
(5) Summarizing  the above  construction (2)-(4),  we obtain the surjective map to $\Pi(\tau)$ (which we may call the $\Pi_2$-morphism): 
\vspace{-0.1cm}
\begin{equation}
\begin{array}{c}
\label{eq:Pi2Morphism}
[\tilde{a}_1\tilde{b}\tilde{d}_1,\tilde{a}_1\tilde{c}\tilde{d}_1]_{\mathcal{R}_g} 
\times
[\tilde{a}_2\tilde{b}\tilde{d}_2,\tilde{a}_2\tilde{c}\tilde{d}_2]_{\mathcal{R}_g}
\qquad \longrightarrow\!> \qquad 
\Pi(\tau).\quad\  \\
\qquad [T_1]\times [T_2]\qquad \mapsto\qquad \chi\big( [{\tilde a}_1^{-1} T_1{\tilde d}^{-1}_1]\circ [1/({\tilde a}_2^{-1} T_2{\tilde d}_2^{-1})]\big)
 \end{array}
\vspace{-0.1cm}
\end{equation}
which is equivariant with the  homomorphisms 
\vspace{-0.1cm}
\begin{equation}
\begin{array}{cccc}
&[\tilde{a}_1\tilde{b}\tilde{d}_1,  \tilde{a}_1\tilde{b}\tilde{d}_1]_{\mathcal{R}_g}\  \times  \  [ \tilde{a}_1\tilde{c}\tilde{d}_1 ,  \tilde{a}_1\tilde{c}\tilde{d}_1]_{\mathcal{R}_g} \!\!& \to\!> &\! G(\tau) \!\! \\
&[\tilde{a}_2\tilde{b}\tilde{d}_2,\tilde{a}_2\tilde{b}\tilde{d}_2]_{\mathcal{R}_g}\ \times \  [\tilde{a}_2\tilde{c}\tilde{d}_2,\tilde{a}_2\tilde{c}\tilde{d}_2]_{\mathcal{R}_g}  
\!\!& \to\!> &\! G(\tau)\!\!
\end{array}
\vspace{-0.1cm}
\end{equation}
from the local inertia groups to the global inertia group $G(\tau)$ (\eqref{eq:DivisionCommutative},  \eqref{eq:TildeInertia} and \eqref{eq:FinalInertia}) (where $\to\!>$ is the surjective morphism given by the composition of $\chi$ with the $\circ$-quotient of the division inertia groups).

\subsection{\normalsize Final Remarks on the $\Pi$-class}

\noindent
{\bf 1.}  The constructed $\pi_2$-classes $\Pi(\tau)$ \eqref{eq:FinalPi2} and their  associated inertia group $G(\tau)$ \eqref{eq:FinalInertia} are  consequences of the three concepts:

i)  {\it Semi-positive presentation} $(L,\mathcal{R})$ (Definition \ref{Semi-positive} ) of $\pi_1(X,*)$. 

ii)  {\it Geometric liftings} $(W,\widetilde{L},\mathcal{R}_g)$ (Definition \ref{W} and \ref{GeoFundRel}).

iii)  {\it NC-twin} (Definition \ref{NCtuple} 4. ). 

\noindent
Good examples in \S2.2, including discriminant complements for finite, affine or elliptic root systems and Yoshinaga's presentation for line arrangements, seem to suggest that these three concepts are not independent of each other, and there might be a unified structure  (recall \S1(7) Problems, \S3.2.3 Question and Footnote 21). How to  distinguish such good classes from some artificial cases (recall  Warnings in \S2.2 7.\ and Remark 5.5)?  Might some categorification suggest such new framework?\!\!   

\medskip
\noindent
{\bf 2.}  
Malcev \cite{M1, M2} and Lambek \cite{L} gave  infinite sequences of conditions for a  monoid $A^+$ in order to be embeddable in its localization group $A$. Cancellativity is the first condition among them. Any finite subset of the conditions is not sufficient for injectivity.  
Then, it is interesting to find the homotopical meaning of Malcev's and/or  Lambek's conditions. That is, could disturbing their conditions up to some level for the monoid $A^+(L,\mathcal{R}_m)$ imply the existence of certain higher homotopy classes of $X$  up to some level? 
 The questions posed above may apply, in particular,  to the classifying space of the monoid $A^+$ studied in the Appendix I.

\medskip
\noindent
{\bf 3.} Since we don't know yet how to characterize``good classes", we  tentatively propose to study the following rather vague {\bf Problems}.\!\!

\medskip
\noindent
{\bf Problem.}  Consider the $\mathbb{Z}[A(L,\mathcal{R})]$-module generated by all $\Pi$-classes:
\begin{equation}
\label{eq:total Pi}
 \Pi(W,\tilde{L},\mathcal{R}_g):= \underset{ \substack{twin \ \tau\ of\ the\ \\ monoid\ A^+(L.\mathcal{R}_m)} }{\oplus} \  \mathbb{Z}[A(L,\mathcal{R})] \cdot \Pi(\tau)
\end{equation}
where (1) each term $\mathbb{Z}[A(L,\mathcal{R})] \cdot \Pi(\tau)$ is the $\mathbb{Z}[A(L,\mathcal{R})]$ submodule generated by $\Pi(\tau)$, 
(2) the summation index runs over all reduced and minimal  NC-twins of the monoid 
$A^+(L,\mathcal{R}_m)$.

Consider the natural homomorphism
\begin{equation}
\label{eq:PiMap}
 \Pi(W,\tilde{L},\mathcal{R}_g)\ \rightarrow \  \pi_2(X,*)
\end{equation}
(1)   Determine the image and the kernel of the morphism \eqref{eq:PiMap}, i.e.\   characterize the image, and  
 find a generator system of the kernel.\!\!\

\smallskip
\noindent
(2) When is the  projection image of 
the module \eqref{eq:total Pi} in 
\vspace{-0.15cm}
$$
 \pi_2(X, *) \otimes (\mathbb{Z}A(L,\mathcal{R})/\text{augmentation ideal})
\vspace{-0.15cm}
 $$
 invariant  under a change of the geometric fundamental relations?

\smallskip
\noindent
(3) There seems to be a ``correlation" between the size of geometric fundamental relations and that of the  inertia group.  Can we clarify this?  In particular, can we characterize the extremal case of geometric fundamental relations when the inertia groups are trivial  (recall  {\bf Question 3.2.3.})?
 %

\section{ Hattori's  second homotopy class} 
We apply the construction of $\Pi$-classes to \S2.2 Examples 6 (a) and (b) of Yoshinaga's presentation \cite{Y2} for   line arrangement complements. 



\subsection{\normalsize (a): three generic lines in a plane}

\smallskip
Recall  the  generic 3 lines system $\mathcal{A}_3$ (\S2.2 Example 6(a)), and 
Yoshinaga's presentation \eqref{eq:3linesMonoid} of the fundamental group of $M(\mathcal{A}_3)$. The  associated monoid $A^+(L,\mathcal{R}_m(\mathcal{A}_3))$  carries NC-twins: 
\vspace{0.1cm}

\quad  $\tau_{23}$ \eqref{eq:Twin23}, \ \  $\tau_{31}$ \eqref{eq:Twin31},  \  \  $\tau_{12}$ \eqref{eq:Twin12}  \ \ \ and  \ \ \ $\tau_{12}^*$, \ \  $\tau_{12}'^*$ \eqref{eq:Twin*},  
\vspace{0.1cm}

\noindent
which are the main objects of the present section.  A geometric lifting of the presentation \eqref{eq:3linesMonoid}  is constructed by using the  Hattori homotopy type $T^3_2$ of $M(\mathcal{A}_3)$ (see the proof of Theorem \ref{ThmY-H} for the construction). 


\smallskip
\begin{thm} 
\label{ThmY-H}

We have the following statements which are independent of the choice of the liftings  of the fundamental relations.                                                                                                              
 
\smallskip
\noindent
{\rm  (1)} The  inertia groups  $G(\tau_{23}),G(\tau_{31}),G(\tau_{12}),G(\tau_{12}^*)$\! 
 and $G(\tau_{12}'^*)$ 
are trivial.\!\!\!\! Hence the $\Pi$-classes $\Pi(\tau_{23}),\Pi(\tau_{31}),\Pi(\tau_{12}),\Pi(\tau_{12}^*)$\! 
 and $\Pi(\tau_{12}'^*)$ are singletons.\!\!\!\!\!
 
\smallskip
\noindent
{\rm (2)} The $\Pi$-classes $\Pi(\tau_{23}),\Pi(\tau_{31})$,$\Pi(\tau_{12})$ and $\Pi(\tau_{12}^*)$ represent  Hattori's second homotopy class, and  $\Pi(\tau_{12}'^*)=0$ (see the proof for the definition).\!\!\!\!\!

\smallskip
\noindent
{\rm (3)} 
Using the  $\pi_1(M(\mathcal{A}_3),*)$-action, we have the following relations. 
\begin{equation}
\label{eq:Pi-clases} 
\gamma_3^{-1}\Pi(\tau_{23})\ =\ \gamma_2 \Pi(\tau_{31}) \ =\ \Pi(\tau_{12}) \ =\ \Pi(\tau_{12}^*) .
\end{equation}
\end{thm}

\noindent
{\it Proof.} 
The rest of the present subsection is devoted to the proof. 

\medskip
\noindent
{\bf Liftings of the twins.} \  For the calculations of the inertia groups and the $\Pi$-classes, we fix liftings of the twins in terms of the free monoid $\mathcal{F}^+\simeq L^*$ as given in \eqref{eq:l=3}, \eqref{eq:l=4} and \eqref{eq:l=5} of the proof of Proposition \ref{NCList}. The final result (Theorem \ref{ThmY-H}) is independent of the liftings (Theorem \ref{FinalTheo}).

\medskip
\noindent
{\bf Notation.}  In the following proof, we shall denote the generators $\tilde{\gamma}_1,\!\tilde{\gamma}_2$ and $\tilde{\gamma}_3$ of the free monoid $\mathcal{F}^+\!\simeq\!L^*$ by $1,2$ and $3$ for simplicity.\footnote{
 We warn of a coincidence of notation in the present section  \S6.1 and that  in previous sections. Namely, in \S4 and 5, two NC-tuples in an NC-twin are (often) distinguished by the subscripts 1 and 2. However, in \S6.1,  the subscripts $12,23,31$ are used as the names of the NC-twins. In order to avoid the confusion, we shall separate two NC-tuples in an NC-twin by putting ``\ ' \ " to the second one.
 }

\medskip
\noindent
{\bf 1.  Geometric fundamental relations}

We recall   
 Hattori's description of the homotopy type  $T_2^3$  of $M(\mathcal{A}_3)$. Then, we  lift the set $L=\{1,2,3\}$ of the generators to the homotopy classes of the three loops $1,2,3: [0,1] \to T_1^3$ (denoted by the same symbols $1,2$ and $3$). We also lift the algebraic fundamental relations $\mathcal{R}_m$ in \eqref{eq:3linesMonoid} 
to the geometric  fundamental relations (Definition \ref{GeoFundRel}):\!\!
 \vspace{-0.05cm}
\begin{equation}
 \vspace{-0.05cm}
\label{eq:geoABC}
\mathcal{R}_g(123\!\overset{A_g}{\sim}\!132)\!=\!\{[A_g]\}, \mathcal{R}_g(123\!\overset{B_g}{\sim}\!312)\!=\!\{[B_g]\}, \mathcal{R}_g(123\!\overset{C_g}{\sim}\!231)\!=\!\{[C_g]\}\!\!\!\!\!\!\!\!\!
\end{equation}
given by the homotopy classes of the maps  $A_g,B_g,C_g\! :\! [0,1]^2\! \to\! T_2^3$.\footnote
{Be cautious that the geometric maps $A_g, B_g$ and $C_g$ have ambiguity modulo $\pi_2$.}

\smallskip
In \cite{H}, Akio Hattori described the homotopy type of the complement of a union of $k$ generic hyperplanes (defined over $\mathbb{R}$) in $\mathbb{C}^n$ for $k\!\ge \!n$  as the $n$-th skeleton $T^k_n$ of the $k$-dimensional real torus $T^k$.  
 \vspace{-0.05cm}
In the case of $\mathcal{A}_3$, $k\!=\!3$ and $n\!=\!2$. So, using the universal covering $\mathbb{R}^3$ of $T^3$, we have the homotopy equivalence:
\vspace{-0.2cm}
$$
 M(\mathcal{A}_3)  \quad \sim \quad  T_2^3:=\tilde{T}_2^3/\mathbb{Z}^3 
$$
 where $\tilde{T}_2^3:=(\mathbb{R}\!\times\!\mathbb{R}\!\times\!\mathbb{Z}) \cup (\mathbb{R}\!\times\!\mathbb{Z}\times\mathbb{R}) \cup (\mathbb{Z}\times\mathbb{R}\times\mathbb{R})$ is a subset of $\mathbb{R}^3$ on which $\pi_1\simeq \mathbb{Z}^3$ acts freely by  translations. The base point $*\in T_2^3$  is chosen at the image of $(0,0,0)$.
 Actually, we make a homotopy equivalence $M(\mathcal{A}_3)\sim T_2^3$ in such a manner that the three loops $\gamma_1,\gamma_2$ and $\gamma_3$ in $M(\mathcal{A}_3)$ (recall \S2.2 Example 6) are homotopic to the three loops in $T_2^3$, which are the images of three parallel unit movements from $(p,q,r)\in\mathbb{Z}^3$ to $(p\!+\!1,q,r), (p,q\!+\!1,r)$ and $(p,q,r\!+\!1)$ for $(p,q,r)\in \mathbb{Z}^3$  in $\tilde{T}_2^3$, respectively.  

 
  \vspace{0.3cm}
\hspace{2cm} 
\includegraphics[width=0.42 \textwidth]{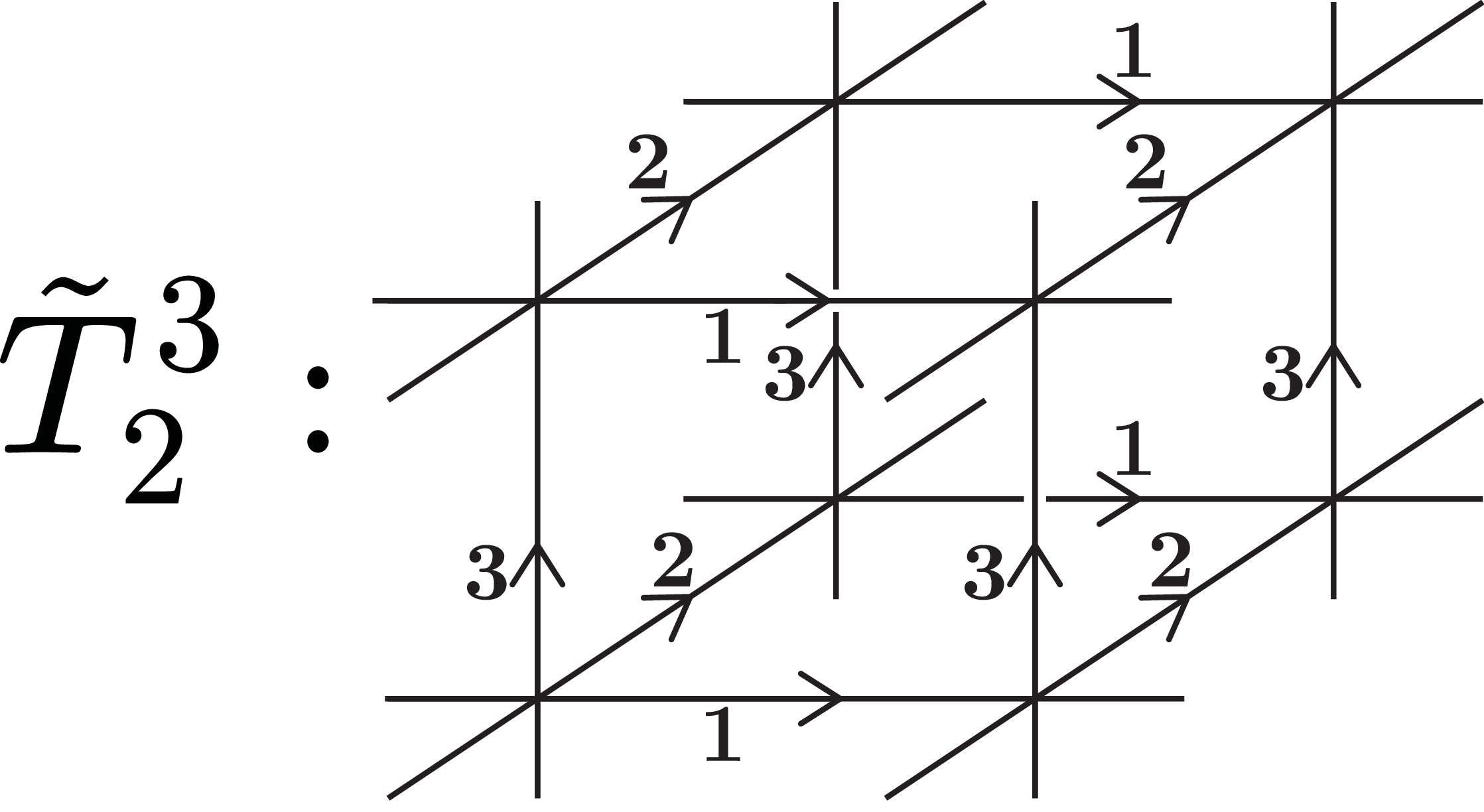}

\vspace{0.2cm}
\centerline{\quad Figure 6.1:  \ $\tilde{T}_2^3$ and the paths 1, 2 and 3}

\medskip

 Consequently, $\tilde{T}_2^3$ is the universal covering of $T_2^3$, and  $\pi_1(M(\mathcal{A}_3),*)$ is a rank-3 free abelian group $\mathbb{Z}^3$ generated by $\gamma_1,\gamma_2$ and $\gamma_3$, and $\pi_2(M(\mathcal{A}_3),*)$ is a rank-1 free module over $\mathbb{Z}\pi_1(M(\mathcal{A}_3),*)$ generated by the image in $T_2^3$ of the boundary  $\partial([0,1]^3) \! \subset\! \tilde{T}_2^3$ of the cube, i.e.\  the {\footnotesize image of the homotopy class of $(S^2,*) \tilde{\rightarrow} (\partial([0,1]^3),(0,0,0))$} in $\tilde{T}_2^3$ to $T_2^3$. 
 Let us call this $\pi_2$-class, {\footnotesize up to the action of $\pi_1(M(\mathcal{A}_3),*)$},  the {\it Hattori  second homotopy class} of $M(\mathcal{A}_3)$.

 We come  to the description of the liftings of loops $L$ and relations $\mathcal{R}$.
 
 \noindent
 1.   The lifting of $L$ is given by the pair $(W,\widetilde{L})$  of $W:=$ the 1-skeleton $T_1^3$ of the CW-complex $T_2^3$ and  $\widetilde{L}:=$ the set of  three homotopy classes of the loops  which are the images of the three parallel unit movements in $\tilde{T}_2^3$.  
 
 \noindent
 2. The lifting $\mathcal{R}_g$ consists of three maps $ A_g, B_g, C_g:  [0,1]^2 \to T_2^3$, defined as follows. Using 1., we  lift in $\tilde{T}_2^3$ the boundary path maps $123(132)^{-1}, 123(312)^{-1}$ and $123(231)^{-1} \ :\ \partial[0,1]^2 \to T_2^3$ of the algebraic relations \eqref{eq:3linesMonoid}. 
 
 \vspace{-0.05cm}
\hspace{-0.2cm} 
\includegraphics[width=0.79 \textwidth]{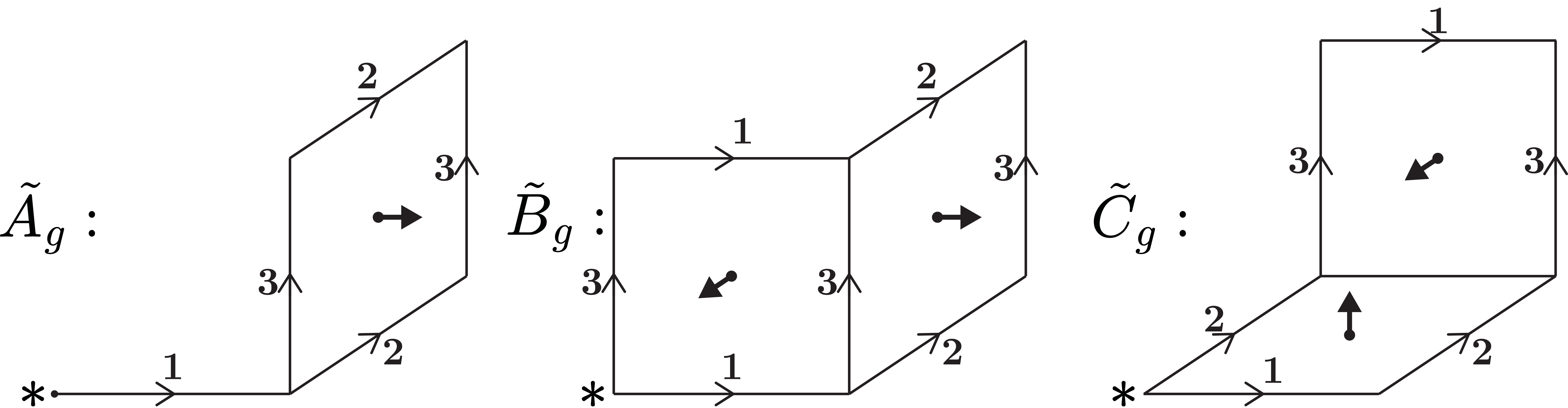}

\vspace{0.12cm}
    \centerline{\quad Figure 6.2:  \ Geometric fundamental relations $\tilde{A}_g$, $\tilde{B}_g$ and $\tilde{C}_g$,}
     \vspace{-0.1cm}
    \centerline{\small where the orientations of the faces are indicated by gothic arrows}

\medskip
 We first define the maps  $\tilde{A}_g,\tilde{B}_g, \tilde{C}_g :[0,1]^2\!\to \! \tilde{T}_2^3$ by extending the maps from  $\partial [0,1]^2$ to  $[0,1]^2$. Extensions exist and  are ambiguous  up to $\pi_2(\tilde{T}_2^3,*)$. 
 }
   One choice (in a minimal sense) is given in Figure 6.2. Composing these extended maps with the  projection $\tilde{T}_2^3\to T_2^3$,  we get the maps ${A_g}, {B_g}, {C_g}$  $:[0,1]^2\to T_2^3$,  
 and their classes  $[A_g]\in [123, 132], [B_g]\in [123,312]$ and $[C_g]\in [123,231]$. This defines \eqref{eq:geoABC}. 
 
 
  Suppose we replace  $A_g,B_g$, $C_g$  by other geometric fundamental relations $A_g',B_g'$ and $C_g'$.  Then the differences $A_g'\!-\!A_g,B_g'\!-\!B_g$,$C_g'\!-\!C_g$ determine elements in $\pi_2(M(\mathcal{A}_3),*)$. For a balanced twin, however,  the resulting changes cancel each other, so that the $\Pi$-class remains unchanged.

 
 \bigskip
\noindent
{\bf 2. Inertia groups}
 
 \noindent 

Using the geometric fundamental relations \eqref{eq:geoABC}, the seven NC-tuples given in Proposition \ref{NCList} give rise to the seven  $\mathcal{R}_g$-sets.
\begin{equation}
\label{eq:A3Rgsets}
\begin{array}{clllccccccccc}
 &[1(23),1(32)]_{\mathcal{R}_g}, \!&\!\!\!\!\!\!\! [(23)123,(32)123]_{\mathcal{R}_g},  &\!\!\! [123(31), 123 (13)]_{\mathcal{R}_g},\!\!\!\!\!\!\\
 & [(31)2, (13)2]_{\mathcal{R}_g}, \!&\!\!\!\!\!\!\! [13(12), 13(21)]_{\mathcal{R}_g}, &\!\!\! [(12)32, (21) 32]_{\mathcal{R}_g},\!\\
&  [3(12)23,3(21)23]_{\mathcal{R}_g}\!
  \end{array}
  \end{equation}
  
Before going further, we show a basic fact  that these sets  are singletons. This is due to the following facts: \!\!\! 
 
(i) On each $\mathcal{R}_g$-set, say $[x,y]_{\mathcal{R}_g}$, the local inertia group $[x,x]_{\mathcal{R}_g} (\simeq [y,y]_{\mathcal{R}_g})$ acts simply and transitively (\S3.2.2  Proposition \ref{InertiaAction}). 

(ii) If the geometric lifting $\mathcal{R}_g$ is singleton (Definition \ref{GeoFundRel}), then the fundamental group of the graph  $G(L^*,\mathcal{R}_m(\mathcal{A}_3))$ \eqref{eq:MonoidGraph3} based at $x\!\in \! L^* \!\! \simeq \! \mathcal{F}^+$\! surjects onto the inertia group $[x,x]_{\mathcal{R}_g}$   \eqref{eq:SingletonLocalInertia}.

(iii) The connected components of  $G(L^*,\mathcal{R}_m(\mathcal{A}_3))$
containing the seven  NC-pairs (123,132), (23123, 32123), (312, 132), (12331,12313), (1232, 2132), (12132,13212)  and (1312,1321) are trees. 

\smallskip
\noindent
{\it Proof.} This is verified by the explicit descriptions of the components of the graph given in the following Figure 6.3, where we call a connected component of  $G(L^*,\mathcal{R}_m(\mathcal{A}_3))$ containing an element $x$ by $G_x$. 
 The base points of the fundamental groups are surrounded by rectangles.  \hfill{$\Box$}
 
  We remark that the sequences of e-transformations given  in
\eqref{eq:l=3}, \eqref{eq:l=4} and \eqref{eq:l=5} are exactly the short routes connecting the two vertices surrounded by rectangles in the graphs.\!\! 
   
 \vspace{0.6cm}
\hspace{0.3cm} 
\includegraphics[width=0.25 \textwidth]{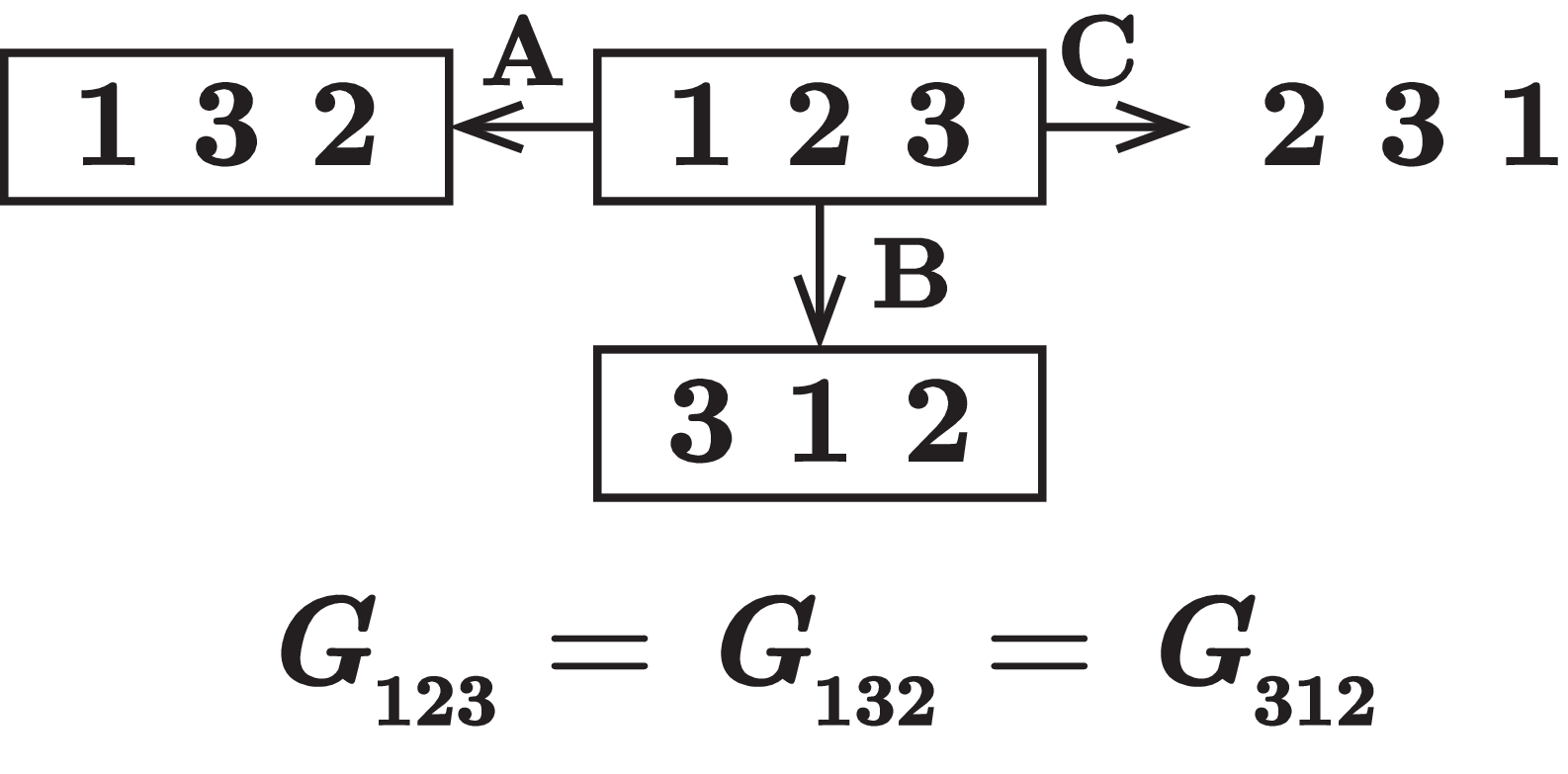}

\vspace{-1.6cm} 
\hspace{4.5cm}
\includegraphics[width=0.4 \textwidth]{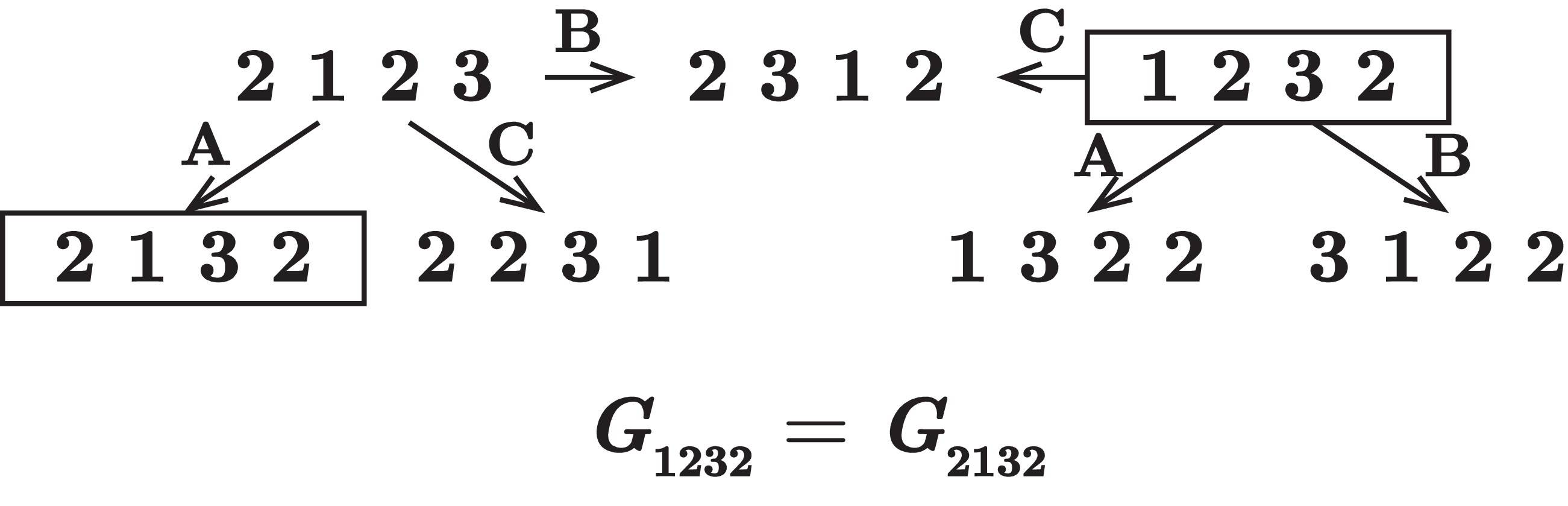}

\vspace{0.7cm}
\hspace{-0.2cm}
\includegraphics[width=0.33 \textwidth]{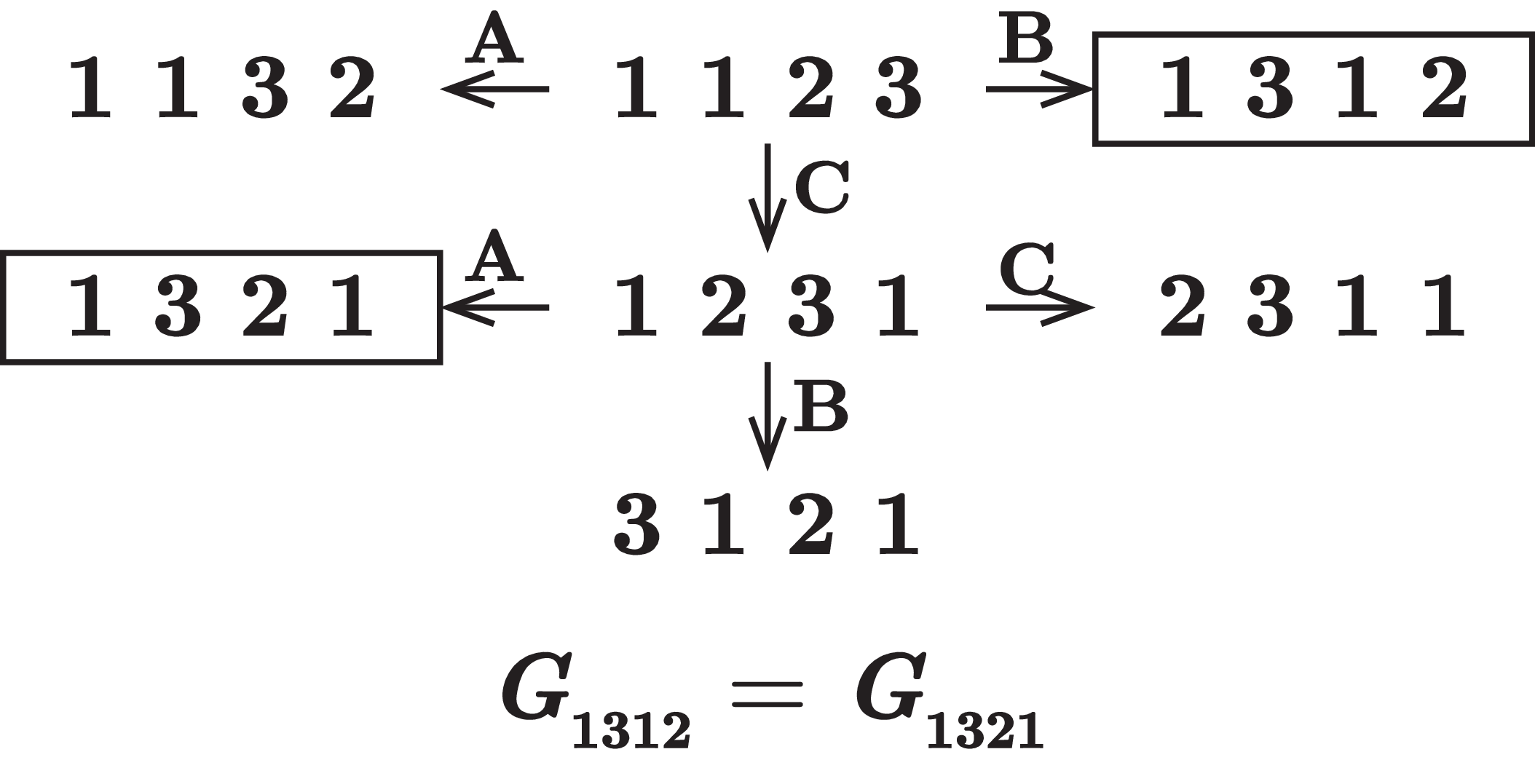}

 \vspace{-2.13cm}  \hspace{5.2cm} 
 \includegraphics[width=0.29 \textwidth]{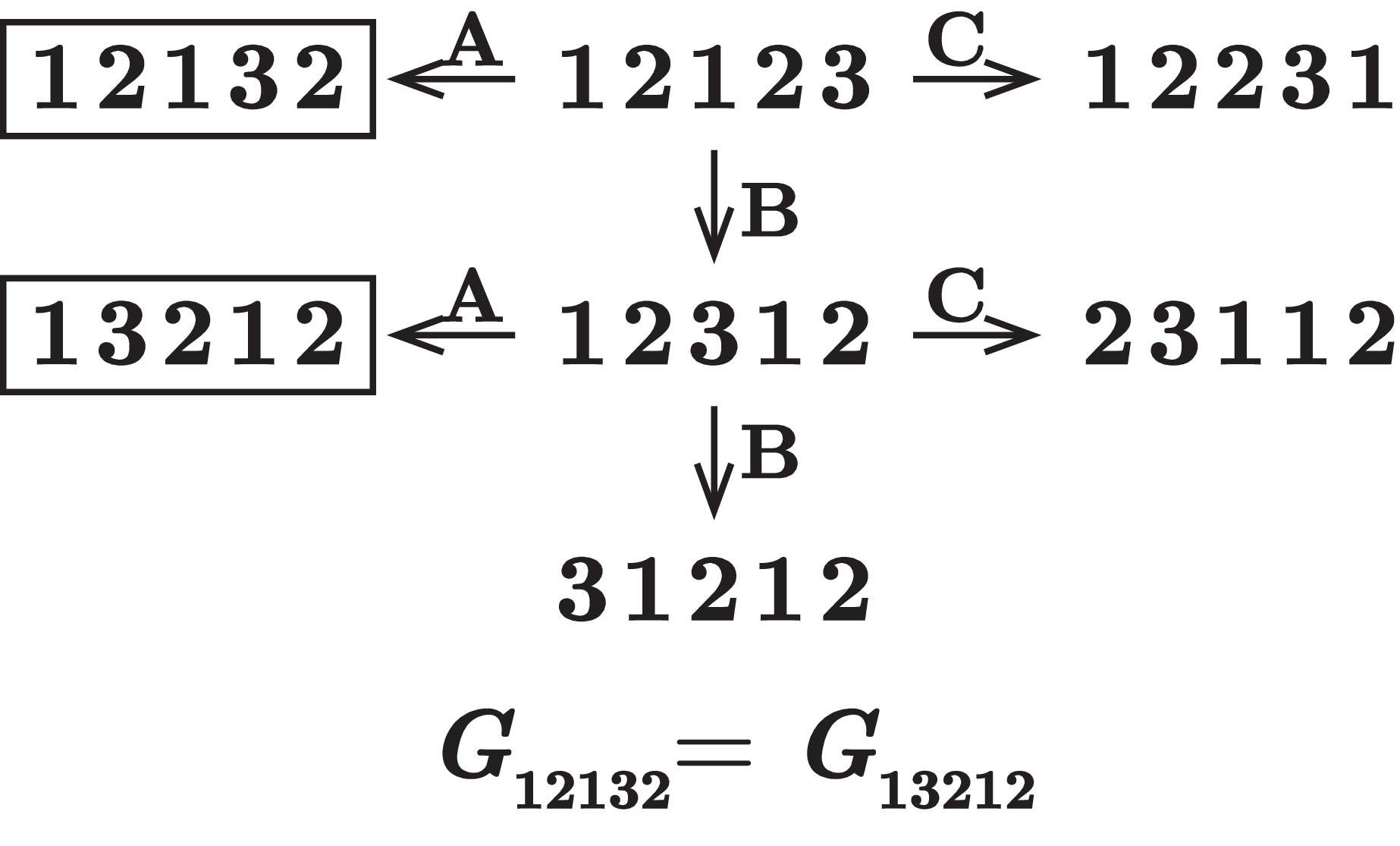}

\vspace{0.4cm}
\hspace{0.6cm}
\includegraphics[width=0.62 \textwidth]{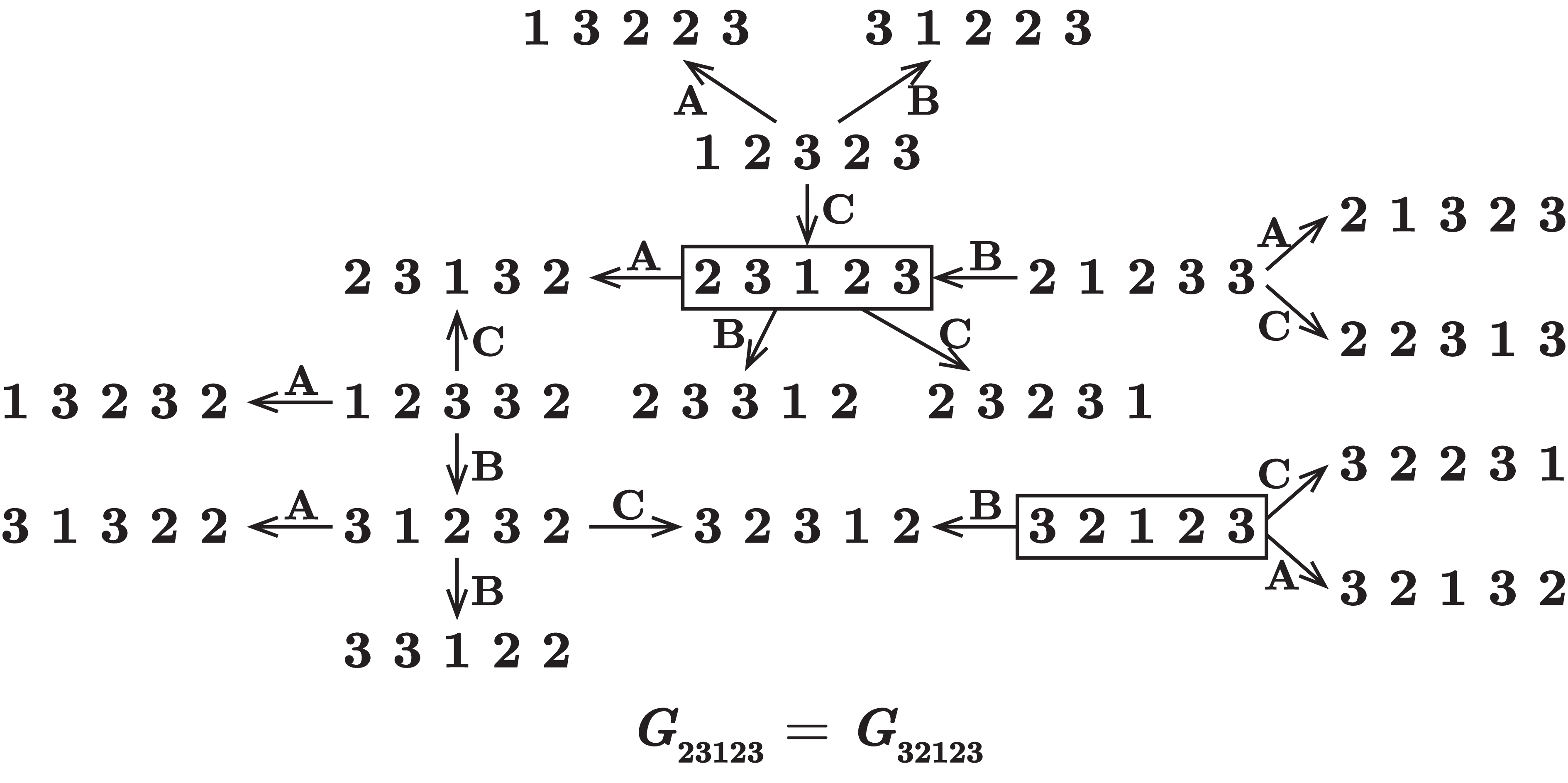}

\vspace{0.4cm}
\hspace{0.3cm}
\includegraphics[width=0.67 \textwidth]{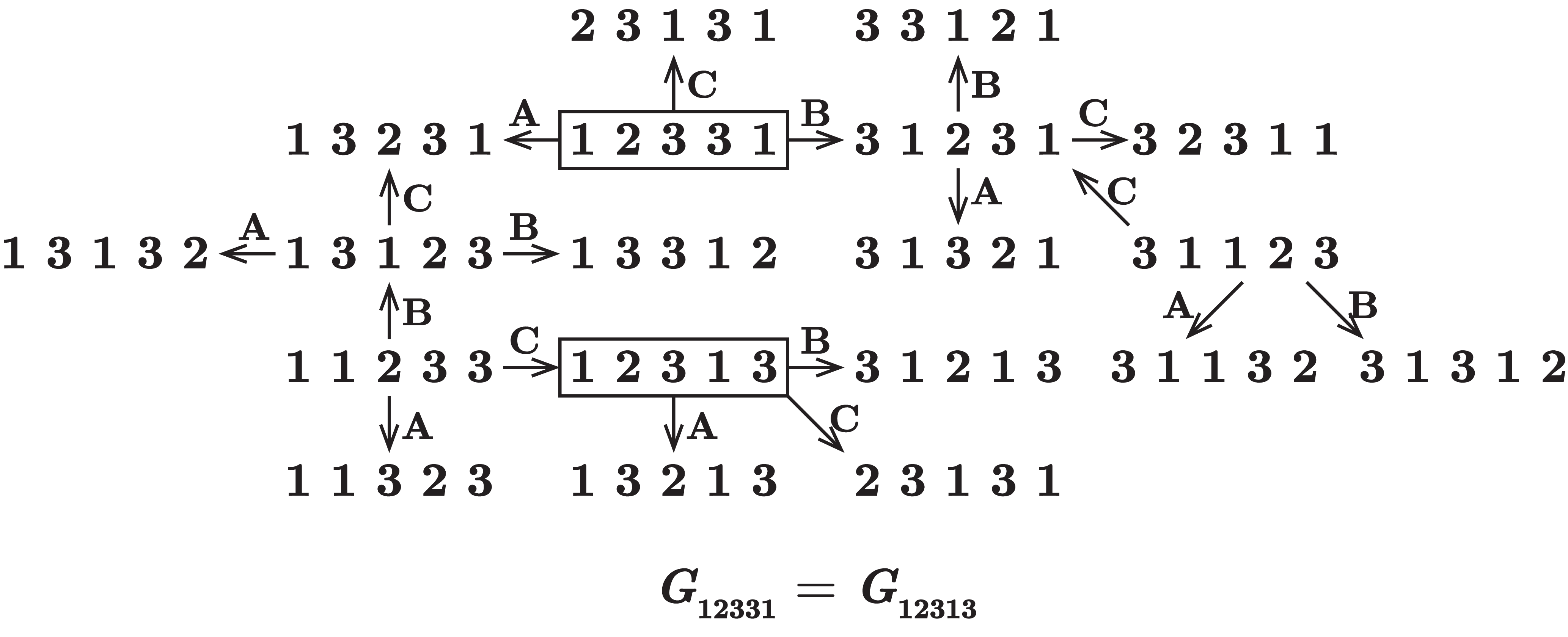}

\vspace{0.05cm}
    \centerline{\quad Figure 6.3:  \ Components of graph $G(L^*,\mathcal{R}_m(\mathcal{A}_3))$,}

 \bigskip
 
 \noindent
 {\bf Notation.} Let us give names to the single elements in the 7 $\mathcal{R}_g$-sets.
 \begin{equation}
 \label{eq:Tij}
\begin{array}{clcl}
&[T_{23}]\in [1(23),1(32)]_{\mathcal{R}_g}
  \!\! \!\!& \!\! \!\! \text{and}  \!\!& \!\!
[T_{23}']\in [(23)123,
(32)123]_{\mathcal{R}_g},  \!\! \!\!\\
&[T_{31}]\in [123(31), 123 (13)]_{\mathcal{R}_g}
 \!\! \!\!& \!\! \!\! \text{ and}  \!\!& \!\! 
[T_{31}']\in  [(31)2, (13)2]_{\mathcal{R}_g},  \!\! \!\!\\
&[T_{12}]\in  [13(12), 13(21)]_{\mathcal{R}_g}
 \!\! \!\!& \!\! \!\! \text{ and}  \!\!& \!\! 
[T_{12}']\in  [(12)32, (21) 32]_{\mathcal{R}_g},  \!\! \!\!\\
&[T_{12}^*]\in [3(12)23,3(21)23]_{\mathcal{R}_g}.&&
\end{array}
\vspace{0.1cm}
\end{equation}
 whose representing maps $T_{ij}, T_{ij}'$ and $T_{12}^*$ are obtained by replacing the algebraic e-transformations $\tilde{\gamma}_1\tilde{\gamma}_2\tilde{\gamma}_3 \!\overset{A}{\sim} \!\tilde{\gamma}_1\tilde{\gamma}_3\tilde{\gamma}_2$, 
$\tilde{\gamma}_1\tilde{\gamma}_2\tilde{\gamma}_3 \!\! \overset{B}{\sim} \!\!
\tilde{\gamma}_3\tilde{\gamma}_1\tilde{\gamma}_2$,
$\tilde{\gamma}_1\tilde{\gamma}_2\tilde{\gamma}_3 \!\! \overset{C}{\sim} \!\!
\tilde{\gamma}_2\tilde{\gamma}_3\tilde{\gamma}_1$ in the sequences 
\eqref{eq:l=3}, \eqref{eq:l=4} and \eqref{eq:l=5}
 by the geometric maps $A_g, B_g,C_g$ \eqref{eq:geoABC} given in  $T_2^3$.

 \smallskip
 \noindent
 {\bf Question.} Are all $\mathcal{R}_g$-sets of the monoid $A^+(L,\mathcal{R}_m(\mathcal{A}_3))$ singletons?

\bigskip
\noindent
{\bf 3.  Second homotopy classes associated with five  NC-twins}

In the previous {\bf 2.}, we have seen that the defining domains for the $\Pi_2$-morphism  \eqref{eq:Pi2Morphism} for the 5 NC-twins  \eqref{eq:Twin23}-\eqref{eq:Twin*}  
$$
\begin{array}{rrlllllll}
& [1(23),1(32)]_{\mathcal{R}_g}  \!\!&\times&\!\! [(23)123, (32)123]_{\mathcal{R}_g} &=& \{ ([T_{23}], [T'_{23}])\} , \\
&[123(31), 123 (13)]_{\mathcal{R}_g} \!\!& \times &\!\! [(31)2, (13)2]_{\mathcal{R}_g} &=& \{ ([T_{31}],[T'_{31}])\}, \\
&[13(12), 13(21)]_{\mathcal{R}_g} \!\!& \times &\!\! [(12)32, (21) 32]_{\mathcal{R}_g} &=& \{([T_{12}],[T'_{12}])\},\\
&  [13(12), 13(21)]_{\mathcal{R}_g} \!\!& \times &\!\! [3(12)23,3(21)23]_{\mathcal{R}_g} &=& \{([T_{12}],[T_{12}^*])\},\\
 &  [(12)32, (21) 32]_{\mathcal{R}_g} \!\!& \times &\!\! [3(12)23,3(21)23]_{\mathcal{R}_g} &=& \{([T_{12}'],[T_{12}^*])\}.\\
 \end{array}
 $$
 are singletons.
We now apply the $\Pi_2$-morphism on them and determine explicitly  in three steps 
 their image homotopy classes in $\pi_2(M(\mathcal{A}_3),*)$. 

\medskip
\noindent
{\bf Step 1.} 
We  apply the $\cdot$-division morphism \eqref{eq:Relative2class} to the 7
$\mathcal{R}_g$-elements \eqref{eq:Tij} and obtain 7 division $\mathcal{R}_g$-elements in the division $\mathcal{R}_g$-sets \eqref{eq:TwinDivisionRgset};
 \begin{equation}
 \label{eq:Hij}
\begin{array}{clcl}
&[H_{23}]:=1^{-1}[T_{23}]\in [23),32)]
 \!\!\!\!\!\!\!\!&\!\!\!\!\!\!\! \text{and} &
[H_{23}']:=[T_{23}'](123)^{-1}\in [23,
32], \\
&[H_{31}]:=(123)^{-1}[T_{31}]\in [31, 13]
 \!\!\!\!\!\!\!\!&\!\!\!\!\!\!\!  \text{ and} & 
[H_{31}']:=[T_{31}']2^{-1}\in  [31, 13], \\
&[H_{12}]:=(13)^{-1}[T_{12}]\in  [12, 21]
 \!\!\!\!\!\!\!\!&\!\!\!\!\!\!\!  \text{ and} & 
[H_{12}']:=[T_{12}'](32)^{-1}\in  [12, 21], \\
&[H_{12}^*]:=3^{-1}[T_{12}^*](23)^{-1}\in [12,21].\!\!\!\!&&
\end{array}
\vspace{0.1cm}
\end{equation}

The defining domains of  the maps $H_{ij}$,  $H_{ij}'$ (for $ij=23, 31$ or $12$) and $H_{12}^*$ are drawn in the following figures, where  domains of the maps $T_{ij}$, $T_{ij}'$ and $T_{ij}^*$ are the small rectangles inside the domains of $H_{ij}$\! (Figure 5.1).\!\!\!

 \vspace{0.5cm}
\hspace{-0.1cm} 
\includegraphics[width=0.695\textwidth]{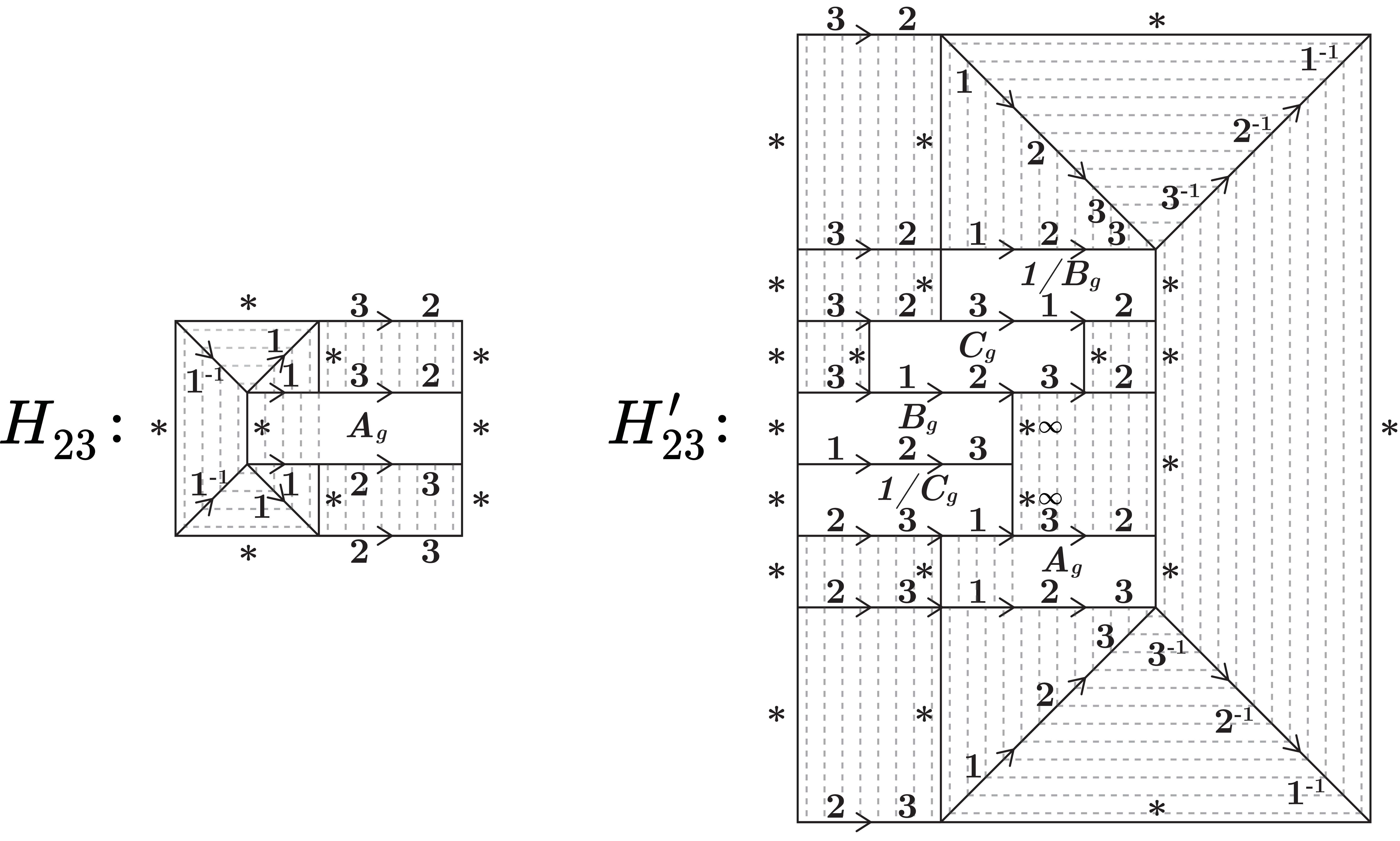}

\vspace{-0.25cm}

\hspace{-0.1cm} 
\includegraphics[width=0.695\textwidth]{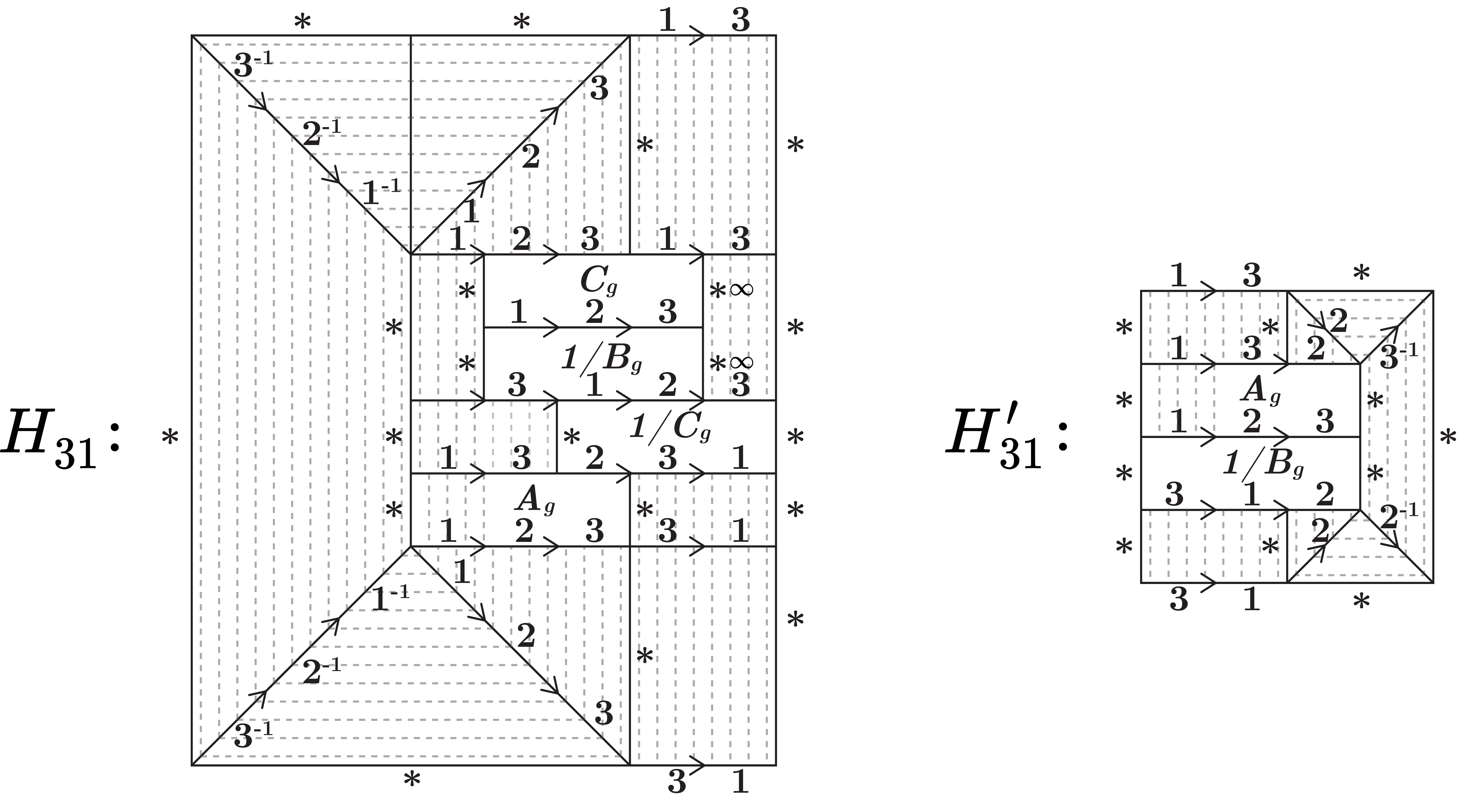}

\vspace{-0.15cm}

\hspace{-0.1cm} 
\includegraphics[width=0.695\textwidth]{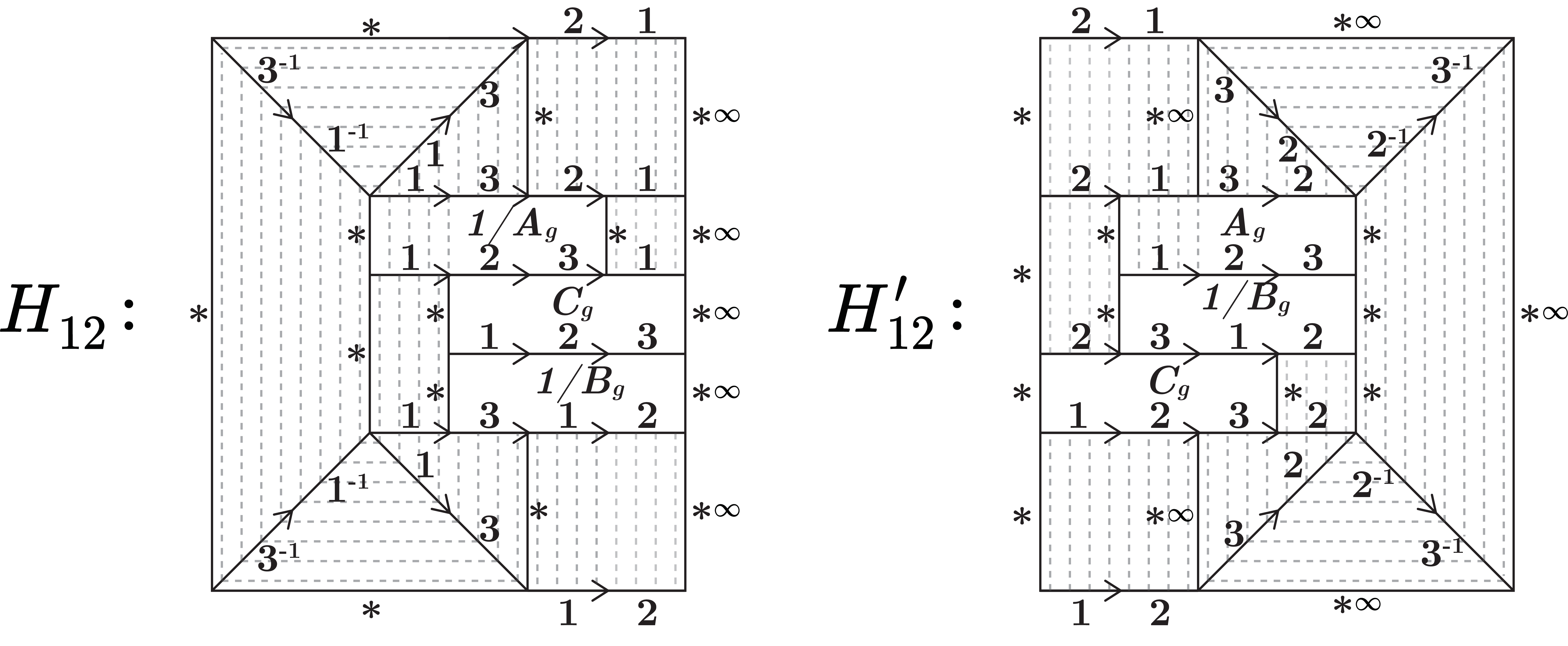}

\vspace{-0.05cm}
\hspace{1.7cm}
\includegraphics[width=0.40\textwidth]{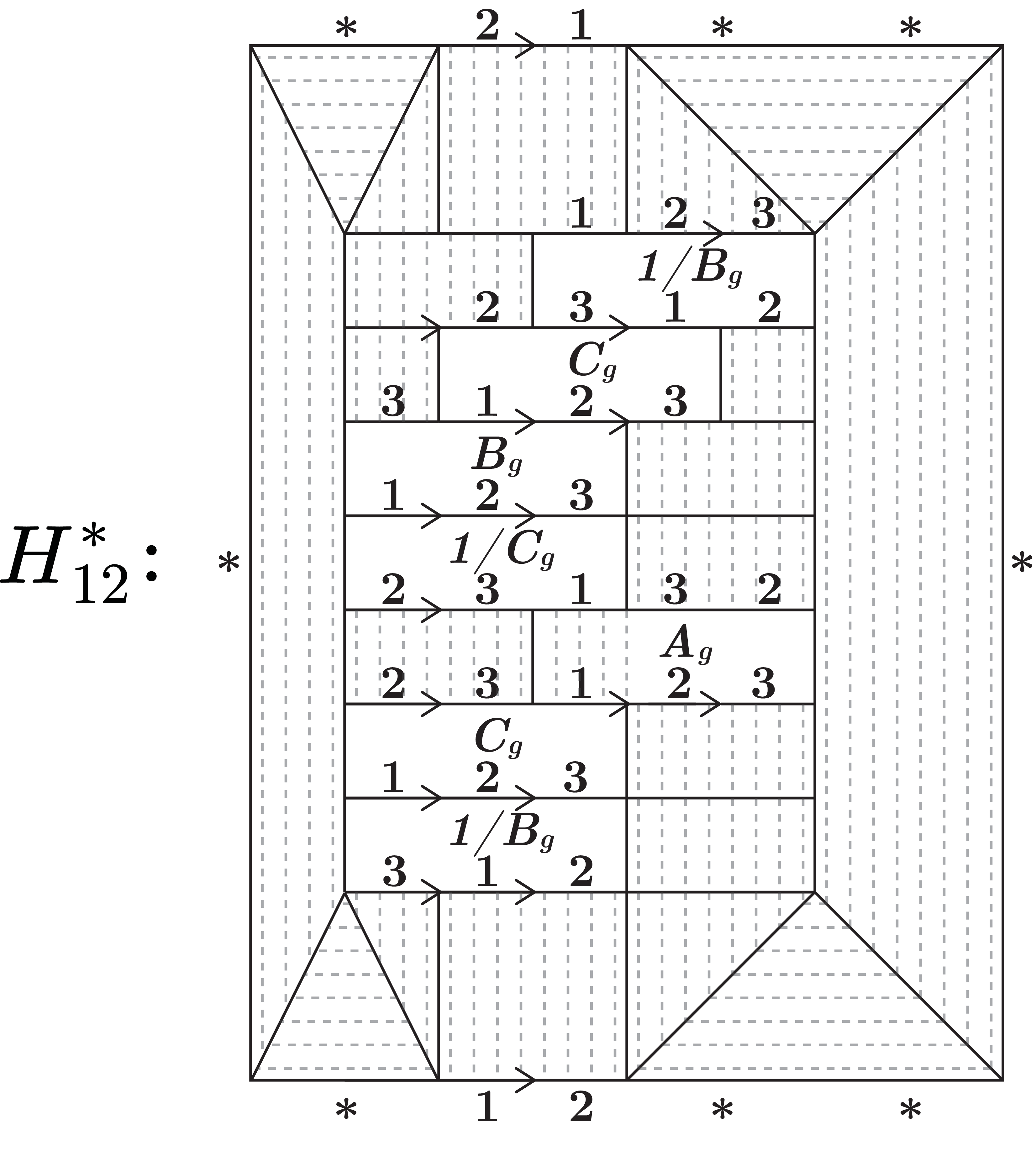}

\vspace{-0.1cm}
\centerline{\ Figure 6.4:  Homotopy equivalences   $H_{ij}, H_{ij}' , H_{ij}^* \in [ij, ji]$}
\noindent {\small The meaning of labels ``$\infty$" on some edges is explained in the next Step 2.\!\!}\!\!


\bigskip
\noindent
{\bf Step 2.}
Consider the \(\circ\)-differences \eqref{eq:TildePi2} of the corresponding pairs of division \(R_g\)-elements  $[H_{ij}],[H_{ij}'], [H_{ij}^*] \in [ij,ji]$  given in Step 1.
Then, we apply to them the relative realization morphism  $\chi$ \eqref{eq:immersion}, and  obtain elements \eqref{eq:FinalPi2} in $\pi_2(X,*)$.  
\begin{equation}
\label{eq:3pi2}
\begin{array}{cccl}
&\chi([H_{23}]\circ[1/H_{23}'])&\in& \Pi(\tau_{23}) \\
&\chi([H_{31}]\circ [1/H_{31}'])
& \in& \Pi(\tau_{31})   \\
&\chi([H_{12}] \circ [1/H_{12}'])
&\in & \Pi(\tau_{12}) \\
&\chi([H_{12}] \circ [1/H_{12}^*])
&\in & \Pi(\tau_{12}^*)\\
&\chi([H_{12}'] \circ [1/H_{12}^*])
&\in & \Pi(\tau_{12}'^*)
\end{array}
\end{equation}
They are the single $\Pi$-classes, and 
 the global inertia groups $G(\tau_{ij}), G(\tau_{ij}^*)$ and $G(\tau_{ij}'^*)$ are also trivial  since the local inertia groups are trivial.

\smallskip
Our next goal is to show that these classes coincide with the Hattori class of the second homotopy group $\pi_2(M(\mathcal{A}_3),*)$, except for the last case: 
\begin{equation}
\label{eq:Pi'*}
\chi([H_{12}'] \circ [1/H_{12}^*])=0. 
\end{equation}
and therefore  
\begin{equation}
\label{eq:PiDifference}
\chi([H_{12}] \circ [1/H_{12}']) - \chi([H_{12}] \circ [1/H_{12}^*])=-\chi([H_{12}'] \circ [1/H_{12}^*]) =0
\end{equation}
(the proof for the last case is separated in {\bf II. Case for $\mathcal{D}_{12}^*$.} at the end). 

\smallskip
We start to describe the homotopy maps for the first 4 classes  \eqref{eq:3pi2} explicitly. We first  redefine the domains of the  maps 
as follows, and call them  $\mathcal{D}_{23}, \mathcal{D}_{31}$, $\mathcal{D}_{12}$ and $\mathcal{D}_{12}^*$, respectively.  

\vspace{-0.3cm}
$$\begin{array}{l}
\mathcal{D}_{ij}:= ((\text{Domain of }H_{ij})\! \cup\!  (\text{Domain of }H'_{ij}))/\text{retraction   for  }
{\tiny ij \!\in \!\{23,31,12\}}\\
\mathcal{D}_{12}^*:= ((\text{Domain of }H_{12})\! \cup\!  (\text{Domain of }H^*_{12}))/\text{retraction}
\end{array}
$$
For  $\mathcal{D}_{ij}$  ($ij\in\{23,31,12\}$), we proceed as follows.

1) For each $ij\in\{23,31,12\}$, take the disjoint union of the pair of the domains of definition for $H_{ij}$ and $H_{ij}'$ in Figure 6.4. 

2) Patch the lower boundary edges $ij$ of the pair of domains, and also patch the upper boundary edges $ji$ of the pair of the domains, where the orientation of the domain for $H_{ij}'$ is reversed. 

3) Retract the patched union along the  dotted lines in Figure 6.4, i.e. retract the areas which are not labeled by  $A,1/A,B,1/B$  or $C,1/C$ (here, we recall that $H_{ij}$ and $H_{ij}'$ take constant values along the dotted lines). 

4) The edges labeled by ``$*$" on the left sides of the domains of $H_{ij}\circ (1/H_{ij}')$ shrink to a point, which we call $*_{ij}$ (Figure 6.5). They are the base points of $\mathcal{D}_{ij}$ (recall (2) of the definition of $\chi$ \eqref{eq:immersion}).

5) Two edges $ij$ and $ji$ starting at $*_{ij}$ meet at the antipodal point of $*_{ij}$ so that their union separates $\mathcal{D}_{ij}$ into two  components $H_{ij}$ and $1/H'_{ij}$.

We omit details of a parallel description  for $\mathcal{D}_{12}^*$.

\medskip
\noindent
{\bf I.  Case $\mathcal{D}_{ij}$ for  $ij \! \in \! \{23,31,12\}$.}   

We observe the following  marvelous fact, for which we do not know  a conceptual explanation. 

\medskip
\noindent
{\bf Fact.}\
{\it  All three domains  $\mathcal{D}_{ij}$ ($ij \! \in \! \{23, 31, 12\}$) give the same cell decomposition of a 2-sphere $S^2$ by the 6 cells $A_g, B_g, C_g$ and $1/A_g,1/B_g,1/C_g$, where this sphere $S^2$ is the defining domain of the $\Pi$-classes.}  

\smallskip
However, be cautious that the base points $*_{ij}$ are located differently.  

\smallskip
Choosing and fixing a point of $S^2$ as the point $\infty$ at infinity, we develop the cell-decomposition of the sphere $S^2 \simeq\! \mathbb{R}^2 \! \cup \! \{\infty\}$ given in the {\bf Fact.} to that of a plane $\mathbb{R}^2$ in Figure 6.5. The $*$-edges in Figure 6.4  which give rise to the point $\infty$ at infinity are labeled by $\infty$  (details are left to the reader).

 \vspace{0.25cm}
\hspace{0.7cm} 
\includegraphics[width=0.61 \textwidth]{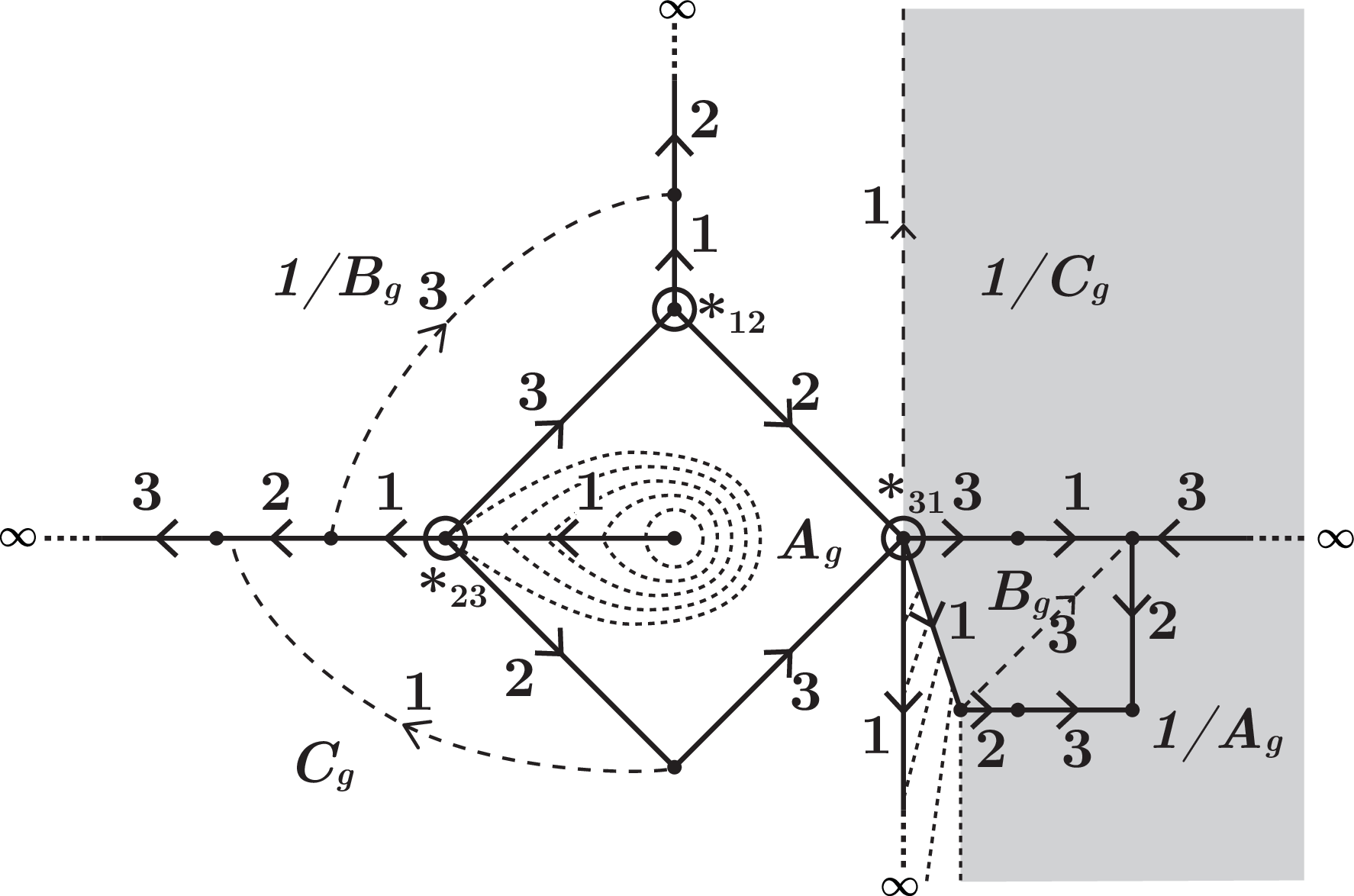}

\vspace{0.2cm}
\centerline{\quad Figure 6.5:  Developed figure of the cell decomposition of $\mathcal{D}_{ij}\simeq S^2$.}

\smallskip
\noindent
The three base points for the three cases $\mathcal{D}_{ij}$ are denoted by $*_{ij}$ for $ij\in \{23,31,12\}$, from where we can easily recover the decomposition of the sphere $S^2$ into the union of the domains of $H_{ij}$ and $1/H_{ij}'$ by cutting along the loop $iji^{-1}j^{-1}$ starting at $*_{ij}$.  Explanations on the grey zone and dotted lines in the figure shall be given in the next Step 3.

\medskip
\noindent
{\bf Step 3.}\  We draw in Figure 6.6  the image $\tilde{I}$ of the maps $\mathcal{D}_{ij} \! \simeq\! S^2 \!\to\! \tilde{T}_2^3$ which is a lifting of the $\pi_2$-map $H_{ij}\circ(1/H_{ij}'): \mathcal{D}_{ij} \to  T_2^3$ obtained by the use of lifted maps $\tilde{A}_g,\tilde{B}_g$ and $\tilde{C}_g$.  Liftings are ambiguous  up to the action of $\mathbb{Z}^3\simeq \pi_1(M(\mathcal{A}_3),*)$. We choose one lifting by fixing the image of the central point of area $A_g$ in Figure 6.5   to  $(-1,0,0)\in \tilde{T}_2^3$.\footnote
{Natural liftings to put the image of the base points $*_{ij}$ at  $(0,0,0)\in \tilde{T}_2^3$ are obtained by translating this lifting. The translation data give the data of $\pi_1(M(\mathcal{A}_3),*)$-action in \eqref{eq:Pi-clases}.
}

 \vspace{0.2cm}
\hspace{1.7cm} 
\includegraphics[width=0.45 \textwidth]{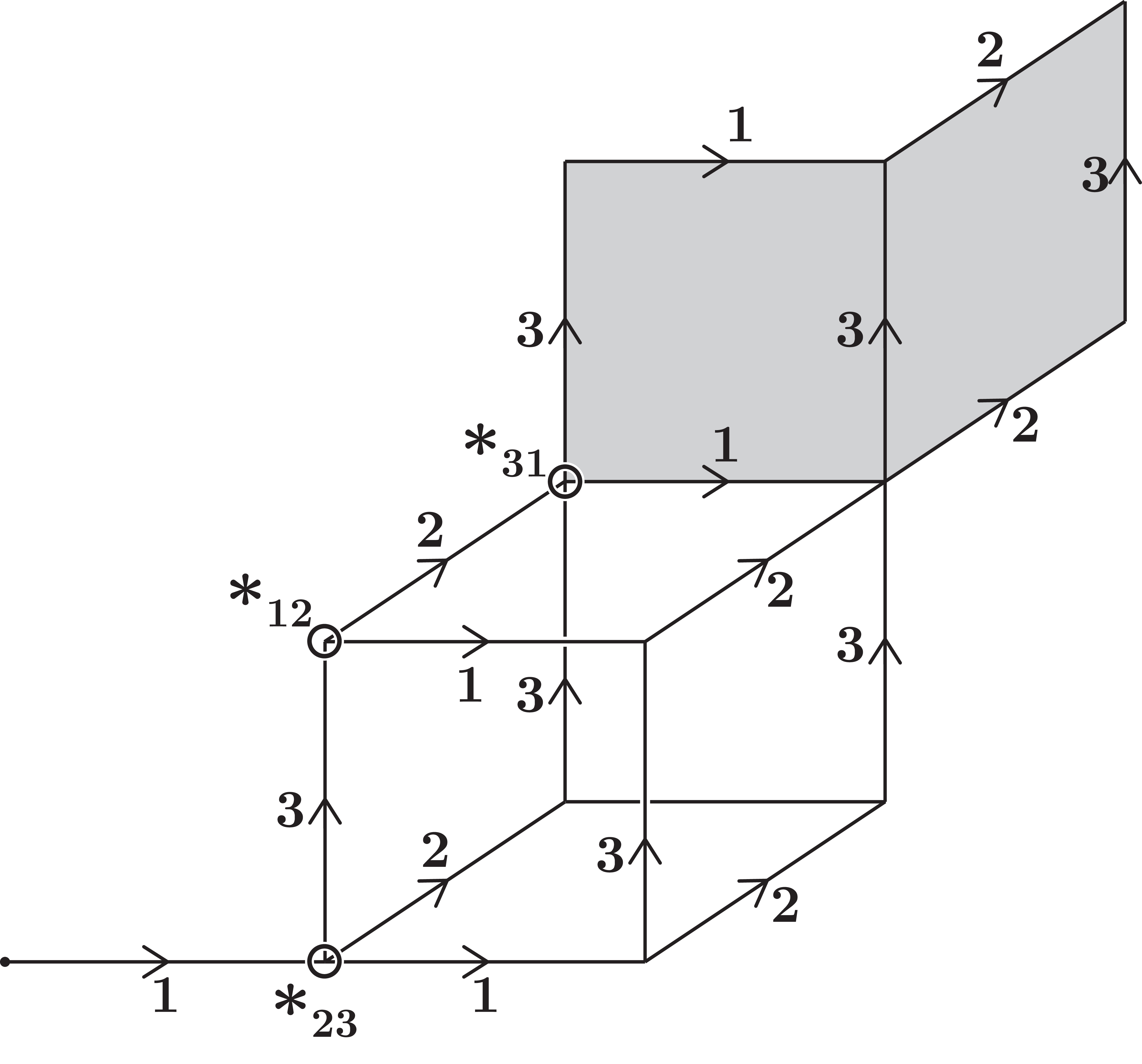}

\vspace{0.1cm}
\centerline{\quad Figure 6.6:  \ Image $\tilde I$ defining the lifted $\Pi$-classes $\Pi(\tau_{ij})$ in $\tilde{T}_2^3$}

\bigskip
\noindent
The image $\tilde I$ consists of the following three parts: 

\noindent
Part 1. The  single edge $\overset{1}{\cdot\!\!-\!\!\!\longrightarrow\!\!\!-\!\!\!-\!\!\!\cdot}$ in the left side of $\tilde I$ is the image of the dotted lines area inside the domain of $A$ in Figure 6.5, obtained  by retracting the area according to the dotted lines.
Another edge $\overset{1}{\cdot\!\!-\!\!\!\longrightarrow\!\!\!-\!\!\!-\!\!\!\cdot}$ starting from the point $*_{31}$ is the image of dotted lines area inside the domain of $A^{-1}$ in Figure 6.5, obtained  by retracting the area according to the dotted lines.

\smallskip
\noindent
Part 2. The cube  $\partial([0,1]^3)$ in the middle part of $\tilde I$ is the image of the white area in the  LHS of Figure 6.5, where the additional dotted arrows give the cell decomposition of the domain corresponding to the face decomposition of the image cube.
All faces of the cube are oriented towards the inside. 

\smallskip
\noindent
Part 3.  Two grey square faces on the  upper side of $\tilde I$ are the image of the grey area of Figure 6.5.  Each face is covered by two oppositely oriented faces in Figure 6.5 and they are retractable to Part 2.

\smallskip

These properties of the lifted maps: $\mathcal{D}_{ij}\to \tilde{T}_2^3$ are naturally inherited by the $\pi_2$-maps: $\mathcal{D}_{ij}\to T_2^3$. That is, the $\pi_2$-maps are homotopic (up to the choices of base points) to a $\pi_2$-map which covers once the projection image  in $T_2^3$ of the cube  $\partial([0,1]^3)$ and hence represents  Hattori's second homotopy class!

\medskip
\noindent
{\bf II. Case for $\mathcal{D}_{12}^*$.}

We give a proof of \eqref{eq:Pi'*} and \eqref{eq:PiDifference}.

Recall the explanation 5) of $\mathcal{D}_{ij}$ (at the beginning of {\bf Step 2.}) so that the path $12(21)^{-1}$ starting at the base point $*_{12}$ in $\tilde{T}_2^3$ (see Figure 6.6) cuts the class $\Pi(\tau_{12})$ into two pieces of relative homotopy classes $\chi(H_{12})$ and $\chi(H_{12}')$ with the boundary  $12(21)^{-1}$. Actually, 
$\chi(H_{12}')$ is the ``top cover part'' of the class. That is, the map $H_{12}'$ defines a relative homotopy class in $\pi_2( \tilde{T}_2^3,  \tilde{T}_1^3)$ homotopic to the map: \!\!\!\! 
$$ 
(*)   \qquad  \hfill  \qquad     [0,1]^2 \to \tilde{T}_2^3, (a,b)\mapsto (a,b,0).\qquad\qquad \qquad
$$

Let us show 

\medskip
\noindent
{\bf Assertion.} {\it The map $H_{12}^*$ defines a relative homotopy class in $\pi_2( \tilde{T}_2^3,  \tilde{T}_1^3)$ homotopic to the map $(*)$.}
\begin{proof}
This is shown by retracting the images of $H_{12}^*$ to that of $(*)$ in $\tilde T_2^3$ as follows. 
Recall the defining domain of the map $H_{12}^*$ (Figure 6.4). We decompose the domain into 13 regions, each of which is homeomorphically mapped to a 2-cell of $\widetilde{T}^3_2$. So we label each region according to the type ($\#$) of the image 2-cell and its orientation as given in Figure 6.7 (more precisely, we may attach the data where the cell is located. However, for simplicity, we omit details). 
At each step, two regions mapped onto the same 2-cell with opposite orientations become adjacent along a common edge and can be cancelled by a homotopy through relative maps
%
(the flow of the retractions is indicated by pathlines in Figure 6.7)  according to the order of adjacency  given by the number $t=0,1,2,3$ on the edges. 
After all these cancellations, only the white region labeled \(12\) remains unretracted. Its restriction is precisely the map \((*)\).


Details of the verification are left to the reader.  

\smallskip
\noindent
$(\#)$ A face in $\tilde{T}^3_2$ surrounded by a path $iji^{-1}j^{-1}$ is called of type $ij$.

\bigskip

\vspace{-0.13cm}
\hspace{1.7cm}
\includegraphics[width=0.48\textwidth]{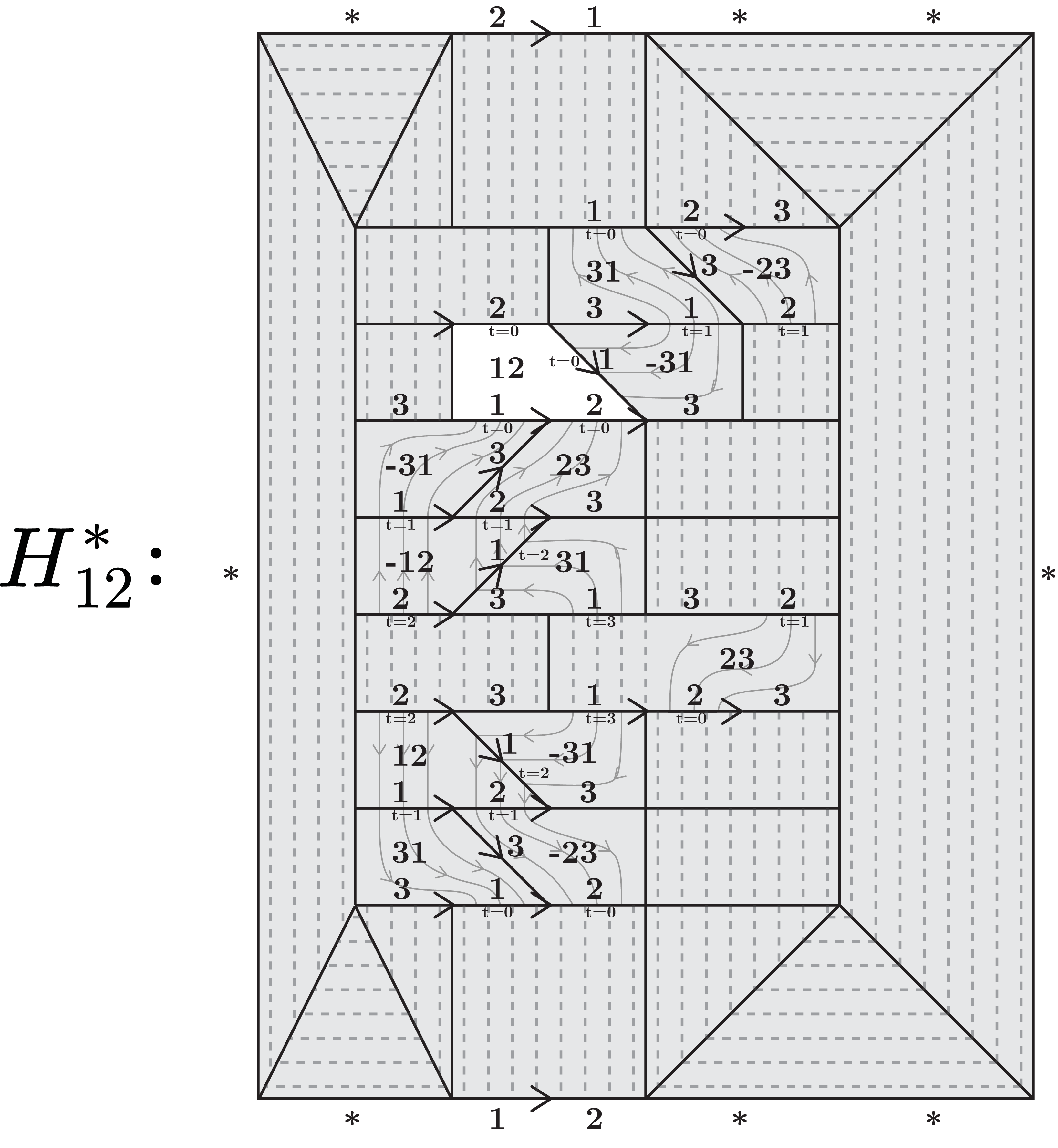}

\centerline{Figure 6.7: Successive retractions of the domain of \(H^*_{12}\).}

\medskip
Thus $H'_{12}$ and $H^*_{12}$ represent the same relative homotopy class. Therefore,  \eqref{eq:Pi'*} is proven and \eqref{eq:PiDifference} follows.
\vspace{-0.2cm}
\end{proof}

\smallskip
This completes the proof of Theorem \ref{ThmY-H} \hfill{$\Box$ }

\medskip
\noindent
{\bf Summary for (a).}\    Let $M(\mathcal{A}_3)$ be the complement of three generic lines in $\mathbb{C}^2$ defined  over $\mathbb{R}$. Consider Yoshinaga's presentation of its fundamental group  (\S2.2, 6(a)).  Then the associated monoid  $A^+(L,\mathcal{R}_m(\mathcal{A}_3))$ carries 3 distinguished reduced and minimal  NC-twins (\S4.2 (a)). We choose geometric fundamental relations $A$, $B$ and $C$ (Figure 6.2) so that the global inertia groups of the NC-twins are trivial (Figure 6.3).  We observe that all three $\Pi$-classes \eqref{eq:3pi2}  of the NC-twins are singletons consisting of translates of Hattori's generator of $\pi_2(M(\mathcal{A}_3),*)$ whose 
base points are given by $*_{23},*_{31}$ and $*_{12}$  (see Figure 6.6).\!\!\!

\medskip
\noindent
\begin{rem}
{\rm   
Why does presentation \eqref{eq:linesPositiveHomogeneous} take this more roundabout route, which seems to retain more information? Is this related to the minimal cell decomposition of \(M(\mathcal A)\)?
}
\end{rem}

\noindent
\subsection{\normalsize (b): the  line arrangement of type $\rm{ A}_3$} 

\smallskip
There  are  11 NC-tuples in degree 5,  6 NC-tuples in degree 7, and  16 NC-tuples in degree 9. There are no twins among them. 

For a singleton geometric lifting $\mathcal{R}_g$ of Yoshinaga's fundamental relations, let us show that the division inertia groups for them are trivial by showing the components of the graph  $G(L^*,\mathcal{R}_m)$ corresponding to them are trees.

\smallskip
\noindent{deg=5 }: 

\hspace{2.6cm}
\includegraphics[width=0.35 \textwidth]{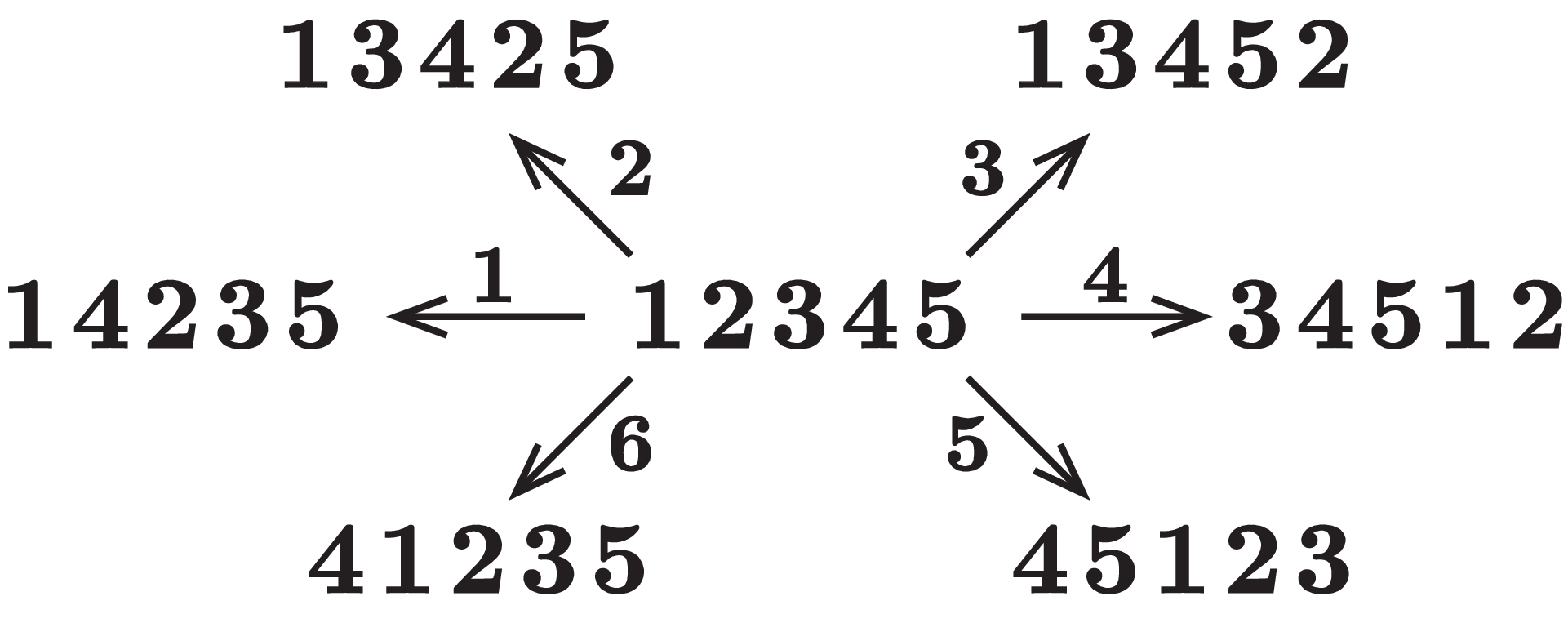}

\vspace{-0.2cm}
\noindent{deg=7 }: 

\vspace{0.3cm}
\hspace{-0.80cm}
\includegraphics[width=0.45 \textwidth]{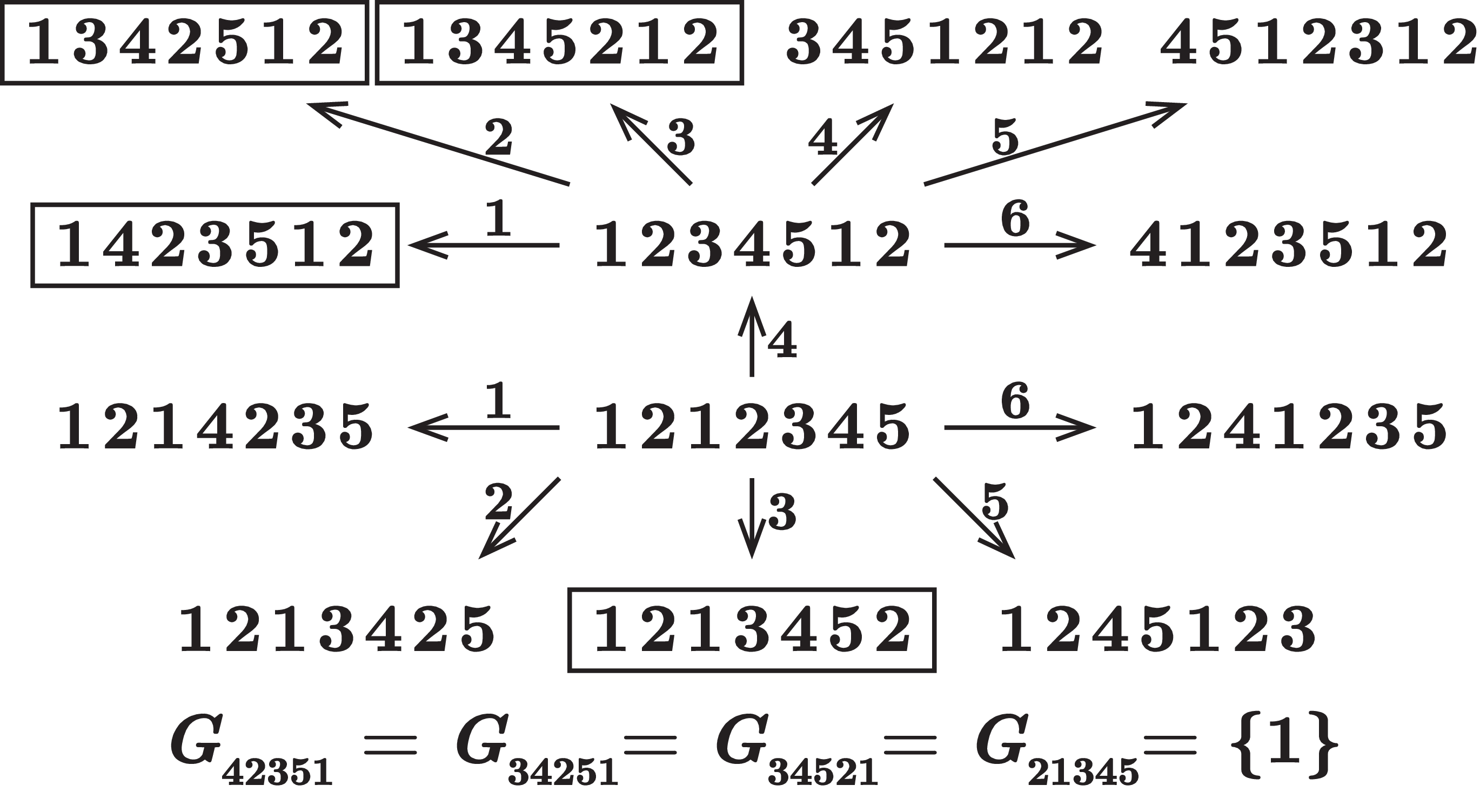}

\vspace{-1.3cm}
\hspace{4.3cm}
\includegraphics[width=0.45 \textwidth]{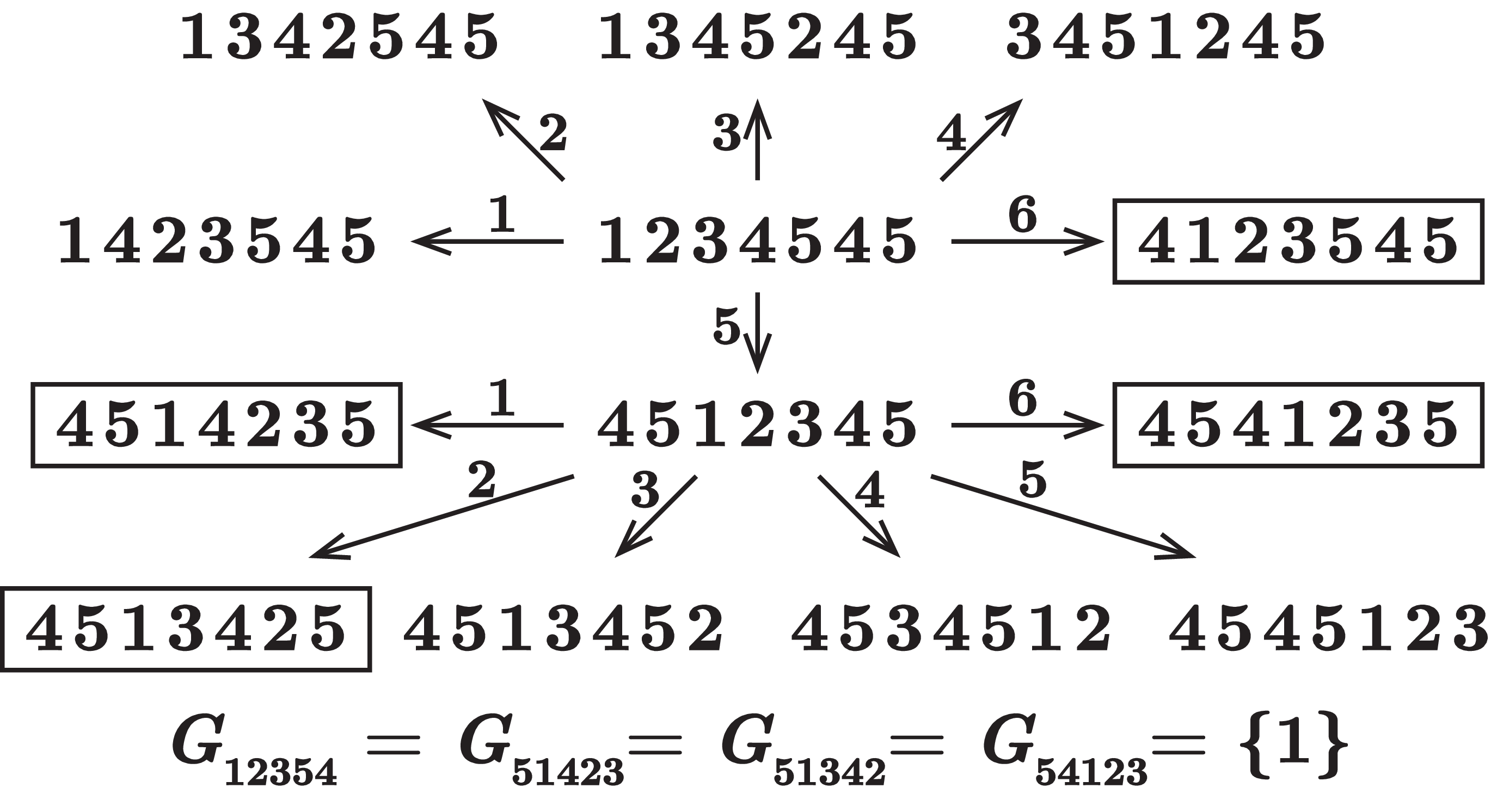}

\noindent{deg=9 }:

\vspace{0.cm}
\hspace{1.3cm}
\includegraphics[width=0.57 \textwidth]{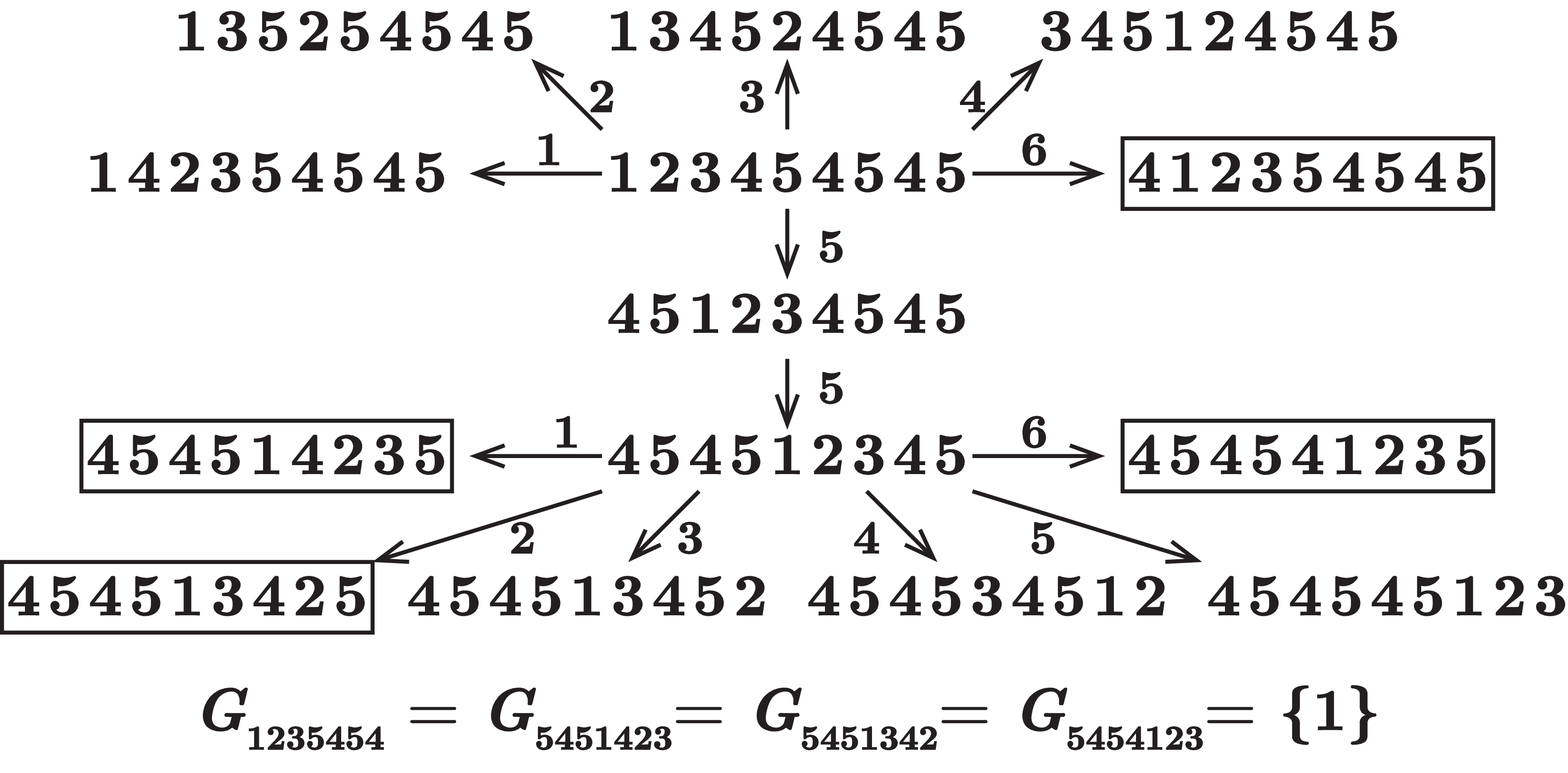}

\vspace{0.4cm}
\hspace{1.3cm}
\includegraphics[width=0.57 \textwidth]{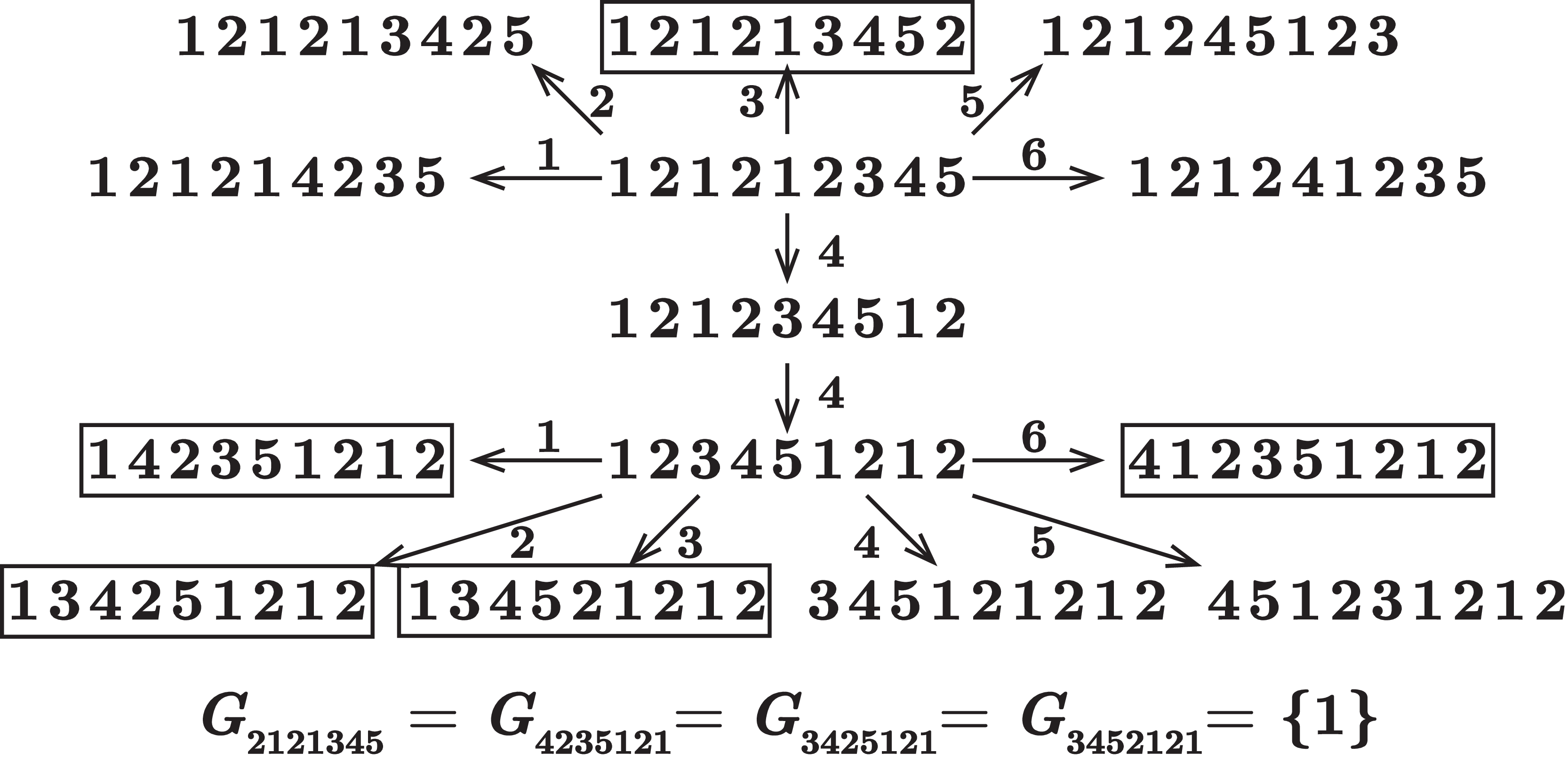}

\smallskip
    \centerline{\quad Figure 6.8:  \ Components of graph $G(L^*,\mathcal{R}_m(\mathcal{A}_{A_3}))$.}

\subsection {\normalsize  Discussions  on  real line arrangements}

\smallskip
The two examples (a) and (b) of Yoshinaga's presentation \eqref{eq:linesPositiveHomogeneous} have  led to  the following two (partly conjectural) observations of the ``resonances"  between geometry (i.e.\ second homotopy classes of $M(\mathcal{A})$) and  algebra (i.e.\ NC-twins of the monoid $A^+(L,\mathcal{R}_m(\mathcal{A}))$).\!\!\!\!

\medskip
\noindent
 (a):  \ \  Existence of the Hattori class in  $\pi_2(M(\mathcal{A}_3),*)$

 $\leftrightarrow$ \ \
 Existence of {\it particular}  NC-twins $\tau_{ij}$   in $A^+(L,\mathcal{R}_m(\mathcal{A}_3))$ 
 
 \qquad which give rise to the Hattori class  
 {(here, ``particular'' means 
 
 \qquad reduced, minimal and, possibly,  some indecomposablility).}

\smallskip
\noindent
 (b): \  Vanishing  of $\pi_2(M(\mathcal{A}_{\rm{A}_3}),*)$
 
  $\leftrightarrow$ \ \ Non-existence of a (particular) NC-twin  in  $A^+(L,\mathcal{R}_m(\mathcal{A}_{\rm{A}_3}))$ in 
  
  \qquad spite of the existence of NC-tuples.\!\!

\smallskip
Can we find a mechanism for how this geometry of arrangements can create or cannot create the algebraic NC-twins? \footnote
{As a naive first question, we may ask:
for which pair of normal crossing lines  do we get NC-twins producing non-trivial $\pi_2$-classes. 
It is well known that any line arrangement $\mathcal{A}$ (except for central arrangements, including the case of the center at infinity) 
 always a nodal point where a pair of lines, say $H_i,H_j\in \mathcal{A}$, are crossing, but no other lines are passing, where one gets a commutativity relation between two (local) loops turning around the crossing lines. 
Since Yoshinaga's fundamental relations \eqref{eq:linesPositiveHomogeneous} are homogeneous of degree $n(:=\!\#$of lines), if $n>2$, in the monoid $A^+(L,\mathcal{R}_m(\mathcal{A}))$, 
the commutativity relations may presumably be realized by an NC-relation \eqref{eq:NCRelation}. When do such NC-relations form twins and when do they not? 
An example of arrangement 
defined by the equation $((x-y)^2)-1)((x+y)^2)-1)=0$ (where $M(\mathcal{A})$ is a $K(\pi,1)$-space). This has 4 normal crossing points, but there does not exist $\pi_2$. }

Here, we note that the above observations are strongly supported by the fact that the components of the graph $G(L^*,\mathcal{R}_m(\mathcal{A}))$  for those NC-tuples are trees, which leads to the following consequences. 

\noindent
(1) We define the concept of {\it weight} (Definition 4.4) for an NC-tuple, which depends only on its kernel, implying that the $\Pi$-class is independent of the geometric  liftings of the fundamental relations (see Appendix II). 

\noindent
(2) The division inertia groups associated with these NC-tuples are trivial; hence the corresponding global inertia groups are trivial and the \(\Pi\)-classes are singletons.

\noindent
 
However, the graph component for an NC-tuple need no longer be a tree when the length of the NC-tuple is larger than $2\#L$.  Nevertheless, we show that (1) holds for any line arrangement in Appendix II. In the following Q1, we conjecture  (2) holds for any line arrangement,  and  ask further  Q2-Q3. 

\medskip
\noindent
Q1.  For any real affine line arrangement $\mathcal{A}$, show that the division inertia group for any NC-tuple of the monoid $A^+(L,\mathcal{R}_m(\mathcal{A}))$  is trivial so that the homotopy class $\Pi(\tau)$ for any NC-twin $\tau$ is a singleton.

\medskip
\noindent
Q2. For any real affine line arrangement $\mathcal{A}$ which is simplicial (
see \cite{Gr,Gr2}), show that there does not exist a particular  NC-twin of the monoid $A^+(L,\mathcal{R}_m(\mathcal{A}))$.  More generally,  does there not exist any NC-twin?

\medskip
\noindent
Q3. Consider a real affine arrangement $\mathcal{A}_n$  of $n$ generic lines ($n\in \mathbb{Z}_{\ge3}$). Fix a 1:1 correspondence between  the set of 1-cells of Hattori's homotopy type $T_2^n$ of $M(\mathcal{A}_n)$ (see \cite{H}) and Yoshinaga's generator system \eqref{eq:linesPositiveHomogeneous}. 
Fix any geometric lifting in $T_2^n$ of Yoshinaga's fundamental relations.

\smallskip
\noindent
(i)  Describe the correspondence:
\vspace{-0.2cm}
$$\begin{array}{cccl}
&\tau \!&\!\in\!&\!\! \{\text{particular  NC-twins of } A^+(L,\mathcal{R}_m(\mathcal{A}_n))\} \\
\mapsto &\Pi(\tau) \!&\! \in \!&\!\! \{\text{Hattori's second homotopy classes of } T_2^n\}. 
\end{array}
\vspace{-0.2cm}
$$ 
(ii) Characterize the $\mathbb{Z}[A(L,\mathcal{R}(\mathcal{A}_n))]$-submodule of  $\pi_2(M(\mathcal{A}_n),*)$  generated by the images.
 Does it cover the whole  group $\pi_2(M(\mathcal{A}_n),*)$?

\noindent
(iii) Describe the  $\mathbb{Z}[A(L,\mathcal{R}(\mathcal{A}_n))]$-linear dependence relations  among $\{ \Pi(\tau)\mid \tau: particular\} $ up to trivial relations among twins over the same kernel.\!\!\!\!

\medskip
\noindent
 


\noindent

 \section{ Appendix I: 
 Second homotopy classes in classifying spaces of monoids}

 
 We construct a canonical second homotopy class in  the classifying space $BA^+$ for a monoid $A^+$ having an NC-twin. The construction does not refer to a system  of  generators and a system of  fundamental relations of the monoid.  
 We compare the class in $BA^+({L},\mathcal{R}_m)$ with the $\Pi$-class constructed  in  a semi-positively presented space equipped with a geometric lifting   $(W,\tilde{L},\mathcal{R}_g)$. The comparison (conjecturally) leads to an understanding of the meaning and the role of different choices of geometric liftings of the fundamental relations and the meaning of the inertia group.

\medskip
Let $A^+$ be any discrete monoid. Let us recall briefly a description of the classifying space $BA^+$ as a simplicial set (see, for instance, \cite{Len, Mc, Se}). 

\begin{itemize}
\item $BA^+$ has a single 0-simplex, called $*$. 
\item For each element $m\in A^+$, there exists an oriented 1-simplex in $BA^+$ called $(m)$ which starts and ends at $*$.
\item For any oriented tuple of elements $m_1,\cdots,m_n \in A^+$ ($n\in\mathbb{Z}_{>1}$), there exists an oriented n-simplex called $(m_1,\cdots,m_n)$ in $BA^+$  such that its 
boundary faces patch with 
$(n\!-\!1)$-simplices $(m_2,\cdots,m_n)$, $(m_1,\cdots,m_km_{k\!+\!1},\cdots,m_n)$ for $k\!=\!1,\cdots,n\!-\!1$, and $(m_1,\cdots,m_{n\!-\!1})$. 
\item If the oriented tuple contains units 1, then it is identified with a shorter tuple obtained by removing the units.
\end{itemize}

We note that even if an element $m\in A^+$ may not be invertible in the monoid $A^+$, the associated loop $(m)$ has the inverse loop $(m)^{-1}$  in $BA^+$ which we shall use systematically.\footnote
{This fact fits with the construction of $\Pi$-classes  in the present paper where we  introduce division morphism in \S5 using the inverse loops for the construction of $\Pi$-classes. In  the present Appendix, we give a construction of second homotopy classes parallel to that in \S5.
}
It is well known that the fundamental group of $BA^+$ is naturally isomorphic to the localization $A$ of the monoid $A^+$.  The natural localization morphism $A^+\to A$  induces a  natural map $BA^+ \to BA$, where $BA$ is well known to be an Eilenberg-Mac Lane space. Thus, our interest is to construct second homotopy classes in $BA^+$  from the algebraic data of the monoid $A^+$.

\bigskip
\noindent
{\bf I.   \normalsize  Relative $\pi_2$-classes of $BA^+$ associated with  NC-tuples.}   

\smallskip
We turn our attention to monoids $A^+$ which are not cancellative. 

\smallskip
\noindent
Recall \S4.1 that a tuple $(a,b,c,d) \in (A^+)^4$  is called an NC-tuple over the kernel $(b,c)\in (A^+)^2$ if it satisfies a non-cancellative equivalence relation \eqref{eq:NCRelation}:
$
abd \ \sim \ acd \ 
$
but $b\not{\!\!\sim}\ c$. 
%
Associated with the LHS of the relation  \eqref{eq:NCRelation}, let us consider the 3-simplex $(a, b,d)$ in $BA^+$  (see Figure 7.1). The existence of the simplex implies that the loop $(b)$ is homotopic to the concatenation of three loops  $(a)^{-1}$, $(abd)$ and $(d)^{-1}$. Let us denote  the homotopy equivalence from $(b)$ to $(a)^{-1}\cdot (abd)\cdot (d)^{-1}$ by 
$$
(b) \quad \underset{abd}{\sim} \quad (a)^{-1}\cdot (abd)\cdot (d)^{-1} 
$$
Note that, not only does the homotopy equivalence between the loops $(b)$ and $ (a)^{-1}\cdot (abd)\cdot (d)^{-1}$ exist, but we have also chosen a particular relative 2-homotopy class given by the  3-simplex $(a,b,d)$ in $BA^+$.

\vspace{0.2cm}
\hspace{3.2cm} 
\includegraphics[width=0.27\textwidth]{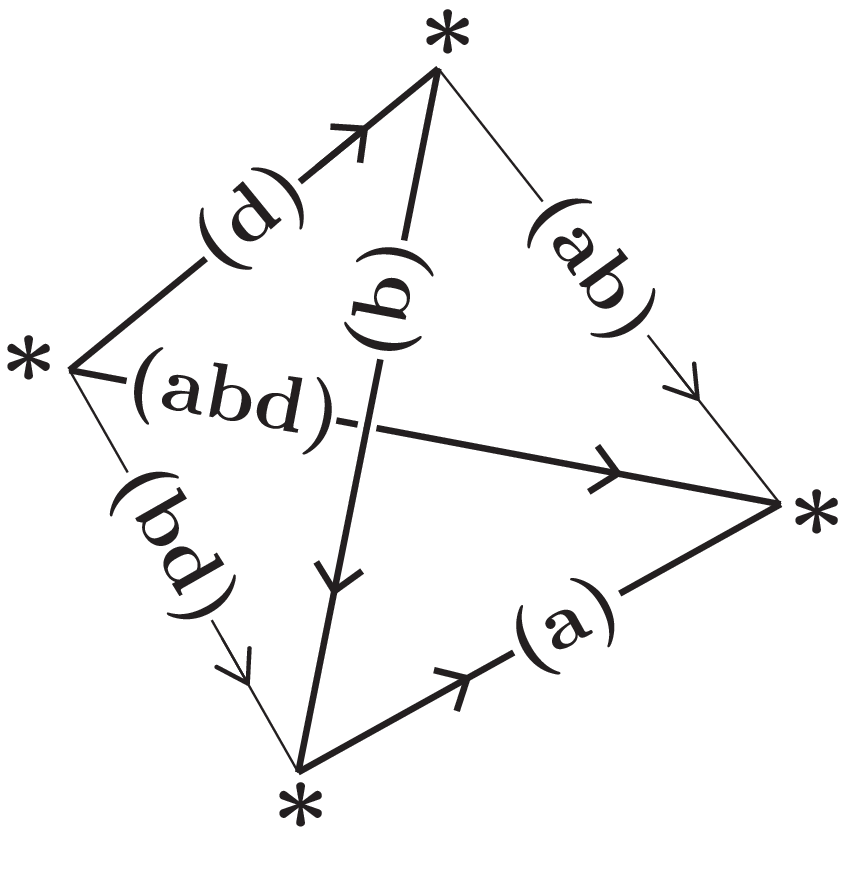}

\vspace{-0.1cm}
\centerline{ Figure 7.1: \   The 3-simplex $(a,b,d)$ and its edges in $BA^+$}

\medskip
Parallel to the above construction, let us consider the 3-simplex $(a,c,d)$ associated with the RHS of the relation  \eqref{eq:NCRelation}.  So, we obtain the other homotopy equivalence from $(c)$ to $(a)^{-1}\cdot (acd)\cdot (d)^{-1}$:
$$
(c) \quad \underset{acd}{\sim}  \quad (a)^{-1}\cdot (acd)\cdot (d)^{-1} .
$$
We note that $abd\sim acd$ in $A^+$ implies that the loops $(abd)$ and $(acd)$ are the same loop in $BA^+$. That is, $(a)^{-1}\cdot (abd)\cdot (d)^{-1} = (a)^{-1}\cdot (acd)\cdot (d)^{-1}$. Concatenating the equivalence $\underset{abd}{\sim}$ with the inverse equivalence $\underset{1/(acd)}{\sim}$ of $\underset{acd}{\sim}$,\footnote
{Here we used notation $1/(acd)$ for the inverse homotopy map since it is not the inverse w.r.t.\ the monoid product but w.r.t.\ the product $\cdot_2$ (cf.\ Notational Remark in \S3.1).
}
 we obtain a homotopy equivalence between $(b)$ and $(c)$:
\begin{equation}
\label{eq:classify}
\qquad\quad  \{ (b)\! \underset{abd/(acd)}{\sim\ \ \ } \!\! (c)\}  \quad \in \ \ \pi_2(BA^+, (b)\vee (c),*)\  .
\end{equation}
That is, associated with the non-cancellative relation \eqref{eq:NCRelation}, we obtain 
a particular homotopy equivalence  $(\underset{abd/(acd)}{\sim})$ as an element of the second  homotopy group relative to the loop $(b)\vee(c)$.

\bigskip
\noindent
{\it Note} 14.  If either $d$ (resp.~$a$) is equal to 1 in \eqref{eq:NCRelation}, the homotopy equivalence \eqref{eq:classify} is given much more simply by 
 the relation obtained by concatenating two 2-simplices: $(a,b)$ and $1/(a,c)$ along loops $(a)$ and $(ab)=(ac)$  (resp.~two 2-simplices: $(b,d)$ and $1/(c,d)$ along loops $(bd)=(cd)$ and $(d)$).\!\!\!  

\medskip
\noindent
{\bf II.   \normalsize  $\pi_2$-classes of $BA^+$ associated with NC-twins.}   

\smallskip
\noindent
Finally, let an NC-twin $\tau:= \big( (a_1,b,c,d_1),(a_2,b,c,d_2)\big)$ of NC-tuples over  the same kernel $(b,c)\in (A^+)^2\setminus \Delta(A^+)$ be given (Definition \ref{NCtuple}). 
Then due to the same kernel, the associated two relative classes have the same boundary whose difference becomes trivial. Since \(\pi_2((b)\vee(c),*)=0\), the homotopy exact sequence identifies this difference with a unique element of \(\pi_2(BA^+,*)\).
Thus, the $\Pi$-class (constructed in the classifying space) of the twin $\tau$  is defined by the $ \cdot_2 $-difference of the two relative classes \eqref{eq:classify}. 
\begin{equation}
\label{eq:pi22}
\begin{array}{rcl}
\vspace{0.1cm}
\Pi(BA^+; \tau) :=(\underset{a_1bd_1/(a_1cd_1)}{\sim}) \cdot_2 (\underset{a_2cd_2/(a_2bd_2)}{\sim})\in   \pi_2(BA^+,*) . \!\!\!\!\!\!\!\!\!\!
\end{array}
\end{equation}
{\it Note} 15. In the notation  \eqref{eq:pi22} for $\Pi$-class in $\pi_2(BA^+,*)$,  we add ``$BA^+$'' in order to distinguish it from the notation for the $\Pi$-class  $\Pi(\tau)$ \eqref{eq:FinalPi2} which is a subset of $\pi_2(X,*)$.

\medskip
\noindent
{\bf III. Comparisons with geometry.}

\noindent
In present {\bf III.} and next {\bf IV.}, we come back to the geometry of a topological space $X$ equipped with a semi-positive presentation $(L,\mathcal{R})$, and compare the $\Pi$-classes in $X$ and in the classifying space. 

Recall the setting in \S3, where $(W,\tilde{L},\mathcal{R}_g)$ is a geometric  lifting of a semi-positive presentation $(L,\mathcal{R})$ of the fundamental group of $X$ (recall Definitions  \ref{W} and \ref{GeoFundRel}), where $A(L,\mathcal{R})$ and $A^+(L,\mathcal{R}_m)$ denote  the fundamental group of $X$ and its monoid presentation, respectively. 

\smallskip
\noindent
{\bf Observation.}  Let \(m\in A^+(L,R_m)\), and let
$ \widetilde m=\ell_1\cdots\ell_r\in L^* $
be a word representing \(m\). The simplex \((\ell_1,\ldots,\ell_r)\) determines a distinguished relative homotopy class between the edge path
$ (\ell_1)\cdots(\ell_r) $
and the single edge \((m)\), where the relative homotopy class $\pi_2(BA^+(L,\mathcal{R}_m),(m)\cup(l_1)\cup\cdots\cup(l_r) ,*)$ is uniquely determined due to the contractibility of the simplex. We shall call this class  the {\it canonical class associated with the word} \(\widetilde m\) with the evident interpretation when $r=0$.

\smallskip
We first use the geometric lifting $\widetilde{L}$ of the generator system $L$ (Definition \ref{W}). 
The lifting $\widetilde{L}$ determines, up to based homotopy, a map 
$$
W \ \longrightarrow \ BA^+(L,\mathcal{R}_m),
$$
sending each generator loop $\tilde\ell$ to the corresponding loop $(\ell)$ in $BA^+(L,\mathcal{R}_m)$. Fixing one, there is a natural commutative  diagram up to based homotopy:\!\!\! 
\begin{equation}
\label{eq:ClassifyingFibration}
\begin{array}{cccccc}
&&&W && \\
&&\!\!\! \swarrow \!\!\!\!\!\!\!\!\!\!\! & &\!\!\!\!\!\!\!\!\!\!\! \searrow \!\!\! & \\
& BA^+(L,\mathcal{R}_m) \!\!\!\!\!\!\!\!\!\!& &
& &\!\!\!\!\! X \\
\vspace{-0.4cm}
\\
&&\!\!\! \searrow \!\!\!\!\!\!\!\!\!\!\! & &\!\!\!\!\!\!\!\!\!\!\! \swarrow \!\!\! & \\
\vspace{-0.3cm}
\\
&&&\!\!\!\! B(A(L,\mathcal{R}))\!\!\!\! &&
\end{array}
\vspace{-0.3cm}
\end{equation}
where 

\noindent
(1) the northeast arrow $\searrow$ is the natural embedding $W\to X$,


\noindent
(2)  the southwest arrow $\searrow$ is the natural map induced by the localization morphism $A^+(L,\mathcal{R}_m)\to A(L,\mathcal{R})$, 

\noindent
(3)  the southeast arrow $\swarrow$ is the first step for constructing the Postnikov tower for $X$, i.e.\ the classifying map.  

\smallskip
\noindent
{\it Proof of the commutativity of} \eqref{eq:ClassifyingFibration}. Consider the two maps from $W$ to $B(A(L,\mathcal{R}))$ turning the left side and the right side of the diagram \eqref{eq:ClassifyingFibration}. Then, both of them, over the free group $\pi_1(W,*) =\mathcal{F}$ \eqref{eq:FreeGroup}, induce the correspondence $\tilde{\ell} \mapsto l\in A(L,\mathcal{R})$, where the target $B(A(L,\mathcal{R}))$ is an Eilenberg-Mac Lane space. Thus, the two maps are based homotopic. $\Box$

\bigskip


\noindent
{\bf Question 1.}  We ask when the following set is non-empty? 
\begin{equation}
\label{eq:MapPsi}\begin{array}{ccl}
\mathcal{M}(\widetilde{L}) \!\!& \!\!:=\!\! &\!\! \{ (\Psi,K) \mid  \Psi: BA^+(L,\mathcal{R}_m)  \rightarrow  X \text{ is a continuous map,}\!\!\!\! \\
&&\qquad\qquad  \text{$K$ is a based homotopy making the upper}\!\!\! \\
&&\qquad\qquad   \text{triangle  of \eqref{eq:ClassifyingFibration} commute} \}
\end{array}
\end{equation} 
For simplicity, we suppress \(K\) from the notation and write simply \(\Psi\) for the pair \((\Psi,K)\).

For a given $\Psi\in \mathcal{M}(\widetilde{L})$,  the lower triangle of the diagram \eqref{eq:ClassifyingFibration} becomes automatically  commutative, since $\pi_1(W,*)\to \pi_1(BA^+(L,\mathcal{R}_m),*)=A(L,\mathcal{R})$ is  surjective and hence the commutativity of the upper triangle implies 
$\Psi_*: \pi_1(BA^+(L,\mathcal{R}_m),*) \overset{\sim}{\longrightarrow} \pi_1(X,*)$, 
and, under the identifications of both groups with \(A(L,R)\), \(\Psi_*\) is the identity. Hence, the two  classifying maps in the lower triangle are identified. 

Consequently, $\Psi$ induces further a $\mathbb{Z}[A(L,\mathcal{R})]$-module homomorphism
$$
\Psi_{*}\ : \  \pi_2(BA^+(L,\mathcal{R}_m),*) \ \longrightarrow \ \pi_2(X,*)
$$

\smallskip
We next ask about the geometric lifting in $X$ of fundamental relations (Definition \ref{GeoFundRel}). 
For a map $\Psi \in \mathcal{M}(\widetilde{L})$, we associate a singleton geometric lifting of the fundamental relations $\mathcal{R}$ as follows:
\begin{equation}
\label{eq:ClassifyLifting}
\mathcal{R}_g^\Psi(p_\mu,q_\mu)\  :=\  \{\Psi_*(p_\mu\simeq q_\mu)\} \quad (\mu \in M).
\end{equation}
Here \(p_\mu\simeq q_\mu\) denotes the canonical relative homotopy class in \(BA^+(L,R_m)\); recall {\bf Observation} at the beginning of present {\bf III}). 
In defining \(\Psi_*(p_\mu\simeq q_\mu)\), we use the fixed homotopy \(K\) to identify the boundary word paths with the prescribed loops in \(W\).

Then for a given  geometric lifting $\mathcal{R}_g$, we consider the set
\begin{equation}
\mathcal{M}(\widetilde{L},\mathcal{R}_g)\ :=\ \{\Psi\in \mathcal{M}(\widetilde{L}) \mid \mathcal{R}_g^\Psi\subset \mathcal{R}_g \}.
\end{equation}

\medskip
\noindent
{\bf Question 2.}  (1)  When is the set $\mathcal{M}(\widetilde{L},\mathcal{R}_g)$ non-empty?

Suppose there exists   $\Psi\in \mathcal M(\widetilde{L},R_g)$.  Let $T=(a,b,c,d)$ be an NC-tuple over the kernel $(b,c)$ of the monoid  $A^+(L,\mathcal{R}_m)$,  and let $\widetilde{T}=(\tilde{a},\tilde{b},\tilde{c},\tilde{d})$  be its lifting,
where we recall that, by Corollary 5.4, the division $\mathcal{R}_g$-set is independent of the choices of $\tilde{a}$ and $\tilde{d}$.

Due to the {\bf Observation}  at the beginning of present {\bf III.}, we have homotopy equivalences between 
$\text{the word path corresponding to } \tilde{b}$ and $(b)$ and between $\text{the word path corresponding to } \tilde{c}$ and $(c)$ in $BA^+(L,\mathcal{R}_m)$.  Then, combining these homotopies with the fixed homotopy $K$, we obtain the equivalences  $\Psi((b))\simeq \tilde{b}, \ \Psi((c))\simeq \tilde{c}$ to the loops in $W$. 

\noindent
(2) Is the relative homotopy class \eqref{eq:classify} mapped into the embedded division $\mathcal{R}_g$-set \eqref{eq:DivisionRgset} by the morphism $\Psi_{*}^{rel}: \pi_2(BA^+, (b)\vee (c),*) \to \pi_2(X,W,*)$?  
\begin{equation}
\label{eq:RelativePi-class}
\Psi^{rel}_{*}(  \{ (b)\! \underset{abd/(acd)}{\sim\ \ \ } \!\! (c)\} )\  \in \ \chi\big(\Pi(\widetilde{T})\big)
\end{equation}

\noindent
Remark that if there exists one $\Psi$ satisfying \eqref{eq:RelativePi-class}, then,  for any 
$\Psi'\in \mathcal M(\widetilde{L},R_g)$, we have  the equivalence: 
\begin{equation}
\label{eq:ClassifyingLocalInertia}
\begin{array}{cl}
& \Psi^{' rel} \text{  satisfies \eqref{eq:RelativePi-class} }  \\
 \Leftrightarrow & \Psi_{*}^{rel}(  \{ (b)\! \underset{abd/(acd)}{\sim\ \ \ } \!\! (c)\} )\cdot_2( \Psi_{*}^{'rel}(  \{ (b)\! \underset{abd/(acd)}{\sim\ \ \ } \!\! (c)\} ))^{-1} \  \in \  \chi \big(G(\widetilde{T}, \tilde b)\big).
 \end{array}
 \end{equation}
Then, running all $\Psi'$ satisfying the LHS of \eqref{eq:ClassifyingLocalInertia}, what subgroup of the division inertia group $\chi \big(G(\widetilde{T}, \tilde b)\big)$ is generated by the  elements on the right-hand side of \eqref{eq:ClassifyingLocalInertia}?

\medskip
\noindent
{\bf Question 3.}  Let a map  $\Psi\in \mathcal M(\widetilde{L},R_g)$ and an NC-twin $\tau$ 
of the monoid  $A^+(L,\mathcal{R}_m)$ be given. 
Assuming an affirmative answer to Question 2 (2),\!\!\!\! 
\begin{equation}
\label{eq:Pi-class}
\Psi_{*}(\Pi(BA^+; \tau))  \  \in \ \Pi(\tau)
\end{equation}
follows. Indeed, we apply Question 2(2) to the two NC-tuples of \(\tau\), using the same liftings \(\widetilde b,\widetilde c\) of their common kernel. We then ask how the classes obtained by varying \(\Psi\in\mathcal M(\widetilde L,R_g)\) fill the coset \(\Pi(\tau)\), or equivalently, what subgroup of \(G(\tau)\) is generated by their differences.

Remark that if there exists one $\Psi$ satisfying \eqref{eq:Pi-class}, then,  for any 
$\Psi'\in \mathcal M(\widetilde{L},R_g)$, we have  the equivalence: 
\begin{equation}
\label{eq:ClassifyingInertia}
\begin{array}{cl}
& \Psi' \text{  satisfies \eqref{eq:Pi-class} }  \\
 \Leftrightarrow & \Psi_{*}( \Pi(BA^+; \tau) )- \Psi'_{*}( \Pi(BA^+; \tau)) \  \in \ G(\tau).
 \end{array}
 \end{equation}
Then, running all $\Psi'$ satisfying the LHS of \eqref{eq:ClassifyingInertia}, what subgroup of the inertia group $G(\tau)$  is generated by the elements on  the right-hand side of \eqref{eq:ClassifyingInertia}?

\bigskip
\noindent
{\bf IV. Reverse comparison.} 

There is also a comparison problem in the opposite direction. Does there exist a map
$$
 \Phi:X\longrightarrow BA^+(L,R_m)
  $$
such that the composite
$$ 
X\xrightarrow{\Phi}BA^+(L,R_m)\longrightarrow B(A(L,R)) 
$$
is based homotopic to the classifying map \(X\to B(A(L,R))\)?

For an NC-twin \(\tau\), we may then ask whether
\begin{equation}
\label{eq:OppositDirection}
\Phi_*(\Pi(\tau)) = \{\Pi(BA^+;\tau)\}.
 \end{equation}

Since \(\Pi(\tau)\) is a single coset of \(G(\tau)\), an affirmative answer implies
$$
 \Phi_*(G(\tau))=0. 
 $$

Conversely, if \(\Phi_*\) sends one element of \(\Pi(\tau)\) to
\(\Pi(BA^+;\tau)\) and annihilates \(G(\tau)\), then \eqref{eq:OppositDirection}  holds.
Thus, the reverse comparison problem is naturally related to the trivialization of the inertia group.

 \section{ Appendix II: Independence of Yoshinaga's fundamental relations}
 
 \vspace{-0.1cm}
 \centerline{Kyoji Saito,  Zitian Xia, Masahiko Yoshinaga}

\bigskip
 In the present Appendix II, the balancedness of NC-twins, shown in Corollary 4.5 for the particular case of Example (a), is shown for all line arrangements in Corollary  8.3, by showing certain homological linear independence among  Yoshinaga's fundamental relations \eqref{eq:linesPositiveHomogeneous}. 

Let $n\in \mathbb{Z}_{\ge0}$ be given.  
Recall the set of  inversions for  a given permutation $\sigma\in \mathfrak{S}_n$ (cf.\ \cite{St} 1.3.13)
\begin{equation}
\mathrm{Inv}(\sigma):=\{ (i,j)\mid 1\le i< j\le n,\ \sigma(i)>\sigma(j)\}.
\end{equation}
Consider the characteristic function of the set $\mathrm{Inv}(\sigma)$:
\begin{equation}
\chi(\sigma):=\sum_{(i,j)\in \mathrm{Inv}(\sigma)} (i,j)  
\end{equation}
  in the free module $\underset{1\le i< j\le n}{\sum}\!\! \mathbb{Z} \cdot (i,j)$ generated by the symbols $(i,j)$.

\begin{prop}
\label{YoshiIndependence}
Consider a real line arrangement $\mathcal{A}$ equipped with a flag $\mathcal{F}$ and the associated system of  fundamental relations \eqref{eq:linesPositiveHomogeneous}. Then the  set of characteristic functions of the  permutations $i(C)$ for $C\in  {\mathrm ch}_2^\mathcal{F}(\mathcal{A})$ (recall Theorem 2.5)
\begin{equation}
\label{eq:YoshiPermutation}
\qquad \quad \chi(i(C))    \qquad  for \quad C\in {\mathrm ch}_2^\mathcal{F}(\mathcal{A})
\end{equation}
  are linearly independent and generate a direct summand of  $\underset{1\le i< j\le n}{\sum}\!\! \mathbb{Z} \cdot (i,j)$.\!\!\! 
\end{prop}
\begin{proof} Let $\al_1,\cdots,\al_n$ be the system of linear forms defining the arrangement $\mathcal{A}$. Consider the morphism
$$
\al=(\al_1,\cdots,\al_n)\ :\ M(\mathcal{A}) \ \rightarrow \ (\mathbb{C}^\times)^n.
$$
Then, the associated morphism  $\al_*\! :\! \mathrm{H}_2(M(\mathcal{A}),\mathbb{Z}) \rightarrow \mathrm{H}_2(\mathbb{C}^\times)^n,\mathbb{Z})$ between their homology groups is a split injection since the homology groups involved are free abelian and the dual $\al^*\!:\!  \mathrm{H}^2(\mathbb{C}^\times)^n,\mathbb{Z}) \to \mathrm{H}^2(M(\mathcal{A}),\mathbb{Z})$ is a surjection (Arnold, Brieskorn).  We identify the modules:\!\!\!\!
$$
 \underset{1\le i< j\le n}{\sum}\!\! \mathbb{Z} \cdot (i,j) \ \simeq \ \mathrm{H}_2( (\mathbb{C}^\times)^n,\mathbb{Z})\simeq \overset{2}{\wedge} \mathbb{Z}^n ,  \quad (i,j) \mapsto e_i\wedge e_j
$$
where $e_1,\cdots,e_n$ is the standard basis of $\mathbb{Z}^n$. 

For $C\in {\mathrm ch}_2^\mathcal{F}(\mathcal{A})$, let $\sigma_C$ denote the associated oriented minimal two-cell (see Yoshinaga \cite{Y1} \S5.2, Ito-Yoshinaga \cite{I-Y} \S4.2). Then,  the associated classes $[\sigma_C]$ form an integral basis of $\mathrm{H}_2(M(\mathcal{A}),\mathbb{Z})$ (\cite{Y1}\S5.2, \cite{I-Y} Theorem 4.1).  Finally, let  us show 
\begin{equation}
\label{eq:BasisIdentify}
\al_*[\sigma_C] \ =\ - \chi(i(C))  \qquad (C\in {\mathrm ch}_2^\mathcal{F}(\mathcal{A})).
\end{equation}
Set
$$
P=\gamma_1\cdots\gamma_n  \qquad  Q_C=\gamma_{i_1}(C)\cdots \gamma_{i_n}(C) .
$$
Yoshinaga's positively oriented attaching boundary (\cite{Y2} Theorem 6.5) is conjugate to the chamber meridian $E(C)=P^{-1}Q_C$. If the equality $P=Q_C$ is written as the standard relator $r_C=PQ_{C}^{-1}$, then
$$
r_C\ =\ PE(C)^{-1}P^{-1}.
$$
Both words have trivial abelianization. In $\rm{H}_2((\mathbb{C}^\times)^n,\mathbb{Z})\simeq \rm{H}_2(\mathbb{Z}^n,\mathbb{Z})\simeq\Gamma_2F_n/\Gamma_3F_n\simeq \overset{2}{\wedge}\mathbb{Z}^n$, successive adjacent interchanges give 
$$
[r_C]\ =\ \underset{(i,j)\in \mathrm{Inv}(i(C))}{\sum}e_i\wedge e_j = \ \chi(i(C)).
$$
Conjugation is invisible in this quotient, while inversion changes the sign. Hence$[E(C)]=-[r_C]$. The class-two/Hopf description of the coordinate torus therefore gives the equality \eqref{eq:BasisIdentify}. Since the sign is uniform, the basis and the split injectivity of $\al_*$  prove the assertion that these elements  $\chi(i(C))$ generate a direct summand
\end{proof}

As an important consequence of Proposition \ref{YoshiIndependence}, we show the existence of the concept of {\it relative weight} on the language $L^*$ (Corollary \ref{relativeWeight}) implying the balancedness of twins (Corollary \ref{balanced}).  

Consider  two words $U,V \in L^*$ which are equivalent in the monoid $A^+(L,\mathcal{R}_m(\mathcal{A}))$.  Consider a sequence of elementary transformations which defines the equivalence between $U$ and $V$ (recall the description of the equivalence relation in $A^+(L,\mathcal{R}_m)$ given in \S2.1).  Let us count the (oriented) number of fundamental relations which appear in the sequence and consider their formal sum  in $\sum_{C\in {\mathrm ch}_2^\mathcal{F} (\mathcal{A})} \mathbb{Z} \cdot i(C)$ (cf.\ Definition \ref{Weight}) and call it the weight of the sequence.\footnote
{One should be cautious that the inverse transformation $\overset{1/i(C)}{\sim}$ of the elementary transformation $\overset{i(C)}{\sim}$ in the sequence is counted as $-i(C)$ but not as the inverse permutation $i(C)^{-1}$.}

\medskip
\begin{cor} \label{relativeWeight} The weight of a sequence between two words $U$ and $V$ depends only on $U$ and $V$ but does not depend on a choice of the sequence. 
\end{cor}
\begin{proof}
 Suppose $\#L=n$ and let us consider a map 
\vspace{-0.1cm}
$$
\vspace{-0.2cm}
\chi: L^* \to   \underset{1\le i< j\le n}{\sum}\!\! \mathbb{Z} \cdot (i,j), \quad   U \mapsto \sum_{1\le i<j\le n}  \chi_{ij}(U) \cdot (i,j)
\vspace{-0.3cm}
$$
where
$$
\begin{array}{rcl}
\chi_{ij}(U) &:=& \#\{\text{pair of letters $i$ and $j$ in $U$ where } \\
& & \quad \  (\text{position of } i ) - (\text{position of } j)>0\}
\end{array}
$$

We  consider a sequence of elementary transformations $U=W_0,W_1,\cdots,$ $ W_l=V$ defining the equivalence of two words $U$ and $V$. Since
$$
\chi(W_{i+1}) -\chi(W_i)=
\begin{cases}
\chi(i(C)) \quad\ \  if \ \ \  W_i \overset{i(C)}{\sim} W_{i+1} \\
-\chi(i(C)) \quad if \ \ \ W_i \overset{1/i(C)}{\sim} W_{i+1} \\
\end{cases}
$$

 Let $\sum_{C\in {\mathrm ch}_2^\mathcal{F}(\mathcal{A})} w_C\cdot  (i(C))$  for some $w_C\in \mathbb{Z}$  be the weight of the sequence.  Then we see
$$
\begin{array} {rcl}
\chi(V)-\chi(U) &= & \sum_{i=0}^{l-1} (\chi(W_{i+1}) -\chi(W_i)) \\
&=&  \sum_{C\in {\mathrm ch}_2^\mathcal{F}(\mathcal{A})} w_C\cdot  \chi(i(C)).
\end{array}
$$
Since $\chi(i(C))$ for $C\in {\mathrm ch}_2^\mathcal{F}(\mathcal{A})$ are linearly independent (Proposition \ref{YoshiIndependence}), the coefficients $w_C$ ($C\in  {\mathrm ch}_2^\mathcal{F}(\mathcal{A})$) are uniquely determined by the LHS which obviously depends only on $U$ and $V$ but not on the sequence.
\end{proof}

Let $(\tilde{a},\tilde{b},\tilde{c},\tilde{d}) \in \mathcal{F}^+$ be any lifting of  an NC-tuple $(a,b,c,d)$. 

\begin{cor} 
\label{balanced} The weight of the lifted NC-relation $\tilde{a}\tilde{b}\tilde{d} \sim \tilde{a}\tilde{c}\tilde{d}$ depends only on the lifting $(\tilde{b},\tilde{c})$ of the kernel. Therefore, any lifting of a twin for a fixed lifting of the kernel is balanced.
\end{cor}
\begin{proof}
We have only to show that the weight of an NC-tuple is determined by the lifting of the kernel. 
Since all Yoshinaga relations preserve the multiplicity of each letter, \(\widetilde b\) and \(\widetilde c\) have the same multiplicity vector, and, therefore, $\chi(\tilde{a}\tilde{b}\tilde{d})-\chi(\tilde{a}\tilde{c}\tilde{d})=\chi(\tilde{b})-\chi(\tilde{c})$. 
\end{proof}

\section{Appendix III:  NC-tuples for line arrangements (a) and (b)\!\!\!}

\centerline{Kyoji Saito and Zitian Xia}

\bigskip
In the present Appendix III, we give  further computational data on NC-tuples for the two line arrangements (a) and (b). These data provide further support for the discussion in §6.3 and show a clear contrast between the two associated monoids in the behavior of their NC-tuples.


\medskip
\noindent
{\bf (a)  Reduced and minimal  NC-tuples for $A^+(L, \mathcal{R}_m(\mathcal{A}_3))$.}

\smallskip
We give the complete list of NC-tuples of degree at most 6.  It exhibits several new NC-twins already in degree at most 6. It also shows different behaviors of K- and T-decomposability, suggesting that the present experimental notions of indecomposability may require further refinement.

\medskip
\noindent
{\it Notation.} 
We mean by $NC_d^{rm}(b,c)$ the set of reduced and minimal NC-tuples over the kernel $(b,c)$ of degree at most $d$.  
We put a superscript \(K\) (resp. \(T\)) on a tuple if it is not \(K\)-indecomposable (resp. not \(T\)-indecomposable).

   The symbols \(1,2,3\) 
   denote the three generators $\tilde\gamma_1,\tilde{\gamma}_2,\tilde{\gamma}_3$, whereas \(e\) denotes the unit element.

$$
\begin{array}{cl}
NC_6^{rm}(13,31) &\!\!\! = \{ (e, 2), (123,e), (23,32),
 (32,12)^T \},\\

NC_6^{rm}(23,32) &\!\!\!=\{ (1,e), (e,123), (12,13)^T,  (13,31)\},\\

NC_6^{rm}(12,21)&\!\!\!=\{ (e,32), (13,e), (3,23),   (31,3)\}, \\

NC_6^{rm}(113,311)&\!\!\!=\{ (e,23), (123,e), (32,2)^T \}, \\

NC_6^{rm}(213,321) &\!\!\!=\{ (e,23), (1,2), (31,e) \},\\

NC_6^{rm}(223,322) &\!\!\!=\{ (31,e), (1,13)^T, (e,123) \},\\
\end{array}
$$
$$
\begin{array}{cl}

NC_6^{rm}(233,332) &\!\!\!=\{(1,31), (123,e), (e,123) \},\\

NC_6^{rm}(133,331) &\!\!\!=\{ (123,e), (23,2), (e,123)\},\\

NC_6^{rm}(112,211) &\!\!\!=\{ (13,3), (3,32)^T,  (31,3), (e,323)^T\},\\

 NC_6^{rm}(122,221) &\!\!\!=\{ (13,3)^T, (3,23), (3,32), (313,e)^T\}, \\
 
 NC_6^{rm}(1213,3211) &\!\!\!=\{ (1,2), (e,23)^T\},\\
  
 NC_6^{rm}(2213,3212) &\!\!\!=\{ (1,2), (31,e)^T\},\\
 
 NC_6^{rm}(1223,2321)  &\!\!\!=\{(13,e)^T\},\\
 
 NC_6^{rm}(1213,3112)  &\!\!\!=\{ (3,3) \},\\
 
 NC_6^{rm}(1223,3212) &\!\!\!=\{ (3,3)\},\\
 
  NC_6^{rm}(2113,3211)  &\!\!\!=\{ (31,e), (e,23)^T \},\\
  
    NC_6^{rm}(2213,3221)  &\!\!\!=\{ (31,e)^T, (e,23) \},\\
    
        NC_6^{rm}(2223,3222)  &\!\!\!=\{(31,e) \},\\
      
       NC_6^{rm}(1113,3111)  &\!\!\!=\{ (e,23)\},\\
        
          NC_6^{rm}(1313,3311)  &\!\!\!=\{ (e,23)\},\\
          
            NC_6^{rm}(2131,3112)  &\!\!\!=\{(e,32)^T\}, \\
  
 NC_6^{rm}(1213,2131) &\!\!\!=\{ (13,e)^K, (3,2)^{K,T}, (e,23)^{K,T}\},\\
 
 NC_6^{rm}(1221,2112)  &\!\!\!=\{(3,3)^{K,T}\},\\

 NC_6^{rm}(1313,3131)  &\!\!\!=\{ (e,2)^K \},\\

 NC_6^{rm}(1331,3113) &\!\!\!=\{ (e,2)^K\},\\

 NC_6^{rm}(2233,3232)  &\!\!\!=\{(31,e)\},\\
   
 NC_6^{rm}(2323,3232)  &\!\!\!=\{ (1,e)^K \},\\

 
 NC_6^{rm}(2331,3213) &\!\!\!=\{ (1,e)^K, (e,2)^K\},\\

 NC_6^{rm}(2332,3223)  &\!\!\!=\{(1,e)^K\},\\

 NC_6^{rm}(11313,31131) &\!\!\!=\{ (e,2)^K\},\\
 
 NC_6^{rm}(11331,31113)  &\!\!\!=\{(e,2)^K\},\\
 
  NC_6^{rm}(13113,31131)  &\!\!\!=\{ (e,2)^K \},\\
 
 NC_6^{rm}(13131,31113)  &\!\!\!=\{ (e,2)^K \},\\
 
NC_6^{rm}(13313,31331)  &\!\!\!=\{ (e,2)^K \},\\
 
NC_6^{rm}(13313,33131)  &\!\!\!=\{ (e,2)^K \},\\
 
NC_6^{rm}(13331,31313)  &\!\!\!=\{ (e,2)^K \},\\
 
NC_6^{rm}(13331,33113)  &\!\!\!=\{ (e,2)^K \},\\
 
NC_6^{rm}(21313,32131)  &\!\!\!=\{ (e,2)^{K,T} \},\\
 
NC_6^{rm}(21331,32113)  &\!\!\!=\{ (e,2)^{K,T} \},\\
 
NC_6^{rm}(22331,32213)  &\!\!\!=\{ (e,2)^K \},\\
 
NC_6^{rm}(2321,3212)  &\!\!\!=\{ (1,3)^{K,T},   (31,e)^{K,T},  (e,32)^K \},\\
 
NC_6^{rm}(23213,32321)  &\!\!\!=\{ (1,e)^{K,T} \},\\
 
NC_6^{rm}(23223,32232)  &\!\!\!=\{ (1,e)^K\},\\
 
NC_6^{rm}(23223,32322)  &\!\!\!=\{ (1,e)^K \},\\
 
NC_6^{rm}(23232,32223)  &\!\!\!=\{ (1,e)^K \},\\
 
NC_6^{rm}(23233,32332)  &\!\!\!=\{ (1,e)^K \},\\
 
NC_6^{rm}(23311,32113)  &\!\!\!=\{ (1,e)^K \},\\
 
NC_6^{rm}(23313,32331)  &\!\!\!=\{ (1,e)^K,(e,2)^K \},\\
 
NC_6^{rm}(23321,32213)  &\!\!\!=\{ (1,e)^{K,T} \},\\
 
NC_6^{rm}(23322,32223)  &\!\!\!=\{ (1,e)^K \},\\
 
NC_6^{rm}(23323,32332)  &\!\!\!=\{ (1,e)^K \},\\
 
NC_6^{rm}(23331,32133)  &\!\!\!=\{ (1,e)^K \},\\
 
NC_6^{rm}(23331,33213)  &\!\!\!=\{ (e,2)^K \},\\

NC_6^{rm}(23332,32233)  &\!\!\!=\{ (1,e)^K \},\\

NC_6^{rm}(23332,32323)  &\!\!\!=\{ (1,e)^K \},\\
                     
\end{array}
$$

\medskip
\noindent
{\bf (b)  Reduced and minimal  NC-tuples for $A^+(L, \mathcal{R}_m(\mathcal{A}_{A_3}))$}

\smallskip
We give the numbers\footnote
{
\label{Test}
Each number counts kernel-reversal orbits under $(a,b,c,d)\leftrightarrow(a,c,b,d)$.} 
of reduced and minimal NC-tuples through degree \(17\) in the following table. 

A separate cumulative check of the computed reduced and minimal NC-tuples of degree at most \(11\) found no two distinct tuples with the same ordered kernel \((b,c)\), before imposing either indecomposability condition.

We also examine the two experimental indecomposability conditions of Remark 4.2. For degrees \(10\le d\le17\), the computed sets of \(K\)-indecomposable and \(T\)-indecomposable NC-tuples coincide, not only in cardinality but tuple by tuple. 
This finite-range coincidence between the “outer” \(T\)-indecomposability condition and the “inner” \(K\)-indecomposability condition in Example (b), in contrast with their different behavior in Example (a), suggests that some intrinsic  structure on the associated monoid may underlie this phenomenon; we leave its investigation for future work.
$$
\begin{array} {cccc}
\text{degree $d$}  & \text{reduced and minimal}& \text{ T-indecomposable} & \text{ K-indecomposable} \\
       5   & 11 &  11 &   11\\
       6  & 0 &  0 &   0\\
       7  & 6  & 6&    6 \\
       8  &  0 & 0&  0  \\
       9  &  16 & 16&  16 \\
       10 &105 &  7& 7 \\
              
      11   & 538 & 48&    48 \\
       12  & 2504 & 54&   54 \\
   13   & 12394 &144 &  144 \\
       14  &  61474 & 224& 224\\
       15 &  306149 & 487 &  487 \\
 16 & 1526218 & 848& 848 \\
       17& 7613868 & 1718 & 1718 \\
       
 \end{array}
 $$

 The computations reported in this appendix were carried out by Zitian Xia using custom deterministic Python and C++ programs. For the \(K\)-indecomposability calculation, the complete reduced/minimal universe was checked in degrees \(10\) and \(11\); in degrees \(12\) through \(17\), a proved reduction was used to restrict the computation to a finite domain containing all \(K\)-indecomposable tuples, and this reduced domain was checked exhaustively. The \(T\)-indecomposable column was obtained in an earlier computation with separately implemented checks. The resulting \(K\)- and \(T\)-indecomposable sets coincide object by object in degrees \(10\) through \(17\).  
AI-assisted tools were used in developing and checking the programs and arguments, but the numerical data reported here were obtained from deterministic program runs. No proof assistant was used for the computations reported in this appendix. The code, object lists, proofs of the reductions, certificates, and checking instructions have been retained. 
The computational materials are available from the authors upon request.\!\!\!

\bigskip
 \noindent
{\bf Acknowledgements:} 
 
The author expresses his gratitude to Daisuke Kishimoto for clarifying the homotopical structures studied in the present paper,  to Masahiko Yoshinaga for an introduction to the homogeneous presentations of  fundamental groups of complements of real line arrangements and to Shunsuke Tsuchioka for informing him  about the inversion set in Appendix II.  

The author is grateful to Benoit Guerville-Ball\'e for independently checking with Lean the author's hand calculations for Example (b) up to degree \(9\) at an early stage of the present  work.

The author is also grateful to Zitian Xia for carrying out the extensive computer calculations reported in Appendix III and for developing the programs used in these computations.

He also expresses his gratitude to  Hiro Lee Tanaka, Kenji Iohara, Yoshihisa Saito, Hiroshi Ohta, Atsushi Takahashi, Tadashi Ishibe, Akishi Ikeda, Barbara Baumeister and Georges Neaime for interesting discussions during the preparation of the present work. 

 Special thanks go to  
 Yoshihisa Obayashi for drawing all  figures (more than 30!) in the present paper.

 The author also thanks  RIMS, IPMU, OIST, MPIM, and MFO where he stayed while preparing the present work.

The present work was partially supported by JSPS KAKENHI Grant Numbers 18H01116 and 23H01068.

\end{document}